\documentclass[4pt]{book}
\usepackage{amsmath}
\usepackage{amsthm}
\usepackage{amssymb}
\usepackage{anysize}
\usepackage{helvet}
\newtheorem{theorem}{Theorem}[section]

\newtheorem{lemma}[theorem]{Lemma}
\newtheorem{corollary}[theorem]{Corollary}
\newtheorem{prop}[theorem]{Proposition}
\newtheorem{claim}[theorem]{Claim}

\theoremstyle{definition}

\theoremstyle{remark}
\newtheorem{example}[theorem]{Example}
\newtheorem{remark}[theorem]{Remark}
\newtheorem{exercise}[theorem]{Exercise}
\newtheorem{con}[theorem]{Conjecture}
\usepackage{graphics}
\usepackage[arrow, matrix, curve]{xy}
\usepackage{amssymb}
\usepackage{anysize}
\marginsize{3cm}{3cm}{3cm}{3cm}

\newcommand{\C}{\mathbb{C}}
\newcommand{\R}{\mathbb{R}}
\newcommand{\G}{{\rm{G}}}
\newcommand{\Z}{\mathbb{Z}}

\newcommand{\Q}{\mathbb{Q}}

\newcommand{\A}{{\rm{A}}}

\newcommand{\HHH}{{\rm{H}}}
\newcommand{\V}{{\rm{V}}}
\newcommand{\U}{{\rm{U}}}
\newcommand{\LL}{{{\rm{L}}}}
\newcommand{\scr}{\scriptscriptstyle}
\numberwithin{equation}{section}

\begin{document}
\pagestyle{empty}

\hspace{1cm}\\
\hspace{1cm}\\
\begin{center}

\huge \textbf{Density of rational points on\\ commutative group varieties and\\ small transcendence degree\\
(long version)}
\end{center}
\hspace{1cm}\\
\hspace{1cm}\\
\begin{center}
\large Aleksander Lech Momot\\
ETH Zurich
\end{center}
\hspace{1cm}\\
\hspace{1cm}\\
\section*{Abstract} The purpose of this paper is to combine classical methods from transcendental number theory with the technique of restriction to real scalars.
We develop a conceptual approach relating transcendence properties of algebraic groups to results about the existence of homomorphisms to group varieties over real fields. Our approach gives a new perspective on Mazur's conjecture
on the topology of rational points. We shall reformulate and generalize Mazur's problem in the light of transcendence theory and shall derive conclusions in the direction of the conjecture. Next to these new theoretical insights, the aim of our application motivated \textit{Ansatz} was to improve classical results of transcendence, of algebraic independence in small transcendence degree and of linear independence of algebraic logarithms. Thirty
new corollaries, most of which are generalizations of popular theorems, are stated in the seventh chapter. For example we shall prove: \\
\\
\textit{Let $a_1,a_2, a_3$ be three linearly independent complex numbers, let $\wp(z)$ be a Weierstra\ss\,\,function with algebraic invariants and let $b$ be a non-zero complex numbers. If
the four numbers satisfy certain hypotheses, then one among the six numbers $\wp(a_j)$, $e^{\scr ba_j}$ is transcendental.}\\
\\
\textit{Let $\wp(z)$ be a Weierstra\ss\,\,function with algebraic invariants and complex multiplication by $\sqrt{-d}$ for a square-free integer $d > 1$.
If $\wp(\omega)$ is defined and algebraic, then either $\omega/|\omega|$ is algebraic or $\omega/|\omega|$ and $\wp(i\omega)$ are algebraically independent.}\\
\\
\textit{Let $\wp(z)$ be a Weierstra\ss\,\,function with algebraic invariants and lattice $\Lambda$ and let $\omega$ be a complex number such that $\wp(\omega)$ is defined and algebraic. Then
$r\omega$ is transcendental for each $r > 0$ or $r\omega \in \Lambda$ for some $r > 0$.}\\
\\
This elaborated long version of our work is essentially self-contained and should be admissible for master students specializing in transcendental number theory and arithmetic geometry.

\tableofcontents

\pagestyle{plain}
\small{
\chapter{Introduction} 

\section{Analytic subgroups of commutative algebraic groups and the methods of Schneider, Gel'fond and Baker}
\large{Many results from transcendental number theory can be reformulated as theorems about complex-analytic subgroups of a commutative algebraic
group $\G$ defined over a subfield $F$ of $\C$. In the simplest case of a one-parameter subgroup, which will be the only case treated here, one considers a complex-analytic homomorphism $\Phi: \C \longrightarrow \G(\C)$ which parametrizes the analytic subgroup $\Phi(\C)$ in question. It is typically required that $\Phi$ satisfies arithmetic conditions 
concerning its differential $\Phi_{\ast}$ and the magnitude of the group of rational points $\xi \in \G(F)$ in $\Phi(\C)$. The conclusion is then the existence of a proper algebraic subgroup $\G$ related to $\Phi(\C)$. We illustrate this along three
unifying approaches which will also serve as a major play ground for this paper.
\subsection{Schneider's method} In 1934 Gel'fond \cite{Gel} and Schneider \cite{Schn} independently solved Hilbert's seventh problem. That is, they showed that for algebraic $\alpha \neq 0,1$ and irrational algebraic $\beta$ the number $\alpha^{\beta}$ is transcendental. The proof given in \cite{Schn} can be sketched as follows.\\
Using Siegel's lemma, Schneider constructs a sequence of not identically vanishing exponential polynomials $f_n(z) = {\rm{P}}_n(z, \alpha^{z})$ which take small algebraic values at points $d_1 + \beta d_2 \in \Z + \beta \Z$ such that $|d_1|$ and $|d_2|$ do not exceed a suitably chosen number $r_n > 0$. From Liouville's inequality it follows then in the second step that $f_n$ must vanish
at those points. In the third step Schwarz lemma is applied. It yields that $f_n$ takes small values within a radius $s_n$ which is essentially bigger than $r_n$ and it is deduced that $f_n$ even vanishes along all points $d_1 + \beta d_2 \in \Z + \beta \Z$ with $|d_1|, |d_2| \leq s_n/(1+|\beta|)$. If $n$ is large,
the argument can be repeated for any given radius, and one infers that $f_n$ vanishes at all points $\Z + \beta \Z$. Finally, Nevanlinna theory implies that $f_n$ must vanish identically, and the required contradiction is established.\\
Fifteen years later, in 1949, Schneider \cite{Schn1} conceptualized the approach sketched above and formulated an important
theorem about algebraic dependence of meromorphic functions which, influenced by works of Gel'fond, also involves higher derivatives taking algebraic values. Lang \cite{Lang} extracted from the technically complicated proof the case when no higher derivatives are involved and established a simple and transparent approach to ``Schneider's method''. An important feature
of this method is that it allows strong applications without involving higher derivatives or algebraic zero estimates.\\
In the framework of one-parameter subgroups of algebraic groups Schneider's method yields the following unifying theorem. We denote by
$\G$ a commutative algebraic group variety of dimension $\dim\,\G \geq 2$ over $\overline{\Q}$. We let ${\rm{g}}_a$ (resp.\,${\rm{g}}_m$) be the dimension of the maximal unipotent (resp.\,multiplicative) factor of $\G$, so that
$\G$ is isomorphic to $\G_c \times \mathbb{G}_{a,F}^{{\rm{g}}_a} \times \mathbb{G}_{m, F}^{{\rm{g}}_m}$ with an algebraic group $\G_c$ over $\overline{\Q}$. We consider a one-parameter homomorphism $\Phi: \C \longrightarrow \G(\C)$. The rank of its kernel is denoted by $\rm{k}$ and the rank of the group of algebraic logarithms $\Phi^{\scr-1}\big(\G(\overline{\Q})\big)$ by
$\rm{r}$.
\begin{theorem} \label{1.1.1}\label{Schn0} Assume that
$${\rm{r}}\cdot(\dim\,\G - 1) - 1 \geq {2\dim\,\G - (2{\rm{g}}_a + {\rm{g}}_m) -{\rm{k}}}.$$
Then the image of $\Phi$ is not Zariski-dense in $\G(\C)$.

\end{theorem}
The theorem can be derived from the main result of Schneider \cite{Schn1}. It is also a formal consequence of Waldschmidt \cite[Theorem\,4.1]{Wal17}.
\subsection{Gel'fond's method} 
In 1949 Gel'fond \cite{Gel} proved general results asserting that among certain sets of numbers related by the exponential function at least two are algebraically independent.
The maybe most famous one teaches that $\alpha^{\beta}$ and $\alpha^{\beta^{\scr 2}}$ are algebraically independent for algebraic $\alpha \neq 0,1$ and cubic $\beta$. The essential ingredient
in Gel'fond's proof was a fundamental lemma relating algebraic numbers to small values of sequences of integer polynomials which satisfy certain asymptotic growth conditions. The scheme
of proof is analogous to Schneider's method. Roughly speaking, Liouville's estimate and Schwarz lemma are replaced by more advanced tools. In the third step of the proof Liouville's inequality is replaced by the fundamental lemma and for the final step a result is needed
which bounds the total number of zeros of an exponential polynomial. An important aspect of ``Gel'fond's method for algebraic independence'' is 
that- in contrast to Schneider's \textit{Ansatz}- one cannot avoid to involve higher derivatives without shrinking the range of applications essentially.\\  
Essential innovations of Gel'fond's original approach were established in the 1970's in two ways. Firstly, Lang \cite{LL}, Tijdeman \cite{Tij1} and Brownawell \cite{Brown} published
technically better versions of the fundamental lemma. Secondly, although Gel'fond's original zero estimate was strong enough to yield striking results, it required many technical hypotheses. Tijdeman \cite{Tij} succeeded in proving a zero estimate for exponential polynomials in which any dispensable technical condition is removed from the hypotheses. This lead to the famous theorem of Brownawell \cite{Brown} and Waldschmidt \cite{Wal9} in 1973.
In the end of the 1980's, after the algebraic zero estimates of Masser-W\"ustholz \cite{MW2}, \cite{MW} and Philippon \cite{P7} were available, it was then possible to write
down theorems unifying Gel'fond's method in the framework of one-parameter homomorphisms. This was accomplished in Tubbs \cite{R1} and Tubbs \cite{R2}.\\
The main results of Tubbs \cite{R1} can be restated in the following way. Let $F$ be a subfield of $\C$ with transcendence degree $\leq 1$ over $\Q$ and let $\G$ be a commutative group variety of dimension $\dim\,\G \geq 2$ over $F$ with Lie algebra $\frak{g} = {\rm{Lie}}\,\G$. As above we consider a complex-analytic homomorphism $\Phi: \C \longrightarrow \rm{G}(\C)$. We let $\frak{T}$ be the smallest linear subspace of $\frak{g}$ over $F$ such that $\Phi_{\ast}(\C)$ is contained in $\frak{T} \otimes_F \C$. We denote by $\rm{r}$ the rank of $\Phi^{\scr -1}\big(\G({F})\big)$ and by $\rm{k}$ the rank of the kernel of $\Phi$. As above, ${\rm{g}}_a$ (resp.\,${\rm{g}}_m$) refers to the maximal unipotent (resp.\,multiplicative) factor. Finally
a number $\delta$ is defined to be $1$ if $\dim\,\frak{T} = {\rm{g}}_a = 1$ and $0$ otherwise.
\begin{theorem} \label{5.1.1} Assume that
$$(1+{\rm{k}}){\rm{r}}\cdot(\dim\,{\rm{G}} -\dim\,\frak{T}) - {\rm{r}} \geq 2\dim\,{\rm{G}} - (2{\rm{g}}_a + {\rm{g}}_m) + \delta.$$ 
Then the image of $\Phi$ is not Zariski-dense in $\rm{G}(\C)$.
\end{theorem}
If $\rm{k} \leq 1$, then Theorem \ref{5.1.1} follows directly from the first two results of Tubbs \cite{R1}. If $\rm{k} \geq 2$, then $\Phi(\C)$ is an elliptic curve and, since $\dim\,\G \geq 2$, automatically a proper algebraic subgroup of $\G(\C)$.   
\subsection{Bakers's method} 
Hilbert's seven problem and its solution by Gel'fond \cite{Gel} and Schneider \cite{Schn} from 1934
were also point of departure of Baker's theory. The assertion that $\alpha^{\beta}$ is transcendental for algebraic $\alpha \neq 0, 1$ and irrational algebraic $\beta$ is equivalent to the statement that two non-zero algebraic logarithms $\omega_0, \omega_1$ of the exponential function $e^z$ are linearly independent over $\Z$ if and only if
they are linearly independent over $\overline{\Q}$. This follows by setting $\alpha = e^{\omega_1}$ and $\beta = \omega_1/\omega_0$. Baker \cite{Baker1}-\cite{Baker3} generalized the theorem in 1966/67 to an arbitrary set of linearly independent algebraic logarithms. His results have such deep consequences that a lot of papers were
written on this subject and many academic careers began with research in this field. Here we only quote Masser's paper \cite{Mass}, because it will
play a role in our applications, and refer for a more detailed development to Baker \cite{Baker} and W\"ustholz \cite{Wussi}. As in Gel'fond's \textit{Ansatz}, Baker's method is by constructing auxiliary functions and estimating derivatives, but in a conceptually more ingenious way. One considers
a linear form $\Lambda = \sum_{j = 1}^k \beta_j\omega_j$ with algebraic numbers $\beta_j$ and $\Z$-linearly independent algebraic logarithms $\omega_j$.
Assuming that $\Lambda = 0$, Baker constructs integer polynomials ${\rm{P}}_n$ and auxiliary functions
$$f_n(z_1,...., z_k) = \sum_{d_1=0}^{L_n} ... \sum_{d_k=0}^{L_n}{\rm{P}}_n(d_1,...,d_k) e^{\scriptscriptstyle (d_1+d_k\beta_1)\omega_1 z_1}\cdot...\cdot e^{\scriptscriptstyle (\lambda_{k-1}+d_k\beta_{k-1})\omega_{k-1}z_{k-1}}.$$
with many zeros of high order along points $(l,...,l) \in \Z^k$ within a suitably chosen radius $r_n$. The final contradiction is then achieved by relating these zeros to a non-vanishing Vandermonde matrix. One central feature of ``Baker's method for linear independence'' is that it in fact provides quantitative results. Roughly speaking, the assertion that $\Lambda$ does not vanish is strengthened to the effect that a positive lower bound for $\left|\Lambda\right|$ is given.  
As to the qualitative theory, the finishing touches were put by W\"ustholz in the end of the 1980's. Using his algebraic zero estimates from \cite{Wussi1}, he generalized Baker's qualitative results to arbitrary commutative group varieties and obtained in \cite{Wu44} the Analytic Subgroup Theorem.\\
To state the theorem, let $\G$ be a commutative group variety over $\overline{\Q}$ and denote by $\Phi: \C \longrightarrow \G(\C)$ a complex analytic homomorphism. Similarly as above $\frak{T}$ is the smallest linear subspace of $\frak{g}$ over $\overline{\Q}$ such that
$\Phi_{\ast}(\C)$ is contained in $\frak{T} \otimes_{\overline{\Q}} \C$. 
\begin{theorem} \label{AST} If the image of $\Phi$ contains a non-trivial algebraic point in $\G(\overline{\Q})$, then $\frak{T}  \subset \frak{g}$ equals the Lie algebra of an algebraic subgroup ${{\rm{H}}} \subset \G$. 
\end{theorem}
In order to reveal the link to the two previously stated results, we reformulate the Analytic Subgroup Theorem in a slightly modified way. As above we define $\rm{r}$ to be the rank of the group of algebraic logarithms $\Phi^{\scr-1}\big(\G(\overline{\Q})\big)$. 
\begin{center}
 \begin{tabular} {lp{13cm}}
$(\ast)$ & Assume that ${\rm{r}}\cdot(\dim\,{\rm{G}} -\dim\,{\frak{T} }) \geq 1$. Then the image of $\Phi$ is not Zariski-dense in $\G(\C)$.
 \end{tabular}
 \end{center}
The equivalence between $(\ast)$ and the original statement is proved by induction on $\dim\,\G$.
\subsection{End of the story?}
The approaches of Schneider, Gel'fond and Baker are not restricted to theorems about analytic subgroups of a commutative group variety. However, they cover a large and representative amount
of transcendence results in the framework of analytic subgroups of algebraic groups. A second important aspect is that often different methods either need to be combined
with each other or contribute to the same type of applications. For instance, the qualitative part of Baker's theory 
can be deduced using ideas of Gel'fond, Schneider and Lang together with zero estimates. This was observed maybe first by Bertrand-Masser \cite{BeMa} and accomplished by Waldschmidt \cite{Wal17}. A third interesting feature is that after years of research activity the ideas of Schneider, Gel'fond and Baker reached a mature state which most probably cannot be improved in an essential manner. Systematic studies as the one of Roy \cite{Diaz1} confirm
this view. By way of contrast, it is expected that the first two theorems hold under much weaker estimates resp.\,that the Analytic Subgroup Theorem admits a generalization to fields of positive transcendence degree which avoids complicated estimates in the hypotheses. The ``folklore'' believe is that one cannot achieve more, but is rather far away from best possible results.\\
\\
\\
\\
\\
\section{From real-analytic homomorphisms to density of rational points}
Here is the idea of our paper: As in the classical setting we start with an algebraic group $\G$ over a subfield $F$ of $\C$, but will assume that $F$ is stable with respect to complex conjugation. Instead of complex-analytic homomorphisms $\Phi: \C \longrightarrow \G(\C)$ with values in the complex Lie group $\G(\C)$, we will consider
real-analytic homomorphisms $\Psi: \R \longrightarrow \G(\C)$ to the real-analytic Lie group underlying $\G(\C)$. We will associate to $\G$ an algebraic group $\mathcal{N} = \mathcal{N}_{F/K}(\G)$ over $K = F \cap \R$ and to $\Psi$
we will assign a real-analytic homomorphism $\Psi_{\mathcal{N}}: \R \longrightarrow \mathcal{N}(\R)$. The role of the vector space $\frak{T}  \subset \frak{g}$ over $F$ will be assumed by the smallest linear subspace $\frak{t} \subset \frak{g}$ over $K$ such that $\Psi_{\ast}(\R)$ is contained in $\frak{t} \otimes_K \R$. Here the algebraic group $\mathcal{N}$ is the Weil restriction of $\G$ over $K$. Its real points are canonically identified with the complex points of $\G$ and its geometry reflects algebraic morphisms from algebraic varieties ${\rm{V}}' \otimes_K F = {\rm{V}}' \times_{\mbox{spec}\,K}\mbox{spec}\,F$ with models ${\rm{V}}'$ over $K$. We shall develop further the theory of Weil restrictions and, at the end, we shall obtain a\\
\\
\textbf{``Machinery``.} A real-analytic homomorphism $\Psi: \R \longrightarrow \G(\C)$ with arithmetic properties is related to the existence of a map $v: \G \longrightarrow \G' \otimes_K F$ onto a positive-dimensional group variety which is definable over $K$. \\
\\
The approach is in contrast to the classical theory where a connection between a complex-analytic homomorphism with arithmetic properties and the existence of an algebraic subgroup of the given group variety is established. The \textit{raison d'\^etre} of our approach
is twofold: It allows better applications (see Subsect.\,1.2.1) and gives new theoretical insights (see Subsect.\,1.2.2).
\subsection{From real-analytic homomorphisms to real fields of definitions. An example}
Our ``machinery'' works only under more restrictive conditions, but then it yields definitely better results. The reason for this is, roughly speaking, that in the transition from an algebraic group over $F$ to an algebraic group over $K= F \cap \R$ the dimension of the algebraic group is doubled, whereas the dimension of the linear space $\frak{t}$ over $K$ remains the same. So,
transcendence theory can be eventually applied where it could not be applied before. We shall illustrate this by a lengthy example. For a list of symbols we refer to Sect.\,1.5.\\
\textbf{The {six exponential theorem} and the four exponentials conjecture.} The {six exponential theorem} asserts that
for a pair $a_1, a_2$ and a triple $b_1, b_2, b_3$ of respectively $\Z$-linearly independent complex numbers at least one among the six exponentials $e^{a_ib_j}$ is transcendental. Let $\G = \mathbb{G}_m^{\scr 3}$ and let $\Phi(z) = \big(e^{b_1z}, e^{b_2z}, e^{b_3z}\big)$ be the homomorphism with values in $\G(\C)$. Then $b_1, b_2$ and $b_3$ are $\Z$-linearly independent if and only if
$\Phi(\C)$ is Zariski-dense in $\G(\C)$. Hence, the six exponential theorem teaches that if the rank $\rm{r}$ of the group of algebraic logarithms $\Phi^{\scr -1}\big(\G(\overline{\Q})\big)$ is $\geq 2$, then $\Phi(\C)$ is not Zariski-dense in $\G(\C)$. It is a direct consequence of Theorem\,\ref{1.1.1}. The four exponential conjecture is the same statement as the six exponential theorem, except that the triple $b_1, b_2, b_3$ is replaced by a pair $b_1, b_2$. 
It follows from Roy \cite{Diaz1} that this famous conjecture cannot be proved with Schneider's method.\\
\textbf{The real-analytic approach.} In special cases the four exponential conjecture can be attacked with our "machinery". To make our real-analytic approach work, we need to assume that the numbers $a_1, a_2$ are contained in a real line. After replacing $a_1, a_2, b_1, b_2$ by $1, a_2/a_1, b_1a_1, b_2a_1$, we can then achieve that the new $a_1$ and $a_2$ are real without changing the hypotheses. Instead of considering the complex-analytic homomorphism $\Phi(z) = \big(e^{b_1z}, e^{b_2z}\big)$ to the set of complex points of $\G = \mathbb{G}_m^{\scr 2}$, it is then more advantageous to work with the restriction
$\Psi = \Phi_{|\R}$ to $\R$. The reason is that $\Psi$ can be "enlarged" without losing its arithmetic properties.\\
\textbf{From $\Psi$ to $\Psi_{\mathcal{N}}$.} We let $\mathcal{N} = \G \times \G$ and write
$p: \mathcal{N} \longrightarrow \G$ for the projection onto the first factor. By abuse of notation we shall identify $\G$ and $\mathcal{N}$ with their extensions to scalars over $\overline{\Q}$. We denote by $h$ the complex conjugation and for real $r$ we set
$$\Psi^{\scr h}(r) = (h\circ \Psi)(r) = \big(e^{h(b_1)r}, e^{h(b_2)r} \big).$$
In this way we can associate to $\Psi$ a real-analytic homomorphism 
$\Psi_{\mathcal{N}} = \Psi \times \Psi^{\scr h}$ with values in $\mathcal{N}(\C)$ such that
$$p_{\ast}\big(\Psi_{\mathcal{N}}\big) = p \circ \Psi_{\mathcal{N}} = \Psi$$
and 
\begin{equation}
\Psi(r)\,\,\mbox{is algebraic}\,\,\Longleftrightarrow\,\,\Psi_{\mathcal{N}}(r)\,\,\mbox{is algebraic}.
\end{equation}
Let $j = 1,2$. To obtain a good representation of the derivative of $\Psi_{\mathcal{N}}$, we shall consider the real and the imaginary parts $s_j = Re\,b_j$ and $t_j = Im\,b_j$. Then $\Psi_{\mathcal{N}}$ and the homomorphism 
\begin{equation}
 \Psi'(r) = \big(e^{2s_1r}, e^{2s_2r}, e^{2t_1r}, e^{2t_2r}\big)
\end{equation}
differ by an isogeny of $\mathcal{N}$.\\
\textbf{Applying transcendence theory in the case ${\rm{s}} \geq 3$.} In the lucky case when the group $\Z s_1 + \Z s_2 + \Z it_1 + \Z it_2$ has rank ${\rm{s}} = 3$ or ${\rm{s}} = 4$, 
the Zariski-closure ${\rm{T}} \subset \mathcal{N}$ of $\Psi_{\mathcal{N}}(\R)$ is a torus of dimension three or four. We let $\Phi_{\mathcal{N}}$ be the extension of $\Psi_{\mathcal{N}}$ to $\C$. If now all numbers $e^{a_ib_j}$ are algebraic, then it follows from (1.2.1) that the rank $\rm{r}$ of the group of algebraic logarithms $\Phi_{\mathcal{N}}^{\scr -1}\big({\rm{T}}(\overline{\Q})\big)$ is $\geq 2$. But in this situation Theorem \ref{1.1.1} yields that $\Phi_{\mathcal{N}}(\C)$ is not Zariski-dense in ${\rm{T}}(\C)$. Since the latter is a contradiction, 
at least one among the numbers $e^{a_ib_j}$ must be transcendental.\\
\textbf{Descent to $K$ in the case ${\rm{s}} = 2$.} If the group $\Z s_1 + \Z s_2 + \Z it_1 + \Z it_2$ has rank ${\rm{s}} = 2$, then the Zariski closure ${\rm{T}}$ has dimension two and the projection $p: \mathcal{N} \longrightarrow \G$ from above restricts to an isogeny between $\rm{T}$ and $\G$. To examine this case, we identify the Lie algebra $\frak{n} = {\rm{Lie}\,\mathcal{N}}$ with $\overline{\Q}^{\scr 4}$ and its complexification $\frak{n}(\C) = \frak{n} \otimes_{\overline{\Q}} \C$ with ${\C}^{\scr 4}$. In the spirit of our approach we view $\C^{\scr 4}$ as a real vector space and identify
a vector $(z_1, ..., z_4) \in \C^{\scr 4}$ with $(x_1, y_1, ...., x_4, y_4) \in \R^{\scr 8}$ where $z_j = x_j + iy_j$. Here the coordinates
$z_j$ are defined over $\overline{\Q}$, and the coordinates $x_j$ and $y_j$ are defined over $K = \overline{\Q} \cap \R$. With respect to the latter coordinates we have 
$$\big(\Psi_{\mathcal{N}}\big)_{\ast}(r) = \big(s_1r, t_1r; s_2r, t_2r; s_1r, -t_1r; s_2r, -t_2r\big).$$
If now all four exponentials $e^{a_ib_j}$ are algebraic, then- since $a_1$ and $a_2$ are real- all numbers $e^{a_is_j}$ and $e^{ia_it_j}$ are algebraic, too. So, $\omega_1 = a_1s_1, \omega_2 = a_1s_2, \omega_3 = ia_1t_1$ and $\omega_4 = ia_1t_2$ are algebraic logarithms. As the group $\Z s_1 + \Z s_2 + \Z it_1 + \Z it_2$ has rank ${\rm{s}} = 2$ by assumption, so does the group $\Z \omega_1 + \Z \omega_2 + \Z \omega_3 + \Z \omega_4$. 
Baker's theorem implies that the vector space $K \omega_1 + K \omega_2 + K \omega_3 + K \omega_4$ has dimension two over $K$. It follows then by linear algebra that the smallest linear subspace $\frak{t}_{\mathcal{N}} \subset \frak{n}$ over $K$, such that
$$\big(\Psi_{\mathcal{N}}\big)_{\ast}(a_1) = (\omega_1, \omega_3; \omega_2, \omega_4; \omega_1, -\omega_3; \omega_2, -\omega_4) $$ lies in the $\R$-span of $\frak{t}_{\mathcal{N}}$, has dimension $\dim\,\frak{t}_{\mathcal{N}} = 2$. Since $p_{|\rm{T}}$ is an isogeny, $\frak{t} = p_{\ast}(\frak{t}_{\mathcal{N}} ) $ is then the smallest subspace of $\frak{g}$ over $K$ such that $\Psi_{\ast}(\R) \subset \frak{t} \otimes_{K} \R$. Here we identify $\frak{t} \otimes_{K} \R$ with its canonical image in $\frak{g}(\C)$. Note that the equality
$$\dim\,\frak{t} = \dim\,\G$$
holds. Our "machinery" will teach us that this equality admits a theoretical explanation: it comes from the existence of an isogeny $v: {{\rm{G}}} \longrightarrow {{\rm{G}}}' \otimes_K \overline{\Q}$ to an algebraic group variety with model ${{\rm{G}}}'$ over $K$ such that $v_{\ast}(\Psi)$ takes values in ${\rm{G}}'(\R)$. This is a simple case of a \textit{descent to $K$} as will be defined in Sect.\,2.3.\\
\textbf{Determining $\G'$ in the case ${\rm{s}} = 2$.} In our simple case ${{\rm{G}}}'$ can be determined in terms $s_j$ and $t_j$ up to isogeny. To see this, let $\mathbb{S}$ be the one-dimensional torus $\mbox{spec}\,\Z[x,y]/(x^{\scr 2} + y^{\scr 2} - 1)$ over $\Z$ with neutral element $(1,0)$ and denote by $\mathbb{S}_K$ the extension to scalars over $K = \overline{\Q} \cap \R$. The assignment $\mu(x,y) = x + iy$ defines an isomorphism between the group $\mathbb{S}(\C)$ and the multiplicative group $\mathbb{G}_m (\C) = \C^{\ast}$, and this isomorphism identifies $\mathbb{S}(\R)$ with the unit circle ${\rm{S}}^{\scr 1} \subset \C.$\footnote{The map $\mu$ is defined over ${\rm{R}} =\Z[i]$ and stems from the isomorphism of ${\rm{R}}$-algebras
$\mu^{\sharp}: {\rm{R}}[t, t^{\scr -1}] \longrightarrow {\rm{R}}[x,y]/(x^{\scr 2} + y^{\scr 2} - 1)$ which maps $t$ to the class of $x+iy$.} In this way we can parametrize the points of $\mathbb{S}(\R)$ by the function $e^{i r}$ with real $r$.
In contrast to that, the points in the connected component of unity of $\mathbb{G}_m(\R) = \R^{\ast}$ are parametrized by $e^{r}$. Furhermore, by Example B.5.2 each one-dimensional torus over $K$ is either isogenous to the multiplicative group over $K$ or to $\mathbb{S}_K$. So, the torus ${{\rm{G}}}'$ from above is isogenous to the square of the multiplicative group over $K$ if $t_1 = t_2 = 0$, so that in (1.2.2) we get $\Psi'(r) = \big(e^{2s_1r},e^{2s_2r},1, 1\big)$; it is isogenous to the square of $\mathbb{S}_K$ if $s_1 = s_2 = 0$ and $\Psi'(r) = \big(1, 1, e^{2it_1r},e^{2rit_2}\big)$; or else it is isogenous to the torus obtained from $\mathbb{G}_m \times_{\Z} \mathbb{S}$ by base change to $K$.\\
 \textbf{Summary.} The example shows that the four exponentials conjecture holds in the "lucky case" when the two numbers $a_1$ and $a_2$ are collinear over $\R$ and the quadruple $s_1, it_1, s_2, it_2$ generates a group of rank ${\rm{s}} \geq 3$. This was already observed in
  Diaz \cite{Diaz}. For the "less lucky case" when $a_1$ and $a_2$ are still collinear over $\R$, but ${\rm{s}} = 2$, it was elaborated how a counter example to the four exponentials conjecture can be related to a homomorphism onto an algebraic group over $\R$. In many applications the involved group variety $\G$ does not allow a non-trivial homomorphism $v: \G \longrightarrow \G' \otimes_K {F}$ onto an algebraic group with model $\G'$ over $K = F \cap \R$ such that $v_{\ast}(\Psi) \subset \G'(\R)$, and then the "less lucky case" can be treated successfully alike.     
      
\subsection{From real fields of definition to density of rational points}
The theoretical gain of our approach stems from its relation to Mazur's conjecture on the topology of rational points:\\
\\
\textit{``At a lecture at Columbia University in 1992 Berry Mazur 'tried out' the following conjecture on his
audience: Let ${{\rm{V}}}$ be a smooth variety over $\Q$ with the property that ${{\rm{V}}}(\Q)$ is Zariski-dense. Then
the topological closure of ${{\rm{V}}}(\Q)$ in ${{\rm{V}}}(\R)$ consists of a (finite) union of connected components of ${{\rm{V}}}(\R)$'' \footnote{from Silverman's review of \cite{M1} on MathSciNet.}}\\
\\
Mazur's conjecture, if true, would generalize the fact that $\R$ is the only complete proper subfield of $\C$ and that the set of rational points of $\mathbb{A}^{\scr 1}$ is dense in the real-analytic manifold $\mathbb{A}^{\scr 1}(\R)$. And although in their paper \cite{SD}, Colliot-Th\'el\`ene, Skorobogatov and Swinnerton-Dyer 
give a counterexample to Mazur's conjecture, the following special case of abelian varieties is still open and point of departure for arithmetic considerations (see Mazur \cite[Conj.\,5]{M}).
\begin{con} \label{ConA} If $\rm{A}$ is a simple abelian variety over $\Q$ such that $\rm{A}(\Q)$ has positive rank, then the closure of $\rm{A}(\Q)$ in $\A(\R)$ with respect to the analytic topology contains the connected component of unity of $\rm{A}(\R)$.

\end{con}
Our results below indicate that the following generalization of Mazur's conjecture for abelian varieties should hold. In the formulation we use the fact that every connected closed subgroup of a real Lie group is a real Lie subgroup.
\begin{con}  \label{ConB} Let $\rm{A}$ be a simple abelian variety over $\overline{\Q}$ of dimension $g$ and let $\Gamma \subset \rm{A}(\overline{\Q})$ be a subgroup of positive rank. Let $\rm{C}$ be the closure of $\Gamma$ in $\rm{A}(\C)$ with respect to the analytic topology and let $\dim\,\rm{C}$ be its dimension
as a real Lie group. Then $\dim\,\rm{C} = 2g$, or $\dim\,\rm{C} = g$ and there exists an isogeny $v: \A \longrightarrow \A' \otimes_K \overline{\Q}$ to an abelian variety with model $\A'$ over $K = \overline{\Q} \cap \R$ such that $v(\Gamma) \subset \rm{A}'(\R)$.
\end{con}
Conjecture \ref{ConB} allows a broader view on Mazur's question concerning the topology of rational points and relates it to a density problem. Morally, it asserts that if a ``large'' subgroup $\Gamma \subset \rm{A}(\overline{\Q})$ is not dense in $\rm{A}(\C)$ with respect to the analytic topology, then there should be 
a theoretical justification for that. This theoretical justification is the existence of an isogeny $v: \A \longrightarrow {\rm{A}}' \otimes_K \overline{\Q}$ to an abelian variety with model $\A'$ over $K$ such that $\Gamma$ is mapped to the proper subset of real points in $\A'(\C)$.
By Conjecture \ref{ConA} the image $v(\Gamma)$ should then generate a dense subgroup
of the connected component of unity of ${\rm{A}}'(\R)$.\\
As will follow from Corollary \ref{mm}, Conjecture \ref{ConA} is equivalent to
\begin{con}  \label{ConD} Let $\rm{A}$ be a simple abelian variety over $\overline{\Q}$ and let $\Gamma \subset \rm{A}(\overline{\Q})$ be a subgroup of positive rank. Let $\rm{C}$ be the closure of $\Gamma$ in $\rm{A}(\C)$ with respect to the analytic topology and denote by ${\rm{C}}^o$ the connected component of unity of ${\rm{C}}$.
Then there exists a subspace $\frak{c}$ of $\frak{a} = {\rm{Lie}}\,{\rm{A}}$ over $K = \overline{\Q} \cap \R$ with the property that ${\rm{C}}^o = \rm{exp}_{\A}(\frak{c} \otimes_K \R)$.
\end{con}
It is not hard to bridge the gap between Conjecture \ref{ConB} and our \textit{Ansatz}. This is illustrated in 
\begin{example} Let $\rm{E}$ be an elliptic curve over $\overline{\Q}$ and let $\Gamma \subset \rm{E}(\overline{\Q})$ be a subgroup of rank three. Suppose that the closure $\rm{C}$ of $\Gamma$ with respect to the analytic topology is connected. If $\rm{C} \neq \rm{E}(\C)$, then $\rm{C}$ is isomorphic to $\R/\Z$ and there is a real-analytic homomorphism $\Psi: \R \longrightarrow \rm{E}(\C)$ with non-zero kernel and such that $\Gamma \subset \Psi(\R)$. 
 We set $\mathcal{N} = \mathcal{N}_{\overline{\Q}/K}({\rm{E}})$, replace $\Psi$ by $\Psi_{\mathcal{N}}$ and extend $\Psi_{\mathcal{N}}$ to a complex-analytic homomorphism $\Phi_{\mathcal{N}}: \C \longrightarrow {\mathcal{N}}(\C)$. Although neither $\mathcal{N}$ nor $\Psi_{\mathcal{N}}$ have been exactly defined yet, one is able to divine what this objects mean: they are analogs of the group variety $\mathcal{N}$ and the homomorphism $\Psi_{\mathcal{N}}$ from the previous subsection. 
In particular, the dimension of $\mathcal{N}$ is twice the dimension of $\rm{E}$, so that transcendence theory can be applied.
Using notations as in Theorem\,\ref{1.1.1} (with $\mathcal{N}$ instead $\G$), we have then $\dim\,\mathcal{N} = 2$, ${\rm{k}} \geq 1$ and ${\rm{r}} = {\rm{k}} + {\rm{rank}}_{\Z}\, \Gamma \geq 4$. It follows from the theorem that $\Phi_{\mathcal{N}}(\C)$ is a proper algebraic subgroup of the Weil restriction ${\mathcal{N}}$. As will be explained in Ch.\,2,
there is then an isogeny $v: {\rm{E}} \longrightarrow {\rm{E}}' \otimes_K F$ such that $v(\Gamma) \subset \rm{E}'(\R)$.  \end{example}
Summarizing, we see that our approach is, at least conjecturally, related to the question of density of $\Psi(\R)$ via the question of existence of real models.\\
\\
We conclude this section generalizing Conjecture \ref{ConB} to an arbitrary commutative group variety. This is also meant to motivate our definition of \textit{weak descent to $K$} in Subsect.\,2.3.2. As to our knowledge there is no straightforward generalization.
One obstruction is that if $\rm{G}$ is an arbitrary commutative group variety and $\Gamma$ is not dense in $\rm{G}(\C)$ then the closure of $\Gamma$ in $\G(\C)$ with respect to the analytic topology is a real-analytic submanifold of $\rm{G}(\C)$, but not necessarily with finitely many components (for example, a lattice in $\mathbb{G}_m(\C) = \C$).
Therefore, in the general setting one should start with a connected proper real-analytic submanifold $\rm{C}$ of $\rm{G}(\C)$ which contains a dense torsion-free subgroup of algebraic points. But even in this case Conjecture \ref{ConB} is not true for arbitrary commutative varieties, as the following example shows.
\begin{example} Let $\rm{E}'$ be an elliptic curve over $K = \overline{\Q} \cap \R$, set ${\rm{E}} = {\rm{E}}' \otimes_K \overline{\Q}$ and let $\G$ be an extension of ${\rm{E}}$ by a one-dimensional torus over $\overline{\Q}$. Suppose that $\G$ is neither isogenous to a group variety with model over $K$ nor an isotrivial extension.
Let $\xi \in \G(\overline{\Q})$ be an algebraic point without torsion which projects to $\rm{E}'(\R)$. It follows that $\Gamma = \Z\xi$ generates a proper real-analytic submanifold of $\G(\C)$. However, the conclusion in Conjecture \ref{ConB} does not hold for $\G$ and $\Gamma$ because- by construction- there is no isogeny $v$ as predicted in the conjecture. 
But there is a substitute for the isogeny in the conjecture. Namely, if we let $v$ be the projection from $\G$ to $\rm{E}$ and let $\G'$ be the real curve $\rm{E}'$, then we get a surjective homomorphism onto $\G' \otimes_K \overline{\Q}$ such that $v(\Gamma) \subset \G'(\R)$.

 \end{example}
It turns out that a surjective homomorphism $v: \G \longrightarrow \G' \otimes_K \overline{\Q}$ as in the example is the key to a generalization of Conjecture \ref{ConB}.
\begin{con}\label{Conc} Let $\G$ be a commutative group variety over $\overline{\Q}$ and let $\rm{C}$ be a connected proper real-analytic submanifold of $\G(\C)$.
 Assume that $\rm{C}$ contains a subgroup $\Gamma \subset \G(\overline{\Q})$ of algebraic points which is dense in $\rm{C}$ with respect to the analytic topology and such that each non-trivial point $\xi \in \Gamma$ generates a Zariski-dense subgroup of $\G(\C)$. Then there exists an algebraic group $\G'$ of positive dimension over $K = \overline{\Q} \cap \R$ and a surjective algebraic homomorphism
$v: \G \longrightarrow \G' \otimes_K \overline{\Q}$ with the property that $v(\rm{C}) = \G'(\R)$. 
\end{con}
We leave it to the reader to verify that Conjecture \ref{Conc} implies Conjecture \ref{ConB}.

\subsection{The origins} Mazur's conjecture for abelian varieties is point of departure of Waldschmidt's article \cite{Walden} from 1992 on the density of rational points on commutative group varieties.
Conjecture \ref{ConA} is proved there under the weaker hypopaper that ${\rm{rank}}_{\Z}\,\A(\Q) \geq (\dim\,\A)^{\scr 2} - \dim\,\A + 1$. Twelve years later, in 2004, Diaz \cite{Diaz} applied the technique of restrictions to scalars to improve transcendence results in elementary cases, though admittedly without referring to Weil restrictions in an explicit manner. In the elementary setting of the exponential function it was sufficient to simply consider conjugate numbers.
Despite many further sources treating the ``density of rational points'' from various points of view, as far as we know the two quoted articles are the only ones where restrictions of scalars are applied to transcendence problems. And it is not hard to see that the potential of this technique has not yet been exhausted. 
\section{Outline of the paper} \large{In the second chapter we conceptualize the density problem by introducing the notion of descent to $K$ and then formulate six theorems about real-analytic one-parameter homomorphism with arithmetic properties. The results of this chapter are stated mainly without proof. The chapter is rather a continuation of the introduction: We make precise what we indicated above and develop the formal theory.\\
Sect.\,2.1 and Sect.\,2.2 prepare the reader with the background needed to understand our definition of descent. We recall all important concepts from geometry in the first two sections. This is meant to synchronize terminology.\\
In Sect.\,2.3 the definition of descent to $K$ and weak descent to $K$ is given and elaborated. The main criterion for descents formulated in Subsect.\,2.5.3 plays an important role in this paper.\\
The underlying object of a descent to $K$ is a real-analytic one-parameter homomorphism $\Psi$ with values in the set of complex points of an algebraic group. In Sect.\,2.4 we consider a refined descent problem. We take a surjective homomorphism $\pi: \G \longrightarrow \U$ of algebraic groups, a homomorphism $\Psi$ to $\G(\C)$, and ask under which conditions $\pi_{\ast}(\Psi)$ descends (weakly) to $K$? The two theorems formulated in this section are technical, but very useful when it comes to applications. Together with the main criterion for descent to $K$ they form the kernel of our ``machinery''.\\
The results from Sect.\,2.1-Sect.\,2.4 do not rely on transcendence theory. By way of contrast, in Sect.\,2.5 we shall combine our ``machinery'' with the methods of Schneider, Gel'fond and Baker. The outcome is a ``matrix'' of two times three transcendence results. By this we mean that respectively two theorems are assigned to one of the three approaches. The respectively first result assigned to an approach is a statement about descent to $K$, whereas the respectively second one deals with weak descent to $K$. The technical, repetitive and maybe unusual shape of our transcendence theorems will be justified later by a quite fruitful list of applications. The last section of this chapter is informal. There we give a further hint how to deal with our approach.\\
\\
The first half of Ch.\,3 is devoted to conjugate varieties and, what comes along with the latter, conjugate sheaves and morphisms. In the second half we show the existence of Weil restrictions for finite Galois extension $F/K$ and recall some functorial properties. In particular,
we study the Lie algebra of the Weil restriction of a group variety. The results stated in this chapter are not new, although some aspects seem not yet been elaborated in literature. This is the only regular chapter of the paper whose statements are not essentially "self-made".\\
\\
The overall aim of the fourth chapter to prove the main criterion for descent and weak descent to $K$ (Theorem \ref{weakdescent}). In Sect.\,4.1 we introduce the notion of a {plurisimple algebraic group}. We call an algebraic group variety over a field \textit{plurisimple} if it is isogenous to a product of simple algebraic groups. We show that each group variety over a field admits a \textit{universal homomorphism} onto a maximal plurisimple quotient which
factors each further homomorphism onto a plurisimple algebraic group. The study of morphisms to commutative algebraic groups with universal properties is not a new idea. In particular, our notion of a universal homomorphism is a special case of a ``\textit{morphisme universel}'' as introduced in Serre \cite{Serre2}. But this is not the point. The crucial observation is that plurisimple groups and universal homomorphisms onto maximal plurisimple quotients constitute appropriate tools to solve descent problems. The challenge here can be sketched as follows: For a group variety $\G$ over $F$ the Weil restriction $\mathcal{N}_{F/K}(\G)$ "controls" morphisms $w: \V' \otimes_K F \longrightarrow \G$ from varieties definable over $K$ {to} $\G$. But descents deal with homomorphisms $v: \G \longrightarrow \G' \otimes_K F$ {from} $\G$ to algebraic groups definable over $K$. In Proposition 4.3.1\,\,it is shown that Weil restrictions also "control" the latter, provided that $F/K$ is a Galois extension of degree $[F:K] = 2$. Roughly speaking, the proposition is an abstract reformulation of our main criterion. In Sect.\,4.4 we then deduce the main criterion.\\
\\
In Ch.\,5 we prove Theorem \ref{inherited} and Theorem \ref{inherited1}. These two theorems rely on Theorem \ref{weakdescent} and are a major component of our transcendence results. Sect.\,5.1 and Sect.\,5.2 deal with two abstract propositions arising in the context of inherited descents. As we will show in Sect.\,5.3, these two propositions imply the theorems.\\
\\
Concerning Ch.\,3-Ch.\,5 we remark the following: Most of the results from these three chapters do not involve analytic concepts and are formulated in a rather abstract fashion. This abstract approach is justified, because it enables to generalize our "machinery" to the ultrametric setting and to obtain $p$-adic analogs of the transcendence theorems from Sect.\,2.5. On the other hand, an elaboration of this would involve plenty of modifications and would require a further paper. Therefore we shall not touch this here.\\  
\\
The sixth chapter is dedicated to the transcendence results from Sect.\,2.5. In Sect.\,6.1 we collect auxiliary lemmas which do not rely on transcendence theory. Each of the following three sections
is then devoted to the proof of respectively two theorems associated to one of the three methods from the introduction. As is expected, the proofs are repetitive to a large extent. On the other hand, each of the three transcendence methods has its own specifics which need to be considered separately.\\
\\
In Ch.\,7 we state applications of the six transcendence theorems. On the one hand, these applications serve as motivation for our \textit{Ansatz}. On the other hand, they rely in part on arguments from Ch.\,3-Ch.\,5 and from the appendices. Therefore we postponed the applications to the last regular chapter of the paper. However, an interested reader will want to have a look at them first or after skimming through the second chapter. The applications are divided into three sections. In the first section we investigate consequences of Theorem \ref{wus1} and Theorem \ref{wus2} (Baker's method), whereas Sect.\,7.2 and Sect.\,7.3 are on a par with Schneider's and Gel'fond's method respectively. This is not quite consistent
with the historical order from the introduction. The reason is that some results from Sect.\,7.2 and Sect.\,7.3 are based on our results about linear independence of algebraic logarithms. As to our knowledge all applications formulated in this chapter are new.}\\
\\
There are three appendices. App.\,A is dedicated to the question when a complex group variety is definable over a subfield of the reals. The results in App.\,A are not new and we apply them rather sporadically. In App.\,A we treat extensions of abelian varieties by linear groups. The contents of this appendix contribute to the proofs of the applications. In App.\,A.1\,\,we list some general facts. Then, in App.\,A.2, we shall consider conjugate extensions of an abelian variety over $\R$. The third section of App.\,A is dedicated to standard uniformizations of linear extensions of elliptic curves. Except Proposition B.2.2 the results in App.\,A are not new. In the last appendix
we fix a Galois extension $F/K$ of degree $[F:K] = 2$ and investigate Weil restrictions over $K$ of simple abelian varieties over $F$. We obtain a useful criterion
to decide whether a simple abelian variety over $F$ is isogenous to an abelian variety with model over $K$ (Proposition\,C.0.4). This criterion seems new and also contributes to the applications.

\section{List of new definitions}
Since we build up a new theory, we will introduce some non-standard definitions. Here is a list of those definitions
which are indispensable for the understanding of the paper.
\begin{center}
\begin{tabular}{ll}
model over $K$ & Sect.\,2.1\\
the twin $\Psi_{[i]}$ &Sect.\,2.2\\
the conjugate $\Psi^{\scr h}$ &Sect.\,2.2\\
$\Psi_{\mathcal{N}}= \Psi \times \Psi^{\scr h}$ & Sect.\,2.2\\
descent to $K$ & Sect.\,2.3\\
weak descent to $K$ & Sect.\,2.3\\
plurisimple group & Sect.\,4.1\\
maximal homomorphism & Sect.\,4.1\\
universal homomorphism & Sect.\,4.1\\
$\delta$-invariant & Sect.\,4.1\\
\end{tabular}
 
\end{center}

\section{List of important symbols} 

\begin{tabular}{llp{13cm}}
$N, M$ & natural numbers\\
$k,l,i, j$ & integer indices\\
$r$ & a real parameter\\
$\subset$ & ``is contained in or equal to''\\ 
$\subsetneq$ & "is contained in but not equal to"\\
$\Q$ & field of rational numbers\\
$\R$ & field of real numbers\\
$\C$ & field of complex numbers\\
$\overline{\Q}$ & algebraic closure of $\Q$ in $\C$\\
$F$ & a field, mostly a subfield of $\C$ stable\\
& with respect to complex conjugation\\
$K$ & a subfield of $F$, mostly $F \cap \R$\\
$Gal(F|K)$ & the Galois group of $F/K$ \\
& if $F/K$ is a Galois extension\\
$h$ & an element in $Gal(F|K)$,\\
& mostly the complex conjugation\\
$\mathbb{G}_{a, F}$ & $= \mbox{spec}\,F[t]$, the additive group over $F$\\
$\mathbb{G}_{m, F}$ & $= \mbox{spec}\,F[t, t^{\scr -1}]$, the multiplicative group over $F$\\
$\mathbb{A}_F^N$ & $= \mbox{spec}\,F[t_1,..., t_N]$, the affine $N$-space over $F$\\
$ \mathbb{P}_F^N$ & $ = \mbox{proj}\,F[t_0,..., t_N]$, the projective $N$-space over $F$\\
$\mathbb{S}_F $ & $= \mbox{spec}\,F[x,y]/(x^{\scr 2} + y^{\scr 2} - 1)$, the torus over $F$\\
$\G, \HHH, \U, \V$ & varieties over $F$, mostly group varieties\\
$\A$ & an abelian variety over $F$\\
${\rm{E}}$ &an elliptic curve over $F$\\
$\LL$ & mostly a linear group over $F$\\
$\rm{T}$ & a torus over $F$\\
$\mathcal{O}_{\G}$ & the sheaf of regular functions\\
& with values in $F$\\
\end{tabular}

\begin{tabular}{lp{11cm}}

$\G(\C)$ & the complex manifold of $\C$-\\
& valued points if $F \subset \C$\\
$\G'$ &variety over $K$, often the $K$-model of $\G$\\
$\G'(\R)$ & the real-analytic manifold of \\
& $\R$-valued points if $K \subset \R$\\
$cl(\G)$ & set of closed points of $\G$\\ 
$u, v, w, \mu, \nu$ & morphisms between varieties, often\\
& isogenies between algebraic groups\\
$v^{\sharp}$ & the homomorphism of structure sheaves\\
& associated to a morphism $v$\\

$\V' \otimes_K F$ & $= \V' \times_{\mbox{spec}\,K} \mbox{spec}\,F$\\
$u \otimes_K F$ & $=u \times_{\mbox{spec}\,K} id_{\mbox{spec}\,F}$\\
$e, e_{\G} $ & the neutral element of $\G$\\
$[k]_{\G}$& the multiplication with $k$ on $\G$\\
$\frak{g}$ & the Lie algebra of $\G$ over $F$\\
$\frak{g}(F')$ & = $\frak{g} \otimes_F F'$ for an extension $F'$ of $F$\\
$\frak{g}'$ & the Lie algebra of $\G'$ over $K$\\ 
$\frak{t}$ & a subspace of $\frak{g}$ over $K$\\
$\frak{t}_{\mathcal{N}}$ & see Subsect. 6.1.2\\
$\partial$ & a vector field in $\frak{g}$\\
$\Lambda$ & the kernel of ${\rm{exp}}_{\G}$, often a lattice in $\C$.\\
$\Psi$ & a real-analytic homomorphism\\
& from $\R$ to $\G(\C)$\\
$\Psi_{\ast}$ & the induced $\R$-linear homomorphism \\
& of Lie algebras\\
$\Psi^{\scr h}$ & the conjugate of $\Psi$\\
$\Psi_{[i]}$ & the twin of $\Psi$\\
$\Psi_{\mathcal{N}}$& $= \Psi \times \Psi^{\scr h}$\\
 $\Phi$ & a complex-analytic homomorphism,\\
& mostly the complexification of $\Psi$\\
$\pi: \G \longrightarrow \U$ & a surjective homomorphism\\
$\mathcal{N}_{F/K}(\G)$ & the Weil restriction of $\G$ over $K$\\
$p_G$ & the functorial map from $\mathcal{N}_{F/K}(\G) \otimes_K F$ to $\G$\\
$\mathcal{N}(v)$ & the map from $\V'$ to $\mathcal{N}_{F/K}(\G)$ associated\\
& to a morphism $v: \V'\otimes_K F \longrightarrow \G$\\
$v_{\mathcal{N}}$ & the morphism between Weil restrictions over $K$\\
& induced by a morphism $v$ of varieties over $F$\\
$\cal{H}(-)$ & see Sect.\,3.3\\

\end{tabular}

\newpage

\chapter{Descent to real subfields}

\section{Algebraic varieties, complex conjugates and analytic structures} Let $F$ denote a subfield of $\C$ and let $\rm{W}$ be the scheme $\mbox{spec}\,F$. For a field extension $j^{\scriptscriptstyle \sharp}: F \longrightarrow F'$ we consider the induced $F$-scheme ${\rm{W}}' = \mbox{spec}\,F'$ with structure morphism $j = \mbox{spec}(j^{\scriptscriptstyle \sharp})$. A \textit{variety} over $F$ is a scheme over $F$ which is irreducible, separated and of finite type over $F$. A variety over $F$ is \textit{smooth} if it is regular as a scheme. We denote by $\rm{G}$ a smooth variety over $F$
and let ${\rm{G}}(F')$ be the set of morphisms from ${\rm{W}}'$ to ${\rm{G}}$ over $F$. An element of ${\rm{G}}(F')$ is called a \textit{point over $F'$} or an \textit{$F'$-rational point}. As usual a point $\xi: {\rm{W}} \longrightarrow \G$ over $F$ will be
identified with its closed image $im\,\xi \in \G$. If $F' \subset F''$ are two extensions of $F$, then the pull-back 
defines an inclusion of ${\rm{G}}(F')$ into ${\rm{G}} (F'')$, and we will eventually
not distinguish between ${\rm{G}}(F')$ and its image in $\G(F'')$ to keep the notation short. For the same reason we will not distinguish between ${\rm{G}}(F'')$ and the $F''$-rational points of ${\rm{G}} \otimes_{F} F' = {\rm{G}} \times_{\mbox{spec}\,F} {\rm{W}}'$.\\
Let $K$ be a subfield of $F$. A \textit{model} of $\rm{G}$ over $K$ is an algebraic variety ${\rm{G}}'$ over $K$ such that $\rm{G}$ is isomorphic to ${\rm{G}}' \otimes_K F$. The variety $\G$ is \textit{definable over $K$} if it admits a model over $K$. It is important to note
that, in general, $\rm{G}$ may have several models over $K$ which are pair-wise non-isomorphic over $K$ (see App.\,A.2).\\
\\
Let $h \in Gal(\C|\R)$ be the complex conjugation and $j: {\rm{G}} \longrightarrow \rm{W}$ the structure morphism of ${\rm{G}}$. We denote by $\cal{G}$ the scheme underlying $\rm{G}$. The scheme $\cal{G}$ together with the morphism $j^{\scriptscriptstyle h} = \mbox{spec}(h)\circ j$ to $\mbox{spec}\,F^{\scriptscriptstyle h}$ define a variety over $F^{\scriptscriptstyle h} = h(F)$. It is called the \textit{complex conjugate} of $\rm{G}$ and is denoted by ${\rm{G}}^{\scriptscriptstyle h}$. The identity map on $\cal{G}$ induces a morphism $\rho$ from $\rm{G}$ to ${\rm{G}}^{\scriptscriptstyle h}$ over $K = F \cap \R$ where we consider ${\rm{G}}$ and ${\rm{G}}^{\scriptscriptstyle h}$ as schemes over $K$. If $\xi: W \longrightarrow \rm{G}$ is an $F$-rational point of $\rm{G}$, then $\rho_{\ast}(\xi) = \rho \circ \xi \circ \mbox{spec}(h)$ defines an $F^{\scr h}$-rational point of ${\rm{G}}^{\scriptscriptstyle h}$.\\
\\ 
Let ${\rm{G}}$ be again a smooth variety over a subfield $F$ of $\C$. The set of complex points ${\rm{G}}(\C)$ of ${\rm{G}}$ is endowed with the structure of a complex manifold and, consequently, with the structure of a real-analytic manifold. For this we refer to Shafarevi\'c \cite[Ch.\,VII]{Sh} and to App.\,A.1. The complex- and the real-analytic structures are determined by the complex algebraic variety ${\rm{G}} \otimes_F \C$. In Subsect.\,3.2.1 we shall show
\begin{prop} \label{ji} The map $\rho_{\ast}: {\rm{G}}(\C) \longrightarrow {\rm{G}}^{\scriptscriptstyle h}(\C)$ from above is a real-analytic isomorphism.
\end{prop}
If the field $F$ is contained in $\R$, then the set of real points ${\rm{G}}(\R)$ is a proper real-analytic submanifold of ${\rm{G}}(\C)$ of dimension $\dim\,{\rm{G}}(\R) = \dim\,{\rm{G}}$ and it is dense in ${\rm{G}}(\C)$ with respect to the Zariski topology. This assertion follows from Silhol \cite[p.\,31]{Sil}; it is in general wrong for singular varieties over $\R$.

\section{Algebraic groups and one-parameter homomorphisms. Definition of $\mathbf{\Psi_{[i]}, \Psi^{\scr h}}$ and $\mathbf{\Psi_{\mathcal{N}}}$} A group variety over $F$ is a group object in the category of varieties over $F$ with the additional property that the neutral element of the group law is a point over $F$. 
We consider a group variety ${\rm{G}}$ over $F$. We denote by $e = e_{\rm{G}} \in {\rm{G}}(F)$ the neutral element and define $\frak{g}$ to be the Lie algebra of left-invariant vector fields on ${\rm{G}}$. For a field extension $F'$ of $F$ we write $\frak{g}(F') = \frak{g} \otimes_F F'$. The space $\frak{g}$ is an algebraic object. However, each algebraic vector field defines a holomorphic vector field on the complex Lie group ${\rm{G}}(\C)$, and we obtain a natural inclusion $\frak{g}$ into the Lie algebra of ${{\rm{G}}}(\C)$. For dimension reasons this yields an identification of $\frak{g}(\C)$ with the Lie algebra of the complex Lie group ${\rm{G}}(\C)$. As recalled in App.\,\,A.2 we get then a unique exponential map ${\rm{exp}}_{\rm{G}} : \frak{g}(\C) \longrightarrow {\rm{G}}(\C)$. 
\begin{prop} \label{lp3333} \label{complex} Let ${{\rm{G}}}_c$ be a complex Lie group and let $\Psi: \R \longrightarrow {{\rm{G}}}_c$ be a real-analytic homomorphism.
Then there exists a unique holomorphic homomorphism
$\Phi: \C \longrightarrow {\rm{G}}_c$ such that $\Phi_{|\R} = \Psi$.
\end{prop}
The proof of the proposition is left to the reader. In what follows we will mostly consider real-analytic homomorphisms $\Psi$ from $\R$ to the analytic manifold induced by a group variety ${\rm{G}}$. By Lie group theory
the extension $\Phi$ admits then a unique differential $\Phi_{\ast}: \C \longrightarrow \frak{g}(\C)$ such that $\Phi = {\rm{exp}}_{\G} \circ \Phi_{\ast}$.\footnote{More precisely, $\Phi_{\ast}$ is defined on ${\rm{Lie}}\,\C = \C\frac{d}{dz}$. But, as usual, we shall identify a tangent vector $c\frac{d}{dz}$ with the complex number $c$.} This is recalled in App.\,A. The restriction of $\Phi_{\ast}$ to $\R$ defines an $\R$-linear homomorphism to $\frak{g}(\C)$ which will be denoted by $\Psi_{\ast}$. Next we define the \textit{twin} of ${\Psi}$ by $\Psi_{[i]}(r) = \Phi(i r)$. The \textit{complex conjugate} of ${\Psi}$ is $\Psi^{\scriptscriptstyle h} = \rho_{\ast} \circ \Psi$. Moreover, we
set $\Psi_{\mathcal{N}} = \Psi \times \Psi^{\scriptscriptstyle h}$ and get this way a homomorphism into the set of complex points of ${\rm{G}} \times {\rm{G}}^{\scriptscriptstyle h}$. 
\section{Descent to real subfields}
\subsection{Descent}
\large{Let $F$ be a subfield of $\C$, set $K = F \cap \R$ and consider a commutative group variety $\rm{G}$ over $F$. Let $\Psi: \R \longrightarrow {{\rm{G}}}(\C)$ be a real-analytic homomorphism.
We say that ${\Psi}$ \textit{descends to $K$} if $\Psi$ is non-zero and if there is an algebraic group ${\rm{G}}'$ over $K$ and an isogeny $v: {\rm{G}} \longrightarrow {\rm{G}}' \otimes_{K}F$
of algebraic groups over $F$ such that $v_{\ast}(\Psi) = v \circ \Psi$ factors through the inclusion of ${\rm{G}}'(\R)$ in ${{\rm{G}}}'(\C)$. In other words, we get a commutative diagram
\begin{equation}
\begin{xy} 
  \xymatrix{ 
\R  \ar[r]^{\Psi}\ar[d]^{v_{\ast}(\Psi)} & {{\rm{G}}}(\C) \ar[d]^{v}\\
{\rm{G}}'(\R) \ar@^{(->}[r] & {{\rm{G}}}'(\C).
}
\end{xy}
\end{equation}
A diagram as above does not exist in general. To get a rough picture of this phenomenon, we look at two easy examples.
\begin{example} \label{exmpl} Let $F = \C$, let ${\rm{G}}'$ be $\mathbb{G}_{a,\R}$ and write ${\rm{G}}$ for $\mathbb{G}_{a,\C}$. Then ${\rm{G}}'(\R) = \R$ and ${{\rm{G}}}(\C) = \C$. A non-zero real-analytic homomorphism $\Psi: \R \longrightarrow {\rm{G}}(\C)$ is multiplication
with $\Psi(1) \in \C^{\ast}$. Hence, if we set $v(z) = z/\Psi(1)$, then $v_{\ast}(\Psi)$ takes values in ${\rm{G}}'(\R) = \R$.
\end{example}
The example throws some light on our definition of descent to $K$. We see that the definition is only interesting if the considered algebraic group ${\rm{G}}$
has a sufficiently rich geometric structure. But it is not sufficient that ${\rm{G}}$ admits a model over $\R$, as the following example shows.
\begin{example} \label{exmpl1} Let $F = \C$ and let ${\rm{T}}$ be the complex torus $\mathbb{G}_{m, \C}$. Then
$$\Psi: \R \longrightarrow {\rm{T}}(\C), \Psi(r) = e^{(1+i)r}$$
does not descend to $\R$. In fact, by Example \ref{realscheiss} below there are two isomorphism classes of tori ${\rm{T}}'$ over $\R$ of dimension $\dim\,{\rm{T}}' = 1$ with complexification
isomorphic to ${\rm{T}}$. One is represented by the multiplicative group over $\R$ and the other one by the commutative group variety $\mathbb{S}_{\R}$ where
\begin{equation}
\mathbb{S} = \mbox{spec}\,\Z[x,y]/(x^{\scriptscriptstyle 2} + y^{\scriptscriptstyle 2} - 1)
\end{equation}
is the algebraic torus over with neutral element $(1,0)$ from Sect.\,1.2. The set $\mathbb{S}(\R)$ is compact, whereas the image of $\Psi$ is not. 
Thus, if $\Psi$ descends to $\R$, then there is an isogeny $v$ of $\rm{T}$ such that $v_{\ast}(\Psi)$ takes real values. But each isogeny of the multiplicative group is multiplication with an integer and maps the image of $\Psi$ to itself. Therefore $\Psi$ cannot descend to $\R$.
\end{example}
The reader will find more on descents in App.\,A.3.
\subsection{Weak descent and ``defects''} 
As above we let $F$ be a subfield of $\C$ and set $K = F \cap \R$. We consider a commutative group variety $\G$ over $F$ and a real-analytic homomorphism $\Psi: \R \longrightarrow \G(\C)$. In this subsection we shall formulate a refined notion of descent. We say that ${\Psi}$ \textit{descends weakly to $K$} if $\Psi$ is non-zero and if there is an algebraic group ${\rm{G}}'$ of positive dimension over $K$ and a surjective homomorphism $v: {\rm{G}} \longrightarrow {\rm{G}}' \otimes_{K} F$
of algebraic groups over $F$ such that $v_{\ast}(\Psi)$ factors through the inclusion of ${\rm{G}}'(\R)$ in ${{\rm{G}}}'(\C)$. Roughly speaking, the definition of weak descent to $K$ is the same as the definition of descent to $K$, except that the isogeny in Diagram (2.3.1) is replaced by a homomorphism onto a positive dimensional algebraic group. In particular, descent to $K$ implies weak descent to $K$. The converse is not true, as the following example shows.
\begin{example} \label{890} Let $\A$ be a complex abelian variety which is not isogenous to an abelian variety definable over $\R$ and set ${\rm{T}}' = \mathbb{G}_{a,\R}$ and ${\rm{T}} = \mathbb{G}_{a,\C}$. Consider non-zero
real-analytic homomorphisms $\Psi_1: \R \longrightarrow \A(\C)$ and $\Psi_2: \R \longrightarrow {\rm{T}}'(\R)$ and let $v$ be the projection from $\A \times \rm{T}$ to $\rm{T}$. Then the homomorphism $\Psi = \Psi_1 \times \Psi_2$ does not descend to $\R$, whereas the image of $v_{\ast}(\Psi) = \Psi_2$ is contained in ${\rm{T}}'(\R)$.
\end{example}
We suppose now that $F$ is stable with respect to complex conjugation and that $[F:K] = 2$. Then the Lie algebra $\frak{g} = \rm{Lie}\,\G$ is a vector space of dimension $\dim_K \frak{g} = 2 \dim\,\frak{g}$ over $K$.
In the following proposition a central theme of the paper is thematized: the connection between "defects" and weak descents to real subfields.  
\begin{prop} \label{real} \label{ooo}\label{opu} Let $\Psi: \R \longrightarrow \G(\C)$ be a non-zero real-analytic homomorphism. Consider $\frak{g}$ as a vector space over $K$ and let $\frak{t} \subset \frak{g}$ be the smallest subspace over $K$
such that $\Psi_{\ast}(\R) \subset \frak{t} \otimes_K \R.$ If $\Psi$ descends weakly to $K = F \cap \R$ via $v: \G \longrightarrow \G' \otimes_K F$, then the "defect" 
$\dim_K\,\frak{g} - \dim\,\frak{t}$ is larger than or equal to the dimension of $\G'$. 

\end{prop}
\begin{proof} The Lie algebra $\frak{g}' = {\rm{Lie}\,\G'}$ is identified with a proper subspace over $K$ of the Lie algebra of $\G' \otimes_K F$ (viewed as a vector space over $K$). We
will show in Corollary 3.3.3 below that $(v\circ \Psi)_{\ast}(\R)$ is contained in $\frak{g}'(\R)$. Since $\R$ is a faithfully flat $K$-module and as $v_{\ast}$ is $K$-linear,
it follows that $v_{\ast}(\frak{t})$ is contained in $\frak{g}'$. Hence, $\frak{t}$ is a proper subspace of $\frak{g}$ over $K$ and $\dim_K\,\frak{g} - \dim\,\frak{t} \geq \dim\,\frak{g}' = \dim\,\G' $.
\end{proof}
Conversely it is not true in general that a positive "defect" implies weak descent to $K$. We will show later that the this holds provided that $F = \overline{\Q}$ (see Corollary\,\ref{E}).
 
\subsection{The main criterion for descent and weak descent}
We denote by $F$ a subfield of $\C$ and define $K$ to be the intersection $F \cap \R$. It is again assumed that $F$ is stable with respect to complex conjugation and that $[F:K] = 2$. We let 
$\G$ be an algebraic group variety over $F$ and let $\Psi: \R \longrightarrow \G(\C)$ be a non-zero real-analytic homomorphism with Zariski-dense image in $\G(\C)$. Recall the definition of $\Psi_{\mathcal{N}}$ from Sect.\,2.2. We write ${{\rm{H}}}$ for the smallest algebraic subgroup of $\G \times \G^{\scr h}$ such that $\Psi_{\mathcal{N}}(\R) \subset {\rm{H}}(\C)$. The algebraic group ${{\rm{H}}}$ equals the intersection $\bigcap {{\rm{V}}}$ of all algebraic subgroups ${{{\rm{V}}}} \subset \G \times \G^{\scr h}$ with the property that $\Psi_{\mathcal{N}}(\R) \subset {{\rm{V}}}(\C)$. In Sect.\,4.4 we shall prove
\begin{theorem} \label{weakdescent} The homomorphism $\Psi$ descends to $K$ if and only if $\dim\,{{\rm{H}}} = \dim\,\G$. And $\Psi$ descends weakly to $K$ if and only if $\dim\,{{\rm{H}}} < 2\dim\,\G$. 
\end{theorem}
The first part of the main criterion is not a deep result, but it will be used everywhere. The condition in the second part means nothing but that $\HHH$ is proper in $\G \times \G^{\scr h}$. Although the second assertion is also standing to reason, its proof is quite involved and occupies almost all of the fourth chapter.

\section{Inherited descent to real subfields} Many applications in transcendence theory involve not only a single algebraic group $\G$ but a homomorphism $\pi: \G \longrightarrow {\rm{U}}$ of algebraic groups. To be more precise, one tries
to deduce transcendence properties of ${\rm{U}}$ from transcendence properties of $\rm{G}$. Typically, $\G$ is an extension of ${\rm{U}}$ by an ``auxiliary'' linear group and $\pi$ is the projection to ${\rm{U}}$.
Similar situations will appear in our real-analytic setting and this section is devoted to an axiomatic approach to such situations. The problem here can be formulated as follows. Let $F$ be a subfield of $\C$ and let $\pi: \G \longrightarrow {\rm{U}}$ be a surjective homomorphism of commutative group varieties over $F$. Consider a commutative diagram of non-zero real-analytic homomorphisms
\begin{equation}\begin{xy} 
  \xymatrix{\R \ar[rr]^{\Psi} \ar[rrd]^{\pi_{\ast}(\Psi)} && \rm{G}(\C) \ar[d]^{\pi}\\
&& {\rm{U}}(\C)}
\end{xy}
\end{equation}
If $\Psi$ descends (weakly) to $K = F \cap \R$, when does $\pi_{\ast}(\Psi)$ descend (weakly) to $K$? 
\begin{example} Let ${\rm{U}} = \rm{A}$ and $\Psi_1$ be as in Example \ref{890} and let $\pi: {\rm{A}} \times {\rm{A}}^{\scr h} \longrightarrow {\rm{U}}$ be the projection. As we will learn in the next chapter, $\Psi = (\Psi_1)_{\mathcal{N}}$ descends to $\R$, whereas $\pi_{\ast}(\Psi) = \Psi_1$ cannot descend to $\R$. 
\end{example}
In the example the group ${\rm{Hom}}(ker\,\pi, {{\rm{U}}}^{\scr h})$ is infinite. The following two theorems which are proved in Ch.\,5 show that this is one essential obstruction for inherited descents. In the statements we suppose that $F$ is stable with respect to complex conjugation. 
\begin{theorem} \label{inherited} If $\Psi$ descends to $K$ and if
${\rm{Hom}}\big(ker\,\pi, {{\rm{U}}}^{\scriptscriptstyle h}\big) = \{0\}$, then $\pi_{\ast}(\Psi)$ descends to $K$. 

\end{theorem} 
Example \ref{890} (with $\pi$ the projection to ${\rm{U}} = \A)$ illustrates that the theorem does not generalize literally to weak descents. For a correct generalization of the previous theorem to weak descents we need to consider the smallest algebraic subgroup ${{\rm{H}}} \subset \G \times \G^{\scr h}$ such that $\Psi(\R) \subset {\rm{H}}(\C)$.
\begin{theorem} \label{inherited1} If $\dim\,\G + \dim\,{\rm{U}} > \dim\,{\rm{H}}$ and if
${\rm{Hom}}\big({{\rm{U}}}^{\scriptscriptstyle h}, {\rm{M}}\big) = \{0\}$ for each quotient $\rm{M}$ of $ker\,\pi$, then $\pi_{\ast}(\Psi)$ descends weakly to $K$. 

\end{theorem} 
The inequality in the last theorem cannot be improved, as the following example demonstrates.
 \begin{example} \label{90+++} Consider Example \ref{890} once again and let $\pi$ be the projection from $\A \times {\rm{T}}$ to ${\rm{U}} = \A$. Then $\pi_{\ast}(\Psi) = \Psi_1$ and $ker\,\pi \backsimeq {\rm{T}}$. The assumption that $\A$ is not isogenous to an abelian variety definable over $\R$ implies by Theorem \ref{weakdescent} that $(\Psi_1)_{\mathcal{N}}$ has a Zariski-dense image.
On the other hand, as will become clear in the third chapter, ${\rm{T}} = \mathbb{G}_{m, \C}$ coincides with ${\rm{T}}^{\scr h}$ and $\Psi_2 = v_{\ast}(\Psi)$ equals its conjugate. Hence, here we have ${{\rm{H}}} = \A \times \A^{\scr h} \times \Delta$ where
$\Delta \subset {\rm{T}} \times {\rm{T}} = {\rm{T}} \times {\rm{T}}^{\scr h}$ is the one-dimensional diagonal. Therefore,
$$\dim\,{\rm{H}} = \dim\,{\rm{A}} + \dim\,{\rm{A}}^{\scr h}+ \dim\,\Delta = 2\dim\,{\rm{U}} + \dim\,ker\,\pi = \dim\,\G + \dim\,{\rm{U}}.$$
And since $\rm{U}$ is an abelian variety, whereas $ker\,\pi \backsimeq {\rm{T}}$ is linear, we have ${\rm{Hom}}\big({{\rm{U}}}^{\scriptscriptstyle h}, {\rm{M}}\big) = 0$
for each quotient $\rm{M}$ of $ker\,\pi$. 
\end{example}

\section{Applying transcendence theory} In what follows we will combine the above theory with the approaches of Schneider, Gel'fond and Baker. This yields six transcendence results where respectively two are aligned to one of the three approaches.  
The six theorems are all situated in the same \\
\\
\textbf{Setting.} The symbol $F$ denotes an algebraically closed subfield of $\C$ which is stable with respect to complex conjugation and we let $K$ be the intersection $F \cap \R$. We consider a surjective homomorphism $\pi: \G \longrightarrow {\rm{U}}$ of algebraic groups over $F$ and denote by $\Psi: \R \longrightarrow \G(\C)$ a real-analytic homomorphism with Zariski-dense image.
This leads to a diagram as in (2.4.1). The homomorphism $\Psi$ induces two invariants ${\rm{r}} = {\rm{rank}}_{\Z}\,\Psi^{\scriptscriptstyle -1}\big({{\rm{G}}}(F)\big) + {\rm{rank}}_{\Z}\,\Psi_{[i]}^{\scriptscriptstyle -1}\big({{\rm{G}}}(F)\big)$ and $\rm{k} = {\rm{rank}}_{\Z}\,ker\,\Psi$. Finally we define $\frak{t} \subset \frak{g}$ to be the smallest subspace over $K$ with the property that $\Psi_{\ast}(\R) \subset \frak{t} \otimes_K \R$. \\
\\
The respective first theorem of the next three subsections deals with inherited descents to $K$, and the respective second theorem with weak descents to $K$. We let ${\rm{H}} \subset \G \times \G^{\scr h}$
be the smallest algebraic subgroup with the property that $\Psi_{\mathcal{N}}(\R) \subset \rm{H}(\C)$. The algebraic group ${\rm{H}}$ will not appear in the hypotheses of the theorems anymore, but it plays a central backstage role. 
The theorems have then the following shape.
\begin{center}
 \begin{tabular} {lp{13cm}}
  $(\ast)$ & \textit{We require that $F, \rm{r}, \rm{k}$ and $\dim\,\frak{t}$ are such that transcendence theory leads to the estimate $\dim\,\G \geq \dim\,{\rm{H}} $.
 If in addition the hypotheses of Theorem \ref{inherited} are satisfied, then $\pi_{\ast}(\Psi)$ descends to $K$.} \\
  &\\
  $(\ast\ast)$ & \textit{We require that $F, \rm{r}, \rm{k}$ and $\dim\,\frak{t}$ are such that transcendence theory leads to the estimate $\dim\,\G + \dim\,{\rm{U}} > \dim\,{\rm{H}}$.
 If in addition the hypotheses of Theorem \ref{inherited1} are satisfied, then $\pi_{\ast}(\Psi)$ descends weakly to $K$.}
 \end{tabular}
\end{center}

\subsection{Two real-analytic theorems about transcendence\\ (Schneider's method)} In this subsection we shall modify Schneider's theory in the above style and state real-analytic versions of Theorem \ref{1.1.1}.
We set $F = \overline{\Q}$ and apply the above setting. In addition we consider a decomposition ${\rm{G}} \backsimeq {\rm{G}}_c \times  \mathbb{G}_{a, \overline{\Q}}^{{\rm{g}}_a} \times \mathbb{G}_{m, \overline{\Q}}^{{\rm{g}}_m}$
with natural numbers ${\rm{g}}_m, {\rm{g}}_a \geq 0$ and an algebraic group $\G_c$ over $\overline{\Q}$. Taking care of a refined decomposition is not quite aesthetic, but once more enlarges the radius of applications. 
\begin{theorem} \label{Schn1} If ${\rm{Hom}}\big(ker\,\pi, {{\rm{U}}}^{\scriptscriptstyle h}\big) = \{0\}$ and
$$({\rm{r}}-2)\cdot\dim\,\G   \geq 3 -(2{\rm{g}}_a + {\rm{g}}_m) -{\rm{k}},$$
then $\pi_{\ast}(\Psi)$ descends to $K$.

\end{theorem}
Note that, as indicated in the introduction, the inequality in the theorem is weaker than the one from Theorem \ref{1.1.1}. 
\begin{theorem} \label{Schn2} If ${\rm{Hom}}\big({{\rm{U}}}^{\scriptscriptstyle h}, {\rm{M}}\big) = \{0\}$ for each quotient $\rm{M}$ of $ker\,\pi$ and if
$$({\rm{r}}-2)\cdot(\dim\,\G + \dim\,{\rm{U}}) \geq {\rm{r}} - (2{\rm{g}}_a + {\rm{g}}_m) -{\rm{k}} + 1,$$
then $\pi_{\ast}(\Psi)$ descends weakly to $K$.
\end{theorem}
\begin{remark} \label{Schn3} If $\G$ is linear, then the summand ``$- (2{\rm{g}}_a + {\rm{g}}_m)$'' on the right hand side can be replaced by ``$- (2{\rm{g}}_a + {\rm{g}}_m + \dim\,\rm{U})$''. We will prove this in Ch.\,6, right after proving the theorem.
 
\end{remark}

\subsection{Two real-analytic theorems about algebraic independence\\ (Gel'fond's method)} The two results of this subsection mirror the impact of our theory on Gel'fond's method.
We assume that $F$ is an algebraically closed field which is stable with respect to complex conjugation and with transcendence degree $\leq 1$ over $\Q$. As above, we
consider a refined decomposition 
${\rm{G}} \backsimeq {\rm{G}}_c \times \mathbb{G}_{a,F}^{{\rm{g}}_a} \times \mathbb{G}_{m,F}^{{\rm{g}}_m} $. The dimension of $\G_c$ is abbreviated by ${\rm{g}}_c$ and we define
$$\delta = \left\lbrace \begin{array}{cl} 1 & \mbox{if}\,{\rm{g}}_a > 0\,\mbox{and}\,\dim\,\frak{t} = 1\\
0 &\mbox{otherwise}
 
  \end{array}
\right. $$
\begin{theorem}\label{mta1} Assume that
$$(1+{\rm{k}}){\rm{r}}\cdot (\dim\,{\rm{G}} -\dim\,\frak{t})  + {\rm{k}}{\rm{r}} \geq \left(2+ \frac{1}{{\rm{g}}_c + 1}\right){\rm{g}}_c + {\rm{g}}_m + 1+ \delta.$$
If ${\rm{Hom}}\big(ker\,\pi, {\rm{U}}^{\scriptscriptstyle h}\big) = \{0\}$, then $\Psi$ descends to $K$. 
\end{theorem}
The next theorem is analogous to Theorem \ref{Schn2}.
\begin{theorem} \label{Gel3} Assume that 
$$(1+{\rm{k}}){\rm{r}}\cdot (\dim\,{\rm{G}} + \dim\,{{\rm{U}}} -\dim\,\frak{t}) - {\rm{r}} \geq 2{\rm{g}}_c + 2\dim\,{{\rm{U}}} + {\rm{g}}_m + \delta.$$
If ${\rm{Hom}}\big({{\rm{U}}}^{\scriptscriptstyle h}, {\rm{M}}\big) = \{0\}$ for each quotient $\rm{M}$ of $ker\,\pi$, then $\pi_{\ast}(\Psi)$ descends weakly to $K$.
\end{theorem}
\begin{remark} \label{Gel31} (compare Remark \ref{Schn3}) If $\G$ is linear, then the summand ``$2\dim\,{\rm{U}}$'' on the right hand side can be replaced by ``$\dim\,{\rm{U}}$''.
 
\end{remark}
\subsection{Two real-analytic theorems about linear independence of algebraic logarithms (Baker's method)} This subsection is devoted to two results which are inspired by the Analytic Subgroup Theorem and deal with linear independence
of real and imaginary parts of algebraic logarithms. Here $F$ is the field of algebraic numbers $\overline{\Q}$ and and the twin $\Psi_{[i]}$ of $\Psi$ is defined as in Sect.\,2.2.
\begin{theorem} \label{mta0}\label{wus1} Assume that the image of $\Psi$ or $\Psi_{[i]}$ contains a non-trivial algebraic point in ${\rm{G}}\big(\overline{\Q}\big)$ and that $\dim\,{\rm{G}} \geq \dim\,\frak{t}$. If
${\rm{Hom}}\big(ker\,\pi, {{\rm{U}}}^{\scriptscriptstyle h}\big) = \{0\}$, then $\dim\,{\rm{U}} = \dim\,\pi_{\ast}(\frak{t})$ and $\pi_{\ast}(\Psi)$ descends to $K$. 
\end{theorem}
The last theorem is again about inherited weak descents.
\begin{theorem} \label{mt0W}\label{wus2} Assume that the image of $\Psi$ or $\Psi_{[i]}$ contains a non-trivial algebraic point in ${\rm{G}}\big(\overline{\Q}\big)$ and that $\dim\,{{\rm{U}}} + \dim\,{\rm{G}} > \dim\,\frak{t}$. If ${\rm{Hom}}\big({{\rm{U}}}^{\scriptscriptstyle h}, {\rm{M}}\big) = \{0\}$
for each quotient $\rm{M}$ of $ker\,\pi$, then $\pi_{\ast}(\Psi)$ descends weakly to $K$.
\end{theorem}
Recall Statement $(\ast)$ following Theorem \ref{AST}. There we reformulated the Analytic Subgroup Theorem by means of an inequality involving $\rm{r}$, $\rm{k}$ and $\dim\,\G$. The same is possible with respect to Theorem \ref{wus1} and Theorem \ref{wus2}. To this end, note that the image of $\Psi$ or $\Psi_{[i]}$ contains a non-trivial algebraic point in ${\rm{G}}\big(\overline{\Q}\big)$ if and only if
$\rm{r} > 0$. Hence, the first hypopaper in Theorem \ref{wus1} holds if and only if $\rm{r}(\dim\,\G - \dim\frak{t}) \geq 1 - \rm{r}$. And
the first condition in Theorem \ref{wus2} is fulfilled if and only if $\rm{r}(\dim\,\G + \dim\,\rm{U} - \dim\frak{t}) \geq 1$.
\section{Instead of a user's manuel} There is no unique best possible way to warm toward our \textit{Ansatz}. However, a rather disadvantageous
approach is by meditating on the statements of the transcendence results and trying to gain insight from this. Certainly, it is possible to build up the theory in a different style, but most probably this will lead to a presentation which is as formal as ours. A better way is by trying out examples and asking: \textit{What can classical theory say? And what is reached, in contrast to the latter, with the approach presented here?}\\
\\
Let's consider the following example which is related to linear independence of logarithms. We take an abelian variety $\A$ and an abelian threefold $\rm{B}$ over $\overline{\Q}$. It is assumed that the two algebraic groups are simple. Set $\G = \A \times \rm{B}$ and choose an algebraic logarithm $\omega \in \frak{g}(\C)$
whose image $\xi = {\rm{exp}}_{\G}(\omega) \in \G(\overline{\Q})$ generates a Zariski-dense subgroup of $\G(\C)$. Let $\frak{t}$ be the smallest subspace of $\frak{g}$ over $K$ with the property $\omega \in \frak{t} \otimes_K \R$. Suppose that $\omega$ is such that
$\dim\,\frak{t}$ is minimal among all logarithms in ${\rm{exp}}_{\G}^{\scr-1}(\xi)$. Finally let $p: \G \longrightarrow \A$ and $q: \G \longrightarrow \rm{B}$ be the projections and set $\Psi(r) = {\rm{exp}}_{\G}(\omega r)$.\\
\\
Here is what classical theory tells us: \textit{The subspace $\frak{t}$ generates $\frak{g}$ over $\overline{\Q}$ and therefore $\dim\,\frak{t} \geq 5$.}\\
And this is what can be inferred with the above approach: \textit{If $\frak{t} = \frak{g}$, then $\Psi$, $p_{\ast}(\Psi)$ and $q_{\ast}(\Psi)$ cannot descend to $K = \overline{\Q} \cap \R$.
On the other hand, if $\frak{t} \neq \frak{g}$, then there are three possibilities:
\begin{center}
\begin{tabular}{r p{12cm}}
1. & It holds that $\dim\,\frak{t} = 8$. Then the homomorphism $p_{\ast}(\Psi)$ descends to $K$, but $q_{\ast}(\Psi)$ does not.\\
2. & We have $\dim\,\frak{t} = 7$. The homomorphism $q_{\ast}(\Psi)$ descends to $K$, but $p_{\ast}(\Psi)$ does not.\\
3. & The equality $\dim\,\frak{t} = 5$ holds. Then all three homomorphisms $\Psi$, $p_{\ast}(\Psi)$ and $q_{\ast}(\Psi)$ descend to $K$.\\
\end{tabular}
\end{center}}
Obviously, the above list is comparatively richer in content than the answer of the classical theory. We suggest to the reader to elaborate
\begin{exercise} \label{opu} Derive the above list of possibilities. 
 
\end{exercise} 
While solving the exercise, one will notice that no statement is deduced from just one among the above theorems. For the proof of the three possibilities one has to apply Proposition \ref{real}, the main criterion for descents to $K$ together with Theorem \ref{wus1} and Theorem \ref{wus2} in a combined manner. We will meet again this structure
of inferring in the proofs of the applications. At the end of Sect.\,7.1\,\,a solution of Exercise \ref{opu} is given.

\chapter{Conjugate varieties and Weil restrictions}
\section{Conjugating varieties, morphisms and sheaves} 
Let $F$ be a field with fixed algebraic closure $\overline{F}$. We consider a variety ${\rm{G}}$ over $F$ with underlying scheme $\cal{G}$ and sheaf of $F$-algebras $\mathcal{O}_{\rm{G}}$. We emphasize right in the beginning
that we will always assume that $\mathcal{O}_{\rm{G}}$ is the sheaf of regular functions with values in $F$, that is, algebraic morphisms to $\mathbb{A}_F^{\scr 1}$. This is certainly no loss of generality, because
the structure sheaf of a variety $\G$ over $F$ is isomorphic (as a sheaf of $F$-algebras) to the sheaf of regular functions on $\G$ with values in $F$. \\
Let $\U \subset \G$ be an open set, let $f \in \Gamma(\U, \mathcal{O}_{\G})$ be a regular function on $\U$ and let $F'$ be a field extension of $F$. Then $f$ induces a regular function $f \otimes_F F': \U \otimes_F F' \longrightarrow \mathbb{A}^{\scr 1}_{F'}$, hence an element in $\Gamma(\U \otimes_F F', \mathcal{O}_{\G \otimes_F F'})$. Via restriction
one receives then a function in $\Gamma(\V, \mathcal{O}_{\G \otimes_F F'})$ for each open subset $\V \subset \U \otimes_F F'$. The elements of $\mathcal{O}_{\rm{G} \otimes_F F'}$ arising from functions in $\mathcal{O}_{\G}$ form a subsheaf of $\mathcal{O}_{\rm{G} \otimes_F F'}$: it is the subsheaf of functions defined over $F$.\\
After we have clarified this for the convenience of the reader, we proceed to the definition of conjugate varieties. To this end, let $K$ be a subfield of $F$ and fix a homomorphism $h \in {\rm{Hom}}_{K}(F, \overline{F})$. For $c \in F^{\scriptscriptstyle h} = h(F)$ and a local section $f$ of $\mathcal{O}_{\rm{G}}$ we define a scalar multiplication on the sheaf $\mathcal{O}_{\rm{G}}$ by $c \ast f = h^{\scriptscriptstyle -1}(c)f$. This way the structure of an $F^{\scriptscriptstyle h}$-variety on $\cal{G}$ is established. It is called the \textit{${h}$-conjugate} of ${\rm{G}}$ and denoted by ${\rm{G}}^{\scriptscriptstyle h}$. 
As agreed above, we replace the sheaf $(\mathcal{O}_{\rm{G}}, \ast)$ of $F^{\scr h}$-algebras by the sheaf of regular functions with values in $F^{\scr h}$ and consider $\G^{\scr h}$ with the latter in what follows.\\
If $j: \cal{G} \longrightarrow \rm{W}$ is the structure morphism of $\G$ to ${\rm{W}} = \mbox{spec}\,F$, then the structure morphism of $\G^{\scr h}$ equals $j^{\scr h} = \mbox{spec}\,(h^{\scr-1}) \circ j$ and its image is ${\rm{W}}^{\scr h} = \mbox{spec}\,F^{\scr h}$.\footnote{Compare the definition given in Sect.\,2.1. There $h$ was the complex conjugation. In this special case
there is no distinction between $h$ and its inverse $h^{\scr -1}$.} As already indicated in
Sect.\,2.1, the identity on $\mathcal{G}$ implies a commutative diagram of schemes over $K$
 $$\begin{xy} 
  \xymatrix{\G \ar[rr]^{\rho} \ar[d]^{j} && \G^{\scr h}\ar[d]^{j^{\scr h}}\\
W \ar[rr]^{ \mbox{spec}\,(h^{\scr -1})} && W^{\scr h}}
\end{xy}$$
Here the morphism $\rho$ depends on $h$ and $\G$, but this is supressed as long as $h$ and $\G$ are fixed.\\
\\ 
To an algebraic morphism $v: {\rm{G}} \longrightarrow {{{\rm{H}}} }$ of varieties over $F$ we associate the conjugate morphism $v^{\scr h}: {\rm{G}}^{\scr h} \longrightarrow {{{\rm{H}}} }^{\scr h}$
of varieties over $F^{\scr h}$. To define the conjugate morphism, fix $h$ and let $\rho_{\G}$ resp.\,\,$\rho_{\HHH}$ be the morphisms $\rho$ as above, associated to $h$ and $\G$ resp.\,$\HHH$.
Then we get a commutative diagram
 $$\begin{xy} 
  \xymatrix{\G^{\scr h} \ar[rr]^{\rho_{\G}^{\scr -1}} \ar[d]^{j^{\scr h}}& & \G \ar[rr]^{v} \ar[d]^{j} && \HHH \ar[rr]^{\rho_{\HHH}} \ar[d]^{j} && \HHH^{\scr h}\ar[d]^{j^{\scr h}}\\
W^{\scr h} \ar[rr]^{ \mbox{spec}\,h} && W \ar[rr]^{id.} && W \ar[rr]^{ \mbox{spec}\,(h^{\scr -1})} && W^{\scr h}}
\end{xy}$$
It results that $v^{\scr h} = \rho_{\HHH} \circ v \circ \rho^{\scr-1}_{\G}$ is a morphism of varieties over $F$. This is the conjugate morphism. It is then clear from the construction that, for all $\xi \in \G(F)$, $v^{\scriptscriptstyle h}$ satisfies
\begin{equation}
 v^{\scriptscriptstyle h}\big(\xi^{\scriptscriptstyle h}\big) = \rho\big(v(\xi)\big) = \big(v(\xi)\big)^{\scriptscriptstyle h}.
\end{equation}
Conjugation of morphisms in particular implies conjugation of regular functions and conjugation of points over $F$: Let $\U \subset \G$ be an open set. Then conjugation yields an isomorphism of rings between $\Gamma(\U, \mathcal{O}_{\G})$ and $\Gamma(\U^{\scr h}, \mathcal{O}_{\G^{\scr h}})$. Here
for $f \in \Gamma(\U, \mathcal{O}_{\G})$ and $\xi \in \G(F)$ it holds that $h\big(f(\xi)\big) = f^{\scr h}(\xi^{\scr h})$. And if $\xi: {\rm{W}} \longrightarrow \G$ is a point over $F$, then
its conjugate $\xi^{\scr h}: {\rm{W}}^{\scr h} \longrightarrow \G^{\scr h}$ is the point
 $\rho_{\ast}(\xi) = \rho \circ \xi \circ \mbox{spec}\,(h)$.\\
 \\
Finally we shall study conjugation of sheaves of abelian groups. So, let $\mathcal{L}$ be a sheaf of abelian groups on ${\rm{G}}$. If ${{\rm{U}}} \subset {\rm{G}}$ is an open subset, then ${{\rm{U}}}^{\scr h}$ is an open subset of ${\rm{G}}^{\scr h}$ in a natural way.
We set $\Gamma\big(\mathcal{L}^{\scriptscriptstyle h}, {{\rm{U}}}^{\scriptscriptstyle h}\big) = \Gamma\big(\mathcal{L}, {\rm{U}}\big)$. Using the same restriction maps we receive a sheaf of abelian groups on ${\rm{G}}^{\scriptscriptstyle h}$. If $\mathcal{L}$ is a sheaf of $\mathcal{O}_{\rm{G}}$-modules, then its conjugate is an $\mathcal{O}_{{\rm{G}}^{\scriptscriptstyle h}}$-module in a natural way. Namely, for $f \in \Gamma\big(\mathcal{O}_{\G^{\scriptscriptstyle h}}, {{\rm{U}}}^{\scriptscriptstyle h}\big)$ and $\sigma \in \Gamma\big(\mathcal{L}^{\scriptscriptstyle h}, {\rm{U}}^{\scriptscriptstyle h}\big) = \Gamma\big(\mathcal{L}, {\rm{U}}\big)$ a multiplication '$\ast$' is defined by $f \ast \sigma = f^{\scriptscriptstyle (h^{\scriptscriptstyle -1})}\sigma$. In particular, the conjugate
of the structure sheaf $\mathcal{O}_{\G}$ is canonically isomorphic to the structure sheaf of regular functions $\mathcal{O}_{\G^{\scr h}}$ as $\mathcal{O}_{\G^{\scr h}}$-module.
Conjugation of sheaves leads to an exact and fully faithful functor of sheaves of abelian groups and thus implies a conjugation in sheaf cohomology. To be more precise, for $k \geq 0$ the identity induces isomorphisms of cohomology groups
\begin{equation}
h^k(\rho): {{{\rm{H}}} }^k\big({{\rm{U}}}, \mathcal{L}\big) \longrightarrow {{{\rm{H}}} }^k\big({{\rm{U}}}^{\scriptscriptstyle h}, \mathcal{L}^{\scriptscriptstyle h}\big).
\end{equation}
Moreover, a homomorphism $u: \mathcal{L}_1 \longrightarrow \mathcal{L}_2$ of sheaves of abelian groups on ${\rm{G}}$ induces a
homomorphism $u^{\scr h}$ satisfying
\begin{equation}
u^{\scr h}\big(\sigma^{\scr h}\big) = \big(u(\sigma)\big)^{\scr h}
\end{equation}
for all local sections $\sigma \in \Gamma({{\rm{U}}}, \mathcal{L}_1)$ on an open set ${{\rm{U}}} \subset \G$. Since $u$ is a global section of the sheaf of homomorphisms between $\mathcal{L}_1$ and $\mathcal{L}_2$, the definition of $u^{\scr h}$ is a special case of conjugation of sheaves. 
 \section{More about conjugate varieties}
 \subsection{Realization in the affine case} 
 Assume that ${\rm{G}}$ is a subvariety of the $N$-dimensional affine space with associated prime ideal $J_{\rm{G}} \subset F[X_1,..., X_N]$ and let $J_{\rm{G}}^{\scr h} \subset F^{\scr h}[X_1,..., X_N]$ be the conjugate prime ideal. 
Then 
the affine $F^{\scriptscriptstyle h}$-algebra
$\big(F[X_1,...,X_N]/J_{\rm{G}}, \ast\big)$ is naturally isomorphic to the affine $F^{\scriptscriptstyle h}$-algebra $\big(F^{\scriptscriptstyle h}[X_1,...,X_N]/J_{\rm{G}}^{\scr h}, \cdot\big)$ with usual multiplication '$\cdot$'. To be more precise,
an isomorphism arises from assigning to the class of a polynomial the class of the polynomial with conjugated coefficients. Therefore ${\rm{G}}^{\scriptscriptstyle h}$ is canonically identified with the subvariety of $\mathbb{A}^N_{F^{\scr h}}$ associated to the ideal
$J_{\rm{G}}^{\scr h}$.
With respect to this identification the map $\rho_{\ast}$ is uniformized by conjugation of affine coordinates by $h$. \\
If ${{{\rm{H}}}}$ is a further affine variety embedded into $M$-dimensional affine space then a morphism $v: \G \longrightarrow \HHH$ is represented by polynomials
${{\rm{Q}}_1,...,{\rm{Q}}_M} \in F[X_1,...,X_N]$. With respect to the above realization of affine conjugate varieties, the polynomials with conjugate coefficients ${{\rm{Q}}_1}^{\scr h},..., {{\rm{Q}}_M}^{\scr h} \in F^{\scriptscriptstyle h}[X_1,...,X_N]$
define then the conjugate morphism $v^{\scriptscriptstyle h}$.

\subsection{Conjugate varieties and base extensions} Consider a commutative square of field extensions
$$\begin{xy} 
  \xymatrix{F_1 \ar@^{(->}[r] & F_2\\
K_1 \ar@^{(->}[r]\ar@^{(->}[r]\ar@^{(->}[u] & \ar@^{(->}[u] K_2}
\end{xy}$$
Let $h_2 \in {\rm{Hom}}_{K_2}(F_2, \overline{F}_2)$ and set $h_1 = h_{2|F_1} \in {\rm{Hom}}_{K_1}(F_1, \overline{F}_1)$. \begin{lemma} For all varieties $\rm{G}$ over $F_1$ there is a canonical isomorphism $${\rm{G}}^{\scriptscriptstyle h_1} \otimes_{F_1^{\scriptscriptstyle h_1}}F_2^{\scriptscriptstyle h_2} \backsimeq \big({\rm{G}} \otimes_{F_1} F_2\big)^{\scriptscriptstyle h_2}.$$
\end{lemma}
\begin{proof} Left to the reader. \end{proof}

\subsection{Conjugate varieties and Galois actions} 
In this subsection $F/K$ is a finite Galois extension and $\rm{G}'$ denotes a variety over $K$ with base extension ${\rm{G}} = {\rm{G}'} \otimes_K F$. We let
 $p: {\rm{G}} \longrightarrow {\rm{G}}'$ be the projection of schemes and fix a $h \in Gal(F|K)$. The previous lemma applied to $K_1 = K_2 = F_1 = K$ and $F_2 = F$ yields identifications
 $${\rm{G}} = {\rm{G}'} \otimes_K F = {\rm{G}'} \otimes_K F^{\scr h} = ({\rm{G}'} \otimes_K F)^{\scr h} = \G^{\scr h}.$$ 
The morphism $\rho$ of $K$-schemes associated to $h$ coincides with the $p$-equivariant automorphism $id_{\rm{G}'} \times_{K} \mbox{spec}\,(h^{\scr-1})$. This way an action of $Gal(F|K)$ on the scheme $\cal{G}$ underlying ${\rm{G}}$ is defined. As seen in the previous subsection, it extends to an action of points over $F$ which fixes ${\rm{G}}(K) = {\rm{G}}'(K)$.\footnote{As explained in the introduction, we identify $\G(K)$ with a subset of $\G(F)$.} 
\begin{lemma} \label{cl1} Let $F/K$ be a Galois extension of fields and let $\G'$ and ${\rm{H}}'$ be varieties over $K$. Consider an algebraic morphism
$v: \G' \otimes_K F \longrightarrow {\rm{H}}' \otimes_K F$. Then $v = v^{\scr h}$ for all $h \in Gal(F|K)$ if and only if $v$ is defined over $K$, that is,
if and only if there is an algebraic morphism $v': \G' \longrightarrow {\rm{H}}' $ such that $v = v' \otimes_K F$.
\end{lemma}
\begin{proof} Well known.
\end{proof}

\section{Conjugate group varieties and the exponential map.\\ Canonical actions} 
In this section we shall work out the previous constructions in the case of a group variety over a subfield of $\C$ and shall relate them to the exponential map.\\
We let $F$ be a subfield of $\C$ and set $K = F \cap \R$. It is assumed that
$[F:K] = 2$ and that $F$ is stable with respect to complex conjugation $h$. We denote
by ${\rm{G}}$ a group variety over $F$ with neutral element $e \in \G(F)$, multiplication $\mu$ and inversion morphism $i$. We leave it to the reader to verify that the conjugates of $e$, $\mu$ and $i$ define the structure of an algebraic group over $F^{\scr h}$ on $\G^{\scr h}$. 
Since the three morphisms and their conjugates are identical as morphisms of schemes over $K$, the real-analytic isomorphism $\rho_{\ast}: {\rm{G}}(\C) \longrightarrow {\rm{G}}^{\scriptscriptstyle h}(\C)$ from Proposition \ref{ji} constitutes then an isomorphism
between the real Lie groups $\G(\C)$ and $\G^{\scr h}(\C)$.\\
We denote by $\frak{g}$ the Lie algebra of left-invariant vector fields of $\G$. A vector field $\partial \in \frak{g}$ is an endomorphism of $\mathcal{O}_{\G}$, so that conjugation induces an isomorphism of vector spaces over $K$
$${{\rm{Lie}}}(\rho): \frak{g} \longrightarrow \frak{g}^{\scriptscriptstyle h}, {{\rm{Lie}}}(\rho)(\partial) = \partial^{\scriptscriptstyle h}$$
such that $\partial$ and $\partial^{\scr h}$ satisfy (3.1.3). To be more precise, if ${{\rm{U}}}^{\scriptscriptstyle h} \subset {\rm{G}}^{\scriptscriptstyle h}$ is an open set and $\partial \in \frak{g}$, then $\partial^{\scriptscriptstyle h} = {{\rm{Lie}}}(\rho)\big(\partial\big)$ acts on $F^{\scr h}$-valued functions $f^{\scriptscriptstyle h}$ on ${{\rm{U}}}^{\scriptscriptstyle h}$ by
\begin{equation}
 \partial^{\scriptscriptstyle h}f^{\scriptscriptstyle h}(e^{\scriptscriptstyle h}) = \partial^{\scriptscriptstyle h}f^{\scriptscriptstyle h}\big(\rho_{\ast}(e)\big) = h\big(\partial f(e)\big).
\end{equation}
The isomorphism extends to an $\R$-linear isomorphism between $\frak{g}(\C)$ and $\frak{g}^{\scr h}(\C)$ which is defined the same way. It will be denoted by the same symbol.
\begin{lemma} \label{111} The isomorphism ${{\rm{Lie}}}(\rho): \frak{g}(\C) \longrightarrow \frak{g}^{\scr h}(\C)$ is defined over $K$ and coincides with the differential of $\rho_{\ast}$. That is, we get a commutative diagram 
$$\begin{xy} 
  \xymatrix{
  \frak{g}(\C) \ar[d]^{exp} \ar[rr]^{{\rm{Lie}}(\rho)} & & \frak{g}^{\scr h}(\C) \ar[d]^{exp} \\
  \G(\C) \ar[rr]^{\rho_{\ast}} && \G^{\scr h}(\C)}
\end{xy}.$$
\end{lemma}
\begin{proof} It is sufficient to prove the assertion in the case when $F = \C$. We shall assume first that ${\rm{G}}$ is a subvariety of $\mathbb{A}^N_{\C}$ and that $e = 0$. The Lie algebra $\frak{g}$ is identified with the tangent space $T_{e}\G$ at the unit element. This is the vector space of $\C$-derivations of of the local ring $\mathcal{O}_{\G, e}$. Using standard coordinates $z_j$, we receive then an embedding of $\frak{g}$ into $T_0(\mathbb{A}^N_{\C}) = \sum_{j=1}^N\C \partial_j$. Here $\partial_j$ acts as $\partial_jf = {\scriptstyle{\frac{\partial}{\partial z_j}}} f(0)$. Let $h_N$ be the real-analytic map arising
from complex conjugation of standard coordinates on $\mathbb{A}^N(\C)$. Then $h_N(e) = h_N(0) = e$. According to Subsect.\,3.2.1 we have $\rho_{\ast} = h_{N|\G(\C)}$.
Statements (3.1.1) and (3.3.1) imply that a tangent vector $\partial \in \frak{g}$ (viewed as an element in $\sum_{j=1}^N\C \partial_j$) is mapped to the tangent vector with conjugate affine coordinates. Indeed, a regular function $f$ on $\G$ is represented
by a complex polynomial in $N$ variables, and the conjugate regular function $f^{\scr h}$ is represented by the polynomial with conjugate coefficients. So, $(\partial^{\scr h}f^{\scr h})(e) = h\big((\partial f)(e)\big)$ for all regular functions $f$ if and only if $\partial^{\scriptscriptstyle h} = h_N(\partial)$. The assertion follows in this situation. And if $\G$ is an arbitrary group variety, then the same argument works after embedding an open neighborhood of the origin into affine space. 
 \end{proof}
For a vector space $\frak{t}^+$ over $K$ we define the \textit{canonical action} of ${Gal(F|K)}$ on $\frak{t} = \frak{t}^+ \otimes_{K} F$ with respect to $\frak{t}^+$ by $h(v \otimes c) = v\otimes h(c)$ for $v \in \frak{t}^+$ and $c \in F$. If, conversely, $\frak{t}$ is a vector space over of finite dimension over $F$ then ${\cal{H}}(\frak{t})$ shall denote the set of $K$-subspaces $\frak{t}^+ \subset \frak{t}$ such that $\dim_{F}\frak{t} = \dim_{K}\frak{t}^+$ and $\frak{t} = F-$span of $\frak{t}^+$. Given $\frak{t}^+ \in {\cal{H}}(\frak{t})$, $\frak{t}$ is identified with $\frak{t}^+ \otimes_{K}F$. Hence, $\frak{t}^+$ defines
a canonical action of $Gal(F|K)$ on $\frak{t}$.\\
Let now $\G'$ is a commutative group variety over $K$ and write $\G$ for the base extension of $\G'$ to scalars over $F$. Then the Lie algebra
$\frak{g}'$ of $\G'$ is canonically contained in the Lie algebra $\frak{g} = \frak{g}'(F)$ of $\G$. Therefore, $\frak{g}' \in {\cal{H}}(\frak{g})$.
\begin{lemma} The canonical action of $Gal(F|K)$ on $\frak{g}$ with respect to $\frak{g}'$ coincides with the action of $Gal(F|K)$ on $\frak{g}$ defined (3.3.1).
 \end{lemma}
\begin{proof} Let $\partial' \in \frak{g}'$, $c \in F$ and write $\partial = \partial' \otimes c$. For all regular functions $f \in \mathcal{O}_{e, \G}$ over $K$ we have then $f(e) = h\big(f(e)\big)$ and
$(\partial')^{\scr h}f(e) = h\big(\partial' f(e)\big) = \partial' f(e)$.\footnote{The notion of a regular function over $K = \R$ was defined in the first subsection of this chapter.} Hence, 

$$\partial^{\scr h}f(e) = h\big(c\cdot \partial' f(e)\big) = h(c)\cdot \partial' f(e) = \big(\partial' \otimes h(c)\big)f(e).$$
As each left-invariant vector field in $\frak{g}$ is determined by the values it takes on functions $f \in \mathcal{O}_{e, \G}$ over $K$, it follows that ${\rm{Lie}}(\rho)\big(\partial\big) = \partial' \otimes h(c)$ for all $\partial \in \frak{g}'$ and all $c \in F$.\end{proof}

\begin{corollary}\label{lp333} Let $\G'$ be a commutative group variety over $K$. Then the image of $\frak{g}'(\R)$ under the exponential map is the connected component of unity ${\rm{G}}^o(\R)$ of ${\rm{G}}'(\R)$.
\end{corollary}
\begin{proof} Write $\G = \G' \otimes_K F$. Since $\G = \G^{\scr h}$, Lemma \ref{111} implies that $\rho_{\ast} \circ {\rm{exp}}_{\G} = {\rm{exp}}_{\G} \circ {\rm{Lie}}(\rho).$ As seen in the previous lemma, ${\rm{Lie}}(\rho)$ acts trivially on $\frak{g}'(\R)$. And by Subsect.\,3.2.3 the set of real points ${\rm{G}}'(\R)$ is fixed by $\rho_{\ast}$. Altogether we find that $\frak{g}'(\R)$ is mapped into ${\rm{G}}^o(\R)$. The assertion follows then because the two real manifolds $\frak{g}'(\R)$ and ${\rm{G}}^o(\R)$ have equal dimension and as
the exponential map is a local diffeomorphism.
\end{proof}

\section{Definition of the Weil restriction}\large{Let $F/K$ be an extension of fields and consider a variety ${\rm{G}}$ over $F$. The \textit{Weil restriction of ${\rm{G}}$ over $K$} is a pair $(\mathcal{N}_{F/K}({\rm{G}}), p_{\rm{G}})$ consisting
of a variety $\mathcal{N}_{F/K}({\rm{G}})$ over $K$ and a morphism 
$p_{\rm{G}}: \mathcal{N}_{F/K}({\rm{G}}) \otimes_{K}F \longrightarrow {\rm{G}}$ over $F$ such that, for all varieties ${{{\rm{V}}}}'$ over $K$ and all morphisms 
$v: {{{\rm{V}}}}' \otimes_{K}F \longrightarrow {\rm{G}}$ over $F$, there is a unique morphism $\mathcal{N}(v): {{{\rm{V}}}}' \longrightarrow \mathcal{N}_{F/K}({\rm{G}})$ with the property that $v = p_{\rm{G}} \circ \big(\mathcal{N}(v) \otimes_{K}F\big).$ That is, we have a unique commutative diagram of morphisms over $F$
$$\begin{xy} 
  \xymatrix{ 
{{{\rm{V}}}}' \otimes_{K}F \ar[rrrrr]^{\mathcal{N}(v)\otimes_{K} F} \ar[drrrrr]^{v}&&&&& \mathcal{N}_{F/K}({\rm{G}})\otimes_{K}F  \ar[d]^{p_{\rm{G}}} \\ 
&&&&& {\rm{G}}}
\end{xy}$$
It follows from the definition that $\mathcal{N}_{F/K}({\rm{G}})$ is unique up to isomorphism over $K$. To formulate an alternative definition of Weil restrictions, assume that $F/K$ is a finite Galois extension. Theorem \ref{Wr1}\,\,below shows that in this situation the restriction to scalars over $K$ exists for every variety ${\rm{G}}$ over $F$. Let ${\rm{Var}}_{F}$ resp.\,${\rm{Var}}_{K}$ be the category of varieties over $F$ resp.\,over $K$ and consider the bi-functors
$$\mathcal{F}_1:  {\rm{Var}}_{K} \times {\rm{Var}}_{F} \longrightarrow \mathbf{Sets}, \mathcal{F}_1({{{\rm{V}}}}', {\rm{G}}) = {\rm{Mor}}_{K}\big({{{\rm{V}}}}', \mathcal{N}_{F/K}({\rm{G}})\big)$$
and
$$\mathcal{F}_2:  {\rm{Var}}_{K} \times {\rm{Var}}_{F} \longrightarrow \mathbf{Sets}, \mathcal{F}_2({{{\rm{V}}}}', {\rm{G}}) = {\rm{Mor}}_{F}\big({{{\rm{V}}}}' \otimes_{K} F, {\rm{G}}\big)$$
For each variety $\G$ over $F$ choose a Weil restriction $\mathcal{N}_{F/K}(\G)$. Then the collection $\{\mathcal{N}_{F/K}({\rm{G}}), p_{\rm{G}}\}_{{\rm{G}} \in {\rm{Var}}_{F}}$ defines a pair $\big(\mathcal{N}_{F/K}, \mathcal{F}_{F/K}\big)$ consisting
of a covariant functor
$$\mathcal{N}_{F/K}: {\rm{Var}}_{F} \longrightarrow {\rm{Var}}_{K}$$
and an equivalence of functors $\mathcal{F}_{F/K}: \mathcal{F}_1 \longrightarrow \mathcal{F}_2$ given by $\mathcal{F}_{F/K}(v) = {\mathcal{N}}(v)$. The Weil restriction is thus the left adjoint of the functor '$-\otimes_{K} F$'.
\section{Existence and basic properties of the Weil restriction} \large{In this section we show the existence of Weil restrictions for finite Galois extensions and state further properties of the latter. To start with, we recall the notion of a descent datum. We let $F/K$ be a finite Galois extension of fields and set $W = \mbox{spec}\,F$. A variety $\G$ with underlying scheme $\cal{G}$ and structure morphism $j: \cal{G} \longrightarrow W$ together with an action $\chi: Gal(F|K) \longrightarrow {\rm{Aut}}(\cal{G})$ is called a \textit{descent datum}
if, for all $h \in Gal(F|K)$, the following diagram of morphisms of schemes commutes

$$\begin{xy} 
  \xymatrix{\cal{G} \ar[rr]^{\chi_h} \ar[d]^j &&\cal{G}\ar[d]^j\\
W \ar[rr]^{{\rm{spec}}\,(h^{\scr -1})} && W}
\end{xy}$$
In Subsect.\,3.2.3 we defined a canonical action of $Gal(F|K)$ on varieties of the form $\G' \otimes_K F$. We denote such an action by $\chi'$.
\begin{theorem} \label{descent} Let $F/K$ be a finite Galois extension and let ${\rm{G}}, \chi$ be a descent datum associated to $F/K$. Then there is a model $\G'$ of $\G$ over $K$ and an isomorphism
$v: \G \longrightarrow \G' \otimes_K F$ of varieties over $F$ which commutes with the actions of $Gal(F|K)$, that is, such that $v \circ \chi = \chi' \circ v$. 
\end{theorem}
\begin{proof} The theorem is shown in Weil \cite{Weil}. Since Weil did not use the language of schemes, we also cite Grothendieck \cite[Exp.\,190]{Gr} for a modern reformulation.
\end{proof}
With this we can state the existence result from Weil \cite{Weil}. For the convenience of the reader we recall that if $\G_j$ are varieties over $F$ and $\xi_j \in \G_j$ are points with Zariski-closure ${\rm{V}}_j \subset \G_j$, then $\prod_j \{\xi_j\}$ defines a point in the fibre product $\prod_{j } \G_j$. This is the generic point of the product $\prod_{j} {\rm{V}}_j \subset \prod_{j} \G_j$.
 \begin{theorem} \label{Wr1} If $F/K$ is a finite Galois extension, then each variety ${\rm{G}}$ over $F$ admits a Weil restriction $(\mathcal{N}_{F/K}({\rm{G}}), p_{\rm{G}})$ over $K$ with the following properties. 
\begin{enumerate}
\item $\mathcal{N}_{F/K}({\rm{G}})  \otimes_{K}F$ equals $\prod_{h \in Gal(F|K)} {\rm{G}}^{\scriptscriptstyle h}.$
\item $p_{\rm{G}}: \mathcal{N}_{F/K}({\rm{G}})  \otimes_{K}F \longrightarrow {\rm{G}}$ is the projection to ${\rm{G}} = \G^{\scr id.}$.
\item For $\tau \in Gal(F|K)$ and a point $\prod_{h} \{\xi_h\} \in \prod_h \G^{\scr h}$ the Galois action on $\mathcal{N}_{F/K}({\rm{G}})  \otimes_{K}F$ is
given by 
$$\left(\prod_{h} \{\xi_{ h}\}\right)^{\scr \tau} = \prod_{\tau h} \{\xi_h^{\scriptscriptstyle \tau}\}.$$ 
Here $\xi_h^{\scriptscriptstyle \tau} = \rho_{\tau h, h}(\xi_h)$ where $\rho_{\tau h, h}: \G^{\scr h} \longrightarrow \G^{\scr \tau h}$ is the morphism of schemes over $K$ stemming from conjugation with $\tau$.
\end{enumerate}\end{theorem}
\begin{proof} [Sketch of proof] We only sketch the proof. Next to the above cited sources, a short modern proof can be found in Huisman's paper \cite[p.\,27]{Hui1}.\\
On the product $\prod_{h \in Gal(F|K)} {\rm{G}}^{\scriptscriptstyle h}$ a descent datum 
$$\chi_h: \prod_{\tau \in Gal(F|K)} {\rm{G}}^{\scriptscriptstyle \tau} \longrightarrow \prod_{\tau \in Gal(F|K)} {\rm{G}}^{\scriptscriptstyle \tau}$$
is defined by requiring that 
$$pr_{\tau} \circ \chi_h = \rho_{\tau, h^{\scr -1} \tau} \circ pr_{{\tau^{-1}h}}$$
for all $h \in Gal(F|K)$. Here $\rho_{\tau, h^{\scr -1} \tau}$ refers to the morphism $\G^{\tau^{\scr -1}h} \longrightarrow \G^{\tau}$
over $K$ arising from conjugation with $\tau$. Let $\mathcal{N}$ be the variety over $K$ associated to this descent datum. Then
$\mathcal{N} \otimes_{K}F = \prod_{h} {\rm{G}}^{\scriptscriptstyle h}$. It was shown in Weil \cite{Weil} that $\mathcal{N}$ together with the projection $p_{\rm{G}} = pr_{id.}$ is a Weil restriction. The idea of proof is the following. A morphism $v: {{{\rm{V}}}} = {{{\rm{V}}}}' \otimes_{K}F \longrightarrow {\rm{G}}$ implies a morphism $w = \prod_h v^{\scriptscriptstyle h}: {{{\rm{V}}}} \longrightarrow \prod_h {\rm{G}}^{\scriptscriptstyle h}$ such that if $\xi' \in {{\rm{V}}}$ is fixed by $Gal(F|K)$, then $w(\xi') \in \left\lbrace \prod_{h \in Gal(F|K)} \{\xi^{\scriptscriptstyle h}\}; \xi \in {\rm{G}}\right\rbrace.$ It follows from Lemma \ref{cl1} that $w$ is defined over $K$, that is, $w = w_{\mathcal{N}} \otimes_{K}F$ for a morphism $w_{\mathcal{N}}$ over $K$.
\end{proof}
Let $F/K$ be a finite Galois extension and let $\rm{G}$ be a variety over $F$. Consider the Weil restriction $\mathcal{N} = \mathcal{N}_{F/K}(\rm{G})$ over $K$. Let $cl(\G)$ (resp.\,$cl(\mathcal{N})$) be the subset of closed points of $\G$ (resp.\,\,of $\mathcal{N}$).
{Theorem\,\ref{Wr1} implies
the following set-theoretical relations.
\begin{corollary} \label{3.3.1} We have canonical identifications
$$cl(\mathcal{N}) = \left\lbrace \prod_{h} \{\xi^{\scriptscriptstyle h}\}; \xi \in cl(\G)\right\rbrace$$
and 
$$\mathcal{N}(K) = \left\lbrace \prod_{h} \{\xi^{\scriptscriptstyle h}\}; \xi \in {\rm{G}}(F)\right\rbrace.$$  
\end{corollary}

\begin{corollary} \label{cl} Let $F/K$ be a finite Galois extension of fields and let $\G$ and $\rm{H}$ be varieties over $F$. Then an algebraic morphism $v: \G \longrightarrow \rm{H}$ induces a morphism $v_{\mathcal{N}}: \mathcal{N}_{F/K}(\G) \longrightarrow \mathcal{N}_{F/K}(\rm{H})$ of Weil restrictions over $K$ such that $v \circ p_{\G} = p_{\rm{H}} \circ (v_{\mathcal{N}} \otimes_K F)$. 

\end{corollary}

\begin{proof} It follows from Corollary \ref{3.3.1} that $u = \prod_h v^{\scr h}$ maps $cl\big(\mathcal{N}_{F/K}(\G)\big)$ to $cl\big(\mathcal{N}_{F/K}({\rm{H}})\big)$. Hence, $u(\xi) = \big(u(\xi)\big)^{\scr h} = u^{\scr h}(\xi^{\scr h}) = u^{\scr h}(\xi)$ for all
$h \in Gal(F|K)$ and all $\xi \in cl\big(\mathcal{N}_{F/K}({\rm{G}})\big)$. Since the sets of closed points are Zariski-dense, $u = u^{\scr h}$ for all $h$. This and Lemma \ref{cl1} imply that $u = v_{\mathcal{N}} \otimes_K F$ with a morphism $v_{\mathcal{N}}$ over $K$. \end{proof}
The next lemma summarizes further properties of Weil restrictions in case of its existence. Since these basic properties will be applied throughout in what follows, the reader
should focus on the statement for a moment. 
\begin{lemma} \label{WR} Let $F/K$ be a finite Galois extension of fields and let ${\rm{G}}, {\rm{G}}_1$ and ${\rm{G}}_2$ be varieties over $F$.
\begin{enumerate}
\item There is a natural isomorphism $\mathcal{N}_{F/K}({\rm{G}}) \backsimeq \mathcal{N}_{F/K}({\rm{G}}^{\scriptscriptstyle h})$ and $p_{{\rm{G}}^{\scriptscriptstyle h}} = (p_{\rm{G}})^{\scriptscriptstyle h}$ for all $h \in Gal(F|K)$. 
\item $\dim\,\mathcal{N}_{F/K}({\rm{G}}) = [F:K]\cdot \dim\,{\rm{G}}$.
\item There is a natural isomorphism $\mathcal{N}_{F/K}({\rm{G}}_1 \times_{F} {\rm{G}}_2) \backsimeq \mathcal{N}_{F/K}({\rm{G}}_1) \times_{K} \mathcal{N}_{F/K}({\rm{G}}_2).$
\item If $K_o \subset \overline{F}$ is an extension of $K$ such that $[FK_o: KK_o] = [F: K]$, then the Weil restriction of ${\rm{G}}_o = {\rm{G}} \otimes_F (FK_o)$ over $KK_o$ exists
and there is a natural isomorphism $\mathcal{N}_{F/K}({\rm{G}}) \otimes_{K}(FK_o) \backsimeq \mathcal{N}_{FK_o/KK_o}({\rm{G}}_o)$.

\end{enumerate}

\end{lemma}
\begin{proof} The first two statements result from Theorem \ref{Wr1}. Statement 3.\,\,follows from the universal properties of products. We proceed to the proof of Statement 4. The Weil restriction $\mathcal{N}_{FK_o/KK_o}({\rm{G}}_o)$ arises from a descent datum 
$$\chi_h^{\scr FK_o}: \prod_{\tau \in Gal(FK_o|KK_o)} \G_o^{\scr \tau} \longrightarrow  \prod_{\tau \in Gal(FK_o|KK_o)} \G_o^{\scr \tau}$$
and the restriction to scalars $\mathcal{N}_{F/K}({\rm{G}})$ is associated to a descent datum
$$\chi_h^{\scr F}: \prod_{\tau \in Gal(F|K)} \G^{\scr \tau} \longrightarrow  \prod_{\tau \in Gal(F|K)} \G^{\scr \tau}.$$
Here the two data are as in the proof of Theorem \ref{Wr1}. Since we have a canonical isomorphism $Gal(FK_o|KK_o) \backsimeq Gal(F|K)$, the datum $\chi_h^{\scr FK_o}$ is equivalent to the descent datum
$\chi_h^{\scr F} \otimes_K (KK_o)$. Hence,
$\mathcal{N}_{FK_o/KK_o}({\rm{G}}_o)$ is isomorphic to $\mathcal{N}_{F/K}({\rm{G}}) \otimes_{K_o} (K_oK).$ This is Statement 4.
\end{proof}
\section{Weil restrictions of group varieties}
\subsection{The Lie algebra of a Weil restriction}
We denote by $F/K$ a finite Galois extension of fields.
\begin{lemma} \label{WR2} Let ${\rm{G}}$ be a group variety over $F$ with multiplication law $\mu$, inversion morphism $i$ and neutral element $e$.
\begin{enumerate}
\item $\mathcal{N}_{F/K}({\rm{G}})$ together with the multiplication law $\mu_{\mathcal{N}}$, inversion morphism $i_{\mathcal{N}}$ and neutral element $e_{\mathcal{N}}$ is a group variety and, for all varieties ${{{\rm{V}}}}'$ over $K$,
$$\mathcal{N}_{F/K}: {\rm{Hom}}_{F}\big({{{\rm{V}}}}' \otimes_K F, {\rm{G}}\big) \longrightarrow {\rm{Hom}}_{K}\big({{{\rm{V}}}}', \mathcal{N}_{F/K}({\rm{G}})\big), \mathcal{N}_{F/K}(v) = v_{\mathcal{N}}$$
is a homomorphism of groups.
\item There is a natural identification
$${{\rm{Lie}}}\,\mathcal{N}_{F/K}({\rm{G}}) = \left\lbrace \sum_{\scr h} \partial^{\scriptscriptstyle h}; \partial \in \frak{g}\right\rbrace $$
with $h$ varying over $Gal(F|K)$. Here we view ${\rm{Lie}}\,\mathcal{N}_{F/K}(\G)$ as a subset of ${\rm{Lie}}\,(\G \times\G^{\scr h}) = \frak{g} \oplus \frak{g}^{\scr h}$. 
\end{enumerate}
\end{lemma}
\begin{proof}
It is left to the reader to verify that, for each $h \in Gal(F|K)$, $\mu^{\scr h}$ and $i^{\scr h}$ define group structures on $\G^{\scr h}$. So, $\prod_{h} \mu^{\scr h}$ and $\prod_{h} i^{\scr h}$ define group structure on $\mathcal{N}_{F/K}(\G) \otimes_K F$. It follows from Lemma \ref{cl1} that these morphisms are defined over $K$. So, they arise from an addition $\mu_{\mathcal{N}}$ and an inversion $i_{\mathcal{N}}$ on the Weil restriction. It is then clear by construction
that $\mathcal{N}_{F/K}$ is a homomorphism of groups. Statement 1.\,\,follows. For the proof of Statement 2. we define
$$\mathfrak{V} = \left\lbrace \sum_{h} \partial^{\scriptscriptstyle h}; \partial \in \frak{g}\right\rbrace$$
with $h$ varying over $Gal(F|K)$. Write $\mathcal{N} = \mathcal{N}_{F/K}({\rm{G}})$ and consider an affine open set ${\rm{U}} \subset {\rm{G}}$ containing the unit element. The affine algebra $\Gamma\big(\mathcal{N}_{F/K}({{\rm{U}}}), \mathcal{O}_{\mathcal{N}}\big)$ is spanned over $K$ by functions
$$\frak{f} = \sum_{\tau} \left(\bigotimes_{\scr \sigma} {f}_{\scriptscriptstyle \sigma}\right)^{\scr \tau}$$
with $\sigma,\tau$ varying over $Gal(F|K)$ and ${f}_{\scriptscriptstyle \sigma} \in p^{\ast}\Gamma\big({{\rm{U}}}^{\scriptscriptstyle \sigma}, \mathcal{O}_{\mathcal{N} \otimes_K F}\big)$. Here $p$ is the projection $\prod_{\scr h} \G^{\scr h} \longrightarrow \G^{\scr \sigma}$. For an arbitrary $\sum_{\scr h} \partial^{\scriptscriptstyle h} \in \mathfrak{V}$ we have
$$\frak{f}' := \sum_{\scr h} \partial^{\scriptscriptstyle h}\frak{f} = \sum_{\scr h, \tau} \left(\partial^{\scriptscriptstyle h} \left(\bigotimes_{\scr \sigma} {f}_{\scriptscriptstyle \sigma}\right)^{\scr \tau}\right) = \sum_{\scr h, \tau} \left(\left(\partial^{\scr h}f_{\tau^{\scr - 1}h}^{\scr \tau }\right) \otimes \bigotimes_{\scr \sigma \neq \tau^{\scr-1}h} f_{\sigma}^{\scr \tau}\right).$$
On the right hand side we replace $\tau^{\scr-1}h$ by $h$. Then statement (3.3.1) together with the fact that $F/K$ is a Galois extension yield 
$$\frak{f}'(e) = \sum_{\scr h, \tau} \big(\partial^{\scr h} f_{\scr h}(e)\big)^{\scr \tau}\cdot \left(\bigotimes_{\scr \sigma \neq h} f_{\scr \sigma}(e)\right)^{\scr \tau} \in K.$$
Since $\frak{f}$ was arbitrary, it follows that the vector stalk of $\sum_{\scr h} \partial^{\scriptscriptstyle h}$ at the unit element $e = e_{\mathcal{N}}$ defines a $K$-derivation of $\mathcal{O}_{\mathcal{N}, e}$. 
This shows that $\sum_{\scr h} \partial^{\scriptscriptstyle h} \in {{\rm{Lie}}}\,\mathcal{N}$ where we view ${\rm{Lie}}\,\mathcal{N}$ as a subspace of ${\rm{Lie}}\,(\G \times \G^{\scr h})$. But $\sum_{\scr h} \partial^{\scriptscriptstyle h}$ was chosen arbitrary, so that
\begin{equation}
 \mathfrak{V} \subset {\rm{Lie}}\,\mathcal{N}.
\end{equation}
Since for each $h \in Gal(F|K)$ the map ${\rm{Lie}}(\rho)(v) = v^{\scriptscriptstyle h}$ is defined over $K$, $\mathfrak{V}$ is a vector space over $K$ of dimension 
$\dim_K \mathfrak{V} = [F:K]\cdot\dim\,{\rm{G}}$. Hence, 
\begin{equation}
 \dim_K \mathfrak{V} = [F:K]\cdot\dim\,{\rm{G}} = \dim\,\mathcal{N} = \dim\,{{\rm{Lie}}}\,\mathcal{N}
\end{equation}
Statements (3.6.1) and (3.6.2) imply Statement 2.
\end{proof}}
\subsection{Simple algebraic groups and Weil restrictions}
If a monic polynomial ${\rm{P}}(x)$ of positive degree with real coefficients is irreducible in $\R(x)$, then there is a complex number $\alpha$ such that ${\rm{P}}(x) = (x-\alpha)$ or ${\rm{P}}(x) = (x-\alpha)\big(x-h(\alpha)\big)$ with the complex conjugation $h$. The following lemma states an analog of this for Weil restrictions of group varieties.
\begin{lemma} \label{cor789} Let $F/K$ be a Galois extension of fields of degree $[F: K] = 2$ and denote by $h$ the generator of $Gal(F|K)$. Let $\G'$ be a simple algebraic group over $K$. Then $\G = \G' \otimes_K F$ is either
a simple algebraic group over $F$ or there exists a simple algebraic group ${\rm{U}}$ over $F$ such that
$\rm{G}$ is isogenous to ${{\rm{U}}} \times {{\rm{U}}}^{\scr h}$.
\end{lemma}
\begin{proof} Let $\pi: \rm{G} \longrightarrow {\rm{U}}$ be a surjective homomorphism onto a simple algebraic group ${\rm{U}}$ of positive dimension over $F$. 
Since ${\rm{G}}'$ is simple, the induced homomorphism ${\mathcal{N}}(\pi): \G' \longrightarrow \mathcal{N}_{F/K}({\rm{U}})$ defines an isogeny between
$\G'$ and its image ${{\rm{H}}}' \subset \mathcal{N}_{F/K}({\rm{U}})$. If ${{\rm{H}}}' = \mathcal{N}_{F/K}({\rm{U}})$, then 
$\G$ is isogenous to ${{\rm{U}}} \times {{\rm{U}}}^{\scr h} = \mathcal{N}_{F/K}({{\rm{U}}}) \otimes_K F$. And if ${\rm{H}}' \neq \mathcal{N}_{F/K}({\rm{U}})$, then ${{\rm{H}}} = {{\rm{H}}}' \otimes_K F$ is a proper group subvariety of the product of simple groups ${{\rm{U}}} \times {{\rm{U}}}^{\scr h}$. Thus, ${\rm{H}}$ is simple over $F$. As $\G$ is isogenous to ${{\rm{H}}}$, the corollary follows.
\end{proof}

\subsection{Exact sequences of group varieties and Weil restrictions} \large{As the following lemma shows, the functor of Weil restrictions can be extended to exact sequences of algebraic groups.

\begin{lemma} \label{op1} Let $F/K$ be a finite Galois extension of fields. Then the functor $\mathcal{N}_{F/K}$ extends to the category of exact sequences of group varieties over $F$. That is, to an exact sequence ${{\rm{L}}} \stackrel{i}{\longrightarrow} {\rm{G}} \stackrel{\pi}{\longrightarrow} A $
of group varieties over $F$ there is associated an exact sequence
$$\mathcal{N}_{F/K}({{\rm{L}}}) \stackrel{i_{\mathcal{N}}}{\longrightarrow} \mathcal{N}_{F/K}({\rm{G}}) \stackrel{\pi_{\mathcal{N}}}{\longrightarrow} \mathcal{N}_{F/K}({\rm{A}}) $$
of group varieties over $K$ in a functorial manner.
\end{lemma}
\begin{proof} Clear. \end{proof}}

\chapter{The main criterion for descent and weak descent}
The aim of this chapter is to prove the main criterion for descent and weak descent (Theorem \ref{weakdescent}). It turns out that the proof is purely algebraic in nature.
In the first section we shall formulate some auxiliary lemmas about plurisimple algebraic groups (to be defined below). In particular, for an algebraic group $\G$ we will define the maximal plurisimple quotient $\overline{\G} = \G/\G_{ps}$ and the $\delta$-invariant $\delta(\G)$. These two notions seem of general interest. The second section deals with intermediate applications to Weil restrictions. Afterwards we formulate and prove an abstract version of the main criterion (Proposition \ref{1234}). In the last section we shall derive the main criterion.
\section{Plurisimple groups and the maximal plurisimple quotient}
\subsubsection{Definition and general properties}
We let $F$ be a field. A commutative group variety $\G$ over $F$ is called \textit{plurisimple} if it is isogenous to a product of simple algebraic groups over $F$.\footnote{It would be more reasonable to call such groups \textit{semi-simple}, \textit{quasi-simple} or even \textit{almost simple}. These notions, however, are already reserved.} 
\begin{lemma} \label{oiip} Let $\G$ be a plurisimple group variety over $F$ and ${\rm{H}}$ be a connected algebraic subgroup of $\G$.
\begin{enumerate}
 \item ${\rm{H}}$ is plurisimple.
\item $\G/{\rm{H}}$ is plurisimple.
\item Let $v: \G \longrightarrow \prod_{j = 1}^k \G_j$ be a maximal homomorphism. If $\dim\,\G > 0$, then the number $\delta(\G)$ of factors depends only on the isogeny class of $\G$ and is an invariant of $\G$.
\item ${\rm{H}}$ admits a complement ${{\rm{V}}}$, that is, a connected algebraic subgroup ${{\rm{V}}} \subset \G$ with the property that the sum ${\rm{H}} + {{\rm{V}}}$ taken in $\G$ equals $\G$ and such that ${\rm{H}} \cap {{\rm{V}}}$ has dimension zero.
\item The $\delta$-invariant is additive, that is, $\delta({{\rm{H}}}) + \delta(\G/{\rm{H}}) = \delta(\G)$. 

\end{enumerate}

\end{lemma}
Statement 5.\,holds with the convention that $\delta(\G) = 0$ if $\dim\,G = 0$. The proof of the lemma is lengthy and left to the reader. We proceed to general group varieties.
\begin{lemma}\label{!!} Let $\rm{G}$ be a commutative group variety of positive dimension over $F$. Then there is a connected linear subgroup $\G_{ps} \subset \G$ such that each homomorphism over $F$ from $\rm{G}$ onto a plurisimple algebraic group factors through the quotient map onto $\G/\G_{ps}$.  
\end{lemma}

\begin{proof}  If a connected subgroup ${\rm{G}}_{ps}$ with the universal property exists, then it is linear. In fact, since abelian varieties are plurisimple (by Poincar\'e's irreducibility theorem) and as the quotient of $\G$ by its maximal linear subgroup $\rm{L}$ is abelian (by Chevalley's theorem), it follows that ${\rm{G}}_{ps} \subset \rm{L}$. So, ${\rm{G}}_{ps}$ must be linear. We are left to prove the existence of ${\rm{G}}_{ps}$.\\
We start with an arbitrary algebraic homomorphism $\pi_1': \rm{G} \longrightarrow \rm{G}_1'$ onto a plurisimple algebraic group $\rm{G}_1'$ and let $\pi_1: {\rm{G}} \longrightarrow {\rm{G}}_1$ be the Stein factorization of $\pi_1'$. If $\pi_1$ factors all further algebraic homomorphisms $\pi_2': \rm{G} \longrightarrow \rm{G}_2'$ onto a plurisimple algebraic group ${\rm{G}}_2'$, we are done. Otherwise, there exists a homomorphism $\pi_2'$ onto a plurisimple algebraic group $\G_2'$ which is not factored by $\pi_1$. Let ${\rm{G}}_2'$ be the image of the homomorphism $\pi_1 \times \pi_2'$ to ${\rm{G}}_1 \times {\rm{G}}_2'$. Let $\pi_2$ be the Stein factorization of the homomorphism $\pi_1 \times \pi_2'$ to ${\rm{G}}_2'$. The homomorphism $\pi_1$ factors through $\pi_2$. So, $ker\,\pi_2 \subset ker\,\pi_1$. If $ker\,\pi_1 = ker\,\pi_2$, then $\pi_1$ would factor $\pi_2$ and $\pi_2'$. Therefore, $ker\,\pi_1 \subsetneq ker\,\pi_2$. Since $\pi_1$ and $\pi_2$ are connected, we infer that $\dim\,ker\,\pi_2 < \dim\,ker\,\pi_1$ (dimension of algebraic groups). By the above the image of $\pi_2$ is a plurisimple algebraic group ${\rm{G}}_2$. If $\pi_2$ factors all further homomorphisms $\pi_3': \rm{G} \longrightarrow \rm{G}_3$ 
onto a plurisimple group $\rm{G}_3$ which are distinct from $\pi_1$ and $\pi_2$, then $\G_{ps} = ker\,\pi_2$ is as required. Otherwise, there exists a homomorphism $\pi_3'$ onto a plurisimple algebraic group such that
the Stein factorization $\pi_3$ of $\pi_2 \times \pi_3'$ satisfies $\dim\,ker\,\pi_3 < \dim\,ker\,\pi_2$. For dimension reasons this process must stop. So, for some natural number $l$, $\dim\,ker\,\pi_l = \dim\,ker\,\pi_{l+1}$ for all choices of algebraic homomorphisms $\pi_{l+1}'$. Then $\G_{ps} = ker\,\pi_l$ fulfills
the universal property. 
 \end{proof}
For a group variety $\G$ over $F$ any surjective homomorphism $\pi: \G \longrightarrow \overline{\G}$ with kernel equal to $\G_{ps}$ will be referred to as \textit{universal}. As mentioned in the introduction, the notion of universal morphisms (``\textit{morphismes universels}'') appears already in Serre \cite[d\'ef.\,3]{Serre2} and our definition
is a special case of that.\\
A homomorphism $v: \G \longrightarrow \prod_{j = 1}^k \G_j$ onto a product of simple algebraic groups $\G_j$ of positive dimension is referred to as \textit{maximal} if $k$ is maximal among all such homomorphisms. For formal reasons
we extend this definition to a group variety $e$ of dimension zero over $F$. We call 
any surjective homomorphism $v: e \longrightarrow \prod_{j=1}^k \G_j$ maximal.
Maximal homomorphisms $v: \G \longrightarrow \prod_{j = 1}^k \G_j$ are not unique in general. But, as will follow from the next lemma and its corollary, the Stein factorization of a maximal homomorphism is unique up to isomorphism.
\begin{corollary} \label{---} Let $\rm{G}$ be a commutative group variety of positive dimension over $F$ and let $v: \G \longrightarrow \prod_j^{k}\G_j$ be a maximal homomorphism. Then the Stein factorization of $v$ is universal and $\delta(\overline{\G}) = k$. Moreover, $\dim\,\overline{\G} > 0$.
 
\end{corollary}
\begin{proof} As $\G_{ps}$ is connected, $\pi$ has connected fibers. Moreover, the Stein factorization ${\rm{St}}(v)$ of a maximal homomorphism $v$ factors through the quotient map $\pi: \G \longrightarrow \G/\G_{ps}$ as, say ${\rm{St}}(v) = u \circ \pi$. If $u$ is no isogeny,
then Lemma \ref{oiip} teaches that $\delta (\G/\G_{ps})$ is strictly larger than the number $k$ of simple factors $\G_j$. Hence, $v$ is not maximal. This yields the second assertion by contradiction.
For the last statement, let $\HHH \subset \G$ be a proper subgroup of maximal dimension. Then $\G/\HHH$ must be simple and $\dim\,\G/\HHH > 0$. Hence, $\G_{ps} \subset \HHH$ and $\dim\,\G/\G_{ps} > 0$. The corollary follows.
 \end{proof}
The above results show that the number $\delta(\G) = \delta(\overline{\G})$ is an invariant which depends only on the isogeny class of $\G$. 
If, more generally, $\G$ is an arbitrary algebraic group over $F$, then the connected unity component $\G^o$ of $\G$ is a group variety over $F$, and we set $\delta(\G)= \delta(\G^o)$. We will call it the \textit{$\delta$-invariant} of $\G$.

\begin{lemma} \label{+} Let $k \geq 1$ be an integer and let ${\rm{G}}_j$, $j = 1,..., k$, be commutative group varieties of positive dimension over $F$ with universal homomorphisms $\pi_j: {\rm{G}}_j \longrightarrow \overline{\G}_j$ onto plurisimple algebraic groups
$\overline{{\rm{G}}}_j$. Then $\pi = \prod_{j = 1}^k \pi_j$ is a universal homomorphism of the product group $\G = \prod_{j=1}^k {\rm{G}}_j.$
\end{lemma}
\begin{proof} Let $p: \rm{G} \longrightarrow \overline{\rm{G}}$ be a universal homomorphism. Then $\pi$ factors through $p$ as, say 
\begin{equation}
\pi = u \circ p.
\end{equation}
If $u$ is no isomorphism, then $u$ is no isogeny for the fibers of $\pi$ are
connected. So, if $u$ is no isomorphism, then Lemma \ref{oiip} implies that we can choose a complement $\overline{{{\rm{V}}}} \subset \overline{\rm{G}}$ of $ker\,u$. We consider then
the quotient map $\mu: \overline{\G} \longrightarrow \overline{\G}/\overline{\rm{V}}$. Taking a simple quotient $\rm{U}$ of $\overline{\G}/\overline{\V}$ of positive dimension, we receive a surjective non-trivial homomorphism $\mu: \overline{\G} \longrightarrow \rm{U}$
such that 
\begin{equation}\overline{\rm{V}} \subset ker\,\mu.
\end{equation}
Write $v = \mu \circ p$ and, for $j = 1,..., k$, let ${\rm{G}}_j^{\ast}$ be the copy of ${\rm{G}}_j$ in $\G = \prod_{j=1}^k \G_j$. 
Since $\pi = \prod_{j = 1}^k \pi_j$, it follows that $\pi_{|{\rm{G}}_j^{\ast}}$ is universal. Hence, $v_{|{\rm{G}}_j^{\ast}}$ factors through $\pi_{|{\rm{G}}_j^{\ast}}$. 
This in turn implies that $v$ factors through $\pi$. Statements (4.1.11) and (4.1.12) imply then 
$$ p^{\scr -1}(ker\,u) + p^{\scr -1}(\overline{{\rm{V}}}) \subset ker\,v.$$
But, as $\overline{{\rm{V}}} + ker\,u = \overline{\G}$, we have $\G = p^{\scr -1}(ker\,u) + p^{\scr -1}(\overline{{\rm{V}}})$. As a result, $v = 0$. We get a contradiction and infer that $u$ is an isomorphism. This yields the lemma. \end{proof} 
\begin{lemma} \label{+++} Let $\G$ be a commutative group variety with universal homomorphism $p: \G \longrightarrow \overline{\G}$ and let $\HHH$ a proper algebraic subgroup of $\G$. Then the image $p(\HHH)$ is a proper algebraic subgroup of $\overline{\G}$.
 
\end{lemma}
\begin{proof} Let $u: \G \longrightarrow \G/{{\rm{H}}} \longrightarrow {{\rm{V}}}$ be a surjective homomorphism onto a simple quotient ${{\rm{V}}}$ of $\mathcal{N}/\G$ of positive dimension. Now, if $p(\HHH) = \overline{\G}$, then $u$ does not factor through the universal homomorphism $u$. Contradiction. The result follows.\end{proof}

\subsubsection{Addendum: An analogy to rings} Let $\rm{R}$ be a commutative ring with unit element. Then the intersection of all maximal ideals of $\rm{R}$ is the \textit{Jacobson
radical} $J(\rm{R})$. It turns out that for a commutative group variety $\G$ over a field $F$ the universal kernel $\G_{ps}$ has a similar meaning. We call an algebraic subgroup $\HHH$ of $\G$ \textit{maximal}, if $\HHH$ is a connected proper subgroup and if for every infinite algebraic subgroup $\Gamma \subset \G$ the algebraic group
generated by $\HHH$ and $\Gamma$ equals $\HHH$ or $\G$. This definition is in analogy to the one of a maximal ideal of $\rm{R}$. The set of all maximal subgroups of $\G$ is denoted by $\max\,\G$.
Finally we define  
$$\epsilon(\G) = \min_{k \geq 1} \left\lbrace \begin{array}{l}\mbox{There are}\,\,k\,\,\mbox{groups}\,\,\HHH_j \in \max\,\G, j = 1,..., k,\\
        \mbox{such that}\,\,\left(\G_{ps} + \bigcap_{j=1}^k\HHH_j\right)/\G_{ps}\,\,\mbox{is finite}                               
                                            \end{array}\right\rbrace. $$  

\begin{lemma} The group $\G_{ps}$ is contained in $\bigcap_{\HHH \in \max\,\G}\HHH$, the quotient $\bigcap_{\HHH \in \max\,\G}\HHH/\G_{ps}$ is finite and the identity $\delta(\G) = \epsilon(\G)$ holds.

\end{lemma}

\begin{proof} We start by proving the first two assertions. The equality is clear if $\dim\G = 1$. So, we shall assume that $\dim\,\G \geq 2$.
If $\HHH \in \max\,\G$, then $\G/\HHH$ is simple. The universality of $\G_{ps}$ implies that $\G_{ps} \subset \HHH$. Hence,
$\G_{ps} \subset \bigcap_{\HHH \in \max\,\G}\HHH$. Next let $v: \G \longrightarrow \prod_{j= 1}^k \G_j$ be a maximal homomorphism and write $p_j: \prod_{j= 1}^k \G_j \longrightarrow \G_j$
for the projection. Let $\HHH_j$ be the connected unity component of $ker\,(p_j \circ v)$. Then $\HHH_j \in \max\,\G$ and $\bigcap_{j=1}^k \HHH_j = ker\,v$. It follows from Corollary \ref{---} that $\big(\bigcap_{j=1}^k \HHH_j\big)/\G_{ps}$ is finite. Hence, $\big(\bigcap_{\HHH \in \max\,\G}\HHH\big)/\G_{ps}$ is finite.\\
Next we verify the equality of numbers. It holds for sure if $k = 1$. So, assume that $k \geq 2$. Then the proof of the second assertion shows that $\delta(\G) \geq \epsilon(\G)$, because $\big(\bigcap_{j=1}^k \HHH_j\big)/\G_{ps}$ is finite. On the other hand, if
 ${\Gamma} = \big(\bigcap_{j=1}^{k-1} \LL_j\big)/\G_{ps}$ is finite for a suitable choice of maximal groups $\LL_1,..., \LL_{k-1} \in \max\,G$, then we set $\G_j = \G/\LL_j$ and let $w_j: \G \longrightarrow \G_j$ be the quotient map. We get a surjective homomorphism $w = \prod_j^{k-1} w_j$ onto the product $\prod_{j=1}^{k-1} \G_j$ of simple groups. 
As ${\Gamma}$ is finite, $w$ has the property that its factorization $u: \overline{\G} \longrightarrow \prod_{j=1}^{k-1} \G_j$ through the universal homomorphism $\pi: \G \longrightarrow \overline{\G}$ is a finite morphism. Hence, $w$ is maximal and Lemma \ref{oiip} implies that $k = \delta(\G) = k -1$. The contradiction shows that $\delta(\G) = \epsilon(\G)$. 
 \end{proof}

\section{The image of a subgroup $\mathbf{\HHH' \subset \mathcal{N}}$ under the universal homomorphism. Plurisimplicity of Weil restrictions}
We return to the theory of Weil restrictions and denote by $F/K$ a finite Galois extension of fields. Here we state two immediate applications of the previous section.
\begin{prop} \label{upp} Let $\rm{G}$ be a commutative group variety over $F$ with universal homomorphism $\pi: \rm{G} \longrightarrow \overline{\rm{G}}$. Let ${{\rm{H}}}'$ be a proper algebraic subgroup of $\mathcal{N}_{F/K}(\rm{G})$. Then $\pi_{\mathcal{N}}({{\rm{H}}}')$ is a proper algebraic subgroup of $\mathcal{N}_{F/K}(\overline{\rm{G}})$.
\end{prop}
\begin{proof} Let $h \in Gal(F|K)$. Then the properties of conjugate morphisms from Sect.\,3.1\,\,imply that $\pi^{\scriptscriptstyle h}$ is universal. By Lemma \ref{+} the morphism $\pi_{\mathcal{N}} \otimes_K F = \prod_h \pi^{\scr h}$ is universal. 
The proposition follows then from Lemma \ref{+++} applied to $\HHH = \HHH' \otimes_K F$.
\end{proof}

In the next proposition we assume that $[F: K] = 2$ and let $h$ be the generator of the Galois group $Gal(F|K)$. In the applications $F$ will be a subfield of $\C$ which is closed with respect to complex conjugation and $K$ will be the intersection $F \cap \R$.
\begin{prop} \label{upp1} Let $\rm{G}$ be a commutative group variety over $F$ and assume that $\G$ is plurisimple. Then $\mathcal{N}_{F/K}(\G)$ is plurisimple. 
 
\end{prop}
\begin{proof} In a first step one reduces to the case when $\G$ is a product of simple groups $\prod_{j=1}^k \G_j$. Since by Lemma \ref{WR} $\mathcal{N} = \mathcal{N}_{F/K}(\G)$ is canonically isomorphic to
 $\prod_{j=1}^k \mathcal{N}_{F/K}(\G_j)$, the assertion is then reduced further to the case when $\G$ is simple. We will assume this from now on.\\
\\
If $\mathcal{N}$ contains no proper simple algebraic subgroup of positive dimension, then $\mathcal{N}$ is simple and hence plurisimple. If $\mathcal{N}$ is not simple, then there is a proper simple algebraic subgroup ${{\rm{L}}}' \subset \mathcal{N}$ of positive dimension. We set $\LL = \LL' \otimes_K F$, ${\rm{U}}' = \mathcal{N}/\LL'$ and ${{\rm{U}}} = {{\rm{U}}}' \otimes_K F$. Moreover,
we let $p': \mathcal{N} \longrightarrow {\rm{U}}'$ be the quotient map and write $p$ for $p'\otimes_K F$. Since $\delta(\mathcal{N} \otimes_K F) = 2$, it follows that $\delta(\LL) = \delta({\rm{U}}) = 1$. In other words, $\LL$ and ${\rm{U}}$ are simple.
Recall that $\mathcal{N} \otimes_K F = \G \times \G^{\scr h}$ and set ${{{\rm{V}}}} = \G \times \{e_{\G^{\scr h}}\}$. Then $\LL \neq {{{\rm{V}}}}$ for $\LL = \LL^{\scr h}$, whereas ${{{\rm{V}}}}^{\scr h} = \{e_{\G}\} \times \G^{\scr h} \neq {{\rm{V}}}$. Hence, ${{{\rm{V}}}}$ is a simple complement
of $\LL$ in $\mathcal{N} \otimes_K F$. This in turn implies that the restriction of the quotient morphism $p$ to ${{\rm{V}}}$ is an isogeny. Let $u: {\rm{U}} \longrightarrow \mathcal{N} \otimes_K F$ be
a homomorphism with image $im\,u = {{\rm{V}}}$ and such that $p \circ u$
is multiplication with an integer $k$. Recall that ${{\rm{V}}}$ is contained in $\mathcal{N} \otimes_F K$ and write $v$ for $p_{\G} \circ u$.
The universal property of Weil restrictions implies the existence of a morphism ${\mathcal{N}}(v): {\rm{U}}' \longrightarrow \mathcal{N}$ such that
$v = p_G \circ ({\mathcal{N}}(v) \otimes_K F)$. We know from the proof of Theorem \ref{Wr1} that $v_{\mathcal{N}} \otimes_K F$ coincides with the morphism
\begin{equation}(v, v^{\scr h}): \V = \V^{\scr h} \longrightarrow \G \times \G^{\scr h},
\end{equation}
and it follows from (3.1.1) and Theorem \ref{Wr1} that $(v, v^{\scr h})= u + u^{\scr h}$. Since $p = p^{\scr h}$ and $[k]_{{\rm{U}}} = \big([k]_{{\rm{U}}}\big)^{\scr h}$, we find
\begin{center}
\begin{tabular}{l}
$(p' \circ v_{\mathcal{N}})\otimes_K F = p_{\ast}(v, v^{\scr h}) =  p_{\ast}(u) + p_{\ast}(u^{\scr h}) =  p_{\ast}(u) + \big(p^{\scr h}\big)_{\ast}(u^{\scr h}) = $\\
$p_{\ast}(u) + \big(p_{\ast}(u)\big)^{\scr h} = [k]_{{\rm{U}}} + \big([k]_{{\rm{U}}}\big)^{\scr h} = [2k]_{{\rm{U}}}.$
\end{tabular}
\end{center}
Therefore, the subgroup $v_{\mathcal{N}}({\rm{U'}}) \subset \mathcal{N}$ is not equal to $\LL' = ker\,p'$. A simplicity argument shows that $v_{\mathcal{N}}({\rm{U'}})$ is a complement of $\LL'$ in $\mathcal{N}$. As $\delta(\mathcal{N}) \leq \delta(\mathcal{N} \otimes_K F) = 2$, this in turn implies that $\mathcal{N}$ is isogenous to $v_{\mathcal{N}}({\rm{U'}}) \times \LL'$, and hence is plurisimple.
\end{proof}
\section{Subgroups of Weil restrictions of plurisimple groups. From subobjects to quotients} We let $F/K$ be a Galois extension and fix a group variety $\G$ over $F$.
In this subsection we study the question whether there exists a surjective homomorphism $v: \G \longrightarrow \G' \otimes_K F$ onto a group variety of positive dimension with model $\G'$ over $K$. We shall solve this problem using the Weil restriction of $\G$ for the case when $[F: K] = 2$. At first glance this seems surprising, because the Weil restriction deals with subobjects: it ``controls'' morphisms
from a variety definable over $K$ with target $\G$. In contrast to this, we ask
whether there are morphisms from $\G$ to a variety definable over $K$. The answer throws a new light on Weil restrictions and is also the final ingredient for the proof of Theorem \ref{weakdescent} in the next section. Although the theorem
involves real-analytic Lie groups, the proposition and its proof are purely algebraic. 
\begin{prop} \label{1234} Suppose that $[F: K] = 2$ and let $\G$ be a group variety over $F$. Let ${\rm{H}}'$ be a proper algebraic subgroup of $\mathcal{N}_{K/F}(\rm{G})$ and assume that $p_{\rm{G}}({{\rm{H}}}) = {{\rm{G}}}$ for ${{\rm{H}}} = {{\rm{H}}}' \otimes_K F$. Then there exist a group variety $\G'$ of positive dimension over $K$ and a surjective homomorphism $v: \G \longrightarrow \G' \otimes_K F$ such that $v_{\ast}\big(p_{\rm{G}|{\rm{H}}}\big) = v \circ p_{\rm{G}|{\rm{H}}}$ is defined over $K$.\footnote{The notion of "being defined over $K$" was explained in Lemma \ref{cl1}.}
\end{prop}
It is easy to see that a homomorphism $v: \G \longrightarrow \G' \otimes_K F$ leads to a proper algebraic subgroup $\HHH'$ of the Weil restriction $\mathcal{N} = \mathcal{N}_{F/K}(\G)$. Namely, letting $\Delta'$ be the proper algebraic subgroup $\mathcal{N}(id.)\big(\G'\big)$ of the Weil restriction of $\G' \otimes_K F$, one can set $\HHH' = v_{\mathcal{N}}^{\scr -1}(\Delta')$.\footnote{The symbol ``$\mathcal{N}(v)$'' has been defined in Sect.\,3.4.} The proposition thus states the converse of this observation. It asserts that the existence of a proper algebraic subgroup $\HHH'$ is indeed sufficient for the existence of a homomorphism $v$. 

\subsection{Three auxiliary lemmas}
The proof of the proposition is unfortunately lengthy and divided into three auxiliary lemmas (Lemma \ref{later1}, Lemma \ref{later3} and Lemma \ref{later4}).
\begin{lemma} \label{later1} Assume that $\G$ is a product $\prod_{j= 1}^k \G_j$ of $k \geq 2$ simple algebraic groups $\G_j$ of respectively positive dimension.
Let ${\rm{H}}'$ be a proper algebraic subgroup of $\mathcal{N}_{K/F}(\G) = \prod_{j=1}^k \mathcal{N}_{F/K}(\G_j)$ and suppose that $p_{\rm{G}}({{\rm{H}}}) = {{\rm{G}}}$ for ${{\rm{H}}} = {{\rm{H}}}' \otimes_K F$. We assume that $\delta({\rm{H}})$ is even and that ${\rm{H}}'$ projects surjectively onto $\mathcal{N}_2 = \prod_{2 = 1}^k \mathcal{N}_{F/K}({{{\rm{G}}}}_j)$.
\begin{enumerate}
 \item There exist a group variety $\G'$ of positive dimension over $K$ and a surjective homomorphism $v: \G \longrightarrow \G' \otimes_K F$ such that $v_{\ast}\big(p_{\rm{G}|{\rm{H}}}\big)$ is defined over $K$.
\item There are distinct $i, j = 1,...., k$ with the property that
${\rm{G}}_i$ is isogenous to the conjugate ${\rm{G}}_j^{\scriptscriptstyle h}$. 
 \end{enumerate}
\end{lemma}
Recall that the identification $\mathcal{N}_{K/F}(\G) = \prod_{j=1}^k \mathcal{N}_{F/K}(\G_j)$ is justified by Lemma \ref{WR}. Since this and the next two lemmas are situated in the same setting, we shall fix some symbols before going into the proof. As long as not otherwise stated, they will be used throughout this subsection.
\begin{center}
\textbf{Table of symbols used during the proof}
\begin{tabular}{lp{11cm}}
&\\
$h$ & the generator of $Gal(F|K)$\\ 
$\G$ & a product $\prod_{j= 1}^k \G_j$ of simple algebraic groups\\
$\mathcal{N}$ &$ = \mathcal{N}_{F/K}(\G) = \prod_{j= 1}^k \mathcal{N}_{F/K}(\G_j)$\\
$\mathcal{N}_1$ & $ = \mathcal{N}_{F/K}(\G_1)$\\
$\mathcal{N}_2$ &$ = \prod_{j=2}^k \mathcal{N}_{F/K}(\G_j).$\\
$\mu_j$ & for $j = 1,2$ the projection ${{\rm{H}}}' \longrightarrow \mathcal{N}_j$\\
$\U_j'$ & $= ker\,\mu_j$\\
$\U_j$ & $= \U_j' \otimes_K F$
\end{tabular}
\end{center}
The symbols $\mathcal{N}_2$, $\mu_j, \U_j'$ and $\U_j$ are not well-defined unless $k \geq 2$. By way of contrast, below we will also treat the case when $k = 1$. In this situation we shall denote by $\mathcal{N}_2$ the finite group variety $e' = \mbox{spec}\,K$ and let, for $j = 1,2$, $\mu_j$ be the unique homomorphism from $\HHH'$ to $e'$ over $K$. Then $\U_j' $ and $\U_j $ are defined as in the table.
\begin{proof}[Proof of Lemma \ref{later1}] We begin by verifying Statement 1.\,\,The proof is divided into several claims. 
\begin{claim} The kernel $\U_2'$ has dimension zero. 

\end{claim}
\begin{proof} The variety $\mathcal{N}_1 \otimes_K F$ is a product of two simple algebraic groups over $F$. Since ${\rm{H}}'$ is proper in $\mathcal{N}$ and as $\mu_2$ is surjective, $\U_2'$ is a proper algebraic subgroup of $\mathcal{N}_1$. Hence, if $\U'_2$
has positive dimension, then $\delta(\U'_2) = \delta(\U_2) = 1$. On the other hand, as $[F: K] = 2$, the $\delta$-invariant of $\mathcal{N}_2 \otimes_K F$ is even. It follows from the additivity of the $\delta$-invariant that if
$\U_1'$ has positive dimension, then $\delta({\rm{H}})$ is odd. We infer the claim by contraposition. 
\end{proof}
If $\mu_2$ is surjective and if $\U_2'$ is finite, then there exists an isogeny $\mu: \mathcal{N} \longrightarrow \mathcal{N}$ with the property that $\V' = \mu({\rm{H}}')$ is the graph of a homomorphism
$w': \mathcal{N}_2 \longrightarrow \mathcal{N}_1$ over $K$. Namely, we have 
$$\U_2' = \HHH' \cap \big(\mathcal{N}_1 \times \{e_{\mathcal{N}_2}\}\big)
 = \U'' \times \{e_{\mathcal{N}_2}\}$$
with an algebraic group $\U'' \subset \mathcal{N}_1$. If $\U'_2$ is finite, then there is an integer $k$ such that $[k]_{\mathcal{N}_1}(\U'') = 0$. Hence, if we take
$\mu = [k]_{\mathcal{N}_1} \times id.$, then $\V' \cap \big(\mathcal{N}_1 \times \{e_{\mathcal{N}_2}\}\big)$ is trivial. In this situation $\V'$ is the graph of a homomorphism $w'$ as claimed.
\begin{claim} For the proof of the lemma one can assume without loss of generality that $\mu$ is the identity.
 
\end{claim}
\begin{proof} Let ${{{\rm{V}}}}' = \mu({\rm{H}}')$ and let $u: \G \longrightarrow \G$ be the isogeny $u =  [k]_{\G_1} \times id.$ Since the $\delta$-invariant depends only on the isogeny class, $\delta({{{\rm{V}}}}' \otimes_K F) = \delta(\HHH)$ is even. Moreover,
with the above choice of $\mu$ and as $\G = \G_1 \times \prod_{j=2}^k \G_j$, we have 
$$p_{\G}(\V) = \big(p_{\G} \circ \mu\big)(\HHH) = \big(u\circ p_{\G}\big)(\HHH) = u(\G) = \G.$$
So, after we have constructed a homomorphism $v$ such that $v_{\ast}(p_{\G|\V})$ is defined over $K$, we can replace $v$ by $v \circ u$ and establish in this way the assertion for $\HHH$. \end{proof}
We will suppose that $\mu$ is the identity map and shall write $w$ for $w' \otimes_K F$. Recalling Corollary \ref{3.3.1}, we consider then closed points
$$\underline{\xi} = \big((\xi_1, \xi_1^{\scr h}),..., (\xi_k, \xi_k^{\scr h})\big) \in {\rm{H}}'.$$
Since $\mathcal{N}_2 \otimes_K F= \prod_{j = 2}^k \big(\G_j \times \G_j^{\scr h}\big)$, there are, for each $j = 2,..., k$, algebraic homomorphisms $w_j: \G_j \longrightarrow \G_1$ and $v_j: \G_j \longrightarrow \G_1^{\scr h}$ such that
$$\xi_1 = \big(p_{\G_1} \circ w\big)(\underline{\xi}) = \sum_{j=2}^k w_j(\xi_j) + \sum_{j=2}^k v_j^{\scr h}(\xi_j^{\scr h}).$$
Hence, for all closed points $\underline{\xi} \in {\rm{H}}'$ we have by (3.1.1)
\begin{equation}
 \xi_1 - \sum_{j=2}^k w_j(\xi_j) = \sum_{j=2}^k v_j^{\scr h}(\xi_j^{\scr h}) =  \left(\sum_{j=2}^k v_j(\xi_j)\right)^{\scr h}.
\end{equation}

\begin{claim} \label{099} The homomorphism $u: \HHH \longrightarrow \G_1$ defined by
$$u(\underline{\xi}) = \xi_1 - \sum_{j=2}^k w_j(\xi_j)$$
for a closed point $\underline{\xi} \in \HHH$ is surjective. Moreover, there exists a $j = 2,..., k$
such that the homomorphism $v_j$ is not trivial.
 
\end{claim}
\begin{proof} Since $\G_1$ is simple, it suffices to show that $u \neq 0$. But if $u = 0$, then $p_{\G}(\HHH)$
is contained in the proper subgroup of $\G$ defined by the relation $\xi_1 = \sum_{j=2}^k w_j(\xi_j)$. This contradicts
the hypotheses of the lemma. So, $u \neq 0$. The second statement is a direct consequence of the first and (4.3.1).
 \end{proof}
We define the homomorphism
$$v: \G \longrightarrow \G_1 \times \G_1^{\scr h} , v(\xi_1,..., \xi_k) = \left(  \xi_1 - \sum_{j=2}^k w_j(\xi_j), \sum_{j=2}^k v_j(\xi_j)\right).$$
It holds that $h = h^{\scr -1}$ because $[F: K] = 2$. So, it results from (4.3.1) that $v_{\ast}(p_{\G})$ maps closed points $\underline{\xi} \in {{\rm{H}}}'$ to closed points $(\xi, \xi^{\scr h}) \in \mathcal{N}_{F/K}(\G_1)$. Since the set of closed points $cl(\HHH')$ is Zariski-dense in $\HHH'$, this implies
that $v_{\ast}(p_{\G}) = \big(v_{\ast}(p_{\G})\big)^{\scr h}$. It follows from Lemma 3.2.2\,\,that $v_{\ast}(p_{\G})$ is defined over $K$. So, if $v$ is surjective, then $v$ is as required for Statement 1.\,\,On the other hand, if $v$
is not surjective, then Claim \ref{099} implies that $p_{\G_1} \circ v \circ p_{\G} = u$ is surjective. Therefore, the algebraic subgroup
$\U' = (v \circ p_{\G})\big(\HHH'\big)$ of $\mathcal{N}_1$ has positive dimension. Proposition\,4.2.2 teaches that $\mathcal{N}_1$ is plurisimple, and it results from Lemma 4.1.1 that $\U'$ admits a complement
$\V'$ in $\mathcal{N}_1$. Hence, if $v$ is not surjective, then we can replace $v$ by the composition of $v$ with the quotient map $\mathcal{N}_1 \otimes_K F \longrightarrow (\mathcal{N}_1/\V') \otimes_K F$, and receive
a homomorphism as required. The first statement follows. And Statement 2.\,\,results directly from Claim \ref{099}.\end{proof}
Next we prove

\begin{lemma} \label{later3} Assume that $\G$ is a product $\prod_{j= 1}^k \G_j$ of simple algebraic groups $\G_j$ of respectively positive dimension.
Let ${\rm{H}}'$ be a proper algebraic subgroup of $\mathcal{N}_{K/F}(\rm{G})$ and suppose that $p_{\rm{G}}({{\rm{H}}}) = {{\rm{G}}}$ for ${{\rm{H}}} = {{\rm{H}}}' \otimes_K F$. If $k = 1$ or if $k = 2$ and $\delta({\rm{H}}) = 3$, then there exist a group variety $\G'$ of positive dimension over $K$ and a surjective homomorphism $v: \G \longrightarrow \G' \otimes_K F$ such that $\G' \otimes_K F$ is simple over $F$ and with the property that $v_{\ast}\big(p_{\rm{G}|{\rm{H}}}\big) = v \circ p_{\rm{G}|{\rm{H}}}$ is defined over $K$.

\end{lemma}
\begin{proof} If $k = 1$, then $\HHH = \U_1 = \U_2$ and 
$\delta(\HHH) = \delta(\G) = 1$. So, $p_{\G|\HHH}$ must be an isogeny. Any isogeny $v: \G \longrightarrow \HHH$ such that $v \circ p_{\G|\HHH}$ is multiplication with an integer
satisfies the requirement with $\G' = \HHH'$.\\
If $k\geq 2$ and $\delta(\HHH) = 3$, then we fix a $j = 1,2$. If $\U_j'$ is finite, then $\HHH'$ is isogenous to $\mu_j(\HHH')$. As a result, $\HHH$ is isogenous to $\mu_j(\HHH') \otimes_K F$. In this situation we would
have $\delta(\HHH) \leq 2$. It follows by contraposition that $\U'_j$ is infinite. Moreover, if $\mu_j(\HHH')$ is a proper algebraic subgroup of $\mathcal{N}_j$,
then we replace $\G$ by $\G_j$ and $\HHH'$ by $\mu_j(\HHH')$, and reduce this way to the situation when $k = 1$ and $\delta(\HHH') = 1$. 
Since this was treated above, we receive the first part of the following claim.
\begin{claim} 
\begin{enumerate}
\item
For the proof of the lemma we may suppose that $\mu_j$ is surjective and $\U'_j$ is infinite for each $j = 1,2$, . 
\item 
If the reduction from the first statement holds, then $\delta(\HHH')= 3$.\footnote{By Proposition\,4.2.6 $\HHH'$ is plurisimple. And \textit{a priori} we have $\delta(\HHH') \leq \delta(\HHH) = 3$.}
\end{enumerate}
\end{claim}
\begin{proof} We only have to show Statement 2. Since $\HHH'$ is a proper algebraic subgroup of $\mathcal{N}$, it follows from the surjectivity of $\mu_2$ that $\U'_2 \neq \mathcal{N}_1 \times \{e_{\mathcal{N}_2}\}$. So, $\delta(\U'_2) = \delta\big(\U_2\big) = 1$.
The same argument yields that $\delta(\U'_1) = \delta\big(\U_1\big) = 1$. Moreover, as the intersection $\U_1' \cap \U_2'$ is trivial by definition, the two algebraic groups
$\U_1' \times \U_1'$ and $\U_1' + \U_2' \subset \HHH'$ are isogenous. And the same argument yields that the two algebraic groups $\U_1 \times \U_1$ and $\U_1 + \U_2 \subset \HHH'$ are isogenous. Therefore,
\begin{equation}
 \delta(\U'_1 + \U'_2) = \delta(\U_1 + \U_2) = 2.
\end{equation}
From the additivity of the $\delta$-invariant we deduce then
\begin{equation}
 \delta\big(\HHH'/(\U_1' +\U_2')\big) \leq \delta\big(\HHH/(\U_1 +\U_2)\big) = 3-2 = 1.
\end{equation}
Statement (4.3.3) implies that $\U_1 +\U_2 \neq \HHH$. Hence, $\U_1' +\U_2' \neq \HHH' $ and
\begin{equation}
 1 \leq \delta\big(\HHH'/(\U_1' +\U_2')\big).
\end{equation}
We deduce the claim from Statements (4.3.2)- (4.3.4).
\end{proof}
Let $\U = ker\,p_{\G|\HHH}$ and recall that we suppose that $k = 2$ and $\delta(\HHH) = 3$. There are two possibilities:
\begin{enumerate}
 \item $\U$ coincides with its conjugate $\U^{\scr h} \subset \HHH$. Then $\U = \U' \otimes_K F$ with an algebraic subgroup $\U' \subset \mathcal{N}$. It results that $\G$ is isomorphic to $(\HHH'/\U') \otimes_KF$. To be more precise, there is an isomorphism
$u: \G \longrightarrow (\HHH'/\U') \otimes_KF$ with the property that $u \circ p_G$ arises from the quotient map $\HHH' \longrightarrow \HHH'/\U'$ by extension of scalars.
The additivity of the $\delta$-invariant implies
$$\delta(\U) =  \delta(\HHH) - \delta(\G) = 3 - 2 = 1.$$
\textit{A fortiori}, 
$$\delta(\U') = \delta(\U) = 1.$$
This and the previous claim give 
$$\delta(\HHH'/\U') = \delta(\HHH') - \delta(\U') = 2 = \delta(\HHH) - \delta(\U) = \delta(\HHH/\U).$$
Thus, there is a simple quotient $\G'$ of $\HHH'/\U'$ such that $\delta(\G' \otimes_K F) = 1$. Let $v$ be the composition of $u$ with the quotient map to $\G' \otimes_K F$. Then $v$ is a required.
\item $\U$ does not coincide with its conjugate $\U^{\scr h} \subset \HHH$. Then $\V = \U + \U^{\scr h}$ is a subgroup of $\HHH$ which coincides with its conjugate. As a result, there is a subgroup $\V' \subset \HHH'$
with the property that $\V = \V' \otimes_K F$. Since $\delta(\U) = 1$, $\U \cap \U^{\scr h}$ is finite. So, $\delta(\V) = 2$ and $\delta(\HHH/\V) = 1$. Moreover, as $\U = ker\,p_{\G|H}$, it follows that 
 $\delta\big(p_{\G}(\V)\big) = \delta(\U^{\scr h}) = 1$. Consequently,
\begin{equation}
 \delta(\HHH/\V) = 1 = 2 - 1 = \delta(\G) - \delta\big(p_{\G}(\V)\big) = \delta\big(\G/p_{\G}(\V)\big).
\end{equation}
Recall that $p_{\G}$ restricts to a surjective homomorphism $ \HHH \longrightarrow \G$. Taking quotients, we receive a surjective homomorphism
$$u: \HHH/\V\ \longrightarrow \G/p_{\G}(\V).$$
Because of (4.3.5) the homomorphism $u$ is an isogeny. We choose an inverse isogeny $\nu$ with the property that $\nu \circ u$ is multiplication with an integer $k$. Next we let $p: \G \longrightarrow \G/p_{\G}(\V)$ and $q': \HHH' \longrightarrow \HHH'/\V'$ be the quotient maps, and write $q = q' \otimes_K F$.
Then
$\nu \circ u \circ p_{\G|\HHH} = q \circ [k]_{\HHH}$ is defined over $K$. Hence, $v = \nu \circ p $ is a required.
\end{enumerate}
The lemma is proved. 
\end{proof}
The final lemma in our trilogy is derived by considering subcases. 
\begin{lemma} \label{later4} Assume that $\G$ is a product $\prod_{j= 1}^k \G_j$ of simple algebraic groups $\G_j$ of respectively positive dimension.
Let ${\rm{H}}'$ be a proper algebraic subgroup of $\mathcal{N}_{K/F}(\rm{G})$ and suppose that $p_{\rm{G}}({{\rm{H}}}) = {{\rm{G}}}$ for ${{\rm{H}}} = {{\rm{H}}}' \otimes_K F$. If $\delta({\rm{H}})$ is odd, then there exist a group variety $\G'$ of positive dimension over $K$ and a surjective homomorphism $v: \G \longrightarrow \G' \otimes_K F$ such that $\G' \otimes_K F$ is simple over $F$ and with the property that $v_{\ast}\big(p_{\rm{G}|{\rm{H}}}\big) = v \circ p_{\rm{G}|{\rm{H}}}$ is defined over $K$.

\end{lemma}

\begin{proof} Because of the previous lemma we may assume that $k \geq 2$. We consider three subcases:\\
\\
If $k \geq 2$ and if $\U_2'$ is finite, then $\U_2$ is finite and $\HHH$ is isogenous to $\HHH^{\ast} = \mu_2(\HHH') \otimes_K F$. Hence, $\delta(\HHH^{\ast}) = \delta(\HHH)$ and $\delta(\HHH^{\ast})$ is odd. Since $[K:F] =2$, $\delta(\mathcal{N}_2 \otimes_K F)$ is even. It follows that 
$\HHH^{\ast}$ is a proper algebraic subgroup of $\mathcal{N}_2 \otimes_K F$.
Replacing $\G$ by $\G^{\ast} = \prod_{j = 2}^{\scr 2} \G_j$ and $\HHH'$ by $\mu_2(\HHH')$, we infer the statement by induction on $k$.\\
\\
If $k \geq 2$ and if $\U_2' = \mathcal{N}_1 \times \{e_{\mathcal{N}}\}$, then $\HHH^{\ast} $ must be a proper subgroup of $\mathcal{N}_2$ for otherwise $\HHH'$ would not be a proper subgroup of $\mathcal{N}$. As $\delta(\mathcal{N}_1 \otimes_F K) = 2$, we find that 
$$\delta\big(\HHH^{\ast}\big) = \delta(\HHH) - \delta(\mathcal{N}_1 \otimes_K F) = \delta(\HHH) - 2.$$
So, $\delta\big(\HHH^{\ast}\big)$ is odd and $\HHH^{\ast}$ is a proper subgroup of $\mathcal{N}_2 \otimes_K F$. Replacing $\G$ by $\G^{\ast} = \prod_{j = 2}^{\scr 2} \G_j$ and $\HHH'$ by $\mu_2(\HHH')$, we deduce the claim by induction on $k$.\\
\\
There is one subcase left: Assume that $k \geq 2$ and that $\U_2'$ is infinite, but a proper subgroup of $\mathcal{N}_1 \times \{e_{\mathcal{N}}\}$. Recall that $\mathcal{N} = \prod_{j = 1}^k \mathcal{N}_{F/K}(\G_j)$ and consider the subgroup $\HHH'_2  \subset \mathcal{N}_{F/K}(\G_2)$ arising
from projection of $\HHH' \subset \mathcal{N}$ to $\mathcal{N}_{F/K}(\G_2)$. There are two possibilities:
\begin{enumerate}
 \item $\HHH_2'$ is a proper algebraic subgroup of $\mathcal{N}_{F/K}(\G_2)$. 
Then $\HHH_2'$ cannot be finite for otherwise $p_{\G|\HHH}$ would not be surjective. So, $\HHH_2'$ is an infinite proper algebraic subgroup $\mathcal{N}_{F/K}(\G_2)$.
Replacing $\G$ by $\G_2$ and $\HHH'$ by $\HHH_2'$, we reduce to the case $k =1$ and $\delta(\HHH) = 1$ which was treated in the previous lemma.
\item $\HHH_2'$ equals $\mathcal{N}_{F/K}(\G_2)$. Then we write $\mathcal{N}_2^{\ast}$ for $\mathcal{N}_{F/K}(\G_2)$. We
replace $\G$ by $\G_1 \times \G_2$ and $\HHH'$ by its image $\V'$ in $\mathcal{N}_1 \times \mathcal{N}_2^{\ast}$ arising from the projection.  Moreover, we replace $\mu_2: \HHH' \longrightarrow \mathcal{N}_2$ by the projection $\mu_2^{\ast}: \V' \longrightarrow \mathcal{N}_2^{\ast}$. Let then $\V = \V' \otimes_K F$ and observe
that $\U_2' = ker\,\mu_2$ is canonically isomorphic to the kernel of $\mu_2^{\ast}$. The assumptions of our last case imply that $\delta(\U_2) = 1$. As a result,

$$\delta(\V) = \delta\big(\U_2\big) + \delta(\mathcal{N}_2^{\ast} \otimes_K F) = 1 + 2 = 3.$$
In other words, we have reduced ourselves to the case treated in the previous lemma.
\end{enumerate}
Everything is proved.
\end{proof}
\subsection{Proof of the proposition. Two corollaries} The three lemmas imply the main result of this section.
\begin{proof}[Proof of Proposition \ref{1234}] Because of Proposition 4.2.1\,\,we may assume that the commutative group variety $\G$ from the statement of Proposition \ref{1234} equals the maximal plurisimple quotient $\overline{\G} = \G/\G_{ps}$. Then $\G$ is isogenous to a product $\prod_{j=1}^k \G_j$ of simple algebraic groups $\G_j$ of respectively positive dimension. Since the $\delta$-invariant is constant on an isogeny class, the proposition is then reduced further to the case when $\G$ equals
$\prod_{j=1}^k \G_j$. In other words, we have reduced ourselves to the setting of the three lemmas above. If in this situation $\delta(\HHH)$ is even and if $\mu_2$ is surjective, then $k \geq 2$ and the proposition follows from Lemma \ref{later1}. If $\delta(\HHH)$ is even and if $\mu_2$ is not surjective, then
$k \geq 2$ and there are three possibilities:
\begin{enumerate} 
\item The algebraic group $\U_2' = ker\,\mu_2$ is finite. Then $\HHH'$ and $\mu_2(\HHH')$ are isogenous. We replace $\G$ by $\G^{\ast} = \prod_{j = 2}^k \G_j$ and $\HHH'$ by $\mu_2(\HHH')$. The proposition is then deduced by induction on $k$.
\item The algebraic group $\U_2'$ equals $\mathcal{N}_1 \times \{e_{\mathcal{N}_2}\}$. Then, letting $\HHH^{\ast} = \mu_2(\HHH')\otimes_K F$, $\delta(\HHH^{\ast}) = \delta(\HHH) - \delta(\mathcal{N}_1 \otimes_K F) = \delta(\HHH) - 2$ is even and we infer the proposition by induction on $k$ as in the first case.
\item The algebraic group $\U_2'$ is infinite, but a proper subgroup of $\mathcal{N}_1 \times \{e_{\mathcal{N}_2}\}$. Then $\delta(\HHH^{\ast}) = \delta(\HHH) - \delta(\U_2) = \delta(\HHH) - 1$ is odd, and we infer the proposition using Lemma \ref{later4}.

\end{enumerate}
So far, we have shown the proposition in the case when $\delta(\HHH)$ is even. Finally, if $\delta(\HHH)$ is odd, then the proposition follows from Lemma \ref{later4}.
\end{proof}
Let $\G$ be a commutative group variety over $F$ and let $w: {\rm{G}} \longrightarrow \prod_{j= 1}^k \G_j$ be a maximal homomorphism. 
\begin{corollary} \label{1234B} Assume the situation of the proposition. Suppose that, with notations as in the proposition, for each $v$ such that
$v_{\ast}(p_{\G|\HHH})$ is defined over $K$ the image $im\,v = \G' \otimes_K F$ is not simple. Then there are distinct $i, j = 1,...., k$ with the property that
${\rm{G}}_i$ is isogenous to ${\rm{G}}_j^{\scriptscriptstyle h}$.

\end{corollary}
\begin{proof} As was observed in the proof of the proposition, we may suppose that $w$ is the identity. Moreover, we may assume that, for all $j = 1,..., k$, $w({{\rm{H}}}') = \HHH'$ projects surjectively onto $\mathcal{N}_{F/K}({{{\rm{G}}}}_j)$. For otherwise we can reduce to the situation where $k = \delta(\HHH) = 1$, and Proposition \ref{1234} yields then a homomorphism $v$ with necessarily simple image. \\
\\
We will adopt the notations from the proof of the three lemmas and the proposition. If the image of $\mu_2$ is not surjective, then $k \geq 2$ and we replace
$\G$ by $\G^{\ast} = \prod_{j= 2}^k \G_j$ and $\HHH'$ by $\mu_2(\HHH')$, so that the assertion follows by induction on $\delta(\G)$. Thus, we will suppose that $\mu_2$ is surjective. If $ker\,\mu_2$ has dimension zero, then
$\delta({\rm{H}})$ is even and the assertion follows from Lemma \ref{later1}. If $\U_2' = ker\,\mu_2$ has positive dimension, then, as $\mu_2$ is surjective and $\HHH'$ is proper, $\U_2' $ is a proper algebraic subgroup of $\mathcal{N}_1 \times \{e_{\mathcal{N}_2}\}$ of the form
${{\rm{U}}}'' \times \{e_{\mathcal{N}_2}\}$. The resulting quotient group $\overline{{\rm{H}}} = {{\rm{H}}}/\U_2'$ is the graph
of a non-trivial quotient map $w: \mathcal{N}_2 \longrightarrow \mathcal{N}_1/{\rm{U}}''$ over $K$. Hence, for $j = 2,..., k$ there are algebraic homomorphisms $w_j, v_j: \G_j \longrightarrow (\mathcal{N}_1/{{\rm{U}}}'') \otimes_K F$, not all trivial, such that
\begin{equation}
 w(\underline{\xi}) = \sum_{j=2}^k w_j(\xi_j) + \sum_{j=2}^k v_j^{\scr h}(\xi_j^{\scr h})
\end{equation}
for each closed point 
$$\underline{\xi} = \big((\xi_2, \xi_2^{\scr h}),..., (\xi_k, \xi_k^{\scr h})\big) \in \mathcal{N}_2.$$
Recall that $\G_j$ is a simple algebraic group for all $j = 2,.., k$. So, a homomorphism $w_j$ resp.\,\,$v_j$ is non-zero if and only if it is an isogeny. Furthermore, the two simple subgroups ${{\rm{V}}} = \G_1 \times \{e_{\G_1^{\scr h}}\}$ and ${{\rm{U}}} = {{\rm{U}}}'' \otimes_K F$  of $\mathcal{N} \otimes_K F$ are distinct for ${{\rm{V}}} \neq {{\rm{V}}}^{\scr h}$, whereas ${\rm{U}} = {\rm{U}}^{\scr h}$. Hence, they have finite intersection. 
Consequently, if there is a $j_{\scr \ast} = 2,..., k$ such that $w_{j_{\scr \ast}} \neq 0$, then $w_{j_{\scr \ast}}$ is onto.
is surjective. As $\big(\mathcal{N}_1/{\rm{U}}''\big) \otimes_K F$ is isogenous to $\G_1$ via the map
$$\G_1 \stackrel{id. \times 0}{\longrightarrow} \G_1 \times \{e_{\G_1^{\scr h}}\} \hookrightarrow \big(\mathcal{N}_1/{\rm{U}}'\big) \otimes_K F,$$
the corollary follows by choosing $i =1$ and $j = j_{\scr \ast}$. On the other hand, if always $w_{j_{\scr \ast}} = 0$, then Statement (4.3.6) implies that there is a $j_{\scr \ast} = 2,..., k$ such that
$v_{j_{\scr \ast}} \neq 0$. Replacing $w_{j_{\scr \ast}}$ by $v_{j_{\scr \ast}}$ in the last argument, we infer the corollary in this situation. Everything is proved.
\end{proof}
From Lemma \ref{later4} one also gets
\begin{corollary} \label{1234C} Assume the situation of the proposition. Suppose that, with notations as in the proposition, for each $v$ such that
$v_{\ast}(p_{\G|\HHH})$ is defined over $K$ the image $im\,v = \G' \otimes_K F$ is not simple. Then $\delta(\HHH)$ is even.

\end{corollary}
\section{Proof of the main criterion for descent and weak descent. Two corollaries} We denote by $F$ a subfield of $\C$ and set $K = F \cap \R$. It is assumed that $F$ is stable with respect to complex conjugation and that $[F: K] = 2$. We let $\G$ be a commutative group
over $F$ and consider a non-zero real-analytic one-parameter homomorphism $\Psi$ with values in $\G(\C)$. The assertion of the main criterion for (weak) descent to $K$ (Theorem \ref{weakdescent}) is that the homomorphism $\Psi$ descends (weakly) to $K$ if and only if ${{\rm{H}}}$, the Zariski-closure of $\Psi_{\mathcal{N}}$, has dimension equal $\dim\,\G$ (resp.\,is a proper algebraic subgroup of $\G \times \G^{\scr h}$).
\begin{proof}[Proof of Theorem \ref{weakdescent}] Let $\mathcal{N} = \mathcal{N}_{F/K}(\G)$. By Corollary 3.3.3\,\,the homomorphism $\Psi_{\mathcal{N}}$ takes values in $\mathcal{N}(\R)$. As the morphism
$\rho_{\ast}: \mathcal{N}(\C) \longrightarrow \mathcal{N}(\C)$ from Sect.\,2.2 fixes real points (see Sect.\,2.3), we find that $\Psi_{\mathcal{N}}(\R) = \big(\Psi_{\mathcal{N}}(\R)\big)^{\scr h}$. So, $\HHH = \HHH^{\scr h}$. It follows that
$\HHH$ admits a model $\HHH' \subset \mathcal{N}$ over $K$.\\
\\
If $\dim\,\HHH = \dim\,\G$ then $w = p_{\G|\HHH}$ is an isogeny. Any isogeny
$v: \G \longrightarrow \HHH$ such that $v \circ w$ is multiplication with an integer is then as required in the definition of descent to $K$. This shows the first part of the theorem.\\
\\
The assertion concerning weak descents has an easy and a more difficult direction. The easy direction is the if-part. Namely, if there is a homomorphism $v: \G \longrightarrow \G' \otimes_K F$ onto an algebraic group of positive dimension such that $v_{\ast}(\Psi)(\R) \subset \G'(\R)$, then ${\rm{W}} = {\rm{W}}^{\scr h}$ for ${\rm{W}} = \G' \otimes_K F$. As the morphism
$\rho_{\ast}: \G'(\C) \longrightarrow \G'(\C)$ fixes real points, we find that 
$$v_{\ast}(\Psi) = \rho_{\ast} \circ v_{\ast}(\Psi) =\big(v_{\ast}(\Psi)\big)^{\scr h}.$$
Hence, the image of $\big(v_{\ast}(\Psi)\big)_{\mathcal{N}}$ is contained in the the set of real points of the diagonal 
$$\Delta = \mathcal{N}(id.)(\G') \otimes_K F \subset {\rm{W}} \times {\rm{W}} = {\rm{W}} \times {\rm{W}}^{\scr h} = \mathcal{N}_{F/K}\big({\rm{V}}\big) \otimes_K F.$$
So, ${\rm{H}} \subset v_{\mathcal{N}}^{\scr - 1}(\Delta)$ is a proper algebraic subgroup in $\G \times \G^{\scr h}$. The easy direction is shown. \\
We proceed to the proof of the more difficult direction which rests on the results from the two previous sections. To this end, we let $\HHH'$ be as in the previous proof and let
$\pi: \G \longrightarrow \overline{\G}$ be the universal homomorphism. Proposition \ref{upp} teaches that ${{\rm{V}}}' = \pi_{\mathcal{N}}({\rm{H}}')$ is a proper algebraic subgroup
of $\mathcal{N}_{F/K}(\overline{\G})$. Since $p_{\overline{\G}} \circ (\pi_{\mathcal{N}}\otimes_K F) = \pi \circ p_{\G}$, we have $\overline{\G} = p_{\overline{\G}}(\V)$ for $\V = \V' \otimes_K F$. So, Proposition \ref{1234} applied to ${{\rm{V}}}'$ implies then the only-if-part in Theorem \ref{weakdescent}. Theorem \ref{weakdescent} follows.
\end{proof}

In the definition of weak descent to $K$ (see Sect.\,2.3) one can assume that the target group ${\rm{G}}'$ is {simple} over $K$. For if $\G'$ is not simple, then $\G'$ admits a simple quotient of positive dimension and one may replace $v$ by the composition of the quotient map with $v$. According to Lemma\,\ref{cor789}, if $\G'$ is simple, then ${\rm{G}}' \otimes_K F$ is either simple over $F$ or it is isogenous to a product ${{\rm{U}}} \times {{\rm{U}}}^{\scr h}$ with ${{\rm{U}}}$ a simple algebraic group over $F$. In view of applications, it is important to examine when the first possibility can be achieved. 
\begin{example} \label{++++} Let $\rm{A}$ be a simple complex abelian variety of positive dimension. Assume that $\rm{A}$ is neither isogenous to its conjugate nor to an abelian variety definable over $\R$ and let $\psi: \R \longrightarrow \rm{A}(\C)$ be a non-zero real-analytic homomorphism. We know from the last chapter that
 the homomorphism $\Psi = \psi_{\mathcal{N}}$ to $({\rm{A}} \times {\rm{A}}^{\scr h})(\C)$ descends to $\R$. On the other hand, the assumptions on $\A$ imply that there is no algebraic group $\G'$ over $\R$ of dimension $0 < \dim\,\G' < \dim\,{\rm{A}} \times {\rm{A}}^{\scr h}$ such that $\G' \otimes_{\R} \C$ is a quotient of ${\rm{A}} \times {\rm{A}}^{\scr h}$. Hence, $\Psi$ cannot descend weakly to $\R$ via a simple quotient $\G' \otimes_{\R} \C$.
\end{example} 
Situations similar to the one in the example will appear frequently in applications, and the corollaries below are designed to deal with them. As in Theorem \ref{weakdescent} we assume that $F$ is stable with respect to complex conjugation and that $[F:K] = 2$. We consider a homomorphism $\Psi$ to $\G(\C)$ and let $\rm{H}$ be the Zariski-closure of $\Psi_{\mathcal{N}}(\R)$ over $F$. Let $v: {\rm{G}} \longrightarrow \prod_{1=j}^k {{{\rm{G}}}}_j$ be an arbitrary maximal homomorphism.
 \begin{corollary} \label{weakdescentodd1} \label{321} If ${{\rm{H}}}$ is a proper algebraic subgroup of $\G \times \G^{\scr h}$, then $\Psi$ descends weakly to $K$ via a simple quotient ${\rm{G}}' \otimes_K F$ unless there are two distinct $i, j = 1,..., k$ such that $\G_i$ and the complex conjugate of $\G_j$ are isogenous. 
 \end{corollary}
 \begin{corollary} \label{weakdescentodd} If $\delta({\rm{H}})$ is odd, then $\Psi$ descends weakly to $K$ via a simple quotient ${\rm{G}}' \otimes_K F$.
\end{corollary}
These two consequences follow from Corollary \ref{1234B} and Corollary \ref{1234C}.

\chapter{Inherited descent} 
This chapter is devoted to the proof of Theorem \ref{inherited} and Theorem \ref{inherited1}. We first formulate two algebraic results which, morally speaking, form the skeleton of the two theorems. Then, in the third section, we show Theorem \ref{inherited} and Theorem \ref{inherited1}.
 
\section{Inhertited $\mathbf{K}$-structures} We let $F/K$ be a Galois extension of fields of degree $[F:K] = 2$ and denote by $h$ the generator of $Gal(F|K)$. We consider a surjective homomorphism $\pi: \G \longrightarrow {\rm{U}}$ of group varieties over $F$. 
The first proposition arises in the context of inherited descents.
\begin{prop} \label{skeleton} Assume that there is an algebraic group $\G'$ over $K$ such that $\G$ is the extension $\G' \otimes_K F$ to scalars over $F$. If
${\rm{Hom}}\big(ker\,\pi, {{\rm{U}}}^{\scriptscriptstyle h}\big)$ is a torsion group, then there is an algebraic group ${\rm{U}}'$ over $K$
and an isogeny $v: {\rm{U}} \longrightarrow {\rm{U}}' \otimes_K F$ with the property that $v_{\ast}(\pi) = v \circ \pi$ is defined over $K$.

\end{prop} 
We recall once again that by Lemma \ref{cl1} the homomorphism $v_{\ast}(\pi) = v \circ \pi$ is defined over $K$ if and only if $v_{\ast}(\pi) = \big(v_{\ast}(\pi)\big)^{\scr h}$.
\subsection{Table of symbols used during the proof}
\begin{center}
\begin{tabular}{lp{7cm}ll p{7cm} }
$\mathcal{N}$ & $ = \mathcal{N}_{F/K}(\G)$ & & $i$ &$ = i' \otimes_K F$ \\

 $\mathcal{N}({\rm{U}})$ &$ = \mathcal{N}_{F/K}({\rm{U}})$ && ${{\rm{H}}}$ & $ = {{\rm{H}}}' \otimes_K F$\\
 ${{{\rm{L}}}}$ &$ = ker\,\pi$ && ${{\rm{H}}}_{{{\rm{L}}}}'$ &$ = {{\rm{H}}}' \cap \mathcal{N}({{{\rm{L}}}})$\\
  $\mathcal{N}({{\rm{L}}}) $ & $= \mathcal{N}_{F/K}({{\rm{L}}}) = ker\,\pi_{\mathcal{N}}$ && ${{\rm{H}}}_{{{\rm{L}}}}$ &$= {{\rm{H}}}_{{{\rm{L}}}}' \otimes_K F$   \\
  $i'$ & $=  {\mathcal{N}}(id.),$ the hom.\,\,from $\G'$ to $\mathcal{N}$ && $p$ & $ = \pi_{\mathcal{N}} \otimes_K F$\\
   ${\rm{H}}'$ &$ = i'\big(\G'\big)$ && $q$ &$ = \pi_{\ast}(p_{\G}) = \pi \circ p_{\G}$ \\
 
          \end{tabular}
          \end{center}
Note that $p_{\G} \circ i$ is the identity of $\G$, so that $p_{\G|\HHH}$ is an isomorphism.
\subsection{Two auxiliary lemmas} The proof of the proposition is divided into two lemmas. 
\begin{lemma} With notations as above the equality $\dim\,{{\rm{H}}}_{{{{\rm{L}}}}} = \dim\,{{\rm{L}}}$ holds.
\end{lemma}
\begin{proof} We have a commutative diagram of exact sequences
$$\begin{xy} 
  \xymatrix{ 
  0 \ar[r] & \mathcal{N}({{\rm{L}}}) \otimes_K F \ar[r] \ar[d]^{p_{{{\rm{L}}}}} & \ar[rr]^{p} \mathcal{N}\otimes_K F \ar[d]^{p_{\rm{G}}}  && \mathcal{N}({{\rm{U}}})\otimes_K F \ar[d]^{p_{{{\rm{U}}}}} \\
   0 \ar[r] & {{{\rm{L}}}} \ar[r]  & \ar[rr] \rm{G} \ar[rr]^{\pi} && {{\rm{U}}}  } 
\end{xy}$$
Let ${{{\rm{W}}}}$ be the subgroup ${{\rm{H}}} \cap p_{\rm{G}}^{\scr -1}({{{\rm{L}}}})$ of ${{{\rm{L}}}} \times {\rm{G}}^{\scriptscriptstyle h}$. Observe that ${{\rm{H}}}_{\rm{L}} \subset {\rm{W}}$. Our first task is to show that ${\rm{W}}/{{\rm{H}}}_{\rm{L}}$ is finite. Since $p_{\G|{{\rm{H}}}}$ is an isomorphism, the restriction $u = p_{{\rm{G}}|{{{\rm{W}}}}}$ is an isomorphism onto its image $\LL$. The inverse map $u^{\scr -1}$ yields a homomorphism
$$\mu = \pi^{\scriptscriptstyle h} \circ p_{{\rm{G}}^{\scriptscriptstyle h}} \circ u^{\scr -1} = q^{\scr h}\circ u^{\scr -1} $$
form ${{{\rm{L}}}}$ to ${{\rm{U}}}^{\scriptscriptstyle h}$. The hypotheses of the proposition imply that $\mu(\LL)$ is a finite torsion group in ${{\rm{U}}}^{\scriptscriptstyle h}$. Since $q^{\scr h}({\rm{W}}) = \mu(\LL)$, we get that
$q^{\scr h}({\rm{W}})$ is a finite torsion group in ${{\rm{U}}}^{\scriptscriptstyle h}$. The same holds for the image $q({{\rm{W}}}) \subset {\rm{U}}$ because ${{{\rm{W}}}} \subset {{{\rm{L}}}} \times {\rm{G}}^{\scriptscriptstyle h}$ and as $p_{\G}(\LL \times \G^{\scr h}) = \LL = ker\,\pi$.
As $p = \pi_{\mathcal{N}} \otimes_K F = (\pi, \pi^{\scr h})$\footnote{For the meaning of ``$(\pi, \pi^{\scr h})$'' compare (4.2.1).}, it follows that 
\begin{equation}
\dim\,p(\rm{W}) = 0.
\end{equation}
Recall that
\begin{equation}{{\rm{H}}}_{{{\rm{L}}}} = {{\rm{H}}} \cap \big(\mathcal{N}({{{\rm{L}}}}) \otimes_K F\big) = {{\rm{H}}} \cap ker\,p.\end{equation}
We observed above that ${{\rm{H}}}_{\rm{L}} \subset {\rm{W}}$. This implies together with Statement (5.1.1) and Statement (5.1.2) that ${{{\rm{W}}}}/{{\rm{H}}}_{{{\rm{L}}}}$ is finite. So, $\dim\,{{\rm{H}}}_{{{\rm{L}}}} = \dim\,{{{\rm{W}}}}$. Moreover, since ${{\rm{W}}}$ isomorphic to ${{\rm{L}}}$ via $p_{\G|\rm{W}}$, we get that $\dim\,{{\rm{W}}} = \dim\,{\LL}$. It follows that
$\dim\,{{\rm{H}}}_{{{\rm{L}}}} = \dim\,{{\rm{L}}}$.
\end{proof}
We set ${{\rm{V}}}' = \pi_{\mathcal{N}}({{\rm{H}}}')$, $\V = \V' \otimes_K F$ and let $w: {{\rm{V}}} \longrightarrow {{\rm{U}}}$ be the restriction of $p_{{{\rm{U}}}}$ to $\V$. 
\begin{lemma} The homomorphism $w$ is an isogeny.
 
\end{lemma}
\begin{proof} We have
$$
\begin{array}{cll}
\dim\,{\rm{U}}' & = \dim\,{{\rm{H}}}' - \dim\,\big({{\rm{H}}} \cap ker\,\pi_{\mathcal{N}}\big)& \mbox{(by additivity of dimension)}\\
& = \dim\,{{\rm{H}}}' - \dim\,{{\rm{H}}}_{{{\rm{L}}}}' & \mbox{(by definition)}\\
&= \dim\,\rm{G}- \dim\,{{\rm{L}}} & \mbox{(by the previous lemma)}\\
&= \dim\,{\rm{U}}.& \mbox{(by additivity of dimension)} 
\end{array}
$$
Moreover, since the restriction $p_{\G|{\rm{H}}}$ is surjective, so is the composition $\pi \circ p_{\G|{\rm{H}}}$. Recalling that $p_{{\rm{U}}} \circ p = \pi \circ p_{\G}$, 
it follows that $w$ is surjective. This and the above identity of dimensions imply the lemma.
\end{proof}
\subsection{Proof of the proposition} We define $\nu' = \pi_{\mathcal{N}} \circ i'$ and $\nu = \nu' \otimes_K F$. Then $\nu'$ is a homomorphism from $\G' $ to $\V' \subset \mathcal{N}({\rm{U}})$. Finally we choose an isogeny $v: {{\rm{U}}} \longrightarrow \V$ such that $v \circ w$ is multiplication with an integer $k$. Recalling that $\pi \circ p_{\G} = p_{\U} \circ p$ and that 
$p_{\G} \circ i$ is the identity on $\G$, we calculate then
\begin{center}
\begin{tabular}{l}
$v_{\ast}(\pi) = v \circ \pi =  v \circ (\pi \circ p_{\G}) \circ i$\\ 
$= v \circ (p_{\U} \circ p) \circ i = v \circ p_{\U} \circ (p \circ i)$\\ 
$= v \circ p_{\U} \circ \nu = v \circ w \circ \nu = [k]_{\V} \circ \nu.$
\end{tabular}
\end{center}
Hence, is $v_{\ast}(\pi) = [k]_{\V} \circ \nu $ defined over $K$. Letting ${\rm{U}}' = \V'$, $v$ is then as required. 
\section{Inherited weak $\mathbf{K}$-structures} The next proposition is related to weak inherited descents. The setting is the same as in the previous section.
\begin{prop} \label{skeleton1} Let ${\rm{H}}' \subset \mathcal{N}_{K/F}(\G)$ be a connected algebraic subgroup and set ${{\rm{H}}} = {{\rm{H}}}' \otimes_K F$. If $\dim\,{\rm{H}} < \dim\,\G + \dim\,{\rm{U}}$ and if
${\rm{Hom}}\big({{\rm{U}}}^{\scriptscriptstyle h}, {\rm{M}}\big) = \{0\}$ for each quotient $\rm{M}$ of $ker\,\pi$, then there is an algebraic group ${\rm{U}}'$ of positive dimension over $K$
and a surjective homomorphism  $v: {\rm{U}} \longrightarrow {\rm{U}}' \otimes_K F$ with the property that $v_{\ast}(\pi \circ p_{\G|{\rm{H}}})$ is defined over $K$. 
\end{prop}
\subsection{Table of symbols used during the proof}
We start by fixing the most important symbols. They are quite the same as in the previous proof and repeated here for the convenience only.\\\begin{center}
\begin{tabular}{lp{7cm}llp{7cm} }
$\mathcal{N}$ & $ = \mathcal{N}_{F/K}(\G)$ && ${{\rm{H}}}_{{{\rm{L}}}}$ & $= {{\rm{H}}}_{{{\rm{L}}}}' \otimes_K F$\\

 $\mathcal{N}({\rm{U}})$ &$ = \mathcal{N}_{F/K}({\rm{U}})$ &&  $p$ & $ = \pi_{\mathcal{N}} \otimes_K F$ \\
 ${{{\rm{L}}}}$ &$ = ker\,\pi$ &&   $q$ &$ = \pi_{\ast}(p_{\G}) = \pi \circ p_{\G}$\\
  $\mathcal{N}({{\rm{L}}}) $ & $= \mathcal{N}_{F/K}({{\rm{L}}}) = ker\,\pi_{\mathcal{N}}$ &&  ${{{\rm{V}}}'}$ &$ = \pi_{\mathcal{N}}({\rm{H}}')$\\ 
  
    ${{\rm{H}}}$ & $ = {{\rm{H}}}' \otimes_K F$ &&  $\V$ & $=\V' \otimes_K F$    \\
    
        ${{\rm{H}}}_{{{\rm{L}}}}'$ &$ = {{\rm{H}}}' \cap \mathcal{N}({{{\rm{L}}}})$&&&\\
        &&&&\\
              \end{tabular}
          \end{center}

\subsection{An auxiliary lemma} We begin the proof of the proposition with an auxiliary result. 
\begin{lemma} \label{refff} If ${{\rm{V}}}' = \mathcal{N}({{\rm{U}}})$, then $p_{{{{\rm{L}}}}}({{\rm{H}}}_{{{{\rm{L}}}}}) = {{\rm{L}}}$.
\end{lemma}
\begin{proof} If ${{\rm{V}}}' = \mathcal{N}({{\rm{U}}})$, then $\pi_{\mathcal{N}}({{\rm{H}}}') =  \mathcal{N}({{\rm{U}}})$. Consider the commutative diagram of exact sequences
 $$\begin{xy} 
  \xymatrix{ 0 \ar[r] & {{\rm{H}}}_{{{{\rm{L}}}}}\ar@^{(->}[d] \ar[r] & \ar[rr]^{w} {{\rm{H}}} \ar@^{(->}[d] && \mathcal{N}({{\rm{U}}})\otimes_K F\ar[d]^{id.} \ar[r] &  0\\
 0 \ar[r] & \mathcal{N}({{\rm{L}}}) \otimes_K F \ar[r] \ar[d]^{p_{{{\rm{L}}}}} & \ar[rr]^{p} \mathcal{N}\otimes_K F \ar[d]^{p_{\rm{G}}}  && \mathcal{N}({{\rm{U}}})\otimes_K F \ar[d]^{p_{{{\rm{U}}}}} \ar[r]& 0\\
   0 \ar[r] & {{\rm{L}}} \ar[r]^j  & \ar[rr] \G \ar[rr]^{\pi} && {{\rm{U}}}  \ar[r] & 0} 
\end{xy}$$
Let $\overline{{{{\rm{L}}}}} = {{{\rm{L}}}}/p_{{{{\rm{L}}}}}({{\rm{H}}}_{{{{\rm{L}}}}})$, $\overline{{\rm{G}}} = {\rm{G}}/p_{{\rm{G}}}({{\rm{H}}}_{{{{\rm{L}}}}})$ and $\overline{{{\rm{H}}}} = {{\rm{H}}}/{{\rm{H}}}_{{{{\rm{L}}}}}$. Taking quotients, we receive a commutative diagram

$$\begin{xy} 
  \xymatrix{  &  & \ar[rr]^{\overline{w}} \overline{{{\rm{H}}}} \ar[d]^v && \mathcal{N}({{\rm{U}}})\otimes_K F\ar[d]^{p_{{{\rm{U}}}}}\\
 0 \ar[r] & \overline{{{\rm{L}}}} \ar[r]^{\overline{j}}  & \ar[rr] \overline{\rm{G}} \ar[rr]^{\overline{\pi}} && {{\rm{U}}}  } 
\end{xy}$$
Observe that $\overline{w}$ is an isomorphism and let 
$$s: {{\rm{U}}} \longrightarrow \mathcal{N}({{\rm{U}}}) \otimes_K F = {{\rm{U}}} \times {{\rm{U}}}^{\scriptscriptstyle h}\,\,\mbox{and}\,\,s^{\scriptscriptstyle h}: {{\rm{U}}}^{\scriptscriptstyle h} \longrightarrow \mathcal{N}({{\rm{U}}}) \otimes_K F$$
be the natural inclusions onto the first and second factor. Then $\sigma_1 = v \circ \overline{w}^{\scriptscriptstyle -1} \circ s$ is a section of $\overline{\pi}$, so that $\sigma_1({\rm{U}}) \cap \overline{j}(\overline{L}) = 0$. Hence, the lower sequence in the last diagram splits and there exists a homomorphism
$\mu: \overline{\G} \longrightarrow \overline{{\rm{L}}}$ such that $\mu \circ \overline{j}$ is the identity and with the property that $\mu \circ \sigma_1$ is trivial. Since the restriction $p_{{\rm{G}}|{\rm{H}}}$ is surjective, so is $v = \overline{p_{{\rm{G}}|{\rm{H}}}}$. Write $\sigma_2 = v \circ \overline{w}^{\scriptscriptstyle -1} \circ s^{\scriptscriptstyle h}$. 
As $v$ is surjective, we get
$$\sigma_1({\rm{U}}) + \sigma_2({\rm{U}}^{\scr h}) = (v \circ \overline{w}^{\scriptscriptstyle -1})\big(\mathcal{N}({\rm{U}})\otimes_K F) = v(\overline{\rm{H}}) = \overline{\G}.$$
Recalling that $\sigma_1({\rm{U}}) \cap \overline{j}(\overline{L}) = 0$, we infer that $(\mu \circ\sigma_2)({\rm{U}}^{\scr h}) = \overline{\LL}$. The hypotheses of the proposition applied to $\rm{M} = \overline{\LL}$ yield $\mu \circ \sigma_2 = 0$. Hence, $\overline{{{\rm{L}}}} = 0$. The claim follows.
\end{proof}
\subsection{Proof of the proposition} Assume that ${{\rm{V}}}' = \mathcal{N}({{\rm{U}}})$. The assumptions of the proposition imply that
\begin{equation}
\dim\,\LL + \dim\,\mathcal{N}(\rm{U}) = \dim\,\LL + 2\dim\,{\rm{U}} = \dim\,\G + \dim\,{\rm{U}} > \dim\,{\rm{H}}'.
\end{equation}
Additivity of dimensions yields
\begin{equation}\dim\,{\rm{V}}' = \dim\,{\rm{H}}' - \dim\,{\rm{H}}'_{\rm{L}}.
\end{equation}
And from the last lemma it results that 
\begin{equation}{{\rm{V}}}' = \mathcal{N}({{\rm{U}}})\Longrightarrow \dim\,{{\rm{H}}}_{{{{\rm{L}}}}}' \geq \dim\,{{\rm{L}}}.
\end{equation}
Statements (5.2.1)-(5.2.3) show that 
$${{\rm{V}}}' = \mathcal{N}({{\rm{U}}})\Longrightarrow{{\rm{V}}}' \neq \mathcal{N}({{\rm{U}}}).$$ 
So, ${{\rm{V}}}' \neq \mathcal{N}({{\rm{U}}}).$ Proposition\,\ref{1234} applied to ${{\rm{V}}}'$ (instead of $\HHH'$) completes the proof.

\section{Proof of the theorems about inherited descent} Let $\Psi$ and $\G$ be as in Theorem\,\ref{inherited}. Then there exists an isogeny $v: \G \longrightarrow \G' \otimes_K F$ such that $v_{\ast}(\Psi)$ takes values in $\G'(\R)$.
Let $w: \G' \otimes_K F \longrightarrow \G$ be an inverse isogeny with the property that $v \circ w$ is multiplication with an integer. Since $w$ is \'etale, $\Psi$ factors through the map $w: \G'(\R) \longrightarrow \G(\C)$ as, say $\Psi = w_{\ast}(\Psi')$. 
One hypopaper of the theorem is that ${\rm{Hom}}\big(ker\,\pi, {{\rm{U}}}^{\scriptscriptstyle h}\big) = 0$. As $w$ is finite, replacing $\pi$ by $\pi_{\ast}(w)$ and $\Psi$ by $\Psi'$ affects
this hypopaper to the effect that ${\rm{Hom}}\big(ker\,\pi_{\ast}(w), {{\rm{U}}}^{\scriptscriptstyle h}\big)$ is a (maybe non-trivial) torsion group. Proposition \ref{skeleton} implies that $(\pi\circ w)_{\ast}(\Psi') = \pi_{\ast}(\Psi)$ descends to $K$. Theorem\,\ref{inherited} follows.\\
For the proof of Theorem\,\ref{inherited1} we recall that $\rm{H} \subset \G$
is the smallest algebraic subgroup with the property that $\Psi_{\mathcal{N}}(\R) \subset \rm{H}(\C)$. We have seen in the proof of Theorem \ref{weakdescent} that $\HHH$ admits a model $\HHH' \subset \mathcal{N}_{F/K}(\G)$.
With this the theorem is then an immediate consequence of Proposition \ref{skeleton1}. 

\chapter{Proof of the transcendence theorems} We recall the setting of our theorems. We consider a surjective homomorphism $\pi: \G \longrightarrow {\rm{U}}$ of algebraic group varieties over an algebraically closed subfield $F$ of $\C$ which is stable with respect to complex conjugation $h$.
We let $\Psi: \R \longrightarrow \G(\C)$ be a real-analytic homomorphism with Zariski-dense image and denote by $\frak{t} \subset \frak{g}$ the smallest subspace over $K = F \cap \R$
with the property that $\Psi_{\ast}(\R) \subset \frak{t} \otimes_K \R$. The respectively first theorem is about descents of $\pi_{\ast}(\Psi)$, whereas the respectively second theorem deals with weak descents of $\pi_{\ast}(\Psi)$. 
\section{Preparations for the proofs}
\subsection{Some auxiliary lemmas} The results stated in this section do not rely on transcendence theory and will be used several times in the proof of the theorems. Together with Theorem \ref{inherited} and Theorem \ref{inherited1}
they exhibit the purely geometric input of the story.\\  
We set $\mathcal{N} = \mathcal{N}_{\overline{\Q}/K}({\rm{G}})$ and $\frak{n} = {{\rm{Lie}}}\,\mathcal{N}$. The symbol $\Psi_{\mathcal{N}}$ was defined in Sect.\,2.2. We also remark that most of the lemmas hold true even if $F$ is not algebraically closed.
\begin{lemma}\label{ff} The real-analytic homomorphism $\Psi_{\mathcal{N}}: \R \longrightarrow \mathcal{N}(\C)$ has the following properties.

\begin{enumerate}
\item The image of $\Psi_{\mathcal{N}}$ is contained in $\mathcal{N}(\R)$ and the image of $(\Psi_{\mathcal{N}})_{\ast}$ is contained in $\frak{n}(\R)$.

 \item $p_{\rm{G}} \circ \Psi_{\mathcal{N}} = \Psi$ and $\Psi^{\scr -1}\big({\rm{G}}(F)) \subset \Psi_{\mathcal{N}}^{\scr -1}\big(\mathcal{N}(K)\big)$.
\item $\big(\Psi_{\mathcal{N}}\big)_{\ast}(\R) = \left\{(v,v^{\scriptscriptstyle h}); v \in \Psi_{\ast}(\R)\right\}$.
\item If $\Phi_{\mathcal{N}}$ is the complexification of $\Psi_{\mathcal{N}}$ according to Proposition \ref{complex},\\
then $\Psi_{[i]}^{\scr -1}\big({\rm{G}}(F)) \subset \Phi_{\mathcal{N}}^{-1}\big(\mathcal{N}(F)\big)$. 
\end{enumerate}
  
\end{lemma}
\begin{proof} Theorem\,\ref{Wr1} and Corollary 3.3.3 yield Statement 1. Theorem\,\ref{Wr1} and Corollary \ref{3.3.1} imply Statement 2. To show the third assertion, we consider the commutative diagram
$$\begin{xy} 
  \xymatrix{ \R \ar[dd]^{\Psi_{\mathcal{N}}}\ar[ddrrr]^{\Psi} & & & \\
&&&\\
\mathcal{N}(\R) \ar@{(->} [r]& \mathcal{N}(\C) = ({\rm{G}} \times {\rm{G}}^{\scriptscriptstyle h})(\C) \ar[rr]^{\,\,\,\,\,\,p_{\rm{G}}} && {\rm{G}}(\C)}
\end{xy}$$
By Corollary \ref{lp333} it induces a commutative diagram of Lie algebras
$$\begin{xy}\xymatrix{
 \R \ar[dd]^{\big(\Psi_{\mathcal{N}}\big)_{\ast}}\ar[ddrrr]^{\,\,\,\,\,\,\Psi_{\ast}} & & \\
&&&\\
\frak{n}(\R) \ar@{(->} [r]& \frak{n}(\C) = \big(\frak{g}\oplus \frak{g}^{\scriptscriptstyle h}\big)(\C) \ar[rr]^{\,\,\,\,\,\,\,\,\,\,(p_{\rm{G}})_{\ast}} && \frak{g}(\C)}
\end{xy}.$$
We have $\frak{n}(\R) = \left\{(\partial,\partial^{\scriptscriptstyle h}); \partial \in \frak{g}(\C) \right\}$ by Lemma \ref{WR2}. With this Statement 3.\,\,follows. We move to the fourth assertion. By Faltings-W\"ustholz \cite{FaWu} we can choose an embedding of ${\rm{G}}$ into $\mathbb{P}^N_F$. Then the real-analytic homomorphism $\Psi$ is locally uniformized by tuples of power series
$$\left(\sum_{l \geq 0} a_{l0}r^l,..., \sum_{l \geq 0} a_{lN}r^l\right).$$
It follows then from Subsect.\,3.2.1 that $\Psi^{\scriptscriptstyle h}$ is locally uniformized by tuples
$$\left(\sum_{l \geq 0} h(a_{l0})r^l,..., \sum_{l \geq 0} h(a_{lN})r^l\right).$$
Let $r \in \R$ be such that $\Psi_{[i]}(r) \in {\rm{G}}(F)$. Then
$$\sum_{l \geq 0} a_{lj}(ir)^l\big/\sum_{l \geq 0} a_{lk}(ir)^l \in  F$$
for all $j = 0,..., N$ and all $k = 0,..., N$ such that $\sum_{l \geq 0} a_{kl}(ir)^l \neq 0$. As a result,
$$\sum_{l \geq 0} h(a_{lj})(-ir)^l\big/\sum_{l \geq 0} h(a_{lk})(-ir)^l \in F^{\scriptscriptstyle h} = F.$$
Hence, $\big(\Psi_{[i]}\big)^{\scr h}(-r)\in {\rm{G}}^{\scriptscriptstyle h}(F)$. So, 
$$ \Phi_{\mathcal{N}}(ir) = \Psi_{[i]}(r) \times \big(\Psi_{[i]}\big)^{\scr h}(-r) \in \big(\G \times \G^{\scr h}\big)(F),$$
and we deduce Statement 4.
\end{proof}
We denote by $\frak{t}_{\mathcal{N}} \subset \frak{n}$ the smallest subspace over $K$ such that $({\Psi}_{\mathcal{N}})_{\ast}(\R) \subset \frak{t}_{\mathcal{N}} \otimes_K \R$. 
\begin{lemma}\label{op=} The equality $\dim\,\frak{t}_{\mathcal{N}} = \dim \frak{t}$ holds. 
\end{lemma}
\begin{proof} Recall the homomorphism 
$${{\rm{Lie}}}(\rho): \frak{g}(\C) \longrightarrow \frak{g}^{\scriptscriptstyle h}(\C), {{\rm{Lie}}}(\rho)(v) = v^{\scriptscriptstyle h}$$
from Sect.\,3.3 and the canonical isomorphisms
$$\frak{n}(\C) = \big(\frak{g} \oplus \frak{g}^{\scriptscriptstyle h}\big)(\C)\,\,\mbox{and}\,\,\frak{n}(\R) = \{(v, v^{\scriptscriptstyle h}); v \in \frak{g}(\C)\}$$
from Lemma\,\ref{WR2}. Lemma 3.3.1 implies that ${\rm{Lie}}(\rho)$ is defined over $K$, that is, ${{\rm{Lie}}}(\rho)\big(\frak{g}\big) = \frak{g}^{\scriptscriptstyle h}$ where we identify $\frak{g}$ (resp.\,$\frak{g}^{\scr h}$) 
with its canonical image in $\frak{g}(\C)$ (resp.\,$\frak{g}^{\scr h}(\C)$). Together with Statement 3.\,in Lemma\,\ref{ff} we find that 
$$\frak{t}_{\mathcal{N}} \otimes_K \R \subset \left\{\big(v, v^{\scriptscriptstyle h}\big); v \in \frak{t} \otimes_K \R\right\}.$$
The claim follows by linear algebra.
\end{proof}
The $K$-vector space $\frak{t}_{\mathcal{N}}$ generates a linear subspace $\frak{T}  = \frak{t}_{\mathcal{N}} + i\frak{t}_{\mathcal{N}}$ of $\frak{n}$ over $F$. We denote by $\dim\,\frak{T} $ its dimension over $F$.
\begin{lemma} \label{op(} 
The vector space $\frak{T}$ has the following properties.
\begin{enumerate}
\item $\frak{T}$ is defined over $F$ and it is the smallest subspace of $\frak{n}$ over $F$\\
 which contains $\frak{t}_{\mathcal{N}}$. 
\item $({\Psi}_{\mathcal{N}})_{\ast}(\R) \subset \frak{T} \otimes_F {\C}$.
 \item $\dim\,\frak{t} = \dim\,\frak{T}.$ \footnote{For the proof of the theorems one only needs the obvious estimate ``$\dim\,\,\frak{T} \leq \dim\,\frak{t}.$'' The equality will be used once in the proof of Corollary 7.1.19.}
\end{enumerate}
\end{lemma}
\begin{proof} The first two statements are clear. We move to the third statement. Since $\Psi_{\mathcal{N}}(\R) \subset \mathcal{N}(\R)$, we have $\frak{t}_{\mathcal{N}} \subset \frak{n}$ by Corollary \ref{lp333}.  
As ${\frak{n}} \in {\cal{H}}\big({\frak{n}}(F)\big)$, a basis of $\frak{t}_{\mathcal{N}}$ over $K$ is also basis of $\frak{T}$ over $F$.\footnote{The symbol ``$\cal{H}(-)$'' has been defined in Sect.\,3.3.}
Together with the previous lemma we find $\dim\,\frak{t} = \dim\,\frak{t}_{\mathcal{N}} = \dim\,\frak{T}$. 
\end{proof}
We let ${{{\rm{H}}}}$ be the Zariski-closure of $\Psi_{\mathcal{N}}(\R)$ in $\mathcal{N} \otimes_K F$, that is, ${\rm{H}}$ is the smallest algebraic subgroup
of $\mathcal{N} \otimes_K F$ such that $\Psi_{\mathcal{N}}(\R) \subset {\rm{H}}(\C)$. Recall the projection $p_{\rm{G}}: \mathcal{N}_{F/K}({\rm{G}}) \otimes_K \overline{\Q} \longrightarrow \rm{G}$.
\begin{lemma} \label{definition} The homomorphism $p_{\G|{\rm{H}}}: {\rm{H}} \longrightarrow \G$ is surjective and there is an algebraic subgroup ${{\rm{H}}}' \subset \mathcal{N}$ with the property that ${{\rm{H}}} = {{\rm{H}}}' \otimes_K F$ and $\Psi_{\mathcal{N}}(\R) \subset {{\rm{H}}}'(\R)$.
 
\end{lemma}
\begin{proof} The first statement holds, because $\Psi$ has a Zariski-dense image in $\G(\C)$. Since $\Psi_{\mathcal{N}}(\R)$ is contained in $\mathcal{N}(\R)$, we have
${{{\rm{H}}}} = {{{\rm{H}}}}^{\scriptscriptstyle h}$. This has been already observed in Sect.\,4.4. The second assertion follows.
\end{proof}
We define $\Phi_{\mathcal{N}}$ to be the complexification of $\Psi_{\mathcal{N}}$ according to Proposition \ref{complex}
and let $\rm{r} = {\rm{rank}}_{\Z}\,\Psi^{\scr -1}\big((\G(F)\big) + {\rm{rank}}_{\Z}\,\Psi_{[i]}^{\scr -1}\big(\G(F)\big)$.
\begin{lemma} \label{rank} The following assertions hold. 
\begin{enumerate}
 \item $\Phi_{\mathcal{N}}(\C) \subset \HHH(\C)$ and $\big(\Phi_{\mathcal{N}}\big)_{\ast}(\C) \subset \frak{T} \otimes_F {\C}$.

\item ${\rm{rank}}_{\Z} \,\Psi_{\mathcal{N}}^{\scr -1}\big(\mathcal{N}(K)\big) \geq {\rm{rank}}_{\Z} \,\Psi^{\scr -1}\big(\G(F)\big)$.
\item ${\rm{rank}}_{\Z}\,\Phi_{\mathcal{N}}^{\scr -1}\big(\mathcal{N}(K)\big) \geq {\rm{r}}$.
\end{enumerate} 
 
\end{lemma}
\begin{proof} The first inclusion in Statement 1.\,\,is a consequence of the previous lemma. The second one follows from Statement 2.\,\,in Lemma \ref{op(} by tensoring with $\C$. The remaining assertions result from Statement 1.\,\,and Statement 3.\,\,in Lemma \ref{ff}. 
\end{proof}
\begin{lemma} \label{k} Let ${\rm{k}} = {\rm{rank}}_{\Z}\,ker\,\Psi$. Then ${\rm{k}} = {\rm{rank}}_{\Z}\,ker\,\Psi_{\mathcal{N}}$.
 
\end{lemma}
\begin{proof} Since $(\Psi(r))^{\scr h} = \Psi^{\scr h}(r)$ for all $r \in \R$, we have 
$ker\,\Psi_{\mathcal{N}} = ker\,\Psi$. This gives the lemma.
\end{proof}
Finally, we consider a decomposition
\begin{equation}
\G \backsimeq \G_c \times \mathbb{G}_{a,F}^{{\rm{g}}_a} \times \mathbb{G}_{m,F}^{{\rm{g}}_m} \end{equation}
with an algebraic group $\G_c$ over $F$. Let $\mathcal{N}_c = \mathcal{N}_{F/K}(\G_c)$, $\mathcal{N}_a = \mathcal{N}_{F/K}\left(\mathbb{G}_{a, F}\right)$ and $\mathcal{N}_m = \mathcal{N}_{F/K}\left(\mathbb{G}_{m,F}\right)$. By Lemma 3.5.5 we get from (6.1.1) a product decomposition 
\begin{equation}
 \mathcal{N} \backsimeq \mathcal{N}_c \times \mathcal{N}_{a}^{{\rm{g}}_a} \times \mathcal{N}_m^{{\rm{g}}_m}
\end{equation}
such that after base change to $F$ the projection $p_{\G}$ maps each of the three factors in (6.1.2) to the respective factor in (6.1.1).
\begin{lemma} \label{0021} If $F$ is algebraically closed, then there is an algebraic group $\HHH_c$ over $F$ and a decomposition
\begin{equation}
 {{\rm{H}}} \backsimeq {{\rm{H}}}_c \times \mathbb{G}_{a,F}^{{\rm{h}}_a} \times \mathbb{G}_{m,F}^{{\rm{h}}_m}
\end{equation}
such that:
\begin{enumerate} 
 \item 
 ${\rm{h}}_a \geq {\rm{g}}_a$ and ${\rm{h}}_m \geq {\rm{g}}_m$. 
\item ${\rm{h}}_c = 0 \Longleftrightarrow {\rm{g}}_c = 0$.
\item $1 \geq {\rm{g}}_a$.
\end{enumerate}

\end{lemma}
\begin{proof} The first assertion is the most difficult one. Since $\mathbb{G}_{a,F}$ and $\mathbb{G}_{m,F}$ coincide with their conjugate varieties, $\mathcal{N}_a^{{\rm{g}}_a} \otimes_K F$ is isomorphic
to a power of $\mathbb{G}_{a,F}$ and the base change $\mathcal{N}_m^{{\rm{g}}_m}\otimes_K F$ is isomorphic to a power of $\mathbb{G}_{m,F} $. We identify $\G$ with the group on the right hand side of (6.1.1) and let $u: \G \longrightarrow \mathbb{G}_{a,F}^{{\rm{g}}_a} \times \mathbb{G}_{m,F}^{{\rm{g}}_m}$ be the projection. Set
 $\HHH'_{u} = u_{\mathcal{N}}(\HHH')$ and define $\HHH_{u} = \HHH_{u}' \otimes_K F$. These two algebraic groups are linear. Moreover,
 since $\mathcal{N}_{a}^{{\rm{g}}_a}$ is unipotent and as $\mathcal{N}_{m}^{{\rm{g}}_m}$
is a torus, we have $\HHH_{u} = \LL_a \times \LL_m$ with algebraic subgroups $\LL_a \subset \mathcal{N}_a^{{\rm{g}}_a} \otimes_K F$ and $\LL_m \subset \mathcal{N}_m^{{\rm{g}}_m} \otimes_K F$.
As $F$ is algebraically closed, $\LL_a$ is isomorphic to a power $\mathbb{G}_{a,F}^{{\rm{h}}_a}$ and $\LL_m$ is isomorphic to a power $\mathbb{G}_{m,F}^{{\rm{h}}_m} $.
Set $w = u_{\mathcal{N}} \otimes_K F$ and recall that $p_{\G}$ respects the decompositions in (6.1.1) and (6.1.2). So, since $u \circ p_{\G} = p_{(\mathbb{G}_a^{{\rm{g}}_a} \times \mathbb{G}_m^{{\rm{g}}_m})} \circ w $, 
it follows that 
\begin{equation}
{{\rm{h}}_a} \geq {{\rm{g}}_a}\,\,\mbox{and}\,\,{{\rm{h}}_m} \geq {{\rm{g}}_m}.
\end{equation}
In the next step we show that there is a section $\sigma: \HHH_{u} \longrightarrow \HHH$. To prove the existence of $\sigma$, let $\HHH_l$ be the maximal linear subgroup of $\HHH$. Taking quotients, we obtain
a surjective homomorphism $\overline{w}: \HHH/\HHH_l \longrightarrow \HHH_{u}/w(\HHH_l)$. By the theorem of Chevalley the quotient $\HHH/\HHH_l$ is an abelian variety. In contrast to this, $\HHH_{u}/w(\HHH_l)$ is linear.
So, $\overline{w} = 0$ and $\HHH_{u} = w(\HHH_l)$. It follows now from Serre \cite[Ch.\,VII, 6., Lemma 2]{Serre1} that the restriction $w_{|\HHH_l}$ admits a section $\tau$. Composing $\tau$ with the inclusion of $\HHH_l$ into $\HHH$,
we receive a section $\sigma$. In the last step we define $\HHH_c = \HHH/\sigma(\HHH_{\LL})$. Using $\sigma$ we get a decomposition
$$\HHH \backsimeq \HHH_c \times \sigma(\LL_a \times \{e_m\}) \times \sigma(\{e_a \} \times \LL_m).$$
as in (6.1.3). Statement 1.\,\,follows from (6.1.4). Statement 2.\,\,concerning ${\rm{h}}_c$ and ${\rm{g}}_c$ is clear from the construction. We move to the last statement. The proof is by way of contradiction. Assume that $\rm{g}_a \geq 2$ and consider the projection
 $u: \G \longrightarrow \mathbb{G}_{a,F}^{{\rm{g}}_a}$. Then $u_{\ast}(\Psi)$ has no Zariski-dense image in $\mathbb{G}_a^{{\rm{g}}_a}(\C)$, and $\Psi$ has no Zariski-dense image in $\G(\C)$. This is a contradiction to the assumptions.
Hence, $\rm{g}_a \leq 1$.\end{proof}

\subsection{Table of symbols used during the proof}
For the convenience of the reader we list all important symbols appearing in the proof of the theorems once again.
\begin{center}
\begin{tabular} {ll}
$F$ & an algebraically closed subfield of $\C$, stable with respect to $h$\\
 $K$ & $ = F \cap \R$\\
$\pi$ & a surjective homomorphism from $\G$ to $\rm{U}$\\
$\frak{g}$ & the Lie algebra of $\G$\\
$\frak{u}$ & the Lie algebra of $\rm{U}$\\
$\Psi$ & a real-analytic homomorphism to $\G(\C)$\\
$\rm{r}$ & $= {\rm{rank}}_{\Z}\,\Psi^{\scr -1}\big(\G(F)\big) + {\rm{rank}}_{\Z}\,\Psi_{[i]}^{\scr -1}\big(\G(F)\big)$\\
${\rm{k}}$ & $= {\rm{rank}}_{\Z}\,ker\,\Psi {=} {\rm{rank}}_{\Z}\,ker\,\Psi_{\mathcal{N}}$ (Lemma \ref{k})\\
$\frak{t} \subset \frak{g}$ & the smallest subspace over $K$ with the\\
& property that $\Psi_{\ast}(\R) \subset \frak{t} \otimes_K \R$\\
$\mathcal{N}$ & $= \mathcal{N}_{\overline{\Q}/K}({\rm{G}})$\\
$\frak{n} $ & $= {{\rm{Lie}}}\,\mathcal{N}$\\
$\frak{t}_{\mathcal{N}} \subset \frak{n}$ & the smallest subspace over $K$ such that \\
& $({\Psi}_{\mathcal{N}})_{\ast}(\R) \subset \frak{t}_{\mathcal{N}} \otimes_K \R$\\
$\frak{T}$ & $= \frak{t}_{\mathcal{N}} + i\frak{t}_{\mathcal{N}} \subset \frak{n}(F)$\\
${{\rm{H}}}' \subset \mathcal{N}$ & the Zariski-closure of $\Psi_{\mathcal{N}}(\R)$ over $K$\\
$\HHH$ & $= \HHH' \otimes_K F$\\
$\Phi_{\mathcal{N}}$ & the complexification of $\Psi_{\mathcal{N}}$ 

\end{tabular}
 
\end{center}
\section{Proof of the two theorems about transcendence} 
\subsection{Proof of Theorem \ref{Schn1}} We let $F = \overline{\Q}$ and 
consider a decomposition $\G \backsimeq \G_c \times \mathbb{G}_{a,\overline{\Q}}^{{\rm{g}}_a}, \mathbb{G}_{m,\overline{\Q}}^{{\rm{g}}_m}.$ We assume that
\begin{equation}
({\rm{r}}-2)\cdot{\dim\,{\rm{G}}} \geq 3 - (2{\rm{g}}_a+{\rm{g}}_m) -\rm{k}.
\end{equation}
With notations as in Lemma \ref{0021} we have
\begin{lemma} \label{001} If $\dim\,{{\rm{H}}} > \dim\,{\rm{G}}$, then
$$(\dim\,{{\rm{H}}}-1)\cdot {\rm{r}} - 1\geq 2\dim\,{{\rm{H}}} - (2{\rm{h}}_a+{\rm{h}}_m) -{\rm{k}}.$$
\end{lemma}
\begin{proof} We suppose that ${\rm{s}} = \dim\,\HHH - \dim\,{\rm{G}} \geq 1$. We have ${\rm{k}} \leq 1$ for otherwise $ker\,\Psi$ would not be discrete in $\R$. Moreover, ${\rm{g}_a} \leq 1$ by Lemma \ref{0021}. This
and (6.2.1) yield that ${\rm{r}} \geq 2$. Hence, we get from (6.2.1)
$$({\rm{r}}-2)\cdot (\dim\,\HHH - 1) = ({\rm{r}}-2)\cdot ({\dim\,{\rm{G}} + {\rm{s}}} - 1) \geq 3 - (2{\rm{g}}_a+{\rm{g}}_m) -\rm{k}.$$
And Lemma \ref{0021} implies that ${\rm{h}}_a \geq {\rm{g}}_a$ and ${{\rm{h}}_m} \geq {{\rm{g}}_m}$. This
yields the claim.
\end{proof}
Theorem\,\ref{Schn0}, Lemma \ref{rank} and Lemma \ref{001} imply that if $\dim\,{{\rm{H}}} > \dim\,{\rm{G}}$, then $\Phi_{\mathcal{N}}$ has no Zariski-dense image in ${\rm{H}}(\C)$. By contraposition
we infer that 
\begin{equation} \dim\,{\rm{H}} \leq \dim\,{\rm{G}}.
\end{equation} 
\begin{lemma} \label{op00} The homomorphism $p_{\rm{G}|{{\rm{H}}}}$ is an isogeny.
 
\end{lemma}
\begin{proof} As follows from Lemma \ref{definition}, $p_{\rm{G}|{{\rm{H}}}}$ is surjective. This and Statement (6.2.2) imply that $\dim\,{\rm{H}} = \dim\,\rm{G}$. Hence, $p_{\rm{G}|{{\rm{H}}}}$ is an isogeny.
\end{proof}
It results that $\Psi$ descends to $K$ via an isogeny $v: \G \longrightarrow {{\rm{H}}} = {{\rm{H}}}' \otimes_K F$ with the property that
$v \circ p_{\rm{G}|{{\rm{H}}}}$ is multiplication with an integer. Finally we recall our technical assumption that ${\rm{Hom}}\big(ker\,\pi, {{\rm{U}}}^{\scriptscriptstyle h}\big) = \{0\}$. Theorem \ref{inherited} implies that $\pi_{\ast}(\Psi)$ descends to $K$. 
\subsection{Proof of Theorem \ref{Schn2} and Remark \ref{Schn3}}
\subsubsection{Proof of the theorem} By assumption $\G$ is decomposed as $\G \backsimeq \G_c \times  \mathbb{G}_{a, \overline{\Q}}^{{\rm{g}}_a} \times \mathbb{G}_{m, \overline{\Q}}^{{\rm{g}}_m}$
and the estimate
\begin{equation} 
({\rm{r}}-2)(\dim\,\G + \dim\,{\rm{U}}) \geq {\rm{r}} -(2{\rm{g}}_a + {\rm{g}}_m) -{\rm{k}} + 1
\end{equation}
holds. Lemma \ref{0021} implies the existence of an induced decomposition of $\HHH$ as in (6.1.3). As in the proof of Lemma \ref{001} one shows
\begin{lemma} \label{002} If $\dim\,{\rm{H}} \geq \dim\,\G + \dim\,{\rm{U}},$ then 
$$(\dim\,{{\rm{H}}}-1)\cdot  {\rm{r}} - 1 \geq 2\dim\,{{\rm{H}}} - (2{\rm{h}}_a + {\rm{h}}_m) -\rm{k}.$$
\end{lemma}
Theorem \ref{Schn0}, Lemma \ref{rank} and Lemma \ref{002} yield the inequality $\dim\,{\rm{H}} < \dim\,\G + \dim\,{\rm{U}}$ by contraposition. Since we assume that ${\rm{Hom}}\big({{\rm{U}}}^{\scriptscriptstyle h}, {\rm{M}}\big) = \{0\}$ for each quotient $\rm{M}$ of $ker\,\pi$, it follows then from Theorem \ref{inherited1} that
$\pi_{\ast}(\Psi)$ descends weakly to $K$. This proves the theorem.
\subsubsection{Proof of the remark} We may assume that the decomposition $\G_c\times \mathbb{G}_{a,\overline{\Q}}^{{\rm{g}}_a} \times \mathbb{G}_{m,\overline{\Q}}^{{\rm{g}}_m} $ is such that the dimensions ${\rm{g}}_m$ and ${{\rm{g}}_a}$ are maximal, whereas ${{\rm{g}}_c}$ is minimal. So, if $\G$ is linear, we have $2\dim\,\G - (2{\rm{g}}_a + {\rm{g}}_m) = {\rm{g}}_m$. Replacing ``$-(2{\rm{g}}_a + {\rm{g}}_m) $'' by ``$-(2{\rm{g}}_a + {\rm{g}}_m +\dim\,{\rm{U}})$'' in (6.2.3), we get
\begin{equation} 
 {\rm{r}}\cdot (\dim\,\G + \dim\,{\rm{U}} - 1) - 1 \geq {\rm{g}}_m + \dim\,{\rm{U}} -{\rm{k}}.
\end{equation}
Furthermore, the Weil restriction of $\G$ is linear, too. \textit{A fortiori}, ${\rm{H}}$ is linear and Lemma 6.1.7 yields 
\begin{equation} 
{\rm{h}}_m = 2\dim\,{\rm{H}} - (2{\rm{h}}_a + {\rm{h}}_m). 
\end{equation}
 Let ${\LL} = ker\,p_{\G|\HHH}$ and write ${\rm{l}}_m$ for the dimension of the multiplicative factor of $\LL$. Because of Lemma 6.1.4 we have a short exact sequence of linear groups
\begin{equation}
 0 \longrightarrow {\rm{L}} \longrightarrow \HHH \stackrel{p_{\G|\HHH}}{\longrightarrow} \G \longrightarrow 0.
\end{equation}
It follows from Serre \cite[Ch.\,VII, 6., Lemma 2]{Serre1} that ${\rm{h}}_m =  {\rm{g}}_m + {\rm{l}}_m$. Setting ${\rm{s}} = \dim\,\LL - \dim\,\rm{U}$, we infer 
\begin{equation} 
{\rm{h}}_m \leq {\rm{g}}_m + \dim\,\rm{U} + \rm{s}.
\end{equation}
We can now complete the proof of the remark: We have $\dim\,\HHH = \dim\,\G + \dim\,\rm{U} + \rm{s}$ and the assumption of the theorem is that $\dim\, \HHH \geq \dim\,\G + \dim\,\rm{U}$. So, ${\rm{s}} \geq 0$ and Statement (6.2.4) yields
\begin{equation} 
 {\rm{r}}\cdot (\dim\,\G + \dim\,{\rm{U}} + {\rm{s}} - 1) - 1 \geq {\rm{g}}_m + \dim\,{\rm{U}} + {\rm{s}} -{\rm{k}}.
\end{equation}
This and Statements (6.2.5)-(6.2.7) lead to the estimate
$${\rm{r}} \cdot (\dim\,{{\rm{H}}}-1) - 1 \geq 2\dim\,{{\rm{H}}} - (2{\rm{h}}_a + {\rm{h}}_m) -\rm{k}.$$
The result follows then the same way as above.

\section{Proof of the two theorems about algebraic independence}
\subsection{Proof of Theorem \ref{mta1}} The proof of Theorem \ref{mta1} goes along the same lines as the one of Theorem \ref{Schn1}, except that instead of Theorem\,\ref{Schn0} we use Theorem\,\ref{5.1.1}.
We saw in Subsect.\,6.2.1 that it suffices to show that the Zariski-closure ${{{\rm{H}}} }'$ of $\Psi_{\mathcal{N}}(\R)$ in $\mathcal{N} = \mathcal{N}_{F/K}(\G)$ has dimension $\dim\,{{{\rm{H}}} }' \leq \dim\,\rm{G}$. We begin the proof by recalling some hypotheses. 
The symbol $F$ denotes an algebraically closed subfield of $\C$ which is stable with respect to complex conjugation and with transcendence degree $\leq 1$ over $\Q$.
We consider a decomposition $\G \backsimeq \G_c \times \mathbb{G}_{a,F}^{{\rm{g}}_a} \times \mathbb{G}_{m,F}^{{\rm{g}}_m}$ and assume that 
\begin{equation}
 (1+{\rm{k}}){\rm{r}}\cdot {(\dim\,{\rm{G}} -\dim\,\frak{t}) + {\rm{k}\rm{r}}} \geq 2{\rm{g}}_c+ \underbrace{\frac{{\rm{g}}_c}{{\rm{g}}_c + 1}}_{=: \gamma'} + {\rm{g}}_m + 1+ \delta.
\end{equation}
Here ${\rm{g}}_c = \dim\,\G_c$ and
$$\delta = \left\lbrace \begin{array}{cl} 1 & \mbox{if}\,{\rm{g}}_a > 0\,\mbox{and}\,\dim\,\frak{t} = 1\\
0 &\mbox{otherwise}
 
  \end{array}
\right. $$
It is clear from the estimate in (6.3.1) that we may suppose that ${\rm{g}}_a, {\rm{g}}_m$ are maximal and $\rm{g}_c$ is minimal under all possible decompositions of $\G$ as above. We then define $\gamma = 1$ if ${\rm{g}}_c \neq 0$ and $\gamma = 0$ if ${\rm{g}}_c = 0$. Note that if ${\rm{g}}_c \neq 0$ then $ 0 < \gamma' = \frac{{\rm{g}}_c}{{\rm{g}}_c + 1} < 1$. (6.3.1) is an estimate between integers. Hence, (6.3.1) holds with $\gamma'$ (in its original form) if and only if it holds with $\gamma'$ replaced by $\gamma$. In other words, we may replace (6.3.1) by
\begin{equation}(1+{\rm{k}}){\rm{r}}\cdot {(\dim\,{\rm{G}} -\dim\,\frak{t}) + {\rm{k}\rm{r}}} \geq 2{\rm{g}}_c + {\rm{g}}_m + 1+ \gamma + \delta.
\end{equation} 
Next we show
\begin{lemma} \label{implies} The estimate $(1+{\rm{k}}){\rm{r}} \geq 2$ holds.
 
\end{lemma}
\begin{proof} We have ${\rm{g}}_a \leq 1$ by Lemma \ref{0021}. Using this and (6.3.2) we get the estimate
$$(1+{\rm{k}}){\rm{r}}\cdot {(\dim\,{\rm{G}} -\dim\,\frak{t}) + {\rm{k}\rm{r}}} \geq \dim\,\G .$$
If $(1+{\rm{k}}){\rm{r}} \leq 1$, then ${\rm{k}\rm{r}} = 0$ and $\dim\,{\rm{G}} -\dim\,\frak{t} \geq \dim\,\G .$
But this contradicts the fact that $\dim\,\frak{t} > 0$. Hence, $(1+{\rm{k}}){\rm{r}} \geq 2.$
\end{proof}
We recall that $\frak{T}\subset \frak{n}$ denotes the smallest linear subspace over $F$ with the property that $\big(\Phi_{\mathcal{N}}\big)_{\ast}(\C) \subset \frak{T} \otimes_F \C$. We consider the product decomposition ${{\rm{H}}} \backsimeq {{\rm{H}}}_c \times \mathbb{G}_{a,F}^{{\rm{h}}_a} \times \mathbb{G}_{m,F}^{{\rm{h}}_m}$ from Lemma \ref{0021}.  
\begin{lemma} \label{00} If $\dim\,{{\rm{H}}} > \dim\,{\rm{G}}$, then
$$(1+{\rm{k}}){\rm{r}}\cdot (\dim\,{{\rm{H}}} -\dim\,\frak{T}) -{\rm{r}}\geq 2{\rm{h}}_c + {\rm{h}}_m + \delta.$$
\end{lemma}
\begin{proof} Lemma 6.1.2 implies that $\dim\,\frak{T}$ equals $\dim\,\frak{t}$. Moreover, ${\rm{g}}_c = 0$ if and only if ${\rm{h}}_c = 0$ by Lemma \ref{0021}.
Let ${\rm{s}} = \dim\,{{\rm{H}}} - \dim\,{\rm{G}}$. For ${\rm{s}} = 1$ one gets
\begin{equation}
(1+{\rm{k}}){\rm{r}}\cdot (\dim\,{{\rm{H}}} -\dim\,\frak{T}) -{\rm{r}} = (1+{\rm{k}}){\rm{r}}\cdot {(\dim\,{\rm{G}} -\dim\,\frak{t}) + {\rm{k}\rm{r}}}.
\end{equation} 
Since $(1+{\rm{k}}){\rm{r}} \geq 2$, the expression
\begin{equation}(1+{\rm{k}}){\rm{r}}\cdot (\underbrace{{\rm{s}} + \dim\,{{\rm{G}}}}_{= \dim\,\HHH} -\dim\,\frak{T}) -{\rm{r}}- 2({\rm{s}}-1) - 1 - \gamma - 2{\rm{g}}_c - {\rm{g}}_m - \delta
 \end{equation}
is monotonically increasing in ${\rm{s}}$. Finally, a short exact sequence as in (6.2.6) gives
\begin{equation}2{\rm{h}}_c + {\rm{h}}_m \leq 2({\rm{s}}-1) + 1 + \gamma + 2{\rm{g}}_c + {\rm{g}}_m.
\end{equation}
Statements (6.3.2)-(6.3.5) imply the required estimate.
\end{proof}
It results from Theorem\,\ref{5.1.1}, Lemma \ref{rank} and Lemma \ref{00} that if $\dim\,{{\rm{H}}} > \dim\,{\rm{G}}$, then $\Phi_{\mathcal{N}}$ has no Zariski-dense image in ${\rm{H}}(\C)$. By contraposition
we infer that $\dim\,{{\rm{H}}} \leq \dim\,{\rm{G}}$. As mentioned in the beginning of the proof, the theorem follows. 
\subsection{Proof of Thm.\,\ref{Gel3} and Remark \ref{Gel31}} The proof of Thm.\,\ref{Gel3} and the remark is very similar to the proof of Theorem \ref{Schn2} and Remark \ref{Schn3}. After all it is left to the reader.

\section{Proof of the two theorems about linear independence of algebraic logarithms}
\subsection{Proof of Theorem\,\ref{mta0}} We apply the notations from Sect.\,6.1 with $F = \overline{\Q}$ and set $\frak{h} = {\rm{Lie}\,\rm{H}}$. Because of Corollary \ref{i} and because of the hypotheses we may assume here and in the next subsection that ${\rm{rank}}_{\Z} \,\Psi^{\scr -1}\big(\G(F)\big) \geq 1$. Lemma \ref{op(} and Theorem\,\ref{AST} applied to $\frak{T}$ yield then
$\frak{T}  = \frak{h}$. In particular, 
\begin{equation}
   \dim \frak{T} = \dim\,{{{{\rm{H}}} }}.
\end{equation}
The assumptions of the theorem, Lemma \ref{op(} and Statement (6.4.1) give the estimate
\begin{equation}
\dim\,{\rm{G}} = \dim\,\frak{t} = \dim\,\frak{T} = \dim\,{{{\rm{H}}} }.
\end{equation}
The same way as in the proof of Lemma 6.2.2 we get
\begin{lemma} \label{op00} The homomorphism $p_{\rm{G}|{{\rm{H}}}}$ is an isogeny.
 
\end{lemma}
As seen already twice above, the assumptions and Theorem \ref{inherited} imply then that $\pi_{\ast}(\Psi)$ descends to $K$. This is one part of the theorem.\\
\\
For the remaining part concerning the identity ``$\dim\,{\rm{U}} = \dim\,\pi_{\ast}(\frak{t})$'' we write $\psi = \pi_{\ast}(\Psi)$ and fix an isogeny $w: {{\rm{U}}} \longrightarrow {{\rm{U}}}' \otimes_K \overline{\Q}$ such that $w_{\ast}(\psi)$ takes values in ${{\rm{U}}}'(\R)$. Statement (6.4.2) and Lemma \ref{op00} yield
\begin{lemma} \label{09} With assumptions as in Theorem \ref{wus1}, we have $\dim\,\frak{t} = \dim\,\rm{G}$.
 
\end{lemma}
Consider the smallest subspace $\frak{s} \subset \frak{u} = {\rm{Lie}\,\rm{U}}$ over $K$
such that $\psi_{\ast}(\R) \subset \frak{s} \otimes_K \R$. Since $\Psi_{\ast}(\R) \subset \pi_{\ast}^{\scr -1}(\frak{s}) \otimes_K \R$ and as $\pi_{\ast}^{\scr -1}(\frak{s}) \subset \frak{g}$ is defined over $K$, the minimality of $\frak{t}$ implies that $\frak{t} \cap \pi_{\ast}^{\scr -1}(\frak{s}) = \frak{t}$. On the other hand, $\frak{s} \subset \pi_{\ast}(\frak{t})$ because $\pi_{\ast}(\frak{t})$ is defined over $K$ in $\frak{u}$. Hence, $\frak{s} = \pi_{\ast}(\frak{t})$.
Moreover, ${\frak{u}}' = {\rm{Lie}}\,{{\rm{U}}}'$ is a subspace of ${\rm{Lie}}\,({\rm{U}}' \otimes_K \overline{\Q})$ over $K$. So, $w_{\ast}^{\scr -1}(\frak{u}')$ is defined over $K$. Corollary \ref{lp333} implies that $(w\circ \psi)_{\ast}(\R) \subset \frak{u}' \otimes_K \R$. Thus, $\frak{s} \subset w_{\ast}^{\scr -1}(\frak{u}')$ by minimality of $\frak{s}$.
We infer that 
\begin{equation}
\dim\,\frak{s} \leq \dim\,{\frak{u}}' = \dim\,{{\rm{U}}}' = \dim\,{\rm{U}}.
\end{equation}
We replace $\G$ by $\rm{U}$, $\pi$ by the identity morphism, $\Psi$ by $\psi$ and $\frak{t}$ by $\frak{s} = \pi_{\ast}(\frak{t})$. Because
of (6.4.3) this new setting satisfies the hypotheses of Theorem 2.9.7. Lemma \ref{09} implies that $\dim\,{\rm{U}} = \dim\,\pi_{\ast}(\frak{t})$. This is the remaining part of the theorem. Everything is proved.
\subsection{Proof of Theorem\,\ref{mt0W}} 

Recall that during the proof of Theorem \ref{mta0} in the previous subsection we showed using the Analytic Subgroup Theorem that the Zariski-closure ${\rm{H}}$ of $\Psi_{\mathcal{N}}(\R)$ has dimension $\leq \dim\,\frak{t}$.
Since we assume that $\dim\,\G + \dim\,{{\rm{U}}} > \dim\,\frak{t},$ we get by the same argument

\begin{lemma} We have $\dim\,\frak{t} = \dim\,{\rm{H}}$ and the estimate $\dim\,\G + \dim\,{{\rm{U}}} > \dim\,{\rm{H}}$ holds.

\end{lemma}
The proof of the theorem is then completed with the help of Theorem \ref{inherited} as in the previous subsections.

\chapter{Applications and further theorems} We can now proceed to the long awaited applications. Throughout we will use the following\\
\\
\textbf{Notations.} We let $\mathcal{L} = exp^{\scr -1}\left(\overline{\Q}\right)$ be the $\Q$-vector space of algebraic logarithms and set $i = \sqrt{-1}$. For a lattice $\Lambda \subset \C$ we define $\wp_{\Lambda} = \wp$ to be the associated Weierstra\ss\,function with invariants $g_2 = g_2(\Lambda)$ and $g_3 = g_3(\Lambda)$. The complex elliptic curve associated to $\Lambda$ is denoted by $\rm{E}$. By abuse of notation, we will sometimes
identify $\rm{E}$ with its Weierstra\ss\,\,model over $F = \overline{\Q(g_2, g_3)}$. We will also identify the exponential map ${\rm{exp}}_{\rm{E}}$ with $\wp$ and the Lie algebra $\frak{e}(\C) =  (\rm{Lie}\,\rm{E}\big)(\C)$ with $\C$. Then the standard coordinate $z = x + iy$ on $\C$ defines a coordinate of $\frak{e}(\C)$ over $F$. And if $F = \overline{F}$ is stable with respect to complex conjugation, then $x$ and $y$ yield coordinates over $K = F \cap \R$ on $\frak{e}(\C)$.\footnote{Warning! If $F$ is not algebraically closed, this is not true in general. For instance, if $F = \Q(\sqrt{-2})$, then $y$ is not defined over $K = F \cap \R = \Q$.} The set of elliptic logarithms $\omega \in \wp^{\scr -1}\big(F \cup \{\infty\}\big)$ is denoted by $\mathcal{L}_{\Lambda}$. As above $h$ is the complex conjugation and we shall write $\Lambda^{\scriptscriptstyle h}$ instead $h(\Lambda)$. 
Note that $\Lambda^{\scriptscriptstyle h}$ has invariants $h(g_2)$, $h(g_3)$, so it is the lattice associated to the complex conjugate ${\rm{E}}^{\scr h}$ of $\rm{E}$. We shall also use the Weierstra\ss\,functions $\sigma = \sigma_{\Lambda}, \eta = \eta_{\Lambda}$ and $\zeta = \zeta_{\Lambda}$. The most important properties of these functions are listed in Waldschmidt \cite{Wal12}. As sketched in App.\,B, these
Weierstra\ss\,\,functions uniformize linear extensions of $\rm{E}$. For some applications the reader will need to have a look at the appendices, especially at App.\,B.3.\\
\\
As mentioned in the outline, in this chapter we shall not follow the historical order from the first part of the introduction anymore and begin with applications concerning real and imaginary parts of linear logarithms.

\section{Linear independence of real and imaginary parts of algebraic logarithms} 
\subsubsection{From Schneider to Masser}
Let $\Lambda$ be a lattice in $\C$ with algebraic invariants. In 1936 Th.\,Schneider succeeded in proving
the transcendence of non-zero algebraic logarithms $\omega \in \mathcal{L}_{\Lambda}$. With the help of the main theorems we are able to generalize the famous result in an essential manner. This is the content of the next three corollaries. In this section $K$ denotes the field $\overline{\Q} \cap \R$.
 
\begin{corollary} Let $\Lambda$ be a lattice in $\C$ with algebraic invariants and let $\omega \in \mathcal{L}_{\Lambda}$.
Suppose that $\Z\omega + \Lambda$ is dense in $\C$ with respect to the analytic topology. Then $\R\omega \cap \overline{\Q} = \{0\}$.  
 
\end{corollary}
\begin{proof} Let $\Psi(r) = \wp(\omega r)$ and $\xi = \wp(\omega)$. The group $\Z\xi \in \rm{E}(\overline{\Q})$ is dense in ${\rm{E}}(\C)$ with respect to the analytic topology, so that $\Psi$ does not descend to $\R$. On the other hand,
the image of $\Psi$ contains an algebraic point distinct from the neutral element. If $\alpha \in \R\omega \cap \overline{\Q} \neq \{0\}$, then $\Psi_{\ast}(\R) \subset \frak{t} \otimes_K \R$ for $\frak{t} = K \alpha$.
Theorem \ref{wus1} implies then that $\Psi$ descends to $K$, and hence to $\R$. It follows by contraposition that $\R\omega \cap \overline{\Q} = \{0\}$.   
\end{proof}

\begin{corollary} Let $\Lambda$ be a lattice in $\C$ with algebraic invariants. Then
 ${{\rm{E}}}$ is isogenous to a curve over $\R$ if and only if there is a non-zero algebraic logarithm
$\omega \in \mathcal{L}_{\Lambda}$ such that $\omega/|\omega| \in \overline{\Q}$.
\end{corollary}
\begin{proof} Let $\Psi(r) = \wp({\scriptstyle \frac{\omega}{|\omega|} }r)$. If $\omega/|\omega| \in \overline{\Q}$, then $\Psi_{\ast}(\R) = \frak{t} \otimes_K \R$ for $\frak{t}= K\omega/|\omega|$ and $\Psi(\R)$ contains a non-trivial algebraic point $\xi \in {\rm{E}}(\overline{\Q})$.
It results from Theorem \ref{wus1} that ${{\rm{E}}}$ is isogenous to a curve over $\R$. Conversely, if ${\rm{E}}$ is isogenous to a curve over $\R$, then Corollary \ref{lp} implies that $\rm{E}$ is isogenous to a curve ${\rm{E}}'$ over $K$. We may assume that $\rm{E}'$ is in Weierstra\ss\,\,form. The associated lattice $\Lambda'$ has real invariants and is stable with respect to complex conjugation. Hence, $\Lambda'$ contains a real element $\lambda' \neq 0$. Thus, there is a period $\lambda \in \Lambda$ and an algebraic number $\alpha$ representing an isogeny $v: {\rm{E}}' \longrightarrow {{\rm{E}}}$ such that $\alpha\lambda' = \lambda$. As a result, $\lambda/|\lambda| = \alpha/|\alpha| \in \overline{\Q}$.
\end{proof}
\begin{corollary} \label{realalg} Let $\Lambda$ be a lattice in $\C$ with algebraic invariants and let $\omega \in \mathcal{L}_{\Lambda}$ be a non-zero algebraic logarithm. If $Re\,\omega$ is algebraic, then $Re\,\omega = 0$.
\end{corollary}
\begin{proof} We define $\G = {\rm{E}} \times \mathbb{G}_{a, \overline{\Q}}$ and let
$\pi$ be the projection to $\U = \rm{E}$. We consider
the homomorphism $\Psi(r) = \big(\wp(\omega r), r\big)$ to $\G(\C)$. Then $\Psi(\Q)$ is contained in $\G(\overline{\Q})$. If $Re\,\omega$ is algebraic, then there exists a subspace $\frak{t} \subset \frak{g}$ over $K$ of dimension $2 < 3 = \dim\,\G + \dim\,\U$ with the property that
 $\Psi_{\ast}(\R) \subset \frak{t}\otimes_K \R $.
Theorem \ref{wus2} implies that $\pi_{\ast}(\Psi)$ descends weakly to $K$. And it follows from Proposition \ref{real} that the real subspace $\R\omega$ of $\frak{e}(\C) = \C$ is defined over $K$. This means that $Re\,\omega$ and $Im\,\omega$ are linearly dependent over $K$. Now,
$Re\,\omega$ is algebraic by assumption. So, if $Re\,\omega \neq 0$, then
$Im\,\omega$ is algebraic. Hence, if $Re\,\omega \neq 0$, then $\omega$ is algebraic. This contradicts Schneider's theorem (see \cite[Thm.\,20]{Wal12}). Therefore, $Re\,\omega = 0$.
 \end{proof}
We also have the following related criterion.
\begin{corollary} Let $\Lambda$ be a lattice in $\C$ with algebraic invariants and choose an arbitrary basis $\lambda_1, \lambda_2 \in \Lambda$. Then $\rm{E}$ is isogenous to a curve with model
over $K$ if and only if either $\lambda_2/|\lambda_2|$ is algebraic or there are integers $d_1, d_2, d_3$ such that $d_1 \neq 0$ and 
$|d_1\lambda_1 + d_2\lambda_2| = |d_3\lambda_2|.$
 
\end{corollary}
\begin{proof} If $\lambda_2/|\lambda_2|$ is algebraic, then the assertion follows from the previous corollary. If 
  $\lambda_2/|\lambda_2|$ is transcendental, then $\Psi(r) = \wp({\scr \frac{\lambda_2}{|\lambda_2|}} r)$ cannot descend to $K$. Indeed, otherwise it would follow from Proposition \ref{real} that $\R\lambda_2 \subset \C$ is defined 
  over $K$. But the latter is equivalent to the statement that $\lambda_2/|\lambda_2|$ is algebraic, and we would receive a contradiction. So, since $\Psi$ does not descend to $K$, Corollary B.1.2 implies that $\tau = \lambda_1/\lambda_2$ has irrational real part $Re\,\tau$. In particular, $\tau$ is not a quadratic irrationality. Since $j(\tau)$ is algebraic, $\tau$ is transcendental by the theorem of Schneider (\cite[Thm.\,20]{Wal12}). The assertion follows now from Corollary \,\ref{op222} and Remark\,\ref{op222+}.
\end{proof}
For a complex number $\omega$ we have $\omega/|\omega| \in \overline{\Q}$ if and only if $\omega$ and $h({\omega})$ are linearly dependent over $\overline{\Q}$. The latter holds if and only if $Re\,\omega$ and $Im\,\omega$ are linearly dependent over $K$. Based upon this observation Theorem\,\ref{wus2} allows the following result. The proof illustrates once again how to deal with the technically complicated theorem.
\begin{corollary} \label{ppopp} Let $\Lambda$ be a lattice in $\C$ with algebraic invariants and suppose that $\alpha \Lambda \nsubseteq \Lambda^{\scriptscriptstyle h}$ for all non-zero algebraic numbers $\alpha$. Let $\omega_1, \omega_2 \in \mathcal{L}_{\Lambda}$ be two non-zero algebraic logarithms which are linearly independent over $\Z$. Then, for all non-zero $\omega' \in \mathcal{L} \otimes \overline{\Q}$, the six numbers
$$1, \omega', Re\,\omega_1, Im\,\omega_1, Re\,\omega_2, Im\,\omega_2$$
are linearly independent over $\overline{\Q}$. 
\end{corollary}
Note that the assertion is wrong if $\Lambda = \Lambda^{\scriptscriptstyle h}$, because in this case we can set $\omega_2 = h(\omega_1)$. 
\begin{proof} Recall that a complex number $u$ lies in $\mathcal{L}$ if and only if $Re\,u$ and $i\cdot Im\,u$
lie in $\mathcal{L}$. Hence, we can write 
$$Re\,\omega' = \alpha_1^{\scriptscriptstyle (1)}\sigma_1^{\scriptscriptstyle (1)} + .... + \alpha_{k_1}^{\scriptscriptstyle (1)} \sigma_{k_1}^{\scriptscriptstyle (1)} + i\beta_1^{\scriptscriptstyle (1)} \tau_{1}^{\scriptscriptstyle (1)} + ... +i \beta_{l_1}^{\scriptscriptstyle (1)} \tau_{l_1}^{\scriptscriptstyle (1)}$$
where $\alpha_j^{\scriptscriptstyle (1)}, \beta_j^{\scriptscriptstyle (1)} \in K$, $\sigma_j^{\scriptscriptstyle (1)} \in \mathcal{L} \cap \R$ and $\tau_j^{\scriptscriptstyle (1)} \in \mathcal{L} \cap i\R$.
We may even assume that this representation of $Re\,\omega'$ enjoys the property that $k_1 + l_1$ is minimal. We find a similar representation for $Im\,\omega'$
involving numbers $\alpha_j^{\scriptscriptstyle (2)}, \beta_j^{\scriptscriptstyle (2)}, \sigma_j^{\scriptscriptstyle (2)}$ and $\tau_j^{\scriptscriptstyle (2)}$. Recall the algebraic torus $\mathbb{S}$ from (2.3.2). It was explained in Sect.\,1.2 that the exponential map of $\mathbb{S}(\R)$ is canonically identified with $e^{ir}$.
This in mind, let ${\rm{U}} = {\rm{E}} \times {\rm{E}}, {{\rm{G}}} = {\rm{U}} \times \mathbb{G}_{a,\overline{\Q}} \times \mathbb{G}_{m,\overline{\Q}}^{k_1 + k_2} \times \mathbb{S}_{\overline{\Q}} ^{l_1 + l_2}$ and write $\pi: \G \longrightarrow \rm{U}$ for the projection. We shall work with the homomorphism 
$$\Psi(r) = \left(\wp(\omega_1 r), \wp(\omega_2 r), r, e^{\sigma_1^{\scriptscriptstyle (1)} r},..., e^{i\tau_{l_2}^{\scriptscriptstyle (2)} r} \right)$$
to $\G(\C)$. We claim that $\Psi$ has Zariski-dense image. Using disjointness of the factors of $\G$, it suffices to verify that the ratio $\omega_1/\omega_2$ is not algebraic. By assumption $\omega_1/\omega_2$ is irrational. Hence,
if ${\rm{E}}$ is not a CM-curve, then Schneider's theorem (\cite[Thm.\,20]{Wal12}) implies that $\omega_1/\omega_2$ is not algebraic.
\begin{claim} \label{815} If $\alpha \Lambda \nsubseteq \Lambda^{\scriptscriptstyle h}$ for all algebraic numbers $\alpha \neq 0$,
then ${\rm{E}}$ is not a CM-curve and not isogenous to its complex conjugate.

\end{claim}
\begin{proof} Since $\alpha \Lambda \nsubseteq \Lambda^{\scriptscriptstyle h}$ for all algebraic $\alpha \neq 0$, ${\rm{E}}$ is not isogenous to its complex conjugate. On the other hand, it follows from Corollary B.3.1 that CM-curves are isogenous to curves with models over $\R$. And Proposition \ref{wu11} implies that elliptic curves, which are definable over $\R$, are isogenous to their complex conjugates. We infer the claim.
\end{proof}
Hence, $\omega_1/\omega_2$ is transcendental and $\Psi$ has Zariski-dense image. This in mind, suppose that the assertion of the corollary is wrong. Then there is a non-trivial relation over $K$ between
$1$ and the non-vanishing real and imaginary parts of the numbers $\omega', \omega_1, \omega_2$.
Since these real and imaginary parts arise from coordinates over $K$ of an algebraic logarithm in $\frak{g}(\C)$, our converse assumption implies that there is a subspace $\frak{t} \subset \frak{g}$ over $K$ of dimension
$$\dim\,\frak{t} < k_1 + l_1 + k_2 + l_2 + 1 + 4 = \dim\,\G + \dim\,\rm{U}.$$
Recalling that $\pi$ is the projection to $\rm{U}$, it follows from Theorem \ref{wus2} that $\pi_{\ast}(\Psi)$ descends weakly to $K$. By Corollary \ref{weakdescentodd1} this is only possible if $\rm{E}$ is either isogenous to a curve which is definable over $\R$ or if $\rm{E}$ is isogenous to its complex conjugate. 
Proposition \ref{wu11} implies that if the latter holds, then $\rm{E}$ must be isogenous to its complex conjugate. This contradicts Claim \ref{815}. So, the corollary follows by contraposition.
\end{proof}
Next we establish a real version of Masser's result on the linear independence of elliptic periods and quasi-periods. To formulate the result, let $\Lambda = \Z\lambda_1 + \Z \lambda_2$ be a lattice in $\C$ with algebraic invariants and set $\eta_1 = \eta(\lambda_1)$ and $\eta_2 = \eta(\lambda_2).$
In 1975 Masser \cite[Ch.\,II, Ch.\,III]{Mass} proved that 
$$\dim_{\overline{\Q}} \{1, \lambda_1, \eta_1, \lambda_2, \eta_2, 2\pi i\} = 2 +2 \dim_{\overline{\Q}} \{\lambda_1, \lambda_2\}.$$
\begin{corollary} Let $\Lambda = \Z\lambda_1 + \Z \lambda_2$ be a lattice in $\C$ with algebraic invariants. Then
\begin{center}
 \begin{tabular}{l}
  $\dim_{\overline{\Q}} \{1, Re\,\lambda_1, Im\,\lambda_1, Re\,\eta_1, Im\,\eta_1, Re\,\lambda_2, Im\,\lambda_2, Re\,\eta_2, Im\,\eta_2, 2\pi i\}$\\
$= 2 + 2 \dim_{\overline{\Q}} \{Re\,\lambda_1,  Im\,\lambda_1,  Re\,\lambda_2, Im\,\lambda_2\}.$
 \end{tabular}

\end{center}
\end{corollary}
\begin{proof} Unfortunately, the proof is a little involved. We divided it into three steps:
\begin{enumerate}
 \item $\rm{E}$ is not isogenous to its complex conjugate, and hence it is not\\ isogenous
 to a curve with model over $K$ (Proposition \ref{wu11}).
\item $\rm{E}$ is isogenous to its complex conjugate, but without complex multiplication.
\item $\rm{E}$ is a CM-curve.
\end{enumerate}
For the proof in the first case, we apply the uniformizations from App.\,B.3. We define ${\rm{U}} =\rm{G}_t$ to be the extension in $\rm{Ext}(\rm{E}, \mathbb{G}_{a, \overline{\Q}})$ associated to $t = 1$. Because of Lemma C.1.4\,\,$\G_t$ is not isotrivial.
 We set $\G = \G_t^{\scr 2} \times  \mathbb{S}_{\overline{\Q}} \times \mathbb{G}_{a,\overline{\Q}}$ and let
 $\pi: \G \longrightarrow {\rm{U}}$ be the projection. For $j = 1,2$ we denote by $\tau_j$ the vector $(t\eta_j , \lambda_j)$. Finally, we define 
$$\Psi(r) = \big({\rm{exp}}_{{\rm{G}}_t}(\tau_1 r), {\rm{exp}}_{{\rm{G}}_t}(\tau_2r), e^{\scriptscriptstyle 2\pi i r}, r\big) = \big(\Psi_o(r), r\big)$$
Lemma \ref{kern} teaches that $\Z\tau_1 + \Z \tau_2$ is the kernel of ${\rm{exp}}_{{\rm{G}}_t}$. Hence, $\tau_1$ and $\tau_2$ are independent over $\C$ for otherwise
${\rm{exp}}_{{\rm{G}}_t}\big(\C\tau_1\big) \subset \G_t(\C)$ would be an elliptic curve, contradicting that $\G_t$ is not isotrivial. Since $\G_t$ and $\mathbb{G}_{m, \overline{\Q}} $ are disjoint, it follows that $\Psi_o$ has Zariski-dense image and non-trivial kernel. 
Moreover, if $u: \G \longrightarrow \mathbb{G}_{a,\overline{\Q}}$ is the projection to the additive factor, then $u_{\ast}(\Psi) = id.$ is injective. This in turn implies 
that $\Psi = \Psi_o \times id.$ has Zariski-dense image. Suppose now that the assertion of the corollary is wrong. Then there is a subspace $\frak{t} \subset \frak{g}$ over $K$ with dimension 
$$\dim\,\frak{t} < 2  + 4 = \dim\,{\rm{U}} + \dim\,\G$$
and such that $\Psi_{\ast}(\R) \subset \frak{t} \otimes_K \R$. Theorem\,\ref{wus2} implies that $\pi_{\ast}(\Psi)$ descends weakly to $K$. We will assume the latter and derive a contradiction.\\
\\
To start with, we show
\begin{claim} \label{min} The homomorphism $\psi(r) = \big(\wp(\lambda_1 r), \wp(\lambda_2r)\big)$ to ${\rm{E}}^{\scr 2}(\C)$ does not descend weakly to $K$.
 \end{claim}
\begin{proof} If $\psi(r) = \big(\wp(\lambda_1 r), \wp(\lambda_2r)\big)$ descends weakly to $K$, then Corollary\,\ref{ppopp} implies that $\rm{E}$ is either isogenous to a curve
with model over $K$ or it is isogenous to its complex conjugate. This is a contradiction to the assumption in the first case (see 1.\,\,above).
\end{proof}
\begin{claim} The equality
$$\dim_{K} \{Re\,\lambda_1,  Im\,\lambda_1,  Re\,\lambda_2, Im\,\lambda_2\}  = \dim_{\overline{\Q}} \{Re\,\lambda_1,  Im\,\lambda_1,  Re\,\lambda_2, Im\,\lambda_2\} = 4$$
holds.
\end{claim}
\begin{proof} Recalling that the four numbers in the left equality are coordinates over $K$ of algebraic logarithms, it follows that
if 
$$\dim_{K} \{Re\,\lambda_1,  Im\,\lambda_1,  Re\,\lambda_2, Im\,\lambda_2\} < 4 = 2\dim\,{\rm{E}}^{\scr 2},$$
then there is a proper subspace $\frak{s} \subset \frak{e}^{\scr 2}$ over $K$ such that $\psi_{\ast}(\R) \subset \frak{s} \otimes_K \R$. We apply Theorem \ref{wus1} to $\psi$, ${\rm{E}}^{\scr 2}$, the identity morphism and to $\frak{s}$. It results that $\psi$ descends weakly to $K$.
But this contradicts Claim \ref{min}. The equality on the left hand side follows.
The identity on the right hand side is clear, because the four numbers are real. 
\end{proof}
Let $p: \G_t^{\scr 2} \longrightarrow {\rm{E}}^{2}$ be the projection. Lemma 4.1.4 implies that $p$ is universal in the sense of Sect.\,4.1.\,\,Since $\pi_{\ast}(\Psi)$ descends weakly to $K$, it results then from Theorem \ref{weakdescent} that
$\big(\pi_{\ast}(\Psi)\big)_{\mathcal{N}}$ has no Zariski-dense image. Observe that $\psi = p_{\ast} \big(\pi_{\ast}(\Psi)\big)$. Universality of $p$ and
Proposition 4.2.1\,\,imply then that $\psi_{\mathcal{N}}$ has no Zariski-dense image. By Theorem \ref{weakdescent} the homomorphism $\psi$ descends weakly to $K$. We receive a contradiction to Claim \ref{min}. So, the assertion follows in the case when $\rm{E}$ is not isogenous to its complex conjugate.\\
\\
We proceed to the proof in the second case. We begin by showing
\begin{claim} \label{;;} If $\rm{E}$ is isogenous to its complex conjugate, but without complex multiplication, then $$\dim_{\overline{\Q}} \{Re\,\lambda_1, Im\,\lambda_1, Re\,\lambda_2, Im\,\lambda_2\} = 2$$
and
$$ \dim_{\overline{\Q}} \{1, \lambda_1, \eta_1, \lambda_2, \eta_2, 2\pi i\} - 2 = \dim_{\overline{\Q}} \{\lambda_1, \eta_1, \lambda_2, \eta_2, \} = 4.$$

\end{claim}
\begin{proof} If $\rm{E}$ is isogenous to a curve with model over $K$, then there exist a non-zero algebraic number $\beta$ such that a sublattice $\Lambda'$ of $\beta\Lambda$ is stable with respect to complex conjugation. Then $\Lambda' = \Z\lambda_1' + \Z \Lambda_2'$ with $\lambda_1' \in \R$ and
$\lambda_2' \in i\R$. As ${\Q}\Lambda' = {\Q}(\beta\Lambda)$, the first identity follows in this situation. Next assume that $\rm{E}$ is not isogenous to a curve with model over $K$ and let $\alpha \neq 0$ be an algebraic number such that $\alpha\Lambda \subset \Lambda^{\scriptscriptstyle h}$.
It follows that $\alpha d_1\lambda_1 + \alpha d_2 \lambda_2 = \lambda_1^{\scriptscriptstyle h}$ with rational $d_1, d_2 \in \Q$. If $d_2 = 0$, then $\lambda_1/|\lambda_1| \in \overline{\Q}.$ But then
${\rm{E}}$ would be isogenous to a curve with model over $K$, as follows from Corollary 7.1.1.\,\,Hence, $d_2 \neq 0$. This implies the first identity in the case when $\rm{E}$ is not isogenous to a curve with model over $K$.
So, the first identity is proved. And the second identity results from Masser's result preceding the corollary.
\end{proof}
By the assumptions of the second case there is an isogeny $v: {\rm{E}}^{\scriptscriptstyle h} \longrightarrow \rm{E}$. We identify ${\rm{E}}^{\scr h}$ with the curve corresponding to the lattice $\Lambda^{\scr h}$. Using identifications as in App.\,B.3, we let then ${{\rm{H}}}_s$ be the extension $v^{\ast}[\G_t]$ in ${\rm{Ext}}({\rm{E}}^{\scriptscriptstyle h}, \mathbb{G}_{a,\overline{\Q}})$ associated to a $s \in \overline{\Q} = {\rm{H}}^{\scr 1}({\rm{E}}, \mathcal{O}_{\rm{E}})$. We consider the induced isogeny $w: {{\rm{H}}}_s \longrightarrow {\rm{G}}_t$ over $v$. For $j = 1,2$ we write ${\eta}_j^{\scriptscriptstyle h} = \eta_{\Lambda^{\scriptscriptstyle h}}(\lambda_j^{\scriptscriptstyle h})$ and
$\sigma_j = (s\eta_j^{\scr h}, \lambda_j^{\scr h})$. By Lemma C.3.1\,\,$\Z\sigma_1 + \Z\sigma_2$ equals the kernel of ${\rm{exp}}_{\HHH_s}.$
The differential $w_{\ast}$ is defined over $\overline{\Q}$
and maps the kernel $\Z\sigma_1 + \Z\sigma_2$ of ${\rm{exp}}_{{\rm{H}}_s}$ into a cofinite submodule of the kernel
$\Z\tau_1 + \Z\tau_2$ of ${\rm{exp}}_{{{\rm{G}}}_t}$. We conclude that 
$$
\eta_1^{\scriptscriptstyle h}, \lambda_1^{\scriptscriptstyle h}, \eta_2^{\scriptscriptstyle h}, \lambda_2^{\scriptscriptstyle h} \subset \overline{\Q}\eta_1 +  \overline{\Q}\lambda_1 + \overline{\Q}\eta_2 + \overline{\Q}\lambda_2.
$$
The definition of the Weierstra\ss\,functions (see \cite{Wal12}, e.g.) implies that for $j = 1,2$ the number $\eta_j^{\scriptscriptstyle h}$ is in fact the complex conjugates of $\eta_j$.
Hence, for $j = 1,2$ we have
\begin{equation}Re\,\eta_j, Re\,\lambda_j, Im\,\eta_j, Im\,\lambda_j \subset  \overline{\Q}\eta_1 +  \overline{\Q}\lambda_1 + \overline{\Q}\eta_2 + \overline{\Q}\lambda_2.
\end{equation}
\begin{claim} For $l = 1,.., 4$ there exist $\odot_l \in \{Re, Im\}$ such that
\begin{equation}\overline{\Q}\eta_1 +  \overline{\Q}\lambda_1 + \overline{\Q}\eta_2 + \overline{\Q}\lambda_2 =  \overline{\Q}(\odot_1\eta_1) + \overline{\Q}(\odot_2\lambda_1) +  \overline{\Q}(\odot_3\eta_2) + \overline{\Q}(\odot_4\lambda_2).
\end{equation}\end{claim}
\begin{proof}[Sketch of the proof] By Claim \ref{;;} the space
$\overline{\Q}\eta_1 +  \overline{\Q}\lambda_1 + \overline{\Q}\eta_2 + \overline{\Q}\lambda_2$
has dimension four over $\overline{\Q}$. Since $\lambda_1$ and $\lambda_2$ are both non-zero, it follows from Claim 7.1.9 that for a suitable choice of $\odot_2$ and $\odot_4$ we have
$$\overline{\Q}\eta_1 +  \overline{\Q}\lambda_1 + \overline{\Q}\eta_2 + \overline{\Q}\lambda_2 =  \overline{\Q}\eta_1 + \overline{\Q}(\odot_2\lambda_1) +  \overline{\Q}\eta_2 + \overline{\Q}(\odot_4\lambda_2)$$
and 
$$\overline{\Q}(\odot_2\lambda_1) + \overline{\Q}(\odot_4\lambda_2) = \overline{\Q}\lambda_1 + \overline{\Q}\lambda_1.$$
Using Claim 7.1.9 and Statement (7.1.1) together with the fact that $\eta_1$ and $\eta_2$ are non-zero, one deduces that for a suitable choice of $\odot_1$ and $\odot_3$ the identity in (7.1.2) holds. This amounts to a lengthy, but straightforward verification based on linear algebra. \end{proof}
The second equality in Claim 7.1.9 and Claim 7.1.10 imply the assertion of the corollary in the second case.\\
\\
After all we only sketch the proof in the third case, that is, when $\rm{E}$ is a CM-curve. If $\rm{E}$ admits complex multiplication, then $\rm{E}$ is isogenous to a curve with model over $K$. Applying an isogeny to a curve over $K$ together with a calculation as above, one shows the two identities
\begin{center}
\begin{tabular}{l}
$\dim_{\overline{\Q}} \{Re\,\lambda_1, Im\,\lambda_1, Re\,\lambda_2, Im\,\lambda_2\}$\\
$ = \dim_{\overline{\Q}} \{\lambda_1, \lambda_2 \}$
\end{tabular}
\end{center}
and
\begin{center}
\begin{tabular}{l}
$\dim_{\overline{\Q}} \{1, Re\,\lambda_1, Im\,\lambda_1, Re\,\eta_1, Im\,\eta_1, Re\,\lambda_2, Im\,\lambda_2, Re\,\eta_2, Im\,\eta_2, 2\pi i\}$\\
$= \dim_{\overline{\Q}} \{1, \lambda_1, \eta_1, \lambda_2, \eta_2, 2\pi i\}$ 
\end{tabular}
\end{center}
hold. The assertion follows then from Masser's result. This was the last open case. Everything is proved.
\end{proof}
\begin{corollary} \label{po} Let $\Lambda$ be a lattice in $\C$ with algebraic invariants and let $\lambda \in \Lambda \setminus \{0\}$. Then the following assertions are equivalent.
\begin{enumerate}
 \item $Re\, \lambda^{\scriptscriptstyle -1 }\Lambda \nsubseteq \Q$.
 
\item For each $\omega \in \mathcal{L}_{\Lambda}\setminus {\Q}\Lambda$ the numbers $\zeta(\omega)\lambda- {\eta(\lambda)}\omega, \lambda$ and\\ their complex conjugates are linearly independent over $\overline{\Q}$.
\end{enumerate}
\end{corollary}

\begin{proof} For a general $\omega \in \mathcal{L}_{\Lambda}\setminus  {\Q}\Lambda$ we define
$$f_{\omega}(z_1, z_2) = \frac{\sigma^{\scriptscriptstyle 3}(z_2-\omega)e^{3\zeta(\omega)z_2 + z_1} }{\sigma^{\scriptscriptstyle 3}(z_2)\sigma^{\scriptscriptstyle 3}(\omega)}.$$
With notations as in App.\,B.3, let $\pi: {\rm{G}}_{\omega} \longrightarrow {\rm{E}}$ be the extension of $\rm{E}$ by $\mathbb{G}_{m, \overline{\Q}} $ associated to $\omega$. By Lemma C.1.4 ${\rm{G}}_{\omega}$ is not isotrivial. The extension admits a uniformization
${\rm{exp}}_{{\rm{G}}_{\omega}} : \C^{\scriptscriptstyle 2} \longrightarrow \mathbb{P}^5(\C)$ over $\overline{\Q}$ given by
$${\rm{exp}}_{{\rm{G}}_{\omega}}(z_1,z_2) = \big[\wp(z_2): \wp'(z_2): 1 : \wp(z_2-\omega)f_{\omega}(z_1,z_2): \wp'(z_2-\omega)f_{\omega}(z_1,z_2): f_{\omega}(z_1,z_2)\big].$$
Set ${\rm{y}} = \zeta(\omega) - \frac{\eta(\lambda)}{\lambda}\omega$. Lemma \ref{kern} teaches that $(-3{\rm{y}}\lambda, \lambda)$ is a period of ${\rm{exp}}_{{\rm{G}}_{\omega}}$. Let $\Psi(r) = {\rm{exp}}_{{\rm{G}}_{\omega}}(-{\rm{y}}\lambda r, \lambda r)$. 
Statement 2.\,\,is wrong if and only if the real and imaginary parts of the two complex numbers in the statement are linearly dependent over $K $. The latter holds if and only if there is a proper subspace $\frak{t} \subset \frak{g}_{\omega} = {\rm{Lie}}\,\G_{\omega}$ over $K$ with the property that $\Psi_{\ast}(\R) \subset \frak{t} \otimes_K \R$.
By Theorem \ref{wus1} this implies that $\Psi$ descends weakly to $K$. Then, either $\Psi$ descends to $K$ or $\Psi$ descends weakly to $K$, but not in the strong sense. In the first case Theorem 2.8.2 teaches that also $\psi = \pi_{\ast}(\Psi)$ descends to $K$. In the second case the homomorphism $v: \G \longrightarrow \G' \otimes_K \overline{\Q}$ in the definition of weak descent has kernel of positive dimension. Hence, $v$ must factor through $\pi$ as, say $v= w \circ \pi$. Since 
${\rm{E}} = im\,\pi$ has dimension one, it follows then that $\psi = \pi_{\ast}(\Psi)$ descends to $K$ via $w$. Corollary\,\ref{lp2} and Corollary\,\ref{lp3} imply that
$\psi $ descends to $K$ if and only if the first statement is wrong. So far, we have shown that if the second statement is not true, so isn't the first. On the other hand,
if the first statement is wrong, so that $\psi$ descends weakly to $K$ as just observed, then it follows from Proposition \ref{real} that Statement 2.\,\,is wrong. 
The corollary is proved.
\end{proof}
The corollary implies that if $\frac{\zeta(\omega)\lambda - {\eta(\lambda)\omega}}{|\zeta(\omega)\lambda - {\eta(\lambda)\omega|}}$ is algebraic for some $\omega \in \mathcal{L}_{\Lambda} \setminus \Lambda_{{\Q}}$ and some $\lambda \in \Lambda$, then ${{\rm{E}}}$ is isogenous to an elliptic curve over $\R$. Conversely:
\begin{corollary} Let $\Lambda$ be a lattice in $\C$ with algebraic invariants and suppose that $\alpha\Lambda \nsubseteq \Lambda^{\scriptscriptstyle h}$ for all non-zero algebraic numbers $\alpha$.
Then, for all $\lambda \in \Lambda \setminus \{0\}$ and all $\omega \in \mathcal{L}_{\Lambda} \setminus \Q\Lambda$, the number
$$\frac{\zeta(\omega)\lambda - {\eta(\lambda)\omega}}{|\zeta(\omega)\lambda - {\eta(\lambda)\omega|}}$$
 is transcendental.
\end{corollary}
Slightly more advanced is the following version of Corollary\,\ref{po}. It demonstrates the utility of Proposition\,\ref{98}. 
\begin{corollary} Let $\Lambda$ be a lattice in $\C$ with algebraic invariants and let $\omega \in \mathcal{L}_{\Lambda} \setminus \Lambda$ be an algebraic logarithm such that $\omega - h(\omega), \omega + h(\omega) \notin \Q\Lambda $.
Then, for all $\lambda \in \Lambda \setminus \{0\}$, the three numbers $Re\,\big(\zeta(\omega)\lambda - {\eta(\lambda)\omega}\big), Im\,\big({\zeta(\omega)\lambda - {\eta(\lambda)\omega}}\big)$ and $\lambda$
are linearly independent over $\overline{\Q}$.
\end{corollary}
\begin{proof}[Sketch of the proof] We use notations as in the proof of Corollary\,\ref{po}. We consider the homomorphism $\Psi(r) = {\rm{exp}}_{{\rm{G}}_{\omega}}({\rm{y}}\lambda r, \lambda r)$ and let $\pi: \G_{\omega} \longrightarrow \rm{E}$ be the projection. Using arguments as in the proof of Corollary 7.1.6, one can show that the assertion of the corollary is true if and only it is true after the three numbers have been replaced by their pendants 
$$Re\,\big(\zeta'(\omega')\lambda' - {\eta'(\lambda')\omega'}\big), Im\,\big({\zeta'(\omega')\lambda' - {\eta(\lambda')\omega'}}\big)\,\,\mbox{and}\,\,\lambda'$$
arising from an isogeny $v: {\rm{E}} \longrightarrow {\rm{E}}'$. So, if $\pi_{\ast}(\Psi)$ descends to $K$, then after transition to an isogeny we may assume that $\rm{E}$ is defined over $\R$ and that $\lambda \in \R = \frak{e}'$. If in this situation the three numbers are linearly dependent over $\overline{\Q}$, then
they are linearly dependent over $K$. But then there is a proper subspace $\frak{t} \subset \frak{g}_{\omega}$ of dimension $\dim\,\frak{t} = \dim \G_{\omega} = 2$ over $K$ such $\Psi_{\ast}(\R)$ is contained
in $\frak{t} \otimes_K \R$. Theorem \ref{wus1} implies that $\Psi$ descends to $K$. And Proposition\,\ref{98} yields that ${\rm{exp}}_{{\rm{E}}}(\omega + h(\omega))$ or ${\rm{exp}}_{{\rm{E}}}(\omega - h(\omega))$ is a torsion point.
This contradicts the assumption. On the other hand, if $\pi_{\ast}(\Psi)$ does not descend to $K$,
then we infer from Corollary C.2.2 that $Re\,\lambda^{\scriptscriptstyle -1}\Lambda \nsubseteq \Q$. The claim follows then from Corollary\,\ref{po}.
 \end{proof}

Let $t$ be a non-zero algebraic number and $\Lambda$ a lattice in $\C$ with algebraic invariants. If $\omega \in \mathcal{L}_{\Lambda} \setminus \Lambda$, then it follows from the definition of ${\rm{exp}}_{{\rm{G}}_t}$ that the value ${\rm{exp}}_{{\rm{G}}_t}\big(t\zeta(\omega), \omega\big)$ is an algebraic point of the extension ${\rm{G}}_t$ in ${\rm{Ext}}({\rm{E}}, \mathbb{G}_{a,\overline{\Q}})$. Moreover, the projection $\pi_t: \G_t \longrightarrow \rm{E}$ is the universal homomorphism. Using these two facts together with Theorem\,\ref{wus1} and Corollary \ref{321}, one proves the next corollary. Details are left to the reader.
\begin{corollary} Let $k \geq 1$ be an integer and let $\Lambda$ be a lattice in $\C$ with algebraic invariants. Suppose that $\alpha\Lambda \nsubseteq \Lambda^{\scriptscriptstyle h}$ for all non-zero algebraic numbers $\alpha$.
Let $\omega_j$, $j = 1,..., k$, be algebraic logarithms in $\mathcal{L}_{\Lambda} \setminus \Lambda$ which are linearly independent over $\Z$. Then the real and imaginary parts
of the numbers $\omega_j, \zeta(\omega_j)$, $j = 1,...., k$, are linearly independent over $K $.
\end{corollary}

\subsubsection{Generalizations of a theorem of Diaz} In his 2004 paper \cite{Diaz}, Diaz found an elementary proof for the following special case of the theorem of Gel'fond and Schneider from the introduction.
\begin{theorem} \label{DD} If $\omega \in \mathcal{L}$ is an algebraic logarithm such that 
the real and the imaginary part of $\omega$ are linearly dependent over $K$, then $\omega$ is either real or purely imaginary.
\end{theorem}
With the help of our main results we are able to generalize Diaz' theorem into several directions.
\begin{corollary} \label{E} Let $\rm{G}$ be a commutative group variety over $\overline{\Q}$ and let $\omega \in \frak{g}(\C)$ be the logarithm of an algebraic point $\xi \in \rm{G}(\overline{\Q})$. Let $\frak{t} $ be the smallest
subspace of $\frak{g}$ over $K$ such that $\omega \in \frak{t} \otimes_{K} \R$. If $\frak{t} \neq \frak{g}$, then there exists a homomorphism $v: \rm{G} \longrightarrow \rm{G}' \otimes_K \overline{\Q}$ onto a commutative group variety of positive dimension with model $\G'$ over $K$ such that $v(\xi) \in \rm{G}'(\R)$.
\end{corollary}  
\begin{proof} This is immediate from Theorem \ref{wus1} by setting $\Psi(r) = {\rm{exp}}_{\G}(\omega r)$ and $\pi = id$.\end{proof}
An interesting consequence of Corollary \ref{E} is
\begin{corollary}  \label{hhh} Let $\rm{G}$ be a simple commutative group variety over $\overline{\Q}$ of dimension $g$. 
For $j = 1,2$ let $\omega_j \in \mathcal{L}_{\rm{G}}$ be an algebraic logarithm and
let $\frak{t}_j$ be the smallest subspace of $\frak{g}$ over $K$ with the property that $\omega_j \in \frak{t}_j \otimes_K \R$.
If $\frak{t}_j \neq \frak{g}$ for $j = 1,2$, then $\frak{t}_1 \cap \frak{t}_2 = \{0\}$ or $\frak{t}_1 = \frak{t}_2$. 
\end{corollary}
\begin{proof} Fix a $j = 1,2$. We apply Theorem\,\ref{wus2} to $\Psi_j(r) = {\rm{exp}}_{\G}(\omega r)$, $\pi_j = id.$ and $\frak{t}_j$ and deduce that 
there is a non-zero homomorphism $w_j: \G \longrightarrow \G_j' \otimes_K \overline{\Q}$ such that $(w_j)_{\ast}(\Psi_j)$ takes values in $\G_j'(\R)$. The simplicity of $\G$ implies that $w_j$ is an isogeny, and we infer that $(w_j)_{\ast}(\frak{t}_j) \subset \frak{g}_j'$. It follows from
Theorem \ref{wus1} that $\dim\,\frak{t}_j \geq \dim\,\G = \dim\,\frak{g}'$. So, $(w_j)_{\ast}(\frak{t}_j) = \frak{g}_j'$.
 Next, let $v_1: \G_j' \otimes_K \overline{\Q} \longrightarrow \G$ be an isogeny with the property that $w_1 \circ v_1$ is multiplication with an integer and set $v_2 = w_2$. The previous yields 
$$(v_1)_{\ast}\big(\frak{g}_1'\big) = \frak{t}_1\,\,\mbox{and}\,\,(v_2)_{\ast}\big(\frak{t}_2\big) = \frak{g}_2'.$$
Define $v = v_2 \circ v_1$. Then
$$v\big({\rm{G}}'_1(\R)\big) \cap {\rm{G}}_2'(\R) \supset v_2\big({\rm{exp}}_{{\rm{G}}}(\frak{t}_1 \cap \frak{t}_2)\big)$$
is a positive dimensional real Lie group. Let $\xi \in \rm{G}'_1(\R)$ be an element which generates a Zariski-dense subgroup of ${\rm{G}}'_1(\R)$
and such that $v(\xi) \in {\rm{G}}_2'(\R)$. Then $v(\xi') = v^{\scriptscriptstyle h}(\xi')$ for each $\xi' \in \Z\xi$. A density argument yields $v = v^{\scriptscriptstyle h}$. It follows from Lemma 3.2.2 that $v$ is defined over $K$. This in turn implies that $\frak{t}_1 = (v_1)_{\ast}^{\scriptscriptstyle -1}\big(\frak{g}_1'\big) = (v_2)_{\ast}^{\scriptscriptstyle -1}\big(\frak{g}_2'\big) = \frak{t}_2.$\end{proof}
Proposition \ref{real} and Corollary \ref{hhh} imply
\begin{corollary} \label{mm} Conjecture \ref{ConB} is equivalent to Conjecture \ref{ConD}.
 
\end{corollary}
The verification of this is left to the reader.
\begin{corollary} Let ${{\rm{A}}}$ be an abelian variety of dimension $g$ over $\overline{\Q}$ and let $u = u(\A)$ be the smallest dimension of a non-trivial simple algebraic subgroup of $\rm{A}$. Let
 $z_1,..., z_g$ be linear coordinates on $\frak{a}(\C)$
 over $\overline{\Q}$. Suppose a non-zero algebraic logarithm
$$\omega = (\omega_1,..., \omega_g) \in \mathcal{L}_A = {\rm{exp}}_{{\rm{A}}}^{ \scriptscriptstyle -1}\big({\rm{A}}(\overline{\Q})\big)$$
with the property that the numbers
$$Re\,\omega_1, Im\,\omega_1,....,Re\,\omega_g, Im\,\omega_g$$
generate a space of dimension $n$ over $K$ such that either $n \leq u$ or $2g-u \leq n < 2g$. Then ${{\rm{A}}}$ contains a simple algebraic subgroup of positive dimension which is isogenous to an algebraic group with model over $K$. 
\end{corollary}
\begin{proof}[Sketch of the proof] Let $\Psi(r) = {\rm{exp}_{\G}}(\omega r)$. As in the proof of Theorem \ref{wus1} (Subsect.\,6.2.1) one shows that $n$ equals the dimension of the Zariski-closure ${\rm{H}} \subset \mathcal{N}_{\overline{\Q}/K}(\G) \otimes_K \overline{\Q} = \A \times \A^{\scr h}$ of $\Psi_{\mathcal{N}}(\R)$.
The algebraic groups $\A \times \A^{\scr h}$ and $\rm{H}$ are plurisimple. Moreover, $u(\A) = u(\A^{\scr h}) = u(\A \times \A^{\scr h})$. So, the hypotheses imply that $\delta(\rm{H})$ is odd. The claim follows then
from Corollary \ref{weakdescentodd} and because $\A$ is plurisimple.\end{proof}

\subsubsection{On Waldschmidt's theorem about the density of rational points}
As mentioned in the introduction, Mazur's conjecture is the motivation for Waldschmidt's article \cite{Walden}. In this paper, Waldschmidt considers first a commutative group variety $\rm{G}$ over a real number field and treats the question when a subgroup $\Gamma \subset \rm{G}^o(\R) \cap \rm{G}(K)$ is dense in the connected component of unity $\rm{G}^o(\R)$ of the real Lie group $\rm{G}(\R)$. In the last part of \cite{Walden} the assumption that $\rm{G}$ is defined over a real subfield is dropped. Instead, the Weil restriction $\mathcal{N}_{\overline{\Q}/K}(\rm{G})$ is introduced which enables to apply the previously established density criterions to subgroups $\Gamma \subset \rm{G}(\overline{\Q})$. The trick is to replace $\Gamma$ by the subgroup $\Gamma_{\mathcal{N}} = \{(\gamma, \gamma^{\scr h}), \gamma \in \Gamma\}$ of $\mathcal{N}_{\overline{\Q}/K}(\rm{G})$.
A spin off of this idea is \cite[Corollary\,5.4 a)]{Walden}:
\begin{theorem}  \label{C} Let $\rm{A}$ be a simple abelian variety of dimension $g$ over $\overline{\Q}$ and let $\Gamma \subset \rm{A}(\overline{\Q})$ be a group of rank
$\geq 4g^{\scriptscriptstyle 2} - 2g + 1$. Then $\Gamma$ is dense in $\rm{A}(\C)$ unless there is an abelian subvariety $\rm{A}' \subset \mathcal{N}_{\overline{\Q}/K}(\rm{G})$ such that
$\Gamma_{\mathcal{N}}/\big(\Gamma_{\mathcal{N}} \cap \rm{A}'(\R)\big)$ has rank $< g^{\scriptscriptstyle 2}-g-1$.
\end{theorem}
Using our results, we are able to simplify the condition in Theorem\,\ref{C} concerning $\Gamma_{\mathcal{N}}$.
\begin{corollary} {Let $\rm{A}$ be a simple abelian variety of dimension $g$ over $\overline{\Q}$ and let $\Gamma \subset \rm{A}(\overline{\Q})$ be a group of rank
$\geq 4g^{\scriptscriptstyle 2} - 2g + 1$. Then $\Gamma$ is dense in $\rm{A}(\C)$ unless there is a proper subspace $\frak{t} \subset \frak{a}$ over $K = \overline{\Q} \cap \R$ with the property
that $\Gamma/\big(\Gamma \cap {\rm{exp}}_{{\rm{G}}}\big(\frak{t} \otimes_K \R\big) \big)$ has rank $< g^{\scriptscriptstyle 2}-g-1$.}
 
\end{corollary}
\begin{proof} Define $\Gamma_{\mathcal{N}} = \left\{(\gamma, \gamma^{\scriptscriptstyle h}), \gamma \in \Gamma\right\}$. Because of Theorem\,\ref{C} we need to show
\begin{center}
 \begin{tabular}{cp{13cm}}
  $(\ast)$ & \textit{If a proper subspace $\frak{t}$ over $K$ exists such that $\Gamma/\big(\Gamma \cap {\rm{exp}}_{{\rm{G}}}\big(\frak{t} \otimes_K \R\big) \big)$ has rank $< g^{\scriptscriptstyle 2}-g-1$,
then there is an abelian subvariety ${\rm{A}}' \subset \mathcal{N}_{\overline{\Q}/K}(\A)$ such that $\Gamma_{\mathcal{N}}/\big(\Gamma_{\mathcal{N}} \cap {\rm{A}}'(\R) \big)$ has rank $< g^{\scriptscriptstyle 2}-g-1$.}
 \end{tabular}
\end{center}
It suffices to prove $(\ast)$ for a free subgroup $\Gamma' =\Z\xi_1 + ...+ \Z\xi_k$ of finite index in $\Gamma$. We may even assume that $\xi_1,..., \xi_l \in {\rm{exp}}_{\rm{G}}\big(\frak{t} \otimes_K \R\big)$ for some $l \leq k$ such that $0 \leq k - l < g^{\scriptscriptstyle 2}-g-1$.\\
For $j = 1,..., l$ let $\omega_j \in \frak{t} \otimes_K \R$ be an algebraic logarithm of $\xi_j$. Let $\frak{t}_j \subset \frak{t}$ be the smallest vector space over $K$ such that $\omega_j \in \frak{t}_j \otimes_K\R$. 
We apply Theorem \ref{wus2} to $\G = \A$, $\pi = id.$, $\Psi_j(r) = {\rm{exp}}_{{\rm{G}}}\big(\omega_j r)$ and $\frak{t}_j$. We infer using simplicity of $\A$ that
$\Psi_j$ descends to $K$ via an isogeny $v_j: \A \longrightarrow \A'_j \otimes_K F$. In particular, $(v_j)_{\ast}\frak{t}_j \subset \frak{a}_j'$, so that $\dim\,\frak{t}_j \leq \dim\,\A'$. Theorem \ref{wus1} yields that $\dim\,\frak{t}_j = \dim\,\A$. 
In particular,
$$\dim_{K}\frak{a} = 2\dim\,\A = \dim\,\frak{t}_j + \dim\,\frak{t}_i$$
for all $i,j = 1,..., l$. So, 
 $$\dim\,\frak{t} < \dim_{K}\frak{a} = \dim\,\frak{t}_j + \dim\,\frak{t}_i$$
for all $i,j = 1,..., l$. Thus, $\frak{t}_i \cap \frak{t}_j \neq \{0\}$ for all $i,j = 1,..., l$. Corollary\,\ref{hhh} implies that $\frak{t}_i  = \frak{t}_j$ for all $i,j = 1,..., l$. It follows that 
\begin{equation} v_1(\xi_j) \in \A_1'(K)
\end{equation}
for all $j = 1,..., l$. Let $w: \A_1' \otimes_K F \longrightarrow \A$ be an isogeny
such that $v_1 \circ w$ is multiplication with an integer. Since $\A'_1$ is simple, the functorially induced morphism
$\mathcal{N}(w): \A'_1 \longrightarrow \mathcal{N}_{\overline{\Q}/K}(\A)$ defines an isogeny onto its image $\A' = \mathcal{N}(w)\big(\A_1')$.
Statement (7.1.3) implies that $\Gamma_l = w^{\scr -1}(\Z\xi_1 + ... + \Z \xi_l) \cap \A'_1(K)$ has rank l. As the isogeny $\mathcal{N}(w)$ is defined over $K$, the group $\mathcal{N}(w)\big(\Gamma_l) \subset \A'(K)$
has rank $l$, too. Define $\tilde{\Gamma} = \mathcal{N}(w)\big(w^{\scr -1}(\Gamma)\big)$. The previous implies that the group $\tilde{\Gamma}/\big(\tilde{\Gamma} \cap \A'(\R)\big)$ has rank $\leq k-l$. Moreover, as $w$ is an isogeny and as $\mathcal{N}(w) \otimes_K \overline{\Q} = (w, w^{\scr h})$, $\Gamma_{\mathcal{N}}$ is a cofinite subgroup of $\tilde{\Gamma}$.
Assertion $(\ast)$ follows.\end{proof}
\subsubsection{Solution to Exercise \ref{opu}} If $\frak{t} \neq \frak{g}$, then Theorem \ref{wus2} applied with $\pi = id.$ and $\U = \G$ yields that $\Psi$ descends weakly to $K$.
Since $\A$ and $\rm{B}$ are disjoint, each positive dimensional quotient of $\G$ has dimension $\geq 2$. It follows from Proposition \ref{real} that
$\dim\,\frak{t} \leq 8$. Assume that $p_{\ast}(\Psi)$ descends to $K$, but $q_{\ast}(\Psi)$ does not. Then $\dim\,\frak{t} = 8$ for otherwise Theorem \ref{wus2} applied with $\pi = q$ and $\U = \rm{B}$
would imply that $q_{\ast}(\Psi)$ descends to $\R$. Next suppose that $q_{\ast}(\Psi)$ descends to $K$, but $p_{\ast}(\Psi)$ does not. Reversing the roles of $\A$ and $\rm{B}$
and applying Proposition \ref{real}, we infer that $\dim\,\frak{t} = 7$. Finally if $p_{\ast}(\Psi)$ and $q_{\ast}(\Psi)$ descend to $K$, then $\dim\,\frak{t} \leq 2 + 3 = 5$ by Proposition \ref{real}, and Theorem \ref{wus1} with $\pi = id.$ and $\U = \G$ implies that $\Psi$ descends to $K$.
From the same theorem we deduce then that $\dim\,\frak{t} = 5$.

\section{Real-analytic generalizations of the six exponentials theorem} In this section we state some nice consequences of Theorem \ref{Schn3} which is inspired by the six exponentials theorem from Sect.\,1.2.
\begin{corollary} \label{ty} Let $\Lambda$ be a lattice in $\C$ with algebraic invariants and let $a_j$, $j = 1,2,3$, be three complex numbers not in $\Lambda$ which are linearly
independent over $\Z$, but collinear over $\R$. Then, for all non-zero complex numbers $b$ such that either $b/|b|$ is transcendental or $ba_1 \notin \R \cup i\R$, at least one among the six numbers $e^{ba_j}$, $\wp(a_j)$ is not algebraic. 
\end{corollary}
\begin{proof} Suppose that all six numbers are algebraic. Set $\G = \mathbb{G}_{m, \overline{\Q}} \times \rm{E}$ and let $p$ (resp.\,$q$) denote the projection to $\mathbb{G}_{m, \overline{\Q}}$ (resp.\,to $\rm{E}$).
Then one deduces from Theorem \ref{Schn1} applied with ${\rm{g}}_m = 1, {\rm{g}}_a = 0$, ${\rm{r}} \geq 3$ and ${\rm{k}} \geq 0$ that the homomorphism $\Psi(r) = \big(e^{b a_1 r}, \wp(a_1 r)\big)$ to the set of complex points of $\G$ descends to $\R$.\footnote{Recall that we identify $\wp(z)$ with the exponential map of $\rm{E}$ for brevity.} Because of Theorem \ref{inherited} the latter is only possible if the two homomorphisms $p_{\ast}(\Psi)$ and $q_{\ast}(\Psi)$ both descend to $\R$. We saw in the previous section that if $q_{\ast}(\Psi)$ descends to $\R$, then $a_1/|a_1|$ is algebraic $(\ast)$.
It follows from Corollary \ref{lp} that if $p_{\ast}(\Psi)$ descends to $\R$, then $ba_1 \in \R \cup i\R$\,$(\ast\ast)$. The assertions $(\ast)$ and $(\ast\ast)$ imply that $b/|b|$ is algebraic. The claim results by contraposition.\end{proof}
Next we prove a variation of the six exponential theorem for elliptic curves. 
\begin{corollary} \label{tyy} Let $\Lambda_1$ and $\Lambda_2$ be two lattices in $\C$ with algebraic invariants. Let $a_j$, $j = 1,2,3$, be three complex numbers not in $\Lambda_1 \cup \Lambda_2$ which are linearly
independent over $\Z$, but collinear over $\R$. Suppose that $\Lambda_2$ is neither isogenous to $\Lambda_1$ nor to $\Lambda_2^{\scr h}$.
Then, for all complex numbers $b$ such that $ba_1/|ba_1|$ is transcendental, at least one among the six values $\wp_1(ba_j)$, $\wp_2(a_j)$ is defined and not algebraic. 

\end{corollary}
The statement of this corollary is similar to the statement of the previous one, but the strategy of proof is rather different.
\begin{proof} We define $\Psi_1: \R \longrightarrow \rm{E}_1(\C)$ to be the homomorphism $\Psi_1(r) = \wp_j(ba_1r)$ and let $\Psi_2: \R \longrightarrow \rm{E}_2(\C)$ be the homomorphism $\Psi_2(r) = \wp_2(a_1r)$.
Moreover, we set $\G = {\rm{E}}_1 \times {\rm{E}}_2 \times {\rm{E}}_2^{\scr h} $ and ${\rm{U}} = {\rm{E}}_1$. We consider the projection $\pi: \G \longrightarrow \rm{U}$ and define $\Psi = \Psi_1 \times (\Psi_2)_{\mathcal{N}}$.
The hypotheses imply that the image of $\Psi$ is Zariski-dense in $\G(\C)$. If now all six numbers in question are algebraic, then the group $\Psi^{\scr - 1}\big(\G(\overline{\Q})\big)$ has rank $\geq 3$. In this situation it follows from Theorem \ref{Schn2} applied with ${\rm{g}}_m = {\rm{g}}_a = 0$, ${\rm{r}} \geq 3$ and ${\rm{k}} \geq 0$ that $\Psi_1 = \pi_{\ast}(\Psi)$ descends to $\R$.
As seen in the previous section, the ratio $ba_1/|ba_1|$ is then algebraic. Contradiction.
\end{proof}
The third result in this section is more abstract.
\begin{corollary} Let $\A$ be a simple abelian variety over $\overline{\Q}$ and let $\Psi: \R \longrightarrow \A(\C)$ be a non-zero real-analytic homomorphism
such that the group of algebraic logarithms $\Psi^{\scr - 1}\big(\A(\overline{\Q})\big)$ has rank $\geq 3$. Let $\rm{C}$ be the closure of $\Psi(\R)$ with respect to the analytic topology. Then ${\rm{C}}$ is a real Lie group of dimension $\dim\,{\rm{C}} \leq 3$. 
\end{corollary}
 \begin{proof} We apply Theorem \ref{Schn0} to $\G = \A$, to the complexification $\Phi$ of $\Psi$ and with ${\rm{g}}_m = {\rm{g}}_a = 0$, ${\rm{r}} \geq 3$ and ${\rm{k}} \geq 0$. It follows that 
$\dim\,\A \leq 3$. Next we apply Theorem \ref{Schn2} to the same setting (without $\Phi$) and with $\pi = id.$ It results
that $\Psi$ descends weakly to $K$ if $\dim\,\rm{A} \geq 2$. Since $\A$ is simple, $\Psi$ descends weakly to $K$ if and only if
it descends to $K$. Now, in general for an abelian variety $\A'$ over $\R$ the dimension of the real Lie group $\A'(\R)$ equals $\dim\,\A'$. So, the assertion of the corollary follows if $\dim\,\rm{A} \geq 2$. On the other hand, if $\dim\,\rm{A} \leq 1$, then $\dim\,\rm{C} \leq 2$. 
The claim is proved.
\end{proof}

\section{Algebraic independence of values connected to $\mathbf{e^z}$ and Weierstra\ss\,elliptic functions}
\subsubsection{On a theorem of Chudnovsky}
Between 1974 and 1981 Chudnovsky proved striking results of algebraic independence related to elliptic functions. One among these results was the algebraic independence of the two numbers
$$\zeta(\omega) - { \frac{\eta(\lambda)}{\lambda}}\omega,{ \frac{\eta(\lambda)}{\lambda}}$$
for a lattice $\Lambda$ in $\C$ with algebraic invariants, an algebraic logarithm $\omega \in \mathcal{L}_{\Lambda} \setminus \Q\Lambda$ and associated values $\eta(\lambda)$, $\zeta(\omega)$. In \cite[Con.\,38]{Wal12} Waldschmidt states a conjecture which generalizes Chudnovsky's theorem.
\begin{con}
Let $\Lambda$ be a lattice in $\C$ and $\lambda \in \Lambda \setminus \{0\}$. If $\omega \notin \Q\lambda \cup \Lambda$ is a complex number, then two among the numbers
$$g_2, g_3, \wp(\omega), \zeta(\omega) - { \frac{\eta(\lambda)}{\lambda}}\omega, { \frac{\eta(\lambda)}{\lambda}}$$
are algebraically independent.
\end{con}
The conjecture is closely related to extensions of $\rm{E}$ by $\mathbb{G}_{a}$. In the direction of the conjecture we prove the following results.
\begin{corollary} \textit{Let $\Lambda$ be a lattice in $\C$ with real, but not algebraic invariants and let $\lambda \in \Lambda \setminus \{0\}$ be such that $Re\,\lambda^{\scriptscriptstyle -1}\Lambda \nsubseteq \Q$. If $\omega \in (\R \cup i\R)\lambda \setminus (\Q\lambda \cup \Lambda)$, then two among the numbers
$$g_2, g_3, \wp_{}(\omega), \zeta_{}(\omega) - {\eta_{}(\lambda)} \omega, {\eta_{}(\lambda)}, \lambda$$
are algebraically independent.}
\end{corollary}
 \begin{proof} With notations as in App.\,A.3 we let $t = 1$ and consider the homomorphism 
 $$\Psi(r) = {\rm{exp}}_{{\rm{G}}_t}\big(\eta_{}(\lambda)r, \lambda r\big)$$
to the extension $\G = {\rm{G}}_t$ of ${\rm{U}} = {\rm{E}}$ by the additive group. It follows from Lemma \ref{kern} that $\Psi$ is $1$-periodic. We let $F_1$ be the algebraically closed field generated by the six numbers $g_2, g_3, \wp_{}(\omega), \zeta(\omega) - \eta(\lambda)\omega, {\eta(\lambda)}, \lambda$ and define $F_2$ to be the algebraically closed field generated by $F_1$ and the conjugates of the six numbers.
Then 
$$\Z + \Z\omega \subset \Psi^{\scr-1}\big(\G_t(F_2)\big) \cup \Psi_{[i]}^{\scr-1}\big(\G_t(F_2)\big).$$
We write $\pi$ for the projection from $\G$ to ${\rm{U}}$ and apply Theorem\,\ref{Gel3} with $\dim\,\frak{t} = 1$, ${\rm{g}}_c = 2$, ${\rm{g}}_m = {\rm{g}}_a = 0$, ${\rm{r}} \geq 2$ and $\rm{k} \geq 1$. It follows that $\pi_{\ast}(\Psi)$ descends weakly to $K$ if ${\rm{trdeg}}_{\Q} F_2 = 1$. Since $g_2, g_3$ are real, but not both algebraic, it holds that if ${\rm{trdeg}}_{\Q} F_1 \leq 1$, then ${\rm{trdeg}}_{\Q} F_2 \leq 1$. Consequently, if ${\rm{trdeg}}_{\Q} F_1 \leq 1$, then $\pi_{\ast}(\Psi)$ descends weakly to $K$. For dimension reasons $\pi_{\ast}(\Psi)$ descends weakly to $K$ if and only if $\pi_{\ast}(\Psi)$ descends to $K$.
And if $\pi_{\ast}(\Psi)$ descends to $K$, then $Re\,\lambda^{\scriptscriptstyle -1}\Lambda \subset \Q$ by Corollary\,B.1.2.\,\,We infer the claim by contraposition.
\end{proof}
Similarly we get
\begin{corollary} \textit{Let $\Lambda$ be a lattice in $\C$ and let $\lambda \in \Lambda \cap \R^{\ast}$ be such that $Re\,\lambda^{\scriptscriptstyle -1}\Lambda \nsubseteq \Q$. If $\omega \in (\R \cup i\R) \setminus (\Q\lambda \cup \Lambda)$, then two among the numbers
$$g_2, g_3, \wp(\omega), \zeta_{}(\omega) - { \frac{\eta_{}(\lambda)}{\lambda}}\omega, { \frac{\eta_{}(\lambda)}{\lambda}}$$
and their complex conjugates are algebraically independent.}
\end{corollary}
\subsubsection{On the algebraic independence of $\pi$ and an elliptic period}

Let $\Lambda$ be a lattice in $\C$ with algebraic invariants. A "folklore'' problem in transcendence theory it the question whether $\pi$ and a lattice element $\lambda \in \Lambda$ can be algebraically dependent. It is known that this is not the case if $\Lambda$ admits complex multiplication (see Waldschmidt \cite{Wal12}). In this direction we can state the following result.
\begin{corollary} \label{ui} Let $\Lambda $ be a lattice in $\C$ with algebraic invariants and suppose that $Re \lambda^{- \scriptscriptstyle 1}\Lambda \nsubseteq \Q$ for some $\lambda \in \Lambda \setminus \{0\}$. Let
$\omega \in \mathcal{L} \setminus i\pi \Q$ be a real or purely imaginary number and set $c = \frac{\lambda\omega}{i\pi}$. Then 
$${\rm{trdeg}}_{\Q}\Q\big(\pi/|\lambda|, {\lambda}/{|\lambda|}, \wp(c)\big) \geq 2$$
and
$$ {\rm{trdeg}}_{\Q}\Q\big({\pi}, \lambda, \wp(c)\big) \geq 2.$$
\end{corollary}
\begin{proof} Let ${\rm{U}} = {\rm{E}}$, ${\rm{G}} = {{\rm{U}}} \times \mathbb{G}_{m, \overline{\Q}} $ and define $\pi: {\rm{G}} \longrightarrow {{\rm{U}}}$ to be the projection. Set $\Psi(r) = \left(\wp_{}\left(\frac{\lambda r}{|\lambda|}\right), e^{i\pi r/|\lambda|}\right)$. The field $F = \overline{\Q({\scriptstyle{\frac{\pi}{|\lambda|}}})}$ is closed with respect to complex conjugation and if $\wp(c) \in F$, then 
$$\Z |\lambda| + \Z i\omega|\lambda|/\pi \subset \Psi^{\scr-1}\big(\G(F)\big) + \Psi_{[i]}^{\scr-1}\big(\G(F)\big).$$
Corollary\,\ref{lp2} and Theorem\,\ref{Gel3} applied  with $\dim\,\frak{t} = 1$, ${\rm{g}}_c = {\rm{g}}_m = 1, {\rm{g}}_a = 0, {\rm{r}} \geq 2$ and $\rm{k} \geq 1$ give then the first assertion.
The second estimate follows similarly by considering $F = \overline{\Q(\pi)}$, $\Psi(r) = \left(\wp_{}\left({\lambda r}\right), e^{i\pi r/\lambda}\right)$ and $\Z + \Z i\omega/\pi$.
\end{proof}
\subsubsection{Density of algebraically independent points}
The final three results of this subsection are not a logical consequence of our theorems, but are proved with the very same techniques. They amend Diaz' observation from Subsect.\,7.1.2. 
\begin{theorem} Let $\rm{G}$ be a simple commutative group variety over a subfield $F$ of $\C$ which is countable and closed with respect to complex conjugation.
Let $\Psi: \R \longrightarrow \rm{G}(\C)$ be a real-analytic homomorphism. If $\mbox{trdeg}_F F(\xi, \xi^{\scriptscriptstyle h}) < 2\dim\,\G$ for uncountably many
$\xi \in \Psi(\R)$, then $\Psi$ descends to $\R$.
\end{theorem}

\begin{proof} Let $K = F \cap \R$ and $\mathcal{N} = \mathcal{N}_{F/K}(\G)$. Suppose the statement is wrong. Then there is an 
uncountable set $\mathfrak{S} \subset \R$ such that 
$\mbox{trdeg}_F F(\xi, \xi^{\scriptscriptstyle h}) < 2\dim\,\G$ for $\xi = \Psi(r)$ and all $r \in 
\mathfrak{S}$. There exists, for all $r \in \mathfrak{S}$, a proper subvariety ${\rm{C}}_r 
\subset \G \times \G^{\scriptscriptstyle h}$ over $\overline{F}$
with minimal dimension and the property that $\Psi_{\mathcal{N}}(r) \in {\rm{C}}_r(\C)$. However, the 
number of pair-wise distinct subvarieties among the ${\rm{C}}_r$ is at most countable. Consequently, 
there is an uncountable subset $\mathfrak{S}_o \subset \mathfrak{S}$ and 
a proper subvariety ${\rm{C}} \subset {\rm{G}} \times {\rm{G}}^{\scriptscriptstyle h}$ over $\overline{F}$ with the property that
 $\Psi_{\mathcal{N}}(\mathfrak{S}_o) \subset {\rm{C}}(\C)$. It is an easy exercise 
to check that $\mathfrak{S}_o$ admits a point of culmination $r_o$ such 
that $\xi_o = \Psi_{\mathcal{N}}(r_o) \in \rm{C}(\C)$. Let ${\rm{U}}$ be an affine 
neighborhood of $\xi_o$ together with a function $f \in F[{\rm{U}}]$ defining 
$\rm{C} \cap {\rm{U}}$. It follows that $f \circ \Psi_{\mathcal{N}} \equiv 0$. 
Hence, $\Psi_{\mathcal{N}}(\R) \subset \rm{C}(\C)$. Since $\rm{G}$ is simple, this in turn implies that ${\rm{C}}$ 
is an algebraic group $\HHH$ which is the Zariski-closure of infinitely 
many points in $\Psi_{\mathcal{N}}(\R) \cap \mathcal{N}(\R)$. Therefore
$\HHH = \HHH^{\scr h}$, so that $\HHH$ admits a model $\HHH' \subset \mathcal{N}$ over $K$ with the property that $\Psi_{\mathcal{N}}(\R) \subset \HHH'(\R)$. It follows that
$\Psi$ descends to $K$.
\end{proof}

\begin{corollary} Let $u \in \C \setminus (\R \cup i\R)$. Then, for $r \in \R$ outside a countable set, 
$e^{ru}$ and its complex conjugate are algebraically independent.
\end{corollary}
\begin{proof} The condition on $u$ guarantees that $\Psi(r) = e^{ru}$ does not descent to $\R$. This is a consequence of Corollary B.3.1. 
We apply the previous corollary and infer the result.
\end{proof}

\begin{corollary} Let $\Lambda$ be a lattice in $\C$ with algebraic invariants and let $u \in \C \setminus \overline{\Q}$ be such that $|u| = 1$. Then, for $r \in \R$ outside a countable set, 
$\wp_{}(ru)$ and its complex conjugate are algebraically independent.
\end{corollary}

\subsubsection{Improvements and analoga of a theorem of Gel'fond} 
A classical theorem of Gel'fond from 1947 asserts that for an algebraic $\alpha \neq 0, 1$ and a cubic $\beta$ the two numbers $\alpha^{\beta}, \alpha^{\beta^{\scriptscriptstyle 2}}$ are algebraically independent. In the next four corollaries
we study results concerning the exponential function which combine Gel'fond's idea of proof with our real-analytic approach. We shall then return to our main play ground, the elliptic world, and derive elliptic analoga of his theorem.
\begin{corollary} \label{oppp} Let $\beta$ be a quadratic irrationality and $\omega$ be a non-zero complex number such that $\omega/h(\omega) \notin \Q(\beta)$.
\begin{enumerate}
 \item If $\frac{\omega}{|\omega|}$ is transcendental, then 
two of the three numbers $\frac{\omega}{|\omega|}, e^{\omega}, e^{\beta \omega}$\\
are algebraically independent.
\item If $\frac{\omega}{|\omega|}$ is algebraic, then two among the numbers
$e^{\omega}, e^{\beta \omega}$ and their\\ conjugates are algebraically independent.
\item If $\frac{\omega}{|\omega|}$ and one of the numbers $e^{\omega}/|e^{\omega}|, e^{\beta \omega}/|e^{\beta \omega}|$ is algebraic, then\\ the two numbers
$e^{\omega}, e^{\beta \omega}$ are algebraically independent.

\end{enumerate}
\end{corollary}
\begin{proof} We write $\Q(\beta) = \Q + \Q\alpha$ with $\alpha \in \R \cup i\R$ and note that $\Q(\beta)$ is closed with respect to complex conjugation. To show the first statement, we shall apply Remark\,\ref{Gel31}. To this end, we 
set ${\rm{G}} = \mathbb{G}_{m,\overline{\Q}}  \times \mathbb{G}_{m,\overline{\Q}}$ and consider
$$\Psi(r) = \left(e^{\omega r/|\omega|}, e^{\beta \omega r/|\omega|}\right).$$
\begin{claim} The homomorphism $\Psi$ does not descend weakly to $\R$.
\end{claim}
\begin{proof} The hypotheses imply that $Re\,\omega, Im\,\omega, \beta Re\,\omega$ and $\beta Im\,\omega$ are $\Z$-linearly
independent. Hence, ${\Psi}_{\mathcal{N}}$ has a Zariski-dense image. Theorem \ref{weakdescent} implies then the assertion of the claim.
\end{proof}
Next we observe that the two numbers $\frac{\omega}{|\omega|}, \frac{h(\omega)}{|\omega|}$ are algebraically dependent and that they generate an algebraically closed field $F_1$ which is stable with respect to complex conjugation.
Let $F_2$ be the algebraically closed field generated by $\frac{\omega}{|\omega|}, e^{\omega}, e^{\beta \omega}$
and their complex conjugates. Then $F_1 \subset F_2$ and
$$\Z + \Z\alpha \subset \Psi^{\scr -1}\big(\G(F_2)\big) \cup \Psi_{[i]}^{\scr -1}\big(\G(F_2)\big).$$
If ${\rm{trdeg}}_{\Q} F_2 = 1$, then $F_1 = F_2$. We apply then Remark\,\ref{Gel31} with $\pi = id., \dim\,\frak{t} = 1$, ${\rm{r}} \geq 2$, ${\rm{g}}_m = 2$, ${\rm{g}}_c = {\rm{g}}_a = 0$ and $\rm{k} \geq 0$.
It is inferred that $\Psi$ descends weakly to $K = F_1 \cap \R$. But this is excluded because of the last claim. Hence, ${\rm{trdeg}}_{\Q} F_2 \geq 2$. Statement 1.\,\,follows.\\
For the remaining statements we recall Theorem\,\ref{DD}. It implies that $e^{\omega}$ and $e^{\beta \omega}$ are transcendental if $\omega/|\omega|$ is algebraic. Then the proof goes along the same lines.
\end{proof}
The corollary admits some nice applications. Since their derivation is straightforward to an large extend, the proofs are only sketched. The first application is a reminiscent of a problem stated in Waldschmidt \cite[p.\,594]{Wal 8}.
\begin{con} If $\omega \in \mathcal{L}$ is a non-zero algebraic logarithm and $\beta$ is a quadratic irrationality, then $\omega$ and $e^{\beta \omega}$ are algebraically independent.
 \end{con}
In the direction of the conjecture we can show 
\begin{corollary} Let $\omega \in \mathcal{L}$ be an algebraic logarithm and $\beta$ be a quadratic irrationality. If $\omega/|\omega|$ and $e^{\beta \omega}$
are algebraically dependent, then $\omega \in \R \cup i\R$.
 
\end{corollary}
\begin{proof} If $\omega \notin \R \cup i\R$, then Theorem\,\ref{DD} implies that $\omega/|\omega|$ is transcendental, and the claim follows from the first statement in Corollary\,\ref{oppp}.
\end{proof}

\begin{corollary} Let $\alpha > 0$ be an algebraic number and $d \geq 2$ be a square-free integer. Then $\pi/\log\,\alpha$ and
$\alpha^{\sqrt{d}}e^{\sqrt{-d}\pi}$ are algebraically independent.

\end{corollary}
\begin{proof} Set $\omega = i\pi + \log\,\alpha$ and $\beta = \sqrt{d}$. Baker's theorem implies that
$\omega/|\omega|$ is not algebraic. Using that $\pi/\log\,\alpha$ and $\omega/|\omega|$ are algebraically
dependent, one completes the proof with the help of Statement 1.\,in Corollary 7.3.9.
\end{proof}
\begin{corollary} Let $\beta \notin \Q(i)$ be a quadratic irrationality. Then $e^{\pi}$ and $e^{\beta(1+i)\pi}$ are algebraically independent.
\end{corollary}
\begin{proof} Let $\omega = i\pi + \pi$. Note that $e^{\pi + i\pi} \in \Q\big(e^{\pi}\big)$ and derive the result with the help of the third statement in Corollary\,\ref{oppp}.
\end{proof}
The next consequences are, to some effect, elliptic analoga of Gel'fond's result. They amend Masser-W\"ustholz \cite[Corollary\,1]{MW4} where a similar result for elliptic curves with complex multiplication is given. 
\begin{corollary} \label{und} Let $\beta$ be an algebraic real number with degree $4$ over $\Q$. Let $\Lambda$ be a lattice in $\C$ with algebraic invariants and
consider an algebraic logarithm $\omega \in \mathcal{L}_{\Lambda}\setminus \Lambda$ such that $\omega/|\omega| \notin \overline{\Q}$. Then the four values $Im\,\wp_{}(\beta\omega)$, $Re\,\wp_{}(\beta\omega)$, $\wp_{}(\beta^{\scriptscriptstyle 2}\omega)$ and $\wp_{}(\beta^{\scriptscriptstyle 3}\omega)$ are defined and two of them are algebraically independent. 
\end{corollary}
\begin{proof} Let $l = 1,2,3$. Denote by $F$ the algebraically closed field generated by $c_l = \wp_{}(\beta^{\scriptscriptstyle l}\omega)$ and the conjugate $h(c_1)$ over $\overline{\Q}$. We will suppose that $F$ has transcendence degree $\leq 1$ and will derive a contradiction.\\
\\
The number $\wp_{}(\beta\omega)$ is transcendental by the theorem of Schneider. Hence, since $F$ has transcendence degree $\leq 1$ over $\Q$, then the field $F$ is the algebraic closure of $\Q\big(c_1, h(c_1)\big)$ and it is stable with respect to complex conjugation. We set $\G = {\rm{E}}^{\scr 3}$, define 
$$\Psi(r) = \big(\wp_{}(\omega r), \wp_{}(\beta \omega r), \wp_{}(\beta^{\scriptscriptstyle 2}\omega r)\big)$$
and consider 
$$\Z + \Z\beta +  \Z \beta^{\scriptscriptstyle 2} + \Z \beta^{\scriptscriptstyle 3} \subset \Psi^{\scr -1}\big(\G(F)\big) + \Psi_{[i]}^{\scr -1}\big(\G(F)\big).$$
Since $\omega/|\omega| \notin \overline{\Q}$, $\omega$ and its complex conjugate are not
linearly dependent over $\overline{\Q}$. 
\begin{claim} {The homomorphism $\Psi$ does not descend weakly to $K$.}
\end{claim}
\begin{proof} Because of Theorem \ref{weakdescent} it suffices to show that the homomorphism $\Psi_{\mathcal{N}}$ has Zariski-dense image. Assume first that ${{\rm{E}}}_{}$ has no complex multiplication. With notations as in Sect.\,3.3 we have ${\rm{Lie}}(\rho)(\omega) = h(\omega)$. If $\Psi_{\mathcal{N}}$ has no Zariski-dense image, then we receive an equation 
$$0 = (a_0 + a_1\beta + a_2 \beta^{\scriptscriptstyle 2}) \omega + (b_0 + b_1\beta + b_2 \beta^{\scriptscriptstyle 2})\alpha h(\omega)$$
with rational $a_j, b_j$ and where $\alpha \in \overline{\Q}$ represents an isogeny in ${\rm{End}}({\rm{E}}^{\scriptscriptstyle h}, {\rm{E}}) \backsimeq \Z$. Since $\omega$ and its complex conjugate are linearly independent over $\overline{\Q}$, we get a contradiction. Finally, if ${\rm{E}}$ has complex multiplication by, say $\sqrt{-d}$, then $\Q(\beta, \sqrt{-d})$ has degree $8$ over $\Q$ for $\beta$ is real. Using this the proof follows similarly.
\end{proof}
Recall that $F$ is stable with respect to complex conjugation and set $K = F \cap \R$. The hypotheses imply: There is a subspace $\frak{t}$ of $\frak{e}^{\scr 2}$ over $K$ of dimension $\dim\,\frak{t} = 2$
such that the image of $\Psi_{\ast}$ is contained in $\frak{t} \otimes_K \R$. From Theorem\,\ref{Gel3} applied with $\pi = id., {\rm{r}} \geq 3$, ${\rm{g}}_c = 3$, ${\rm{g}}_m = 0$, ${\rm{g}}_a = 0$ and $\rm{k} \geq 0$ we infer that ${\Psi}$ descends weakly to $K$. But the latter contradicts the previous claim.\end{proof} 
\begin{corollary} Let $\beta$ be a cubic real irrationality. Let $\Lambda$ be a lattice in $\C$ with algebraic invariants and
consider an algebraic logarithm $\omega \in \mathcal{L}_{}$ such that $\omega/|\omega| \notin \overline{\Q}$. Then two of the three numbers $Re\,\omega, \wp_{}(\beta\omega)$ and $\wp_{}(\beta^{\scriptscriptstyle 2}\omega)$ are algebraically independent. 
\end{corollary}
\begin{proof} Since $\omega/|\omega|$ is transcendental, Corollary \ref{realalg} implies that so is $Re\,\omega$. Set ${\rm{G}} = {{\rm{E}}}^{\scriptscriptstyle 3} \times \mathbb{G}_{a, \overline{\Q}}$, consider the projection $\pi$ to ${\rm{E}}^{\scr 3}$ and define 
$$\Psi(r) = \big(\wp_{}(\omega r), \wp_{}(\beta \omega r), \wp_{}(\beta^{\scriptscriptstyle 2}\omega r), (Re\,\omega)r \big).$$
Using arguments as in the previous proof, one derives the assertion from Theorem \ref{Gel3}.
\end{proof}
\begin{corollary} Let $\Lambda$ be a lattice in $\C$ with algebraic invariants and complex multiplication by $\tau \notin \Q(i)$. Consider a non-zero algebraic logarithm $\omega \in \mathcal{L}_{\Lambda}$ such that $\omega/|\omega| \notin \overline{\Q}$. Then $\omega/|\omega|$ and $\wp(i\omega)$ are defined and algebraically independent. 
\end{corollary}
\begin{proof} We will sketch the proof only in the case when $\tau = \sqrt{-d}$ for a square-free integer $d > 1$. The general case is readily reduced to this. We consider the group $|\omega|\big(\Z + i \Z + \tau\Z + \sqrt{d} \Z\big)$ of rank four and the homomorphism
$$\Psi(r) = \left(\wp_{}\big({\scriptstyle \frac{\omega}{|\omega|}}r\big), \wp_{}\big(i{\scriptstyle \frac{\omega}{|\omega|}}r\big)\right)$$
to the set of complex points of $\G = \rm{E}^{\scr 2}$. The proof runs then as follows: First one verifies that $\Psi$ has Zariski-dense image in $\G(\C)$. Next one 
defines $F$ to be the algebraically closed field generated by $\frac{\omega}{|\omega|}$, $\wp_{}(\sqrt{d}\,\omega)$ and $\wp_{}\big(i\omega\big)$. Since $\frac{\omega}{|\omega|}$ and its conjugate are algebraically dependent,
it follows that if $F$ has transcendence degree $\leq 1$, then it is stable with respect to complex conjugation. In the third step one applies Theorem\,\ref{Gel3} with $\pi = id., \dim\,\frak{t}= 1$, ${\rm{g}}_c = 2$, ${\rm{g}}_m ={\rm{g}}_a = 0$, ${\rm{r}} \geq 4$ and ${\rm{k}} \geq 0$. One derives that $F$
has transcendence degree $\geq 2$. As $\wp_{}(\tau\sqrt{d}\,\omega) = \wp_{}\big(id\omega\big)$ and $\wp_{}(\tau i\omega) = \wp_{}\big(-\sqrt{d}\omega\big)$,
it results then in the last step that $\frac{\omega}{|\omega|}$ and $\wp_{}\big(i\omega\big)$ are algebraically independent. 
\end{proof}

\appendix

\small{
\chapter{Complex group varieties and real fields of definition}
In this appendix we state and prove some mostly known results which link the analytic theory of algebraic groups to the problem of definability over real subfields. 
\section{Real fields of definition} In what follows we denote by ${\rm{G}}$ a complex commutative group variety. 
\begin{prop} \label{op2}\label{2.5.3} Let $\frak{t}^+ \in \cal{H}(\frak{g})$ be such that $\Lambda$ is stable with respect to the canonical action of $Gal(\C|\R)$ on $\frak{g} = \frak{t}^+ \otimes_{\R}\C$. Then there is an algebraic group ${\rm{G}}'$ over $\R$ and an isomorphism $v: {\rm{G}} \longrightarrow {\rm{G}}' \otimes_{\R}\C$ of complex algebraic groups with the property that ${\rm{exp}}_{\G}(\frak{t}^+) \subset v^{\scriptscriptstyle -1}\big({\rm{G}}'(\R)\big)$.
\end{prop}
The notion of a ``canonical action'' and the symbol ``$\cal{H}(\frak{g})$'' were defined in Sect.\,3.3.
\begin{proof} The statement is readily reduced to
\begin{center}
\begin{tabular} {cp{10cm}}
 $(\ast)$ & Let $\Phi: \C^{g} \longrightarrow {\rm{G}}(\C)$ be an \'etale and holomorphic homomorphism, and denote by $\Lambda$ its kernel. Suppose that $\Lambda$ is stable with respect to complex conjugation. Then there is an algebraic group ${\rm{G}}'$ over $\R$ and an isomorphism $v: {\rm{G}} \longrightarrow {\rm{G}}' \otimes_{\R}\C$ of complex algebraic groups with the property that $\Phi(\R^g) \subset v^{\scriptscriptstyle -1}\big({\rm{G}}'(\R)\big)$.
\end{tabular}
\end{center}
We set $\mathcal{N} = \mathcal{N}_{\C/\R}({\rm{G}})$ and write $\Psi = \Phi_{|\R^{\scriptscriptstyle g}}$. Because of Proposition \label{complex} the real-analytic homomorphisms $\Psi^{\scr h} = \rho_{\ast} \circ \Psi$ and $\Psi_{\mathcal{N}} = \Psi \times \Psi^{\scr h}$ can be extended to holomorphic maps 
$\Phi^{\scr h}: \C^g \longrightarrow \G^{\scr}(\C)$ and $\Phi_{\mathcal{N}}: \C^g \longrightarrow \mathcal{N}(\C).$ As above we will suppose that ${\rm{G}}$ is embedded into $\mathbb{P}^N_F$ as a quasi-projective subvariety. Then $\Phi$ is locally represented by power series $\sum a_{Ij} \underline{z}^I$, where $j = 0,..., N$, and $\Phi^{\scr h}$ is locally represented by the conjugate power series $\sum h(a_{Ij}) \underline{z}^I$. Since $\Lambda$ is stable with respect to complex conjugation, it follows
that $\Lambda = ker\,\Phi = ker\,\Phi^{\scr h} = ker\,\Phi_{\mathcal{N}}$. Let ${\rm{C}} = \Psi_{\mathcal{N}}(\R)$. Taking quotients, the map $\Phi_{\mathcal{N}}$ induces a commutative diagram of holomorphic maps 
$$\begin{xy} 
  \xymatrix{ & & {\rm{G}}(\C)\\
{\rm{G}}(\C) \ar[urr]^{id.} \ar[drr] \ar[rr]^{\,\,\,\,\,\,\,\,\,\,\,\,\,\overline{\Phi}_{\mathcal{N}}} & &\ar[u]^{p_{{\rm{G}}}}\mathcal{N}(\C) \ar[d]^{p_{{\rm{G}}^{\scriptscriptstyle h}}}\\
& & {\rm{G}}^{\scriptscriptstyle h}(\C) }
\end{xy}$$
By construction we have $\big(p_{\rm{G}}(\xi)\big)^{\scriptscriptstyle h} = p_{{\rm{G}}^{\scriptscriptstyle h}}(\xi)$ for all $\xi \in \overline{\Phi}_{\mathcal{N}}(\rm{C})$.
Corollary \ref{3.3.1} implies that $\overline{\Phi}_{\mathcal{N}}(\rm{C}) \subset \mathcal{N}(\R)$. Moreover, the image $\overline{\Phi}_{\mathcal{N}}(\rm{C})$ is a Zariski-dense subset of the algebraic subgroup ${\rm{H}} = im\,\overline{\Phi}_{\mathcal{N}}(\rm{C})$ of $\mathcal{N} \otimes_{\R} \C$. 
Since $\overline{\Phi}_{\mathcal{N}}({\rm{C}}) = \big(\overline{\Phi}_{\mathcal{N}}({\rm{C}})\big)^{\scr h}$, we have ${\rm{H}} = {\rm{H}}^{\scr h}$. Consequently,
there is a group variety ${\rm{H}}' \subset \mathcal{N}$ such that $\rm{H} = {\rm{H}}' \otimes_{\R} \C$. Letting $v = \big(\overline{\Phi}_{\mathcal{N}}\big)^{\scr -1}$ and $\G' = {\rm{H}}'$, $(\ast)$ follows.
\end{proof}
Proposition\,\ref{op2} is specified in the context of elliptic curves by the following result. For a $\tau$ in the upper half plane $\mathbb{H}$ we write $\Lambda_{\tau} = \Z + \tau \Z$ and let ${{\rm{E}}} = {{\rm{E}}}_{\tau}$ be the associated 
complex elliptic curve in Weierstra\ss\,\,form. Then ${\rm{Lie}}\,{\rm{E}}$ is identified with $\C$ and 
the exponential map ${\rm{exp}}_{{{\rm{E}}}}$ with $[\wp_{\tau}: \wp_{\tau}': 1]$.
\begin{corollary} \label{op22} \label{lp2} The following assertions are equivalent.
\begin{enumerate}
\item $Re\,\tau \in \Q.$
 \item There is an elliptic curve ${\rm{E}}'$ over $\R$ and an isogeny\\ $v: {{\rm{E}}} \longrightarrow {\rm{E}}' 
\otimes_{\R} \C $ with the property that $v_{\ast}(\R) = \rm{Lie}\,{\rm{E}}'$. 
\end{enumerate}
\end{corollary}
\begin{proof} \textit{'$\mathbf{1. \Longrightarrow 2.}$'} Let $k > 0$ be such that $k(Re\,\tau) \in \Z$. Then $\Lambda = \Z + \Z ki(Im\,\tau)$ is a sublattice of $\Lambda_{\tau}$ which is stable with respect to complex conjugation and such that the fix locus
of the complex conjugation is $\Lambda \cap \R = \Z$. We apply Proposition\,\ref{op2} to $\Lambda$ and infer the existence of an isogeny
$w: {\rm{E}}' \otimes_{\R} \C \longrightarrow {\rm{E}}$ with the property that $w_{\ast}(\rm{Lie}\,{\rm{E}}') = \R$. Any inverse isogeny $v: {{\rm{E}}} \longrightarrow {\rm{E}}' 
\otimes_{\R} \C $ with the property that $v \circ w$ is multiplication with an integer is then as required for Statement 2.\\
\textit{'$\mathbf{2. \Longrightarrow 1.}$'} Write $\frak{e}' = \rm{Lie}\,\rm{E}'$ and let $\Lambda' \subset \frak{e}'(\C)$ be the kernel of the exponential map ${\rm{exp}}_{{{\rm{E}}}'}$. Corollary \ref{lp333} implies that $Gal(\C|\R)$ admits a canonical action on $\frak{e}' \otimes \C$ which 
restricts to an action of $\Lambda'$ and which fixes ${\frak{e}}'$. 
For $c \in \C$ consider the action $\mu_h(c) = v_{\ast}^{\scr -1}\left(\big(v_{\ast}(c)\big)^{\scr h}\right)$ of $Gal(\C|\R)$ induced by $v_{\ast}$. As $\R = v_{\ast}^{-\scr 1}(\frak{e}')$, it follows that $\mu_h$ is complex conjugation. As a result, the lattice $\Lambda^{\ast} = v_{\ast}^{\scr -1}\big(\Lambda')$ in $\C$ is stable with respect to complex conjugation. Thus, $\Lambda = (\deg\,v)\Lambda^{\ast}$ is a sublattice of $\Lambda_{\tau}$ and stable with respect to complex conjugation.
This in turn implies the existence of integers $d_1, d_2,d_3 $ such that $d_1 \neq 0$ and $d_1\cdot h(\tau) = d_2 + d_3\tau$. Looking at imaginary parts, we find that
$d_1 = -d_3$. Consequently, $2d_1 (Re\,\tau) = d_2$. Statement 1.\,\,follows.

\end{proof}
Let now $F$ be an algebraically closed subfield of $\C$ which is stable with respect to complex conjugation $h$.
\begin{prop} \label{op1} Let ${\rm{G}}$ be a semi-abelian commutative group variety over $F$ and let ${\rm{G}}_{r}$ be a real model of the complexification of $\G$. Then $\G_r$ is definable
over $K = F \cap \R$. 
\end{prop} 
\begin{proof} Let $v: \G_r \otimes_{\R} \C \longrightarrow \G \otimes_F \C$ be an isomorphism. We have to prove
\begin{center}
\begin{tabular} {cp{10cm}}
$(\ast)$ & There is an algebraic group ${\rm{G}}'$ over $K$
and an isomorphism of algebraic groups $w: {\rm{G}} \longrightarrow {\rm{G}}' \otimes_{K}F $ over $F$ such that
$(w  \otimes_{F} \C) \circ v$ is defined over $\R$.
\end{tabular}
\end{center}
Set $\A = {\rm{G}}_{r} \otimes_{\R} \C$ and ${\rm{B}}  = {\rm{G}} \otimes_{F} \C$. We consider the embedding ${\mathcal{N}}(v): {\rm{G}}_{r} \longrightarrow \mathcal{N}_{\C/\R}(\rm{B})$ over $\R$. We define $\mathcal{N} = \mathcal{N}_{F/K}({\rm{G}})$. By Lemma 3.5.5 there is a canonical isomorphism $\mathcal{N}_{\C/\R}({\rm{B}}) \backsimeq \mathcal{N} \otimes_{{K}}\R$ and we shall identify these varieties. Consider the algebraic group variety ${\rm{H}} = \big(im\,{\mathcal{N}}(v)\big) \otimes_K \C \subset \mathcal{N} \otimes_K \C$. Since $\mathcal{N} \otimes_{K} \C = {\rm{A}} \times  {{\rm{A}}}^{\scriptscriptstyle h}$ is semi-abelian, ${{{{\rm{H}}} }}_{tor}$ is Zariski-dense in ${{{\rm{H}}} }(\C)$. And ${{{{\rm{H}}} }}_{tor}$ is contained in $\mathcal{N}(F)$. We infer that ${{{\rm{H}}} }$ is definable over $F$,
 that is, ${{\rm{H}}} = {{{\rm{H}}} }_f \otimes_F \C$ with a subvariety ${{{\rm{H}}} }_f \subset \mathcal{N}\otimes_{{K}} F.$ Moreover, ${{{\rm{H}}} }$ is stable with respect to the action of $h$. It follows from Subsect.\,3.1.5 that ${{{\rm{H}}} }_f$ is stable
with respect to the action of $h$ on $\mathcal{N}\otimes_{{K}} F.$ Hence, there is a group variety ${{\rm{H}}}' \subset \mathcal{N}$ such that ${{{\rm{H}}} }_f = {{\rm{H}}}' \otimes_K F$. Let $p_{{\rm{G}}_f}: \mathcal{N} \otimes_{{K}} F \longrightarrow {\rm{G}}_f$ be the projection from Theorem\,\ref{Wr1}. Setting ${\rm{G}}' = {{\rm{H}}}'$ and $w = \big(p_{{\rm{G}}_f|{{\rm{H}}}}\big)^{\scr -1}$, $(\ast)$ follows.
 \end{proof}

\section{Complex twins}
Next we let $F$ be an algebraically closed subfield of $\C$ and write $K = F \cap \R$. We assume that the extension $F/K$ has degree $[F: K] = 2$ and that $F$ is stable with respect to complex conjugation $h$. In this section we associate to a commutative group variety $\G'$ over $K$ its \textit{complex twin} $\G'_{[i]}$, a group variety over $K$ 
which becomes isogenous to $\G' \otimes_K F$ after base change to $F$.\\
We define $\Lambda \subset \frak{g}(\C)$ to be the kernel of the exponential map ${\rm{exp}}_{\G}$ and let ${\rm{Lie}(\rho)}$ be as in Sect.\,3.3. Then ${\rm{G}} = {{\rm{G}}}^{\scriptscriptstyle h}$ and, as seen in Corollary \ref{lp333}, $\Lambda = \Lambda^{\scriptscriptstyle h} = {\rm{Lie}(\rho)}\big(\Lambda\big)$. We note that the set ${\rm{G}}(i\R) = {\rm{exp}}_{\rm{G}}\big(i \cdot \frak{g}'(\R) \big)$ is the fixed locus of the real-analytic map $\rho_{\ast} \circ [-1]_{\rm{G}}: \G(\C) \longrightarrow \G(\C)$. Here $\rho_{\ast}$ is as defined in Sect.\,2.2.
Let $\Lambda_1 = \Lambda \cap \frak{g}'(\R)$ and $\Lambda_2 = \Lambda \cap i\cdot \frak{g}'(\R)$. Then $\Lambda_1 + \Lambda_2$ has finite index in $\Lambda$.\\
\\
\textbf{First construction.} Set $\Lambda_{[i]} = i\Lambda_1 + i\Lambda_2$. The subgroup $\Lambda_{[i]}$ of $\frak{g}(\C)$ is stable with respect to the canonical action of $Gal(\C|\R)$ induced by $\frak{g}'(\R) \in \cal{H}\big(\frak{g}(\C)\big)$ (see Sect.\,3.6). The analytic quotient $\frak{g}(\C)/\Lambda_{[i]}$ is isogenous to $\G(\C)$ (an isogeny is induced by multiplication with $i$) and hence inherits
the structure of a group variety over $F$. We denote the variety in question by ${{{\rm{G}}}}_{[i]}$. As follows from Proposition \ref{op2} and Proposition \ref{op1} the action of $Gal(\C|\R)$ on $\big(\frak{g}(\C),\Lambda_{[i]}\big)$ descends to an action on ${{{\rm{G}}}}_{[i]}$ and defines a $K$-structure on ${{{\rm{G}}}}_{[i]}$. \\
\\
\textbf{Second construction.} Consider $\frak{g}(\C)$ with the canonical action of $Gal(\C|\R)$
with respect to $i\cdot \frak{g}'(\R) \in \cal{H}\big(\frak{g}(\C)\big)$. Since $\Lambda_1 + \Lambda_2$ is stable with respect to this action, Proposition \ref{op2} and Proposition \ref{op1} imply that the group variety ${{{\rm{G}}}}_{[i]}$ defined by the analytic quotient $\frak{g}(\C)/(\Lambda_1 + \Lambda_2)$ carries a $K$-structure.\\
\\
Both constructions lead to canonically isomorphic group varieties over $K$. We identify these two varieties and call the resulting single variety the \textit{complex twin} of ${\rm{G}}'$. It is denoted by ${\rm{G}}_{[i]}'$. From the first construction it is clear that multiplication with $i$ leads to a canonical isogeny $[i]: {{\rm{G}}}_{[i]} \longrightarrow {{{\rm{G}}}}$ over $F$ which, however, is never defined over $K$ if $\dim\,\rm{G} > 0$.
\begin{example} \label{realscheiss} Recall the algebraic group $\mathbb{S}$ from (2.3.2). Then $\mathbb{S}_{\R} = \big(\mathbb{G}_{m, \C}\big)_{[i]}$ and
$\mathcal{N}_{\C/\R}(\mathbb{G}_{m, \C})$ is isogenous to $\mathbb{G}_{m, \R} \times \mathbb{S}_{\R}$.
 
\end{example}
\begin{example} Let ${\rm{E}}'$ be an elliptic curve over $\R$ and write
$\rm{E} = {\rm{E}}' \otimes_{\R} \C$. Then ${\rm{End}}({\rm{E}}')$ is identified with the subgroup of isogenies
 $v \in {\rm{End}}({\rm{E}})$ such that $v_{\ast}(\frak{g}') = \frak{g}'$. Hence, ${\rm{End}}\big({\rm{E}}') \backsimeq \Z$.
The curve ${\rm{E}}'$ and its twin
${\rm{E}}'_{[i]}$ are isogenous over $\R$ if and only if there is a $v \in {\rm{End}}\big({\rm{E}})$ such that
$v_{\ast}(\frak{g}') = i\cdot \frak{g}'$. The latter holds if and only if $\rm{E}$ is a CM-curve. It follows that if $\rm{E}$ has complex multiplication, then $\mathcal{N} =\mathcal{N}_{\C/\R}(\rm{E})$ is isogenous to the power ${\rm{E}}' \times {\rm{E}}'$ and ${\rm{End}}(\mathcal{N}) \backsimeq \mathbf{Mat}_2(\Z)$.
And if $\rm{E}$ admits no complex multiplication, then $\mathcal{N} =\mathcal{N}_{\C/\R}(\rm{E})$ is isogenous to the product of the two non-isogenous twins ${\rm{E}}' \times {\rm{E}}'_{[i]}$ and
${\rm{End}}\big(\mathcal{N}\big) \backsimeq \Z \times \Z$.
\end{example}
\section{Applications to descents} 
We consider a commutative group variety ${\rm{G}}$ over $\C$ together with a non-constant real-analytic homomorphism $\Psi: \R \longrightarrow {\rm{G}}(\C)$. In the next proposition we reformulate the definition of decent to $\R$ in terms of the Lie algebra. 
\begin{corollary} \label{lp} The homomorphism $\Psi$ descends to $\R$ if and only if there exists a subspace $\frak{t}^+ \in {\cal{H}}\big(\frak{g}\big)$ such that $\Psi_{\ast}(\R) \subset \frak{t}^{+}$ and such that the canonical action of $Gal(\C|\R)$ on $\frak{g}$
with respect to $\frak{t}^+$ satisfies ${h}(\Lambda) \subset k^{\scriptscriptstyle -1}\Lambda$ with an integer $k > 0$. 
\end{corollary}
\begin{proof} This is a straightforward application of Proposition \ref{op2}.\end{proof}
Since arithmetic problems are our objective, we are interested in algebraic groups over a fixed algebraically closed subfield ${F}$ of $\C$. 
\begin{corollary} \label{lp3} If $F$ is algebraically closed and stable with respect to complex conjugation and if ${\rm{G}}$ is semi-abelian, then $\Psi$ descends to $K$ if and only if it descends to $\R$.
\end{corollary}
\begin{proof} The assertion results directly from Corollary \ref{lp333} and Proposition \ref{op1}.\end{proof}
We keep assuming that $F$ is an algebraically closed subfield of $\C$ and define ${\rm{r}} = {\rm{rank}}_{\Z}\,\Psi^{\scriptscriptstyle -1}\big({{\rm{G}}}(F)\big) + {\rm{rank}}_{\Z}\,\Psi_{[i]}^{\scriptscriptstyle -1}\big({{\rm{G}}}(F)\big)$. Recall the definition of $\Psi_{[i]}$ from Sect.\,2.2.
\begin{corollary}\label{i} If $F$ is algebraically closed and if ${\rm{r}} > 0$, then $\Psi$ descends to $K$ if and only if $\Psi_{[i]}$ does.
 
\end{corollary}
\begin{proof} For $F = \C$ the claim is a direct consequence of the previous section. Hence, $\Psi$ descends to $\R$ if and only if $\Psi_{[i]}$ does.
We will show that under the additional assumption that ${\rm{r}} > 0$ the lemma follows from this.\\
\\
To this end, we let $\phi$ be the complexification of $\Psi_{\mathcal{N}}$. Since ${\rm{r}} > 0$, there are two
possibilities:
\begin{enumerate}
\item We have ${\rm{rank}}_{\Z}\,\Psi_{[i]}^{\scr -1}\big(\G(F)\big) > 0$. Then ${\rm{rank}}_{\Z}\,\phi^{\scr -1}\big(\mathcal{N}(F)\big) > 0$ by Statement 2.\,in Lemma \ref{ff}. \item We have ${\rm{rank}}_{\Z}\,\Psi^{\scr -1}\big(\G(F)\big) > 0$. Since $\Psi = (\Psi_{[i]})_{[i]}$, Statement 4.\,in Proposition \ref{ff}
implies that ${\rm{rank}}_{\Z}\,\phi^{\scr -1}\big(\mathcal{N}(F)\big) > 0$.

\end{enumerate}
Hence, ${\rm{rank}}_{\Z}\,\phi^{\scr -1}\big(\mathcal{N}(F)\big) > 0$. By Lemma 3.3.4 we may identify $\mathcal{N} \otimes_K \R$ with $\mathcal{N}_{\C/\R}(\G \otimes_F \C)$. As ${\rm{rank}}_{\Z}\,\phi^{\scr -1}\big(\mathcal{N}(F)\big) > 0$, it follows that the Zariski-closure ${\rm{V}}$ of $im\,\phi$ in $\mathcal{N} \otimes_K \C$ is definable over $F$.
To be more precise, there is a group variety ${\rm{V}}_f \in \mathcal{N} \otimes_K F$ with the property that ${{\rm{V}}} = {\rm{V}}_f \otimes_F \C$.
Moreover, since $\psi_{\mathcal{N}}(\R) \in \mathcal{N}(\R)$, we have $\psi_{\mathcal{N}}(\R) = \big(\psi_{\mathcal{N}}\big)^{\scr h}(\R)$. So, ${{\rm{V}}}$ equals ${{\rm{V}}}^{\scr h}$. By Subsection 3.2.2 this implies that ${{\rm{V}}}_f = {{\rm{V}}}_f^{\scr h}$. Hence, ${{\rm{V}}}_f$ is definable over $K$, that is, ${{\rm{V}}}_f = {{\rm{V}}}_r \otimes_{K} F$ for a group variety ${{\rm{V}}}_r \subset \mathcal{N}$. Once we know that $\rm{V}$ is definable over $K$, the proof of the lemma can be completed as follows:\\
If $\Psi$ descends to $K$, then, as noticed in the beginning of the proof, $\Psi_{[i]}$ descends to $\R$. As seen in the previous chapters, the Zariski-closure ${{\rm{V}}}$ of $(\Psi_{[i]})_{\mathcal{N}}$ in $\mathcal{N} \otimes_{\R} \C$ is then isogenous to $\G \otimes_{\R} \C$. An isogeny is given by the restriction to ${\rm{V}}$ of the projection $p_{(\G \otimes_F \C)}$. 
In view of our identification $\mathcal{N} \otimes_K \R = \mathcal{N}_{\C/\R}(\G \otimes_F \C)$, we have $p_{(\G \otimes_F \C)} = p_{\G} \otimes_F \C$. 
Moreover, as seen above, $\V = {{\rm{V}}}_r \otimes_K \C$ for some $\V_r \subset \mathcal{N}$. It results that the restriction to ${{\rm{V}}}_r \otimes_K F$ of the projection $p_{\G}$ is an isogeny onto $\G$. Theorem \ref{weakdescent} implies that $\Psi_{[i]}$ descends to $K$. Since $\Psi(r) = (\Psi_{[i]})_{[i]}(-r)$, reversing the roles
of $\Psi$ and $\Psi_{[i]}$, we deduce that if $\Psi_{[i]}$ descends to $K$, then so does $\Psi$. The corollary follows.
\end{proof}

\chapter{Linear extensions of abelian varieties} 
\section{Some general facts} A reference for this section is the seventh chapter of Serre' book \cite{Serre1}. We consider an abelian variety ${{\rm{A}}}$ over an algebraically closed subfield $F$ and let $\LL$ be a commutative linear algebraic group variety over $F$. An \textit{extension of ${\A}$ by ${\LL}$} is a short exact sequence of algebraic groups
$$0 \longrightarrow {\rm{L}} \stackrel{j}{\longrightarrow} {\rm{G}} \stackrel{\pi}{\longrightarrow } {{\rm{A}}} \longrightarrow 0.$$
For brevity we will eventually consider triples $({\rm{G}}, j, \pi)$ to denote an extension, or even identify an extension with its \textit{fiber space} ${\rm{G}}$. 
Two extensions $({\rm{G}}_1, j_1, \pi_1)$ and $({\rm{G}}_2, j_2, \pi_2)$ are \textit{isomorphic} if there exists an algebraic homomorphism $v: {\rm{G}}_1 \longrightarrow {\rm{G}}_2$ such that the following diagram commutes
$$\begin{xy} 
  \xymatrix{ 
0 \ar[r] & {\rm{L}} \ar[r]^{j_1} \ar[d]^{id.} & {\rm{G}}_1 \ar[r]^{\pi_1} \ar[d]^{v} & \rm{A} \ar[r] \ar[d]^{id.} & 0\\
0 \ar[r] & {\rm{L}} \ar[r]^{j_2}            & {\rm{G}}_2 \ar[r]^{\pi_2}              & \rm{A} \ar[r] & 0 
}
\end{xy}$$
The isomorphism classes $[({\rm{G}}, j, \pi)]$ of linear extensions over $F$ of ${{\rm{A}}}$ by ${\rm{L}}$ form a group $\rm{Ext}(\rm{A}, {\rm{L}})$. The group structure
can be realized in the following way. Recall that a rational map $f: {\rm{A}} \times  \rm{A} \dashrightarrow {\rm{L}}$ is a \textit{rational symmetric factor system} if 
$f(\xi_1, \xi_2) = f(\xi_2, \xi_1)$ and if $f(\xi_2, \xi_3) - f(\xi_1 + \xi_2, \xi_3) + f(\xi_1, \xi_2 + \xi_3) - f(\xi_2, \xi_2) = e_L$ for all $\xi_1, \xi_2, \xi_3$ on a Zariski-dense subset of ${\rm{A}}(F)$. A factor system $f(\xi_1, \xi_2)$ is said to be trivial
if $f(\xi_1, \xi_2) = g(\xi_1+\xi_2) - g(\xi_1) - g(\xi_2)$ for a rational map $g: A \dashrightarrow L$. The group of classes of rational symmetric factor systems is denoted by ${{{\rm{H}}} }^{\scriptscriptstyle 2}_{rat}(\rm{A}, {\rm{L}})_s$. There is a natural map $\chi: \rm{Ext}(\A, {\rm{L}}) \longrightarrow {{{\rm{H}}} }^{\scriptscriptstyle 2}_{rat}(\A, {\rm{L}})_s$ defined in the following way.
For an extension ${\rm{G}}$ choose a rational section $\sigma: \rm{A} \dashrightarrow {\rm{G}}$ and let $f_{\sigma}(\xi_1, \xi_2) = \sigma(\xi_1, \xi_2) - \sigma(\xi_1) - \sigma(\xi_2)$.
Then $f_{\sigma}$ is a rational factor system and $\chi[{\rm{G}}]\,\,=\,\,\mbox{class of}\,\,f_{\sigma}$.
\begin{lemma} The map $\chi$ is well-defined and an isomorphism of groups.

\end{lemma}
\begin{proof} See Serre \cite[Ch.\,VII, 4]{Serre1} 
\end{proof}
The symbol ${\rm{Ext}}(\cdot, \rm{A})$ is a covariant functor from the category of linear
groups into groups, whereas ${\rm{Ext}}({\rm{L}}, \cdot)$ is a contravariant functor from the category of abelian varieties into groups.\footnote{The two functors, of course, extend to arbitrary algebraic groups, but we restrict ourselves to the subcategory of linear commutative groups resp.\,\,abelian varieties.}
This is made precise in what follows.\\
\begin{itemize}
 \item If $v: {{\rm{A}}} \longrightarrow {\rm{B}}$ is a homomorphism of abelian varieties over $F$, then a homomorphism of groups
${\rm{Ext}}^{v}: {\rm{Ext}}({\rm{B}}, {\rm{L}}) \longrightarrow {\rm{Ext}}(\A, {\rm{L}})$ is defined by ${\rm{Ext}}^{v}[{\rm{G}}] = [{\rm{G}}\times_{\rm{B}} \rm{A}]$.
\item If $u: {\rm{L}}_1 \longrightarrow {\rm{L}}_2$ is a homomorphism of linear groups, then, for each extension ${\rm{G}}_1$ of ${{\rm{A}}}$ by ${\rm{L}}_1$
there, exists a commutative diagram
$$\begin{xy} 
  \xymatrix{ 
0 \ar[r] & {\rm{L}}_1 \ar[r]^{j_1} \ar[d]^{u} & {\rm{G}}_1 \ar[r]^{\pi_1} \ar[d]^{v} & \rm{A} \ar[r] \ar[d]^{id.} & 0\\
0 \ar[r] & {\rm{L}}_2 \ar[r]^{j_2}            & {\rm{G}}_2 \ar[r]^{\pi_2}              & \rm{A} \ar[r] & 0 
}
\end{xy}$$
We set ${\rm{Ext}}_u[{\rm{G}}_1] = [{\rm{G}}_2].$ This way a homomorphism 
${\rm{Ext}}_u: {\rm{Ext}}({\rm{A}}, {\rm{L}}_1) \longrightarrow {\rm{Ext}}(\A, {\rm{L}}_2)$ is defined.
\item If 
$$\begin{xy} 
  \xymatrix{ 
0 \ar[r] & {\rm{L}}_1 \ar[r]^{j_1} \ar[d] & {\rm{G}}_1 \ar[r]^{\pi_1} \ar[d]^{w} & \rm{A} \ar[r] \ar[d]^{v} & 0\\
0 \ar[r] & {\rm{L}}_2 \ar[r]^{j_2}            & {\rm{G}}_2 \ar[r]^{\pi_2}              & \rm{A} \ar[r] & 0 
}
\end{xy}$$
is a commutative diagram of extensions, then $w$ factors through ${\rm{Ext}}^v[{\rm{G}}_2]$ and we obtain a diagram
$$\begin{xy} 
  \xymatrix{ 
0 \ar[r] & {\rm{L}}_1 \ar[r]^{j_1} \ar[d]^{u} & {\rm{G}}_1 \ar[r]^{\pi_1} \ar[d] & \rm{A} \ar[r] \ar[d]^{id.} & 0\\
0 \ar[r] & {\rm{L}}_2 \ar[r]            &{\rm{Ext}}^v[{\rm{G}}_2] \ar[r]              & \A \ar[r] & 0 
}
\end{xy}$$
In other words, ${\rm{Ext}}_u[{\rm{G}}_1] = {\rm{Ext}}^v[{\rm{G}}_2]$.
\end{itemize}
\begin{lemma} Let ${{\rm{A}}}$ and ${\rm{B}}$ be abelian varieties over $F$ and $v: \rm{A}\longrightarrow {\rm{B}}$ be an isogeny.
\begin{enumerate}
 \item If ${{\rm{L}}} = \mathbb{G}_{a,F}$, then ${\rm{Ext}}(\rm{A}, {\rm{L}})$ is canonically isomorphic to ${{{\rm{H}}} }^1(\A, \mathcal{O}_{\rm{A}})$ and endowed with the structure of a vector space over $F$.
 The homomorphism
${\rm{Ext}}^v$ is identified with the induced map in cohomology ${{{\rm{h}}} }^{\scr1 }(v): {{{\rm{H}}} }^1(\rm{B}, \mathcal{O}_{\rm{B}}) \longrightarrow {{{\rm{H}}} }^1(\rm{A}, \mathcal{O}_{\rm{A}}).$ 
\item If ${{\rm{L}}} = \mathbb{G}_{m,F}$, then ${\rm{Ext}}(\rm{A}, {\rm{L}})$ is canonically isomorphic to $\rm{Pic}^o(\rm{A})$, the group underlying the dual abelian variety ${{\rm{A}}}^{\ast}$. The homomorphism
${\rm{Ext}}^v$ is identified with the induced dual isogeny $v^{\ast}: \rm{B}^{\ast} \longrightarrow {{\rm{A}}}^{\ast}$.
\end{enumerate}

\end{lemma}
\begin{proof} See Serre \cite[p.\,163]{Serre1} and \cite[Ch.\,VII, 14-17]{Serre1}.
 
\end{proof}

\begin{lemma} \label{91} Let
$$\begin{xy}                                                    
  \xymatrix{                                                                                                                                                                                                                                                                                                                                                                                                                                                           
0 \ar[r] & {\rm{L}} \ar[r]^{j_1} \ar[d]^{w_{{\rm{L}}}} & {\rm{G}}_1 \ar[r]^{\pi_1} \ar[d]^{w} & \rm{A} \ar[r] \ar[d]^{id.} & 0\\
0 \ar[r] & {\rm{L}} \ar[r]^{j_2}            & {\rm{G}}_2 \ar[r]^{\pi_2}              & \A \ar[r] & 0 
}
\end{xy}$$
be a commutative diagram of extensions in ${\rm{Ext}}(\A, {\rm{L}})$ and assume that $w$ is an isogeny.
\begin{enumerate}
 \item If ${{\rm{L}}} = \mathbb{G}_{a,F}^k$ with a $k \geq 1$, then $w$ is an isomorphism.
\item If ${{\rm{L}}} = \mathbb{G}_{m,F}$, then $\pm \big(\deg\,w_{\LL}\big)\cdot [{\rm{G}}_1] = [{\rm{G}}_2] \in {\rm{Ext}}(\rm{A}, {\rm{L}})$.
\end{enumerate}
\end{lemma}
\begin{proof} Note that $w_{{\rm{L}}}$ is an isogeny. If ${{\rm{L}}} = \mathbb{G}_{a,F}^k$, then $w_{{\rm{L}}}$ is an isomorphism. Writing out two exact sequences one checks that $w$ must be bijective. This is the first statement.\\
To prove Statement 2., choose a rational section $\sigma_1: A \dashrightarrow {\rm{G}}_1$. Then $\sigma_2 = v \circ \sigma_1$ defines a rational section
of $\pi_2$, and $w_{\LL}\circ f_{\sigma_1} = f_{\sigma_2}$. But $\LL = \mathbb{G}_{m,F}$, so that there exists an integer $k \neq 0$ such that $w_{\LL} = [k]_{\LL}$.
\end{proof}
\begin{lemma} \label{92} Let $i = 1,2$ and let $w: {\rm{G}}_2 \longrightarrow {\rm{G}}_1$ be an isogeny of fiber spaces associated to the extensions in ${\rm{Ext}}({{{\rm{A}}}}_{i}, L)$.
\begin{enumerate}
 \item There is a commutative diagram
$$\begin{xy} 
  \xymatrix{ 
0 \ar[r] & {{\rm{L}}} \ar[r]^{j_1} \ar[d]^{w_{\LL}} & {\rm{G}}_1 \ar[r]^{\pi_1} \ar[d]^{w} & {{{\rm{A}}}}_1 \ar[r] \ar[d]^{v} & 0\\
0 \ar[r] & {{\rm{L}}} \ar[r]^{j_2}            & {\rm{G}}_2 \ar[r]^{\pi_2}              & {{{\rm{A}}}}_2 \ar[r] & 0 
}
\end{xy}$$
such that $v$ is an isogeny.
\item If ${{\rm{L}}} = \mathbb{G}_{a,F}$, then the fiber space of the extension ${{{\rm{h}}} }^{\scr 1}(v)[{\rm{G}}_2]$ is isomorphic to ${\rm{G}}_1$ as an algebraic group.
\item If ${{\rm{L}}} = \mathbb{G}_{m,F}$, then $v^{\ast}[{\rm{G}}_2] = \pm \big(\deg\,w_{{{\rm{L}}}}\big) \cdot [{\rm{G}}_1]$.
\end{enumerate}
\end{lemma}

\begin{proof} Let $\pi: {\rm{G}}_1 \longrightarrow {{\rm{A}}}$ be the Stein factorization of $\pi_2 \circ w$. Then $\pi$ defines an extension in ${\rm{Ext}}(\A, \LL)$. However, ${\rm{G}}_1$ is an extension in one unique
way up to isomorphism. Hence, on the level of isomorphism classes we have $\pi = \pi_1$. This yields the first statement. Moreover, we get a factorization
$$\begin{xy} 
  \xymatrix{ 
{\rm{G}}_1  \ar[r]^{\overline{w}\,\,\,\,\,\,\,\,\,\,\,}\ar[d]^{\pi_1} & {\rm{Ext}}^v[{\rm{G}}_2] \ar[d]^{pr_{{{{\rm{A}}}}_1}} \ar[r] & {\rm{G}}_2 \ar[d]^{\pi_2}\\
{{{\rm{A}}}}_1 \ar[r]^{id.} & {{{\rm{A}}}}_1 \ar[r]^{v} & {{{\rm{A}}}}_2.
}
\end{xy}$$
The last two claims thus follow from the previous two lemmas.
\end{proof}
\begin{lemma} \label{93} Let 
$$\begin{xy} 
  \xymatrix{ 
0 \ar[r] & {\rm{L}} \ar[r]^{j_1} \ar[d]^{w_{{\rm{L}}}} & {\rm{G}}_1 \ar[r]^{\pi_1} \ar[d]^{w} & \A \ar[r] \ar[d]^{id.} & 0\\
0 \ar[r] & {\rm{L}} \ar[r]^{j_2}            & {\rm{G}}_2 \ar[r]^{\pi_2}              & \A \ar[r] & 0 
}
\end{xy}$$
be a commutative diagram of extensions with $w$ an isomorphism.
\begin{enumerate}
\item If ${\rm{L}} = \mathbb{G}_{a,F}$, then $[{\rm{G}}_1] \in [{\rm{G}}_2]F^{\ast}$ where we use the structure\\ of a vector space on ${{{\rm{H}}} }^{\scr 1}(A, \mathcal{O}_A)$.
\item If ${\rm{L}} = \mathbb{G}_{m,F}$, then $[{\rm{G}}_1] = \pm [{\rm{G}}_2]$.
\end{enumerate}
\end{lemma}
\begin{proof} We know that $w$ is determined by $w_{{{\rm{L}}}}$ up to isomorphism. We start by proving the first statement. Using the identification $\mathbb{G}_a(F) = F$,
$w_{{\rm{L}}}$ is multiplication with a number $c \in F^{\ast}$. Representing ${\rm{G}}_1$ by a rational symmetric factor system, we receive the first claim.\\
If ${{\rm{L}}} = \mathbb{G}_{m,F}$, then $w_{{\rm{L}}}$ is either the identity or the inverse morphism $i:{\rm{L}} \longrightarrow {\rm{L}}$. By Serre \cite[p.\,164]{Serre1}
we have $[({\rm{G}}_1, j_1, \pi_1)] = -[({\rm{G}}_1, j_1 \circ i, \pi_1)]$. This shows the second statement.
\end{proof}

\section{Conjugate extensions and real fields of definition} 
We denote by $F/K$ a Galois extension of fields. We assume that $F$ is algebraically closed and fix an element $h \in Gal(F|K)$. Let ${{\rm{A}}}$ be an abelian variety and ${\rm{L}}$ be a linear group over $F$. If $({\rm{G}}, j, \pi) $ is an extension representing an element in $\rm{Ext}(\rm{A}, {\rm{L}})$, then the \textit{h-conjugate extension} is
$({\rm{G}}^{\scriptscriptstyle h}, j^{\scriptscriptstyle h}, \pi^{\scriptscriptstyle h})$ and represents an element in ${\rm{Ext}}({{\rm{A}}}^{\scriptscriptstyle h}, {\rm{L}}^{\scriptscriptstyle h})$. \begin{lemma} \label{90} The homomorphism $h$ defines an isomorphism of groups $${\rm{Ext}}(h):{\rm{Ext}}({\rm{A}}, {{\rm{L}}}) \longrightarrow {\rm{Ext}}({{\rm{A}}}^{\scriptscriptstyle h}, {{\rm{L}}}^{\scriptscriptstyle h})$$
given by ${\rm{Ext}}(h)[\G] = [\G]^{\scr h} = [\G^{\scr h}]$.
\end{lemma}
\begin{proof} With notations as preceding Lemma C.1.1 we have
$$\big(\chi[\G]\big)^{\scr h} = \big(\mbox{class of}\,f_{\sigma}\big)^{\scr h}
= \mbox{class of}\,f_{\sigma}^{\scr h} = \chi[\G^{\scr h}].$$
The lemma follows. 
\end{proof}
We assume now that the extension $F/K$ has degree $[F:K] = 2$ and denote by $h$ the generator of $Gal(F|K)$. It is still presupposed that $F$ is algebraically closed. We treat the question when an extension of ${{\rm{A}}}$ by $\mathbb{G}_{m,F}$ is isogenous to its conjugate. This problem is, of course,
only reasonable if ${{\rm{A}}}$ is isogenous to ${{\rm{A}}}^{\scriptscriptstyle h}$. We will in fact consider the slightly easier problem when ${{\rm{A}}}$ is the extension to scalars of an abelian variety ${{\rm{A}}}'$ over $K$. 
\begin{prop} \label{98} Let $\pi: {\rm{G}} \longrightarrow \A$ be an extension of $\A = \A' \otimes_{K} F$ by $\mathbb{G}_{m,F}$.
If there exists an isogeny $w: {{{\rm{H}}} }' \otimes_K F \longrightarrow {\rm{G}}$ of algebraic groups over $F$ such that $\pi_{\ast}(w)$ is defined over $K$, then $[{\rm{G}}] + [{\rm{G}}^{\scriptscriptstyle h}]$ or $[{\rm{G}}] - [{\rm{G}}^{\scriptscriptstyle h}]$ is a torsion element in ${\rm{Ext}}({\rm{A}}, \mathbb{G}_{m,F}) = \rm{Pic}^o(\rm{A})$.
 
\end{prop}
\begin{proof} We set $v = \pi_{\ast}(w) = \pi \circ w$ and let $v': {{\rm{H}}}' \longrightarrow {\rm{A}}'$ be a homomorphism such that $v = v' \otimes_{K} F$. Letting ${\rm{H}} = {\rm{H}}'\otimes_K F$ and ${\rm{A}} = {\rm{A}}'\otimes_K F$, we get a commutative diagram of algebraic groups over $F$
$$\begin{xy} 
  \xymatrix{ 
{\rm{H}} \ar[r]^{{w}} \ar[d]^{{w}^{\scriptscriptstyle h}}\ar[rd]^{p} & {\rm{G}} \ar[rd]^{\pi} & &\\
{\rm{G}}^{\scriptscriptstyle h} \ar[rd]^{\pi^{\scriptscriptstyle h}} & \rm{B}\ar[r]^{\mu} \ar[d]^{\mu^{\scriptscriptstyle h} }& \rm{A}\\
& \rm{A} = {{\rm{A}}}^{\scriptscriptstyle h} &
}
\end{xy}$$
Here $p$ is defined to be ${\rm{St}}(v)$, the Stein factorization of $v = \mu \circ p$. The homomorphism $p$ equals $p' \otimes_K F$ where $p'$ is the Stein factorization ${\rm{St}}(v')$ of $v'$. Hence, the extension $p: \HHH \longrightarrow \rm{B}$ is defined over $K$ and coincides with its conjugate. The assumptions of the proposition imply that also $\pi_{\ast}(w)$ equals its conjugate $\big(\pi_{\ast}(w)\big)^{\scriptscriptstyle h}$. So, 
$$\mu \circ p = \pi \circ w = (\pi \circ w )^{\scriptscriptstyle h} = (\mu \circ p )^{\scriptscriptstyle h} =  \mu ^{\scriptscriptstyle h} \circ p.$$
Consequently, $\mu = \mu^{\scriptscriptstyle h}$. This and Statement (3.1.1) yield
 \begin{equation}
\mu^{\ast}[{\rm{G}}^{\scriptscriptstyle h}] = \big(\mu^{\scr h}\big)^{\ast}[{\rm{G}}^{\scriptscriptstyle h}] = \big(\mu^{\ast}\big)^{\scr h}[{\rm{G}}^{\scriptscriptstyle h}] = \big(\mu^{\ast}\big)^{\scr h}[{\rm{G}}]^{\scriptscriptstyle h} = \big(\mu^{\ast}[{\rm{G}}]\big)^{\scriptscriptstyle h}.
\end{equation}
The homomorphism $p$ defines an extension of ${\rm{B}}$ by $\mathbb{G}_{m,F}$. As seen in Lemma\,C.1.5, $p$ is determined by a point $s \in \rm{Pic}^o({\rm{B}})$ up to multiplication with $-1$. And since $p$ is defined over $K$, we have $p = p^{\scr h}$. So, $s = \pm s^{\scr h}$ by definition of the conjugate extension. Let $t \in \rm{Pic}^o(\rm{A})$ be an element corresponding to the extension $\mu^{\ast}[{\rm{G}}]$.
Lemma \ref{91} and Statement (C.2.1) imply that, for some $k \in \Z$, $t = ks$. So, $t = \pm ks = \pm ks^{\scr h} = \pm t^{\scriptscriptstyle h}$. Consequently, $t + t^{\scr h} = 0$ or $t - t^{\scr h} = 0$. So,
$$\mu^{\ast}[{\rm{G}}] - \mu^{\ast}[{\rm{G}}^{\scriptscriptstyle h}] = 0\,\,\mbox{or}\,\,\mu^{\ast}[{\rm{G}}] + \mu^{\ast}[{\rm{G}}^{\scriptscriptstyle h}] = 0.$$
Since $v$ equals $\mu \circ p$ and as $p$ is the Stein factorization of $v$, the homomorphism $\mu$ is finite. Hence, $\mu$ is an isogeny. It follows that $\mu^{\ast}$ is an isogeny. As a result, $[{\rm{G}}] - [{\rm{G}}^{\scriptscriptstyle h}]$ or $[{\rm{G}}] + [{\rm{G}}^{\scriptscriptstyle h}]$ is a torsion element in ${\rm{Pic}}^o(\rm{A})$.
\end{proof}
In the $\mathbb{G}_a$-case all fiber spaces of extensions along a punctured line $F^{\ast}[{\rm{G}}]$ in ${{{\rm{H}}} }^1({\rm{A}}, \mathcal{O}_{\rm{A}})$ are isomorphic. In the special case when ${{\rm{A}}}$ is an elliptic curve ${\rm{E}}$, which is one focus of our applications,
this means that all fiber spaces of extensions are isomorphic as soon as they are not isotrivial. Hence, in this special case it is not reasonable to ask when an extension is isogenous
to its complex conjugate.

\section{Standard uniformizations of extensions of elliptic curves by $\mathbf{\mathbb{G}_a}$ and $\mathbf{\mathbb{G}_m}$} In this section of the appendix we consider linear extensions of elliptic curves more closely. We stress standard uniformizations which are represented by elliptic Weierstra\ss\,\,functions, because many of our applications rely on them. We assume that $F$ is a subfield of $\C$ and let ${\rm{E}}$ be an elliptic curve over $F$ in Weierstra\ss\,\,form and associated to a lattice $\Lambda$. For the definition of the Weierstra\ss\,\,functions
$\wp(z) = \wp_{\Lambda}(z), \sigma(z) = \sigma_{\Lambda}(z)$ and $\zeta = \zeta_{\Lambda}(z)$ we refer to Waldschmidt \cite{Wal12}. Lemma C.1.2 and the differential of the exponential map ${\rm{exp}}_{\rm{E}} = \big[\wp_{}: \wp_{}': 1 \big]$ imply natural isomorphisms
 ${\rm{Ext}}({{\rm{E}}}, \mathbb{G}_{a,F}) \backsimeq {{{\rm{H}}} }^1({\rm{E}}, \mathcal{O}_{{\rm{E}}}) \backsimeq {\rm{Lie}}\,\rm{E} \backsimeq F$. Similarly, the same lemma together with the canonical polarization $\mathcal{L} = \mathcal{O}_{\rm{E}}(e_{\rm{E}})$ yield canonical isomorphisms ${\rm{Ext}}({\rm{E}}, \mathbb{G}_{m,F}) \backsimeq {{\rm{E}}}^{\ast}(F) \backsimeq {\rm{E}}(F)$. Based on these isomorphisms, the Weierstra\ss\,\,functions allow the following \textit{standard uniformization} of a class in ${\rm{Ext}}({\rm{E}}, \mathbb{G}_{a,F})$ resp.\,in ${\rm{Ext}}({\rm{E}}, \mathbb{G}_{m,F})$.\\
\\
Let ${\rm{L}} = \mathbb{G}_{a,F}$ and $t \in F$. Consider the functions
\begin{center}\begin{tabular}{ll}
$f_{t1}(z_1, z_2)$ & $= z_1 + t\zeta(z_2)$\\
$f_{t2}(z_1, z_2)$ & $= \wp'(z_2)f_{t1}(z_1,z_2) + 2t\wp^2(z_2)$ \\
$f_{t3}(z_1, z_2)$ & $= \wp(z_2)f_{t1}(z_1,z_2) + {\scriptstyle \frac{t}{2}}\wp'(z_2)$.
\end{tabular}
\end{center}
The extension $\pi: {\rm{G}} = {\rm{G}}_t \longrightarrow \rm{E}$ associated to $t$ has a {standard uniformization} 
$${\rm{exp}}_{\rm{G}}(z_1, z_2) = \big[\wp(z_2):\wp'(z_2): 1: f_{t3}(z_1, z_2):f_{t2}(z_1, z_2): f_{t1}(z_1, z_2) \big].$$
\\
Let ${{\rm{L}}} = \mathbb{G}_{m,F}$ and let $\omega \in \C$ represent an element in $\rm{E}(F)$. Define the functions
\begin{center}\begin{tabular}{ll}
$f_{\omega1}(z_1, z_2)$ & $= \frac{\sigma^{\scriptscriptstyle 3}(z_2-\omega)e^{3\zeta(\omega)z_2 + z_1} }{\sigma^{\scriptscriptstyle 3}(z_2)\sigma^{\scriptscriptstyle 3}(\omega)}.$\\
$f_{\omega2}(z_1, z_2)$ & $= \wp'(z_2-\omega)f_{\omega1}(z_1,z_2)$ \\
$f_{\omega 3}(z_1, z_2)$ & $= \wp(z_2-\omega)f_{\omega1}(z_1,z_2).$
\end{tabular}
\end{center}
The extension $\pi: {\rm{G}} = {\rm{G}}_{\omega} \longrightarrow {\rm{E}}$ associated to $\omega$ admits a {standard uniformization}
$${\rm{exp}}_{\rm{G}}(z_1,z_2) = \big[\wp(z_2): \wp'(z_2): 1 : f_{\omega3}(z_1,z_2): f_{\omega2}(z_1,z_2): f_{\omega1}(z_1,z_2)\big].$$
These uniformizations appear in Caveny-Tubbs \cite{CT}. The properties of elliptic Weierstra\ss\,functions imply
\begin{lemma} \label{kern} Let $\Lambda$ be a lattice in $\C$ and $\rm{E}$ the associated elliptic curve in Weierstra\ss\,form. Let $t$ and $\omega$ be complex numbers.
\begin{enumerate}
 \item The kernel of the homomorphism ${\rm{exp}}_{\G_t}: \C^{\scr 2} \longrightarrow \rm{E}(\C)$ is generated by elements $(t\eta(\lambda), \lambda)$ with $\lambda \in \Lambda$.
\item The kernel of the homomorphism ${\rm{exp}}_{\G_{\omega}}: \C^{\scr 2} \longrightarrow \rm{E}(\C)$ is generated by elements $(\zeta(\omega)\lambda - \eta(\lambda)\omega, \lambda)$ with $\lambda \in \Lambda$.
\end{enumerate}

\end{lemma}
\begin{proof} For a $\lambda \in \Lambda$ we have $\zeta(z + \lambda) = \zeta(z) + \eta(\lambda)$ and $\sigma(z + \lambda) = -e^{\eta(\lambda)(z + \lambda)}\sigma(z)$, see Waldschmidt \cite{Wal12}. The two statements of the lemma follow from these relations
by an easy verification.
\end{proof}

\chapter{Weil restrictions of simple abelian varieties} 

\section{A criterion for the existence of isogenous\\ models over $K$} We let $F/K$ be a Galois extension of fields of degree $[F: K] = 2$. We define $h \in Gal(F|K)$ to be the element generating the Galois group. The symbol $\rm{A}$ denotes a simple abelian variety over $F$ with Weil restriction $\mathcal{N} = \mathcal{N}_{F/K}(\rm{A})$ over $K$. In this appendix we shall examine when $\A$ is isogenous to an abelian variety which is definable over $K$.
\begin{lemma} \label{opo1} Let ${{\rm{A}}}$ be a simple abelian variety over $F$ and let ${{\rm{A}}}'$ be a proper abelian subvariety of $\mathcal{N}$ of positive dimension. Then ${{\rm{A}}}'$ is a simple abelian variety over $K$ and the projection $p_{\A}$ induces
and isogeny $v: {\rm{A}}' \otimes_K F \longrightarrow \A$. Moreover, conjugation with $h$ induces an isomorphisms of groups
$$\nu_h: {\rm{Hom}}\big({{\rm{A}}}' \otimes_{K}F, {\rm{A}}\big) \longrightarrow {\rm{Hom}}\big({{\rm{A}}}' \otimes_{K}F, {{\rm{A}}}^{\scriptscriptstyle h }\big), \nu_h(u) = u^{\scriptscriptstyle h}.$$
\end{lemma}
\begin{proof} Clear. \end{proof}
It follows that $\mathcal{N}$ is simple over $K$ unless there exists an isogeny $v: {\rm{A}} \longrightarrow {{\rm{A}}}^{\scriptscriptstyle h}$. An isogeny $v$ induces
an endomorphism $w: \mathcal{N} \otimes_{K}F \longrightarrow \mathcal{N} \otimes_{K}F$. To be more precise, for $(\xi_1, \xi_2) \in {\rm{A}} \times  {{\rm{A}}}^{\scriptscriptstyle h} =  \mathcal{N} \otimes_{K}F$ we define $w(\xi_1, \xi_2) = \big(v^{\scriptscriptstyle h}(\xi_2), v(\xi_1)\big)$. Recall (3.1.1) and Corollary \ref{3.3.1}, and let $(\xi, \xi^{\scriptscriptstyle h}) \in cl(\mathcal{N})$.\footnote{The symbol ``cl(-)'' refers to the set of closed points. See Corollary 3.3.3.} We calculate 
$$ w\big(\xi, \xi^{\scriptscriptstyle h}\big) = \big(v^{\scriptscriptstyle h}(\xi^{\scriptscriptstyle h}), v(\xi)\big) = \left(\big(v(\xi)\big)^{\scriptscriptstyle h}, v(\xi)\right) \in \mathcal{N}.
$$
Since the set $cl(\mathcal{N})$ of closed points in $\mathcal{N}$ is Zariski-dense, we conclude that $w = w^{\scr h}$. By Lemma \ref{cl1} the homomorphism $w$ is actually defined over $K$, that is, $w$ results from an endomorphism of $\mathcal{N}$ which will be denoted by the same letter.
\begin{lemma} \label{opo} Let $\rm{A}$ be a simple abelian variety over $F$ with ring of endomorphisms ${\rm{End}}(\rm{A}) \backsimeq \Z$. Let $v: {\rm{A}} \longrightarrow {\rm{A}}^{\scriptscriptstyle h}$ be an isogeny of minimal degree.
Then ${\rm{End}}(\mathcal{N})$ is generated over $\Z$ by the identity and the induced isogeny $w$. If $\mathcal{N}$ is simple over $K$, then ${\rm{End}}(\mathcal{N})$ is an order in a real quadratic number field. 
 
\end{lemma}
\begin{proof} Let $\mu' \in {\rm{End}}(\mathcal{N})$ and write $\mu$ for the endomorphism $\mu' \otimes_K F \in {\rm{End}}\big({\rm{A}} \times {\rm{A}}^{\scriptscriptstyle h}\big)$. Since ${\rm{End}}({\rm{A}}) \backsimeq {\rm{Hom}}({\rm{A}}, {\rm{A}}^{\scriptscriptstyle h}) \backsimeq \Z$ and as $v$ is minimal, there are integers $d_1, d_2$ such that, for all closed points $(\xi_1, \xi_2) \in {\rm{A}} \times {\rm{A}}^{\scriptscriptstyle h}$, $(p_{{\rm{A}}} \circ \mu)(\xi_1, \xi_2) = d_1\xi_1 + d_2v(\xi_2)$. As $\mu$ is defined over $K$, we have $\mu = \mu^{\scr h}$ and $p_{{\rm{A}}^{\scriptscriptstyle h}} \circ \mu = \big(p_{\rm{A}}\big)^{\scriptscriptstyle h} \circ \mu^{\scriptscriptstyle h} = \big(p_{\rm{A}} \circ \mu\big)^{\scriptscriptstyle h}$. Statement (3.1.1) implies $\big(p_{\A^{\scr h}} \circ \mu\big)(\xi_2^{\scr h}, \xi_1^{\scr h}) = d_1\xi_2^{\scr h} + d_2v^{\scr h}(\xi_1^{\scr h})$.
As a result, $\mu = d_1 \cdot id. + d_2 \cdot w$. We infer that ${\rm{End}}(\mathcal{N}) = \Z[w]$.\\
Next we observe that $v^{\scriptscriptstyle h} \circ v$ is multiplication with an integer $k$, because ${\rm{End}}(\rm{A}) \backsimeq \Z$. The next lemma implies that $k > 0$. It results that $w^{\scriptscriptstyle 2} = [k]_{\mathcal{N}}$. Hence, if $\mathcal{N}$ is simple, then ${\rm{End}}(\mathcal{N})$ is an order in the real quadratic number field
${\rm{End}}(\mathcal{N}) \otimes \Q.$
\end{proof}
It follows from the universal properties of Weil restrictions and from Lemma \ref{opo1} that $\A$ is isogenous to an abelian variety with model over $K$ if and only if $\mathcal{N}$ is not simple. In this context we have the following criterion. 
\begin{lemma}\label{Wu} Let ${{\rm{A}}}$ be a simple abelian variety over $F$. If the Weil restriction $\mathcal{N}$ is not simple over $K$, then there exists an isogeny $v: {\rm{A}} \longrightarrow {{\rm{A}}}^{\scriptscriptstyle h}$ such that $\deg\,v = k^{2\dim\,\rm{A}}$ and $v^{\scriptscriptstyle h} \circ v = [k^{\scriptscriptstyle 2}]_{\rm{A}}$
with a natural number $k > 0$.

\end{lemma}
The lemma was suggested to us by W\"ustholz in the case of elliptic curves (oral communication).
\begin{proof} Assume that there is a proper abelian subvariety ${{\rm{B}}}'$ of $\mathcal{N}$ over $K$ and let $u = p_{\A|\rm{B}}$ be the restriction to ${\rm{B}} = {\rm{B}}' \otimes_K F$ of the projection $p_{\rm{A}}$. Lemma D.1.1\,\,implies that $u$ is an isogeny. There exists an inverse isogeny $\mu: {{\rm{A}}} \longrightarrow \rm{B}$ with the property that $u \circ \mu$ is the multiplication $[k]_{{\rm{A}}}$ with a natural number $k$. 
Moreover, $\ker\,u^{\scriptscriptstyle h} = \big(ker\,u\big)^{\scriptscriptstyle h}$. Hence, 
$$k^{\scr 2\dim\,A} = \deg\,\big(u\circ \mu\big) = (\deg\,u)(\deg\,\mu) = (\deg\,u^{\scriptscriptstyle h})(\deg\,\mu).$$
It follows that $v = u^{\scriptscriptstyle h} \circ \mu$ is an isogeny of degree $\deg\,v = k^{2\dim\,\rm{A}}$. We are left to show that
$v^{\scriptscriptstyle h} \circ v$ is multiplication with $k^{\scriptscriptstyle 2}$. Now, $[k]_{\A} = u \circ \mu$, so that 
$$(\mu \circ u) \circ \mu = \mu \circ [k]_{\A} = [k]_{{\rm{B}}} \circ \mu .$$
Since $\mu$ is epi, we get $[k]_{{\rm{B}}} = \mu \circ u$. Recalling that $\rm{B}$ is defined over $K$,
$$
 [k]_{\rm{B}} = \big([k]_{\rm{B}}\big)^{\scriptscriptstyle h} = \big(\mu \circ u\big)^{\scriptscriptstyle h} = \mu^{\scriptscriptstyle h} \circ u^{\scriptscriptstyle h}.
$$
Moreover, 
$$v^{\scriptscriptstyle h} \circ v = (u \circ \mu^{\scriptscriptstyle h}) \circ (u^{\scriptscriptstyle h} \circ \mu) = u \circ (\mu^{\scriptscriptstyle h} \circ u^{\scriptscriptstyle h}) \circ \mu$$
Thus, we get 
$$v^{\scriptscriptstyle h} \circ v = u \circ [k]_{\rm{B}} \circ \mu = [k]_{\rm{A}} \circ \big(u \circ \mu\big) = [k^{\scriptscriptstyle 2}]_{\rm{A}}.$$
Everything is proved.
\end{proof}
In the special case when ${\rm{End}}({\rm{A}}) \backsimeq \Z$ the last lemma can be improved to
\begin{prop} \label{wu11} Let ${{\rm{A}}}$ be a simple abelian variety over $F$ such that ${\rm{End}}({\rm{A}}) \backsimeq \Z$. Then the following assertions are equivalent.
\begin{enumerate}
\item ${{\rm{A}}}$ is isogenous to an abelian variety with model over $K$. 
 \item The Weil restriction $\mathcal{N}$ is not simple over $K$.
\item  There exists an isogeny $v: {{\rm{A}}} \longrightarrow {{\rm{A}}}^{\scriptscriptstyle h}$ with the property that $deg\,v = k^{\scr 2\dim\,\rm{A}}$ for a natural number $k > 0$.
\item ${\rm{Hom}}({\rm{A}}, {{\rm{A}}}^{\scriptscriptstyle h}) \neq 0$ and for all $v \in {\rm{Hom}}({\rm{A}}, {{\rm{A}}}^{\scriptscriptstyle h})$ we have $deg\,v = k^{\scr 2\dim\,\rm{A}}$ with a natural number $k > 0$.
\end{enumerate}
\end{prop} 
\begin{proof}
\textit{'$\mathbf{1. \Longleftrightarrow 2.}$'} The direction ``$\Longrightarrow$'' results from the universal property of Weil restrictions. The converse holds by Lemma \ref{opo1}.\\
\textit{'$\mathbf{2. \Longleftrightarrow 3.}$'} The direction ``$\Longrightarrow$'' was proved in the previous lemma. To show the converse direction, let $v: {{\rm{A}}} \longrightarrow {{\rm{A}}}^{\scriptscriptstyle h}$ be an isogeny with the property that $deg\,v = k^{2\dim\,A}$. 
Then $v^{\scr h} \circ v$ lies in ${\rm{End}}(\rm{A}) \backsimeq \Z$ and $\deg\,(v^{\scr h} \circ v) = \deg\,v^{\scr h} \circ \deg\,v = (\deg\,v)^{\scr 2}$. Together with the previous lemma it follows that $v^{\scr h} \circ v = [k^{\scr 2}]_{\rm{A}}$. The definition of $w$ preceding Lemma \ref{opo} yields
$w^{\scriptscriptstyle 2} = [k^{\scriptscriptstyle 2}]_{\mathcal{N}}$. Thus, $\big(w - [k]_{\mathcal{N}}\big)\circ \big(w + [k]_{\mathcal{N}}\big) = [0]_{\mathcal{N}}.$
As $w \neq \big[\pm k\big]_{\mathcal{N}}$, ${\rm{End}}(\mathcal{N})$ is no division ring. Therefore
$\mathcal{N}$ is not simple. The second statement follows.\\
\textit{'$\mathbf{3. \Longleftrightarrow 4.}$'} Let $v$ be as in Statement 3.\,\,Since the $\Z$-module ${\rm{Hom}}(\A, \A^{\scr h})$ is isomorphic to ${\rm{End}}(\A) \backsimeq \Z$, there is then a generator $v_o$ of ${\rm{End}}(\A)$ and an integer $m$ such that $v = [m]_{{{\rm{A}}}^{\scriptscriptstyle h}}\circ v_{o}$. Since $[m]_{\A^{\scr h}}$ has degree $m^{2\dim\,\A}$, there exists an integer $k_1 > 0$ such that $\deg\,v = k_1^{\scr 2\dim\,{{\rm{A}}}}$ if and only if $\deg\,v_o = k_2^{\scr 2\dim\,{{\rm{A}}}}$ with an integer $k_2 > 0$. In other words, Statement 3.\,\,implies Statement 4. The converse direction is obvious. 
\end{proof}
There are examples of abelian varieties $\rm{A}$ with complex or real multiplication such that an isogeny $v$ of degree $\deg\,v \neq k^{\scr 2\dim\,{{\rm{A}}}}$ exists, although $\mathcal{N}$ is not simple.
\begin{example} If $F = \C$ and $\rm{A}$ is the elliptic curve $\rm{E}_{\sqrt{-2}} = \C/(\Z + \sqrt{-2}\Z)$, then ${\rm{A}} = {\rm{A}}^{\scriptscriptstyle h}$ and $\mathcal{N}$ is not simple over $K$. However, multiplication with $\sqrt{-2}$
induces an isogeny $v: {\rm{A}}\longrightarrow {\rm{A}}^{\scriptscriptstyle h}$ of degree $2$.
\end{example}
The reader should not confuse either the existence of an isogeny between conjugated group varieties with the existence of an isogenous group variety admitting a model over $K$.
\begin{example} Let $F = \C$, $\tau = \sqrt{3}(1+i)$ and let $\rm{A} $ be the associated elliptic curve ${\rm{E}}_{\tau} = \C/(\Z + \Z\tau)$. Then its complex conjugate is ${\rm{E}}_{-h(\tau)}$. An isogeny $v: {\rm{E}}_{\tau} \longrightarrow {\rm{E}}_{-h(\tau)}$ is induced by multiplication with $h(\tau)$.
We have $\deg\,v = \tau \cdot h(\tau) = 6$, and the last lemma implies that $\mathcal{N} = \mathcal{N}_{\C/\R}\big({\rm{E}}_{\tau}\big)$ is simple.
 
\end{example}

\section{An application: The $j$-invariant and real fields of definition} Corollary\,\ref{lp2} and Proposition\,\ref{wu11} admit the following nice application. In the statement we consider the modular $j$-function
$j: \mathbb{H} \longrightarrow \C$ on the upper half plane $\mathbb{H}$.
\begin{corollary} \label{op222} Let $\tau \in \mathbb{H}$. Then
the elliptic curve ${\rm{E}}_{\tau} = \C/(\Z + \Z\tau)$ is isogenous to a curve with model over $\R$ if and only if $Re\,\tau$ is rational or if
there exist integers $d_1, d_2$ such that $d_1$ is non-zero and $|d_1\tau + d_2|$ is rational.
\end{corollary}
\begin{proof} The assertion is clear if $\tau$ is a quadratic irrationality. We may thus assume that $\tau$ is not a quadratic irrationality and that ${\rm{E}}_{\tau}$ and ${\rm{E}}_{-h(\tau)}$ are not CM-curves.
If ${\rm{E}}_{\tau}$ is isogenous to a curve with model over $\R$, then it is isogenous to its complex conjugate ${\rm{E}}_{-h(\tau)}$. As a result, 
there exists a number $\sigma \in \mathbb{H} \cup \R$ such that $\sigma(\Z + \Z\tau)$ is a sublattice of $\Z + \Z h(\tau).$ If $\sigma$ is real, then
$\sigma$ is an integer. In this situation we find as in the proof of Corollary \ref{op22} that $Re\,\tau$ is rational. If $\sigma \in \mathbb{H}$, then $\sigma = d_1h(\tau) + d_2$ with integers $d_1, d_2$ such that $d_1 < 0$. We set $\beta = h(\sigma)$. It follows that 
$$\beta(\Z + \Z\sigma) \subset \Z + \Z h(\sigma)\,\,\mbox{and}\,\,h(\beta)\big(\Z + \Z h(\sigma)\big) \subset \Z + \Z \sigma.$$
Hence $\beta$ represents an isogeny $v: {\rm{E}}_{\sigma} \longrightarrow {\rm{E}}_{-h(\sigma)}$ and $\beta h(\beta)$ represents the endomorphism $v^{\scr h}\circ v$.
Since the curve ${\rm{E}}_{\sigma}$ is isogenous to ${\rm{E}}_{-h(\tau)}$, it has no complex multiplication either. So, $v^{\scr h}\circ v$ is multiplication with an integer and has degree $\big(\beta h(\beta)\big)^{\scr 2}$. Consequently, the isogeny 
$v$ has degree 
$$\sqrt{(\deg\,v)^{\scr 2}} = \sqrt{(\deg\,v)(\deg\,v^{\scr h})} = \sqrt{(\deg\,v \circ v^{\scr h})} = \beta h(\beta) = |\beta|^{\scriptscriptstyle 2} = |\beta|^{\scriptscriptstyle 2 \dim\,{\rm{E}}_{\sigma}}.$$
It follows from Proposition\,\ref{wu11} that $|\beta| = |d_1\tau + d_2|\in \Z$. This shows one direction (if-part).\\
For the converse direction we may suppose that $Re\,\tau$ is irrational for otherwise the claim follows from Corollary \ref{op22}. By assumption there are then integers $d_1 \neq 0$ and $d_2$ such that $|d_1\tau + d_2|$ is rational. Then there exists a further integer $d_3 \neq 0$ such that, letting $\sigma = d_3d_1\tau + d_3d_2 $ and $\beta = h(\sigma)$, we have $\sigma \in \mathbb{H}$,
$\beta(\Z + \Z\sigma) \subset \Z + \Z h(\sigma)$ and $|\beta| \in \Z$. The isogeny between
 ${\rm{E}}_{\sigma}$ and ${\rm{E}}_{-h(\sigma)} = {\rm{E}}_{\sigma}^{\scr h}$ represented by $\beta$ has degree $|\beta|^{\scriptscriptstyle 2} = |\beta|^{\scriptscriptstyle 2 \dim\,{\rm{E}}_{\sigma}}$. It follows from Proposition\,\ref{wu11} that ${\rm{E}}_{\sigma}$ is isogenous to a curve with model over $\R$. And so is ${\rm{E}}_{\tau}$.
 \end{proof}
\begin{remark} \label{op222+} If $\tau$ takes an algebraic value $j(\tau)$, then $\tau$ is either a quadratic irrationality or a transcendental number. In the latter case the numbers $Re\,\tau$ and $|d_1\tau + d_2|$ cannot both be rational unless $d_1 = 0$. For otherwise $Im\,\tau$ is algebraic, and so is $\tau = Re\,\tau + i\cdot Im\,\tau$.
 
\end{remark}

}

\bibliographystyle{amsplain}  

\small{

}

\end{document}